\documentclass{amsart}

\usepackage{tikz-cd}
\usepackage[all]{xy}

\usepackage[top=1in, bottom=1in, left=1in, right=1in]{geometry}
\usepackage{graphicx,hyperref,multicol,amsthm,amssymb,tikz,mathtools,xcolor,tikz,tikz-3dplot,ytableau,enumerate,appendix,comment}
\usetikzlibrary{decorations.pathreplacing,
                matrix,
                positioning,
                }

\newcommand{\nlcalg}{\mathcal{L}^{\mathcal{N}=2}_{\infty}}

\newcommand{\Winf}{\mathcal{W}_{\infty}}

\newcommand{\WNtwo}{\mathcal{W}^{\mathcal{N}=2}_{\infty}}
\newcommand{\Mod}[1]{\ (\mathrm{mod}\ #1)}

\numberwithin{equation}{section}

\newtheorem{theorem}{Theorem}
\numberwithin{theorem}{section}
\newtheorem{definition}[theorem]{Definition}
\newtheorem{lemma}[theorem]{Lemma}
\newtheorem{remark}[theorem]{Remark}
\newtheorem{proposition}[theorem]{Proposition}

\newtheorem{example}[theorem]{Example}
\newtheorem{corollary}[theorem]{Corollary}

\newtheorem{conjecture}[theorem]{Conjecture}

\def\ra{\rightarrow}

\newcommand{\CC}{\mathbb{C}}
\newcommand{\B}[1]{\textbf{#1}}

\newcommand{\T}[1]{\text{#1}}

\newcommand{\fr}[1]{\mathfrak{#1}}

\def\cO{{\cal O}}

\def\cA{{\cal A}}

\def\ra{\rightarrow}

\def\cA{{\mathcal A}}
\def\cB{{\mathcal B}}
\def\cC{{\mathcal C}}
\def\cD{{\mathcal D}}
\def\cE{{\mathcal E}}
\def\cF{{\mathcal F}}

\def\cH{{\mathcal H}}
\def\cI{{\mathcal I}}

\def\cK{{\mathcal K}}

\def\cN{{\mathcal N}}
\def\cO{{\mathcal O}}

\def\cS{{\mathcal S}}

\def\cV{{\mathcal V}}
\def\cW{{\mathcal W}}

\def\ga{{\mathfrak a}}

\def\gd{{\mathfrak d}}

\def\gg{{\mathfrak g}}
\def\gh{{\mathfrak h}}

\def\gl{{\mathfrak l}}

\def\gn{{\mathfrak n}}
\def\go{{\mathfrak o}}
\def\gp{{\mathfrak p}}

\def\gs{{\mathfrak s}}

\title{Universal $2$-parameter $\mathcal{N}=2$ supersymmetric $\cW_{\infty}$-algebra}
\author{Thomas Creutzig}
\address{FAU Erlangen}
\email{creutzigt@math.fau.de}
\thanks{T. Creutzig is supported by DFG project Projektnummer 551865932.}

\author{Volodymyr Kovalchuk}
\address{Department Mathematik, FAU Erlangen}
\email{v.dexter97@gmail.com}
\thanks{V. Kovalchuk is supported by the Alexander von Humboldt Foundation.}

\author{Andrew R. Linshaw}
\address{University of Denver}
\email{andrew.linshaw@du.edu}
\thanks{A. Linshaw is supported by NSF Grant DMS-2401382 and Simons Foundation Grant MPS-TSM-00007694.}

\author{Arim Song}
\address{Department Mathematik, FAU Erlangen}
\email{arimsong11@gmail.com}
\thanks{A. Song is supported by Basic Science Research Program through the National Research Foundation of Korea (NRF) funded by the Ministry of Education(RS-2024-00409689).}

\author{Uhi Rinn Suh}
\address{University of Denver and Seoul National University}
\email{uhrisu1@snu.ac.kr}
\thanks{U.R. S. is supported by NRF Grant \#2022R1C1C1008698 and  Creative-Pioneering Researchers Program by Seoul National University}

\thanks{We thank S. Nakatsuka and J. Fasquel for sharing with us that they had independently calculated the central charges of the commuting Virasoro fields inside $\cW^k(\gs\gl_{n+1|n})$.}
\begin{document}
	\maketitle

	\pagestyle{plain}

\noindent {ABSTRACT. The universal $2$-parameter vertex algebra $\cW_{\infty}$ of type $\cW(2,3,\dots)$ is a classifying object for vertex algebras of type $\cW(2,3,\dots,N)$ for some $N$; under mild hypotheses, all such vertex algebras arise as quotients of $\cW_{\infty}$. In 2017, Gaiotto and Rap\v{c}\'ak introduced a family of such vertex algebras called $Y$-algebras, and conjectured that they fall into groups of three that are mutually isomorphic. This is a common generalization of both Feigin-Frenkel duality and the coset realization of principal $\cW$-algebras in type $A$, and was proven in 2021 for the simple $Y$-algebras (i.e., one label is zero) by the first and third authors. In this paper, we extend this entire story to the $\mathcal{N}=2$ superconformal setting. First, we prove the 2013 conjecture of Gaberdiel and Candu that there exists a universal $2$-parameter vertex algebra $\cW^{\cN=2}_{\infty}$ which is an extension of the $\mathcal{N}=2$ superconformal algebra, and has four additional generators in weights $i, i + \frac{1}{2}, i + \frac{1}{2}, i+1$, for each integer $i > 1$. This admits many $1$-parameter quotients which we call $\mathcal{N}=2$ supersymmetric $Y$-algebras, and we prove the dualities among these algebras which were conjectured in 2018 by Prochazka and Rap\v{c}\'ak. A special case is the coset realization of the principal $\cW$-algebra $\cW^k(\gs\gl_{n+1|n})$ which was conjectured in 1992 by Ito. As a corollary, we obtain the strong rationality of $\cW_k(\gs\gl_{n+1|n})$ for $k = -1 + \frac{1}{n+a+1}$ for all positive integers $n,a$, and we describe its module category. This generalizes Adamovi\'c's 1999 result on $\mathcal{N}=2$ minimal models, which is the case $n=1$.}

\section{Introduction}

$\cW$-algebras are a class of vertex algebras that are a common generalization of affine vertex algebras and the Virasoro algebra. Given a Lie (super)algebra $\mathfrak{g}$ and a nilpotent element $F$ in the even part of $\mathfrak{g}$, the $\cW$-algebra $\cW^k(\mathfrak{g}, F)$ at level $k\in \mathbb{C}$ was defined by Kac, Roan and Wakimoto in \cite{KRW04}, generalizing earlier constructions of Feigin and Frenkel \cite{FF1}. The best studied case is when $F$ is a principal nilpotent; $\cW^k(\mathfrak{g},F)$ is then called a principal $\cW$-algebra and is denoted by $\cW^k(\mathfrak{g})$. Principal $\cW$-algebras satisfy {\it Feigin-Frenkel duality}, which is a vertex algebra isomorphism $\cW^k(\gg) \cong \cW^{k'}(^L \gg)$. Here $^L\gg$ is the Langlands dual Lie algebra, $h^{\vee}$,  $^L h^{\vee}$ are the dual Coxeter numbers of $\gg$, $^L\gg$, and $(k+h^{\vee})(k' + ^Lh^{\vee}) = r$ where $r$ is the lacing number \cite{FF}. For $\gg$ simply-laced, there is another duality called the {\it coset realization} which was proven \cite{ACL}. For generic values of $\ell$, we have a vertex algebra isomorphism 
	\begin{equation} \label{eq:cosetrealization} \cW^{\ell}(\gg)\cong  \text{Com}(V^{k+1}(\gg),V^k(\gg)\otimes L_1(\gg)),\qquad \ell +h^{\vee}=\frac{k+h^{\vee}}{(k+1)+ h^{\vee}},\end{equation} which descends to an isomorphism of simple vertex algebras $\cW_{\ell}(\gg)\cong \text{Com}(L_{k+1}(\gg), L_k(\gg)\otimes L_1(\gg))$ for all admissible levels $k$. This was a longstanding conjecture \cite{BBSS,FaLu,FKW92,KW89}, generalizing the Goddard-Kent-Olive (GKO) construction of the Virasoro algebra \cite{GKO}. A shorter proof had been found later \cite{CN}. When $k$ is an admissible level for $\widehat{\gg}$, $\ell$ is a nondegenerate admissible level for $\widehat{\gg}$, so that $\cW_{\ell}(\gg)$ is strongly rational (i.e., lisse and rational) by \cite{Ar1, Ar2}. There is also a coset realization for type $B$ appearing in \cite{CL3}: 
\begin{equation} \label{cosetBC} \cW^{\ell}(\gs\go_{2n+1})  \cong \text{Com}(V^k(\gs\gp_{2n}), V^k(\go\gs\gp_{1|2n})), \qquad \ell + h^{\vee}_{\gs\go_{2n+1}} =  \frac{k + h^{\vee}_{\go\gs\gp_{1|2n}}}{k + h^{\vee}_{\gs\gp_{2n}}}. 
\end{equation} 
	
	Non-principal $\cW$-algebras as well as $\cW$-superalgebras have become important in physics in recent years; see for example \cite{ArMT,CG,CH,CHR, GR,GRZ,ProI,ProII,PR}. In \cite{GR}, Gaiotto and Rap\v{c}\'ak introduced a family of $1$-parameter vertex algebras $Y_{N_1, N_2, N_3}[\psi]$ called $Y$-algebras, which are indexed by three integers $N_1, N_2, N_3\geq 0$ and $\psi \in \mathbb{C}$. They are associated to interfaces of twisted $\cN=4$ supersymmetric gauge theories with gauge groups $U(N_1)$,  $U(N_2)$, and $U(N_3)$. The interfaces satisfy a permutation symmetry which led Gaiotto and Rap\v{c}\'ak to conjecture a triality of isomorphisms of $Y$-algebras. The $Y$-algebras with one label zero are (up to a Heisenberg algebra) the affine cosets of a family of non-principal $\cW$-(super)algebras of type $A$ which are known as {\it hook-type}. For $n\geq 1$ and $m\geq 0$, define 
$$\cW^{\psi}(n,m) := \cW^k(\gs\gl_{n+m}, F_{n,1^m}),$$ where $F_{n,1^m}$ corresponds to the hook-type partition $(n,1,1,\dots, 1)$ of $n+m$, and $\psi = k + n+m$. It has affine subalgebra $V^{\psi-m-1}(\gg\gl_m)$, and we define
\begin{equation} \cC^{\psi}(n,m) := \text{Com}(V^{\psi-m-1}(\gg\gl_m), \cW^{\psi}(n,m)).\end{equation} When $n=0$, we instead define
\begin{equation}\cC^{\psi}(0,m) := \text{Com}(V^{k-1}(\gg\gl_m), V^k(\gs\gl_m) \otimes \cS(m)).\end{equation} Here $\cS(m)$ is the rank $m$ $\beta\gamma$-system, which has an action of $L_{-1}(\gg\gl_m)$.

Similarly, for $n\geq 1$, $m\geq 0$, and $m\neq n$, we define 
$$\cV^{\psi}(n,m) := \cW^k(\gs\gl_{n|m}, F_{n|1^m}),$$ where $F_{n|1^m}$ is principal in the subalgebra $\gs\gl_n$ and trivial in $\gs\gl_m$, and $\psi = k+n-m$. It has affine subalgebra $V^{-\psi-m+1}(\gg\gl_m)$, and we define
\begin{equation}\cD^{\psi}(n,m) := \text{Com}(V^{-\psi-m+1}(\gg\gl_m), \cV^{\psi}(n,m)).\end{equation}
For $m = n \geq 2$,  we define $\cD^{\psi}(n,n) := \text{Com}(V^{-\psi-n+1}(\gs\gl_n), \cV^{\psi}(n,n))^{\text{GL}_1}$, where $\cV^{\psi}(n,n) := \cW^k(\gp\gs\gl_{n|n}, F_{n|1^n})$. When $n=0$, we set
\begin{equation} \label{intro:D0m} \cD^{\psi}(0,m) := \text{Com}(V^{-k+1}(\gg\gl_m), V^{-k}(\gs\gl_m) \otimes \cE(m)).\end{equation} Here $\cE(m)$ is the rank $m$ $bc$-system, which has an action of $L_{1}(\gg\gl_m)$. Note that \eqref{intro:D0m} is isomorphic to \eqref{eq:cosetrealization} for $\gg = \gs\gl_m$, so we also regard \eqref{intro:D0m} as a GKO coset. Finally, we define $\cD^{\psi}(1,1) :=\cA(1)^{\text{GL}_1}$ where $\cA(1)$ is the rank $1$ symplectic fermion algebra. 
\begin{theorem} \label{intro:trialitytheorem} \cite{CL2} For integers $n\geq m \geq 0$, we have isomorphisms of $1$-parameter vertex algebras
\begin{equation} \label{typeAtriality} \cD^\psi(n, m)  \cong \cC^{\psi^{-1}}(n-m, m) \cong \cD^{\psi'}(m, n),\qquad  \frac{1}{\psi} +\frac{1}{\psi'} =1.\end{equation}
\end{theorem}
 The cases $\cD^\psi(n, 0)  \cong \cC^{\psi^{-1}}(n, 0)$ and $\cD^\psi(n, 0)  \cong \cD^{\psi'}(0, n)$ are just Feigin-Frenkel duality and the coset realization, respectively. The key idea in the proof of \eqref{typeAtriality} that $\cC^\psi(n, m)$ and $\cD^\psi(n, m)$ can be realized explicitly as simple $1$-parameter quotients of the universal $2$-parameter vertex algebra $\cW_{\infty}$ of type $\cW(2,3,\dots)$. This is a classifying object for vertex algebras of type $\cW(2,3,\dots, N)$ for some $N$ satisfying some mild hypotheses. It was known to physicists since the early 1990s, and was constructed rigorously by the third author in \cite{Lin}.

\subsection{$\cW$-algebras with $\mathcal{N}=2$ supersymmetry}
The $\mathcal{N}=2$ superconformal algebra $\text{Vir}^c_{\mathcal{N}=2}$ has even generators $H, L$ of weights $1$ and $2$, and odd generators $G^{\pm}$ of weight $\frac{3}{2}$, which satisfy
 \begin{equation} \label{eq:N=2}
\begin{split}
H(z)H(w) &\sim \frac{c/3}{(z-w)^2}, \qquad H(z)G^\pm(w) \sim \frac{\pm G^\pm(w)}{(z-w)}, \\ 
G^\pm(z) G^\mp(w) &\sim \frac{c/3}{(z-w)^3} \pm \frac{H(w)}{(z-w)^2} + \frac{L(w) \pm 1/2 \partial H(w)}{(z-w)}.
\end{split}
\end{equation} It is isomorphic to the principal $\cW$-algebra $\cW^k(\gs\gl_{2|1})$, where $c = -3(2k+1)$. Like the Virasoro algebra, there is a sequence of levels, namely $k= -1 + \frac{1}{3+a}$ for $a \in \mathbb{Z}_{\geq 1}$, where the central charge is $c = \frac{3(a-1)}{a+1}$, and the simple quotient $\cW_k(\gs\gl_{2|1})$ is strongly rational. The classification of irreducible modules and fusion rules for these vertex algebras were given by Adamovi\'c in \cite{Ad99}.

A natural generalization of $\text{Vir}^c_{\mathcal{N}=2}$ is the principal $\cW$-algebra $\cW^k(\gs\gl_{n+1|n})$, which is a conformal extension of $\text{Vir}^c_{\mathcal{N}=2}$ for $c = -3n (kn + k +n)$ for $k\neq -1$. It is freely generated of type
\begin{equation} \label{stronggenW:intro}
\cW\bigg( 1,2^2, 3^2,\dots, n^2, n+1; \bigg(\frac{3}{2}\bigg)^2, \bigg(\frac{5}{2}\bigg)^2,\dots, \bigg(\frac{2n+1}{2}\bigg)^2\bigg).\end{equation} This means that it has a minimal strong generating set consisting of an even field in weight $1$, two even fields in weights $2,3,\dots, n$, an even field in weight $n+1$, two odd fields in weights $\frac{3}{2}, \frac{5}{2}, \dots, \frac{2n+1}{2}$, and there are no normally ordered relations among these generators. A longstanding conjecture of Ito \cite{I1,I2} says that $\cW^k(\gs\gl_{n+1|n})$ should have the following coset realization. The rank $n$ $bc$-system $\cE(n)$ has an action of the simple affine vertex algebra $L_1(\gg\gl_n) = \cH(1) \otimes L_1(\gs\gl_n)$, where $\cH(1)$ is the rank $1$ Heisenberg algebra. We therefore have the diagonal action $V^{\ell+1}(\gg\gl_n) \rightarrow V^{\ell}(\gs\gl_{n+1}) \otimes \cE(n)$. Set 
\begin{equation} \label{intro:itocoset} \cC^{\ell}(n) = \text{Com}(V^{\ell+1}(\gg\gl_n), V^{\ell}(\gs\gl_{n+1}) \otimes \cE(n)),\end{equation} which is analogous to the GKO coset \eqref{intro:D0m}. Observe that when $\ell$ is admissible for $\hat{\gs\gl}_n$, the simple quotient $\cC_{\ell}(n)$ of $\cC^{\ell}(n)$ coincides with $\text{Com}(\tilde{V}^{\ell+1}(\gg\gl_n), L_{\ell}(\gs\gl_{n+1}) \otimes \cE(n))$, where $\tilde{V}^{\ell+1}(\gg\gl_n)$ is the image of $V^{\ell+1}(\gg\gl_n)$ in the simple quotient $L_{\ell}(\gs\gl_{n+1}) \otimes \cE(n))$. As explained in \cite{GL}, Ito's conjecture is equivalent to the following isomorphism of $1$-parameter vertex algebras
\begin{equation} \cW^k(\gs\gl_{n+1|n}) \cong \cC^{\ell}(n),\qquad (k+1)(\ell+ n+1) = 1.\end{equation} This is well known for $n=1$ since both sides are just $\text{Vir}^c_{\cN=2}$ \cite{KRW04,CL1}, and it was proven in \cite{GL} for $n=2$. By \cite[Theorem 8.1]{CL1}, Ito's conjecture would imply the isomorphism of simple quotients $\cW_k(\gs\gl_{n+1|n}) \cong \cC_{\ell}(n)$ when $\ell$ is admissible. It is apparent that $\cC_{\ell}(n)$ is unitary when $\ell$ is a positive integer, and it was shown to be strongly rational in \cite{ACL}. Therefore the isomorphism would imply the strong rationality of $\cW_k(\gs\gl_{n+1|n})$ at these levels. This is a natural generalization of the strong rationality of the $\mathcal{N}=2$ minimal models \cite{Ad99}.

\subsection{SUSY $\cW$-algebras} An important insight into the structure of $\cW^k(\gs\gl_{n+1|n})$ was obtained from the concept of a supersymmetric (SUSY) $\cW$-algebra, first introduced by Madsen and Ragoucy \cite{MRa} and developed mathematically by Molev, Ragoucy and the fifth author in \cite{MRS21}. Given a basic Lie superalgebra $\mathfrak{g}$ and an odd nilpotent element contained in an $\mathfrak{osp}_{1|2}$-subalgebra of $\mathfrak{g}$, the SUSY $\cW$-algebra is constructed via the supersymmetric analogue of the BRST complex. By its construction, any SUSY $\cW$-algebra naturally has an $\mathcal{N}=1$ supersymmetry and, moreover, it is a conformal extension of the $\mathcal{N}=1$ superconformal algebra $\text{Vir}^c_{\mathcal{N}=1}$. Recently, the fourth and fifth authors and Genra in \cite{GSS25} showed that the SUSY $\cW$-algebra $\cW^k_{\mathcal{N}=1}(\mathfrak{g})$ associated with the odd principal  nilpotent $f$ is isomorphic to the principal $\cW$-algebra $\cW^k(\mathfrak{g})$. This  implies the intrinsic SUSY structure of $\cW^k(\mathfrak{g}).$ In particular, when $\mathfrak{g}=\gs\gl_{n+1|n}$ and $f_{n+1|n}$ is its odd principal nilpotent,  $F_{n+1|n}=\frac{1}{2}[f_{n+1|n},f_{n+1|n}]$ is an even principal nilpotent and  $\cW^k(\gs\gl_{n+1|n})\cong \cW^k_{\mathcal{N}=1}(\gs\gl_{n+1|n},f_{n+1|n})$. The $\mathcal{N}=2$ supersymmetry of $\cW^k(\gs\gl_{n+1|n})$, understood as a conformal extension of $\text{Vir}^c_{\mathcal{N}=2}$ with the remaining generators arranged into $\mathcal{N}=2$ diamonds, essentially follows from the existence of two independent odd nilpotents $f_{n+1|n}$ and $\tilde{f}_{n+1|n}$ satisfying $[f_{n+1|n},f_{n+1|n}]=[\tilde{f}_{n+1|n},\tilde{f}_{n+1|n}]$ and hence 
\begin{equation} 
    \cW^k_{\mathcal{N}=1}(\gs\gl_{n+1|n},f_{n+1|n})\cong\cW^k(\gs\gl_{n+1|n})\cong \cW^k_{\cN=1}(\gs\gl_{n+1|n},\tilde{f}_{n+1|n}).
\end{equation}
The details will be explained in Section \ref{sec:W-algebra type A(n,n-1)}. In addition, all the strong generators of $\cW^k_{\mathcal{N}=1}(\gs\gl_{n+1|n})$ can be obtained by a column determinant formula of a matrix \cite{MRS21}. Using the SUSY Miura map on  $\cW_{\mathcal{N}=1}^k(\gs\gl_{n+1|n})$ which induces the free field realization \cite{{Song24free}}, one can find OPE relations between generators.

\subsection{$\mathcal{N}=2$ universal vertex algebra $\cW^{\mathcal{N}=2}_{\infty}$ and its quotients}
Based on computer calculations, Candu and Gaberdiel conjectured in \cite{CanGab} that there exists a universal $2$-parameter vertex algebra $\cW^{\mathcal{N}=2}_{\infty}$ which is freely generated of type
 \begin{equation} \label{stronggenU:intro}
\cW\bigg(1,2^2, 3^2,\dots; \bigg(\frac{3}{2}\bigg)^2, \bigg(\frac{5}{2}\bigg)^2,\dots \bigg),\end{equation}
and serves as a classifying object for vertex algebras with this generating type, in the sense that they can be obtained as quotients of $\cW^{\mathcal{N}=2}_{\infty}$. In particular, both $\cW^k(\gs\gl_{n+1|n})$ and $\cC^{\ell}(n)$ are known to have this property, so constructing $\cW^{\mathcal{N}=2}_{\infty}$ is an approach to proving Ito's conjecture. In fact, there are many other $1$-parameter vertex algebras of this kind. Ito's conjecture is one of an infinite family of dualities that have been conjectured by Prochazka and Rap{\'c}ak \cite{PR}, which we will prove in a uniform way using $\cW^{\mathcal{N}=2}_{\infty}$.

Our first main result, which is a paraphrasing of Theorems \ref{thm:induction}, \ref{Wn=2 freely generated}, and \ref{one-parameter quotients theorem}, is the following
\begin{theorem} There exists a unique $2$-parameter vertex algebra $\cW^{\mathcal{N}=2}_{\infty}$ with the following features:
\begin{enumerate}
\item It is defined over the localization of the polynomial ring $\mathbb{C}[c,\lambda]$ obtained by inverting $c$, $c-1$, and $c+3$.
\item It is freely generated of type \eqref{stronggenU:intro}, and weakly generated by the fields in weights $1,\frac{3}{2}, 2$.
\item Fields $H$, $G^{\pm}$, $L$ of weights $1, \frac{3}{2}$, $2$ generate a copy of $\text{Vir}^c_{\mathcal{N}=2}$.
\item For each $i \geq 2$, we have fields $W^{i, \bot}, W^{i,\pm}, W^{i,\top}$ of conformal weights $i, i+ \frac{1}{2}, i+ \frac{1}{2}, i+ 1$, which generate an $\mathcal{N}=2$ diamond.
\end{enumerate}
Moreover, $\cW^{\mathcal{N}=2}_{\infty}$ serves as a classifying object for vertex algebras with these properties; any vertex algebra with a strong generating set of type \eqref{stronggenU:intro} (not necessarily minimal) satisfying the above conditions, is a quotient of $\cW^{\mathcal{N}=2}_{\infty}$.
\end{theorem}
The construction is similar to the construction of $\cW_{\infty}$ given in \cite{L};  $\cW^{\mathcal{N}=2}_{\infty}$ is defined as the universal enveloping vertex algebra of a nonlinear Lie conformal superalgebra \cite{DSKI}. In \cite{CanGab}, it was also conjectured that for $i\geq 2$, the fields $\{W^{i, \bot}, W^{i,\pm}, W^{i,\top}\}$ could be assumed primary for $\text{Vir}^c_{\mathcal{N}=2}$, and we will prove this in Appendix \ref{sect:primary structures}; see Corollary \ref{univeral:primarydiamond}. 

We will also show that $\cW^{\mathcal{N}=2}_{\infty}$ is a conformal extension of $\cH \otimes \cW^+_{\infty} \otimes \cW^-_{\infty}$, where $\cW^{\pm}_{\infty}$ are copies of the universal $2$-parameter vertex algebra $\cW_{\infty}$ constructed in \cite{Lin}; see Proposition \ref{prop:base case prop}. The parameters of both $\cW^{\pm}_{\infty}$ are determined from the parameters of $\cW^{\mathcal{N}=2}_{\infty}$, and vice versa. This implies that in the terminology of \cite{SY25}, $\cW^{\mathcal{N}=2}_{\infty}$ is a {\it minimal $\mathcal{N}=2$ supersymmetric extension} of $\cW_{\infty}$; see Corollary \ref{minextension}.

\begin{remark} In a recent paper \cite{CKoL26}, three of us have conjectured the existence of universal $2$-parameter vertex algebras for {\it all} nilpotents in classical Lie types. In type $A$, let $P = (n_0^{m_0}, n_1^{m_1},\dots, n_{t}^{m_t})$ be a partition of $N = \sum_{i=0}^t n_i m_i$ consisting of $m_i$ parts of size $n_i$, where $n_0> n_1 > \cdots > n_t \geq 2$. Let $M = \{m_0,\dots, m_t\}$ be the set of {\it multiplicities}, and if $t \geq 1$, $S = \{d_1,\dots, d_t\}$ be the set of {\it height differences} $d_{i+1} = n_i - n_{i+1}$. If $t = 0$, we set $S = \emptyset$. Let $f_P \in \gs\gl_N$ be the corresponding nilpotent. The conjectural universal object $\cW^{A, S, M}_{\infty}$ depends only on $S$ and $M$, and is expected to admit all the $\cW$-algebras $\cW^k(\gs\gl_n, f_P)$ as $1$-parameter quotients. It is expected to be a conformal extension of the Heisenberg algebra $\cH(|M| - 1)$ tensored with $|M|$ commuting copies of $\cW_{\infty}$, where $|M| = \sum_{i=0}^t m_i$. The case  $\cW^{A, \emptyset, \{1\}}_{\infty}$ is just $\cW_{\infty}$. One can similarly conjecture the existence of universal $2$-parameter vertex superalgebras that govern $\cW$-superalgebras of classical Lie types. For example, a nilpotent in $\gs\gl_{N|N'}$ corresponds to a {\it pair} of partitions $P$ of $N$ and $P'$ of $N'$. The corresponding universal object is uniquely specified by two pairs of sets $S,M$ and $S',M'$, together with an overall height difference $e \geq 0$, and is expected to be a conformal extension of $\cH(|M| + |M'| - 1)$ tensored with $|M| + |M'|$ commuting copies of $\cW_{\infty}$. The case $\cW^{\mathcal{N}=2}_{\infty}$ corresponds to $S = \emptyset = S'$, $M = 1 = M'$, and $e = 1$. \end{remark}

Like $\cW_{\infty}$, $\cW^{\mathcal{N}=2}_{\infty}$ is simple for generic parameter values, but acquires a singular vector along certain curves in the parameter space which we call {\it truncation curves}. This is because the generating set for $\cW^{\mathcal{N}=2}_{\infty}$ typically truncates to a strong finite generating set for the simple quotient of $\cW^{\mathcal{N}=2}_{\infty}$ along these curves. There is an infinite family of such vertex algebras which were conjectured by Prochazka and Rap{\'c}ak \cite{PR} to be quotients of $\cW^{\mathcal{N}=2}_{\infty}$. For $n\geq 1$ and $r,s \geq 0$, we set $\psi = k + r - s + 1$, and define
\begin{equation}\label{eq:intro:Yalg}
    \cC^{\psi}_{\cN=2}(n,r|s):=
    \begin{cases}
        \textup{Com}(V^{k+1}_{\mathcal{N}=1}(\mathfrak{gl}_{r|s}),\mathcal{W}^k_{\mathcal{N}=1}(\gs\gl_{n+r+1|n+s},f_{n+1|n})) &r\neq s-1,\\
        \textup{Com}(V^{k+1}_{\mathcal{N}=1}(\mathfrak{sl}_{r|r+1}),\mathcal{W}^k_{\mathcal{N}=1}(\gp\gs\gl_{n+r+1|n+r+1},f_{n+1|n}))^{U(1)} &r= s-1,
    \end{cases}
\end{equation}
For $n=0$, we interpret $f_{1|0}$ as the zero (odd) nilpotent. Then $\cW^k_{\cN=1}(\gs\gl_{r+1|s},f_{1|0}) = V^k_{\cN=1}(\fr{sl}_{r+1|s})$ for $r\neq s-1$, and $\cW^k_{\cN=1}(\gp\gs\gl_{r+1|r+1},f_{1|0}) = V^k_{\cN=1}(\fr{psl}_{r+1|r+1})$, so that \eqref{eq:intro:Yalg} becomes 
\begin{equation}\label{eq:introparafermions}
    \cC_{\cN=2}^{\psi}(0,r|s):=\begin{cases}
        \T{Com}(V^{k+1}_{\cN=1}(\fr{gl}_{r|s}),V^{k}_{\cN=1}(\fr{sl}_{r+1|s})) & r\neq s-1,\\
        \T{Com}(V^{k+1}_{\cN=1}(\fr{sl}_{r|r+1}),V^{k}_{\cN=1}(\fr{psl}_{r+1|r+1}))^{U(1)} & r= s-1.
        \end{cases}
\end{equation}
The right-hand sides of \eqref{eq:intro:Yalg} and \eqref{eq:introparafermions} can also be expressed in terms of non-SUSY affine and $\mathcal{W}$-algebras, using the isomorphism $\mathcal{W}^k_{\mathcal{N}=1}(\gg,f)\cong \mathcal{W}^k(\gg,F)\otimes \mathcal{F}(\gg^f_0)$, where $[f,f]=2F$ and $\mathcal{F}(\gg^f_0)$ is the $bc\beta\gamma$-system on the parity reversed space of $\gg^f_0$ \cite{LSS}. For $n\geq 1$ we have
$$\cC^{\psi}_{\cN=2}(n,r|s)\cong
\begin{cases}
    \textup{Com}(V^{{k+1}}(\mathfrak{gl}_{r|s}),\mathcal{W}^k(\gs\gl_{n+r+1|n+s},F_{n+1|n}))& r \neq s-1,\\
    \textup{Com}(V^{{k+1}}(\mathfrak{sl}_{r|r+1}),\mathcal{W}^k(\fr{psl}_{n+r+1|n+r+1},F_{n+1|n}))^{U(1)}& r = s-1,
\end{cases}$$
and for $n=0$ we have 
$$
 \cC_{\cN=2}^{\psi}(0,r|s)\cong
 \begin{cases}
     \T{Com}(V^{k+1}(\fr{gl}_{r|s}),V^{k}(\fr{sl}_{r+1|s})\otimes \cE(r) \otimes \cS(s))& r\neq s-1\\
     \T{Com}(V^{k+1}(\fr{sl}_{r|r+1}),V^{k}(\fr{psl}_{r+1|r+1})\otimes \cE(r) \otimes \cS(r+1))^{U(1)} & r=s-1.
 \end{cases}$$

We will prove that all the vertex algebras $\cC^{\psi}_{\cN=2}(n,r|s)$ arise as $1$-parameter quotients of $\cW^{\mathcal{N}=2}_{\infty}$, and we will compute their explicit truncation curves; see Theorem \ref{trunc:main}. In addition, they are conformal extensions of $\cH \otimes \cA^+ \otimes \cA^-$, where $\cA^{\pm}$ are $1$-parameter quotients of $\cW^{\pm}_{\infty}$, respectively. In particular, $\cC^{\psi}_{\cN=2}(n,r|s)$ is a minimal $\mathcal{N}=2$ supersymmetric extensions of either $\cA^+$ or $\cA^-$. 

The most important cases are when either $r = 0$ or $s = 0$, since in these cases $\cC^{\psi}_{\cN=2}(n,r|s)$ is simple as a $1$-parameter vertex algebra. We use the notation
\begin{equation}\label{intro:cnmdnm}
\cE^{\psi}_{\mathcal{N}=2}(n,r) :=\cC^{\psi}_{\cN=2}(n,r|0),\quad \cD^{\psi}_{\mathcal{N}=2}(n,s):=\cC^{\psi}_{\mathcal{N}=2}(n,0|s+1).
\end{equation}
Using the fact that a simple, $1$-parameter quotient of $\cW^{\mathcal{N}=2}_{\infty}$ is uniquely determined by its truncation curve, which in turn is determined by the truncation curves of $\cA^{\pm}$ and vice versa, the triality theorem of \cite{CL2} implies the following result which was conjectured Prochazka and Rap{\'c}ak \cite{PR}.

\begin{theorem} \label{intro:duality} (Theorem \ref{N=2:duality}) For all integers $n,m \geq 0$, we have isomorphisms of $1$-parameter vertex superalgebras
\begin{equation} \cD^{\psi}_{\mathcal{N}=2}(n,m) \cong \cD^{\psi^{-1}}_{\mathcal{N}=2}(m,n),\qquad \cE^{\psi}_{\mathcal{N}=2}(n,m) \cong \cE^{\psi^{-1}}_{\mathcal{N}=2}(m,n). \end{equation}
\end{theorem} 
The case $\cE^{\psi}_{\mathcal{N}=2}(n,0) \cong \cE^{\psi^{-1}}_{\mathcal{N}=2}(0,n)$ is just Ito's conjecture. The case $\cD^{\psi}_{\mathcal{N}=2}(n,0) \cong \cD^{\psi^{-1}}_{\mathcal{N}=2}(0,n)$ is a similar coset realization of $\cW^{k}(\gp\gs\gl_{n+1|n+1}, F_{n+1|n})^{U(1)}$. It can be enhanced to a coset realization of $\cW^{k}(\gp\gs\gl_{n+1|n+1}, F_{n+1|n})$; see Corollary \ref{GKO:pslnn}. In the case $n=1$, this is a coset construction of the small $\mathcal{N}=4$ superconformal vertex algebra.

By \cite[Corollary 14.1]{ACL}, $\text{Com}(L_{\ell+1}(\mathfrak{gl}_{n}), L_{\ell}(\mathfrak{sl}_{n+1}) \otimes \cE(n))$ is strongly rational for $\ell \in \mathbb{N}$. Combining this with the isomorphism of simple quotients $\cE_{\mathcal{N}=2,\psi}(n,0) \cong \cE_{\mathcal{N}=2,\psi^{-1}}(0,n)$, we obtain
\begin{theorem} \label{intro:rationality} (Theorem \ref{thm:rational}) For $\ell \in \mathbb{N}$ and $k = -1 + \frac{1}{n+\ell+1}$, the simple quotient $\cW_k(\gs\gl_{n+1|n})$ is strongly rational. Moreover, Theorem \ref{thm:fusionrulesW} gives a complete list of inequivalent simple modules, including Ramond twisted modules, as well as their fusion rules.
\end{theorem}
This generalizes the strong rationality of the $\mathcal{N}=2$ minimal models $\text{Vir}_{\mathcal{N}=2,c}$ for $c = \frac{3(a-1)}{a+1}$ and $a\in \mathbb{N}$, which was first proved by Adamovi\'c in \cite{Ad99}. 
The classification of simple modules as well as the computation of fusion rules is done via a version of level-rank duality and
this requires us to extend the {\it mirror equivalence theorem} for a pair of rational vertex algebras  $U,V$ \cite{Lin:2016hsa,McRae:2021urf}, to the case when $U,V$ can be vertex superalgebras; see Section \ref{sec:mirror} for this generalization.

In Appendices \ref{appendix:stronggenerators} and \ref{appendix:largerstructures} we extend some results of the first and third authors on cosets of affine vertex algebras $V^k(\gg)$ inside larger vertex algebras \cite{CL1} for reductive Lie algebras $\gg$, to the case when $\gg$ is a basic Lie superalgebra or $\gg\gl_{n|n}$. These results are needed in our description of $\cC^{\psi}_{\cN=2}(n,r|s)$, and are of independent interest.

\subsection{Applications}

Very much like the $\cW_\infty$-algebra and its even spin cousin have many consequences for the structure and representation theory of $\cW$-algebras, we expect to also find a wide variety of applications of our construction of $\cW^{\cN=2}_{\infty}$.
For example, using a convolution operation the dualities of Feigin-Frenkel type between $Y$-algebras were upgraded to a functor between the representation categories of the $\cW$-algebras whose cosets are isomorphic \cite{CLNS}. A similar result should work as well in the $\cN=2$ supersymmetric situation. In addition, by analogy with \cite[Conjecture 1.1]{CKoL26}, we expect that up to a suitable completion and localization, the mode algebra of $\cW^{\cN=2}_{\infty}$ is isomorphic to the affine super Yangian of $\widehat{sl}_{n+1|n}$ \cite{U1,U2}, and the Zhu algebra of $\cW^{\cN=2}_{\infty}$ is isomorphic to the finite super Yangian of $\gs\gl_{n+1|n}$. This would imply that there are surjective maps from this super affine Yangian to the mode algebras of all the vertex algebras $\cC^{\psi}_{\cN=2}(n,r|s)$. This is analogous to the realization of finite $\cW$-superalgebras as truncations of super Yangians \cite{BR}.

Three major topics of current interest are conformal embeddings \cite{AACLMMP, AMP, AKMPP}, the reduction in stages conjecture \cite{GJ, GJ2}, and the iterated reduction conjecture \cite{CDM}. It is work in progress to prove the reduction in stages and iterated reduction conjecture for $\cW^k_{\mathcal{N}=1}(\gs\gl_{n+1+r|n+s}, f_{n+1|n})$, as well as to find conformal embeddings of $\cW$-algebras into $\cW$-superalgebras.

\subsection{Organization} In Sections \ref{sec:SUSYVOA}, \ref{sect:W-alg}, and \ref{sect: N=2}, we recall some basic facts about vertex algebras, SUSY vertex algebras, and the $\cN=2$ superconformal algebra $\text{Vir}^c_{\cN=2}$ and its representation theory. In Section \ref{sec:W-algebra type A(n,n-1)}, we discuss basic properties of $\cW^k(\gs\gl_{n+1|n})$, and in Section  \ref{sect:supersymmetricY} we introduce the $\cN=2$ $Y$-algebras $\cC^{\psi}_{\cN=2}(n,r|s)$, which generalize $\cW^k(\gs\gl_{n+1|n})$. In Section \ref{sect:main}, we construct the universal $2$-parameter vertex algebra $\cW^{\cN=2}_{\infty}$ of type \eqref{stronggenU:intro}. In Section \ref{sec:one param q}, we prove that the $\cN=2$ $Y$-algebras all arise as $1$-parameter quotients of $\cW^{\cN=2}_{\infty}$ and we give their explicit truncation curves. In Section \ref{sect:FFduality}, we prove Theorem \ref{intro:duality}, and in Section \ref{level-rankrationality} we study the module category of the rational vertex superalgebras appearing in Theorem \ref{intro:rationality}. Finally, we include three appendices. In Appendix \ref{sect:primary structures}, we prove that the generators of $\cW^{\cN=2}_{\infty}$ can be organized into {\it primary diamonds}, proving the original conjecture of Candu and Gaberdiel \cite{CanGab}. Finally, in Appendices \ref{appendix:stronggenerators} and \ref{appendix:largerstructures} we extend the results of \cite{CL1,CL2} on cosets of affine vertex algebras inside larger structures, to cosets of affine vertex superalgebras.

\section{Preliminaries} \label{sec:SUSYVOA}

\subsection{Important identities} \label{sec:VOA} We will assume that the reader is familiar with vertex algebras, and we use the same notation as the paper \cite{CL3} of the first and third authors. In particular, we denote the coeffecients of the fields by $a(z)=\sum_{j\in \mathbb{Z}}a_{(j)}z^{-j-1}$ and let $\partial$ be the translation operator. We will make use of the following well-known identities which hold in any vertex algebra $\cA$.
\begin{equation}\label{conformal identity}
			(\partial a)_{(r)}b=-ra_{(r-1)}b,\quad r\in\mathbb{Z}.
		\end{equation}
		\begin{equation}\label{skew-symmetry}
			a_{(r)}b=(-1)^{p(a)p(b)+r+1}b_{(r)}a + \sum_{i = 1}^{\infty}\frac{(-1)^{p(a)p(b)+r+i+1}}{i!}\partial^i \left(b_{(r+i)}a\right),\quad r\in\mathbb{Z}.
		\end{equation}
		\begin{equation}\label{quasi-associativity}
			:a (:bc:):\ =\ :(:ab:)c: +\sum_{i=0}^{\infty}\frac{1}{(i+1)!}\left(:\!\partial^{i+1}(a) (b_{(i)}c)\!:+(-1)^{p(a)p(b)}:\!\partial^{i+1}(b) (a_{(i)}c)\!:\right).
		\end{equation}
		\begin{equation}\label{quasi-derivation}
			a_{(r)}:bc:\ =\ :(a_{(r)}b)c:+(-1)^{p(a)p(b)}: b(a_{(r)}c):+\sum_{i=1}^r\binom{r}{i}(a_{(r-i)}b)_{(i-1)}c, \quad r\geq 0.
		\end{equation}
		\begin{equation}\label{Jacobi}
			a_{(r)}(b_{(s)}c) = (-1)^{p(a)p(b)}b_{(s)}(a_{(r)}c) + \sum_{i=0}^r \binom{r}{i}(a_{(i)}b) _{(r+s-i)}c, \quad r,s\geq 0.
		\end{equation}
		Identities \eqref{Jacobi} are known as {\it Jacobi identities}, and we often denote them using the shorthand $J_{r,s}(a,b,c)$. We use $J(a,b,c)$ to denote the set of all Jacobi identities $\{J_{r,s}(a,b,c)|r,s\geq 0\}$.
	
	A vertex algebra $\cA$ is of type $\cW(a_1^{d_1}, a_2^{d_2},\dots; b_1^{e_1}, b_2^{e_2},\dots)$ if it has a minimal strong generating set consisting of $d_i$ even fields in weight $a_i$ and $e_i$ odd fields in weight $b_i$, for $i = 1,2,\dots$. We refer to \cite{CL3} for notation and basic properties of affine vertex algebras $V^k(\gg)$ and $\cW$-algebras $\cW^k(\gg,F)$, where $\gg$ is a simple, finite-dimensional Lie (super)algebra, and $F$ is a nilpotent element in the even part of $\gg$.

\subsection{SUSY vertex algebras}

Recall that an odd endomorphism $D$ on a vertex algebra $V$ is called an \textit{odd derivation} if it satisfies
\begin{equation} \label{eq: def of supersymmetry}
    [D,a_{(j)}]=(Da)_{(j)}
\end{equation}
for any $a\in V$ and $j\in \mathbb{Z}$. When $V$ has $n$ supercommuting odd derivations $D_1, \dots, D_n$ satisfying $D_i^2=\partial$, we say that \textit{$V$ has $n$ supersymmetries.} We call a vertex algebra equipped with $n$ supersymmetries an \textit{$\mathcal{N}=n$ supersymmetric (SUSY) vertex algebra}. In particular, when $\mathcal{N}=1$, we often omit $\mathcal{N}=1$ and refer to it simply as a SUSY vertex algebra.

\begin{example} \label{ex:SUSY affine}
Let $\gg$ be a finite basic simple Lie superalgebra with a supersymmetric invariant bilinear form $( \, | \, )$ normalized by $(\theta|\theta)=2$ for the even highest root $\theta$ of $\gg.$
Let $\bar{\gg}$ be the parity reversed superspace of $\gg$. The SUSY affine vertex algebra $V^k_{\mathcal{N}=1}(\gg)$ is freely generated by a basis $\{a_i\}_{i\in I}$ of $\gg$ and $\{\bar{a}_i\}_{i\in I}$, where the OPE relations between the generators are 
\begin{equation}\label{eq:SUSY affine, lambda}
\begin{aligned}
       &  \bar{a}_i(z)\bar{a}_j(w)\sim (k+h^\vee) (a_i |a_j)(z-w)^{-1} , \ \ \bar{a}_i(z) a_j(w)\sim(-1)^{p(a_i)p(a_j)}\overline{[a_i,a_j]}(w)(z-w)^{-1}\, , \\
       &  a_i(z)  a_j(w) \sim (-1)^{p(a_i)p(a_j)} \Big(\, [a_i,a_j](w)(z-w)^{-1}+ (k+h^\vee)\, (a_i|a_j)(z-w)^{-2}\, \Big).
\end{aligned}
\end{equation}
This is a SUSY vertex algebra with a supersymmetry $D$ defined by $D(\bar{a})=a$ for any $a\in \gg$. Note that 
$V^k_{\mathcal{N}=1}(\gg)\cong V^k(\gg) \otimes \mathcal{F}(\gg)$, where $\mathcal{F}(\gg)$ is the free field vertex subalgebra of $V^k_{\mathcal{N}=1}(\gg)$ generated by $\bar{\gg}.$
\end{example}

\subsection{$\cN=1$ superconformal algebras}
In this section, we recall the notion of $\mathcal{N}=1$ and $\mathcal{N}=2$ superconformality introduced in \cite{HK07}.
Recall that a vertex algebra $V$ is called \textit{conformal} if it has an even element $L$ such that
\begin{equation*}
    L(z)L(w)\sim \partial L(w)(z-w)^{-1}+2L(w)(z-w)^{-2}+\frac{c}{2}(z-w)^{-4}, \quad L_{(0)}=\partial,
\end{equation*}
and $V$ is diagonalizable with the $L_{(1)}$-action, where the eigenvalues are bounded below. In this case, $L$ is called a conformal vector of central charge $c\in\CC$. Generalizing the notion, one can define superconformal vectors inside $\mathcal{N}=1$ or $2$ SUSY vertex algebra.

Consider a SUSY vertex algebra $V$ with a supersymmetry $D$. We say $V$ is ($\mathcal{N}=1$) \textit{superconformal} if it has an odd element $G$, such that (i) $L:=\frac{1}{2}DG$ is a conformal vector of central charge $c$,  (ii) $G_{(0)}=D$, and (iii) $L$ and $G$ generate the Neveu-Schwarz vertex algebra $\textup{Vir}_{\mathcal{N}=1}^c$ of central charge $c$, that is,
        \begin{equation} \label{eq: N=1 superconformal conditions}
            L(z)G(w)\sim 
            \partial G(w)(z-w)^{-1}+\frac{3}{2}G(w)(z-w)^{-2}, \quad G(z)G(w)\sim 2L(w) (z-w)^{-1}+\frac{2}{3}c\, (z-w)^{-3}.
        \end{equation}
We call such $G$ an ($\mathcal{N}=1$) superconformal vector of central charge $c$. Conversely, in a conformal vertex algebra, one can find a supersymmetry using an odd $G$ satisfying \eqref{eq: N=1 superconformal conditions}.
\begin{lemma} \label{lem: N=1 superconformal conditions}
    Let $V$ be a conformal vertex algebra with a conformal vector $L$ and an odd element $G$ satisfying \eqref{eq: N=1 superconformal conditions}. Then $V$ is $\mathcal{N}=1$ superconformal with a superconformal vector $G$ and $D=G_{(0)}$.
\end{lemma}
\begin{proof}
    It is enough to show that $D:=G_{(0)}$ is an odd endomorphism satisfying $D^2=\partial$. By the second equality in \eqref{eq: N=1 superconformal conditions}, we have $
        [D,D]=[G_{(0)},G_{(0)}]=(G_{(0)}G)_{(0)}=(2L)_{(0)}=2\partial$,
    which proves the desired statement.
\end{proof}

	\subsection{Free field algebras}
	A {\it free field algebra} is a vertex superalgebra $\cV$ with weight grading 
	\[\cV = \bigoplus_{d\in \frac{1}{2}\mathbb{Z}_{\geq 0}}\cV[d],\quad \cV[0] = \mathbb{C}\B{1},\] 
	with strong generators $\{X^i | i\in I\}$ satisfying OPE relations
	\[X^i(z)X^j(w)\sim a_{i,j} (z-w)^{-\Delta(X^i)-\Delta(X^j)}, \quad a_{i,j}\in\mathbb{C}, \quad a_{i,j} =0 \T{ if }\Delta(X^i)+\Delta(X^j)\not\in\mathbb{Z}.\]
	Note $\cV$ is not assumed to have a conformal structure. In \cite{CL3}, the first and third authors introduced the following families of free field algebras.
	\begin{enumerate} 
	\item Even algebras of orthogonal type $\cO_{\text{ev}}(n,k)$ for $n\geq 1 $ and even $k \geq 2$. When $k =2$, $\cO_{\text{ev}}(n,2)$ is just the rank $n$ Heisenberg algebra $\cH(n)$.
	\item Odd algebras of orthogonal type $\cO_{\text{odd}}(n,k)$ for $n\geq 1 $ and odd $k \geq 1$. When $k = 1$, $\cO_{\text{odd}}(n,1)$ is just the rank $n$ free fermion algebra $\cF(n)$.
	\item Even algebras of symplectic type $\cS_{\text{ev}}(n,k)$ for $n\geq 1$ and odd $k \geq 1$. When  $k = 1$, $\cS_{\text{odd}}(n,1)$ is just the rank $n$ $\beta\gamma$-system $\cS(n)$.
	\item Odd algebras of symplectic type $\cS_{\text{odd}}(n,k)$ for $n\geq 1$ and even $k \geq 2$. When  $k = 2$, $\cS_{\text{ev}}(n,2)$ is just the rank $n$ symplectic fermion algebra $\cA(n)$.
	\end{enumerate}
	
	We refer the reader to \cite{CL3} for the construction and key properties of these algebras. The main application of them in this paper is that if $\gg$ is a simple Lie superalgebra with a nondegenerate, supersymmetric bilinear form, $\cW^k(\gg,F)$ admits a large level limit $\cW^{\text{free}}(\gg,F)$ which is a tensor product of free field algebras of these kinds. This is useful for deducing the generating type of orbifolds and affine cosets of $\cW^k(\gg,F)$.

\subsection{Vertex algebras over rings}
We will use the same notation for VOAs over a finitely generated commutative $\mathbb{C}$-algebra $R$ as in \cite[Section 3]{Lin}. Let $\cV$ be a VOA over $R$ with weight grading 
$\cV = \bigoplus_{d \in \frac{1}{2} \mathbb{Z}_{\geq 0}} \cV[d]$ where $\cV[0] \cong R$.
All ideals $\cI \subseteq \cV$ will be graded by weight; this means that $\cI = \bigoplus_{d \in \frac{1}{2} \mathbb{Z}_{\geq 0}} \cI[d]$ where $\cI[d] = \cI \cap \cV[d]$. 
We will call $\cV$ {\it simple} if there are no proper graded ideals $\cI$ such that $\cI[0] = \{0\}$. If $R$ is a field, this is the same as the usual notion of simplicity. 
If $I\subseteq R$ is an ideal, $I$ is a subset of $\cV[0] \cong R$, and we denote by $I \cdot \cV$ the VOA generated by $I$. Then $\cV^I = \cV / (I \cdot \cV)$ is a VOA over $R/I$. Even if $\cV$ is simple as a vertex algebra over $R$, $\cV^I$ need not be simple as a vertex algebra over $R/I$. 

Suppose that $R$ is the coordinate ring $\cO(X)$ of a variety $X$, and that $\cV$ is simple. Let $I\subseteq R$ be an ideal such that $\cV^I$ is {\it not} simple, i.e., $\cV^I$ has maximal proper graded ideal $\mathcal{I}$ with $\mathcal{I}[0] = \{0\}$. Then the quotient 
$$\cV_{I} = \cV^I / \cI$$ is a simple VOA over $R/I$. Letting $Y\subseteq X$ be the closed subvariety corresponding to $I$, we can regard $\cV_I$ as a simple VOA over $Y$. 
For ideals $I_1, I_2$ with this property, we have the corresponding simple VOAs  $\cV_{I_1} = \mathcal{V}^{I_1} / \mathcal{I}_1$ and $\mathcal{V}_{I_2} = \mathcal{V}^{I_2} / \mathcal{I}_2$ over $R/I_1$ and $R / I_2$, respectively. Let $Y_1, Y_2 \subseteq X$ be the closed subvarieties corresponding to $I_1, I_2$, and let $p \in Y_1 \cap Y_2$. Let $\mathcal{V}^p_{I_1}$ and $ \mathcal{V}^p_{I_2}$ be the VOAs over $\mathbb{C}$ obtained by evaluating at $p$. As above, $\mathcal{V}^p_{I_1}$ and $ \mathcal{V}^p_{I_2}$ need not be simple, and we denote their simple quotients by $\mathcal{V}_{I_1,p}$ and $\mathcal{V}_{I_2,p}$. Then $p$ corresponds to a maximal ideal $M_p \subseteq R$ containing both $I_1$ and $I_2$, so we have isomorphisms
\begin{equation}\label{tautological}  \mathcal{V}_{I_1,p} \cong \mathcal{V}_M \cong \mathcal{V}_{I_2,p}.\end{equation} In our main example $\cV = \cW^{\mathcal{N}=2}_{\infty}$, $X$ is the complement in $\mathbb{C}^2$ with coordinates $c,\lambda$, of the curves $c = 0$, $c-1=0$, and $c+3=0$.

\section{$\cW$-algebras and SUSY $\cW$-algebras} \label{sect:W-alg}

Throughout this section, we 
let $\gg$ be a finite dimensional basic simple Lie superalgebra with an $\mathfrak{sl}_2$ triple $(E,2x,F)$ and a nondegenerate supersymmetric invariant bilinear form $( \, |\, )$. We assume that the bilinear form $(\, |\, )$ is normalized by $(\theta|\theta)=2$ for an even highest root $\theta$ of $\gg$. The  $\frac{1}{2}{\mathbb{Z}}$-grading $\gg=\bigoplus_{i\in \frac{1}{2}{\mathbb{Z}}}\gg_i$ given by the eigenvalue with respect to $\text{ad}\, x$ and denote
$ \mathfrak{n}_+:=\bigoplus_{i>0}\gg_i$, $\mathfrak{n}_-:=\bigoplus_{i<0}\gg_i$ and  $\mathfrak{p}:=\bigoplus_{i\leq 0}\gg_i.$ We recall the associated $\cW$-algebras and SUSY $\cW$-algebras introduced in \cite{KRW04} and \cite{MRS21}, respectively.

\subsection{$\mathcal{W}$-algebras} \label{subsec:W-alg} 
Let $V^k(\gg)$ be the affine vertex algebra of level $k$ and $\mathcal{F}(\mathfrak{A})$ be the charged free fermion vertex algebra freely generated by a basis of  $\mathfrak{A}:=\phi_{\bar{\mathfrak{n}}_+}\oplus \phi^{\bar{\mathfrak{n}}_-}$. Here the parity of $\phi_{\bar{a}}$ (resp. $\phi^{\bar{b}}$) is opposite to the parity of $a$ (resp. $b$) and the OPE relation is given by 
 $ \phi_{\bar{a}}(z)\phi^{\bar{b}}(w)\sim  (a|b)(z-w)^{-1}$ and $\phi^{\bar{b}}(z)\phi_{\bar{a}}(w)\sim  (-1)^{p(a)}(a|b)(z-w)^{-1}$
    for $a\in \mathfrak{n}_+$ and $b\in \mathfrak{n}_-$. Let $\Phi_{1/2}$ be the neutral free fermion vertex algebra freely generated by a basis of $\{\Phi_{m}|m\in \gg_{1/2}\}$, where the parity of $\Phi_a$ is the same as the parity of $a$ and the OPE relation is $ \Phi_{m_1}(z)\Phi_{m_2}(w)\sim  (F|[m_1,m_2])(z-w)^{-1}.$

The BRST complex for the Lie superalgebra $\gg$ is $\mathcal{C}^k(\gg,F):= V^k(\gg)\otimes \mathcal{F}(\mathfrak{A}) \otimes \Phi_{1/2}$. Let
$\{u_\alpha|\alpha\in S\}$ and $\{u^\alpha|\alpha\in S\}$ be bases of $\mathfrak{n}_+$ and  $\mathfrak{n}_-$, respectively, satisfying $(u^\alpha|u_\beta)=\delta_{\alpha,\beta}$ and let 
$\{u_\alpha| \alpha \in S_{1/2}\subset S\}$ be the basis of $\gg_{1/2}.$
Consider the odd element 
\begin{equation}\label{eq:diff_W}
    d=\sum_{\alpha\in S}(-1)^{p(\alpha)}:\!\phi^{\bar{\alpha}}(u_\alpha+(F|u_{\alpha}))\!:+\sum_{\alpha\in S_{1/2}} :\!\phi^{\bar{\alpha}} \Phi_\alpha\!:+\frac{1}{2}\sum_{\alpha,\beta\in S} (-1)^{p(\alpha)}:\!\phi^{\bar{\alpha}} \phi^{\bar{\beta}} \phi_{\overline{[u_\alpha,u_\beta]}}\!:\in \mathcal{C}^k(\gg,F),
\end{equation}
where $p(\alpha):=p(u_\alpha),$ $\phi^{\bar{\alpha}}:=\phi^{\bar{u}^\alpha}$, $\phi_{\bar{\alpha}}:=\phi_{\bar{u}_\alpha}$ and $\Phi_{\alpha}=\Phi_{u_\alpha}.$ Then the operator $d_{(0)}$ on $\mathcal{C}^k(\gg,F)$ is an odd differential on $\mathcal{C}^k(\gg,f)$ and the associated $\cW$-algebra of level $k\in \CC$ is defined by $\cW^k(\gg,F):= H(\mathcal{C}^k(\gg,F),d_{(0)})$. 

Let $V^{\psi_k}(\mathfrak{p})$ be the vertex algebra freely generated by a basis of $\mathfrak{p}$, where 
\begin{equation}
    a(z)b(w) \sim \psi_k(a|b) (z-w)^{-2} +  [a,b](w)(z-w)^{-1}
\end{equation}
for $a,b\in \mathfrak{p}$ and $\psi_k(a|b)= k(a|b)+\text{str}_{\mathfrak{n}}(\text{ad}\, a \,\text{ad}\, b).$ Then, the $\mathcal{W}$-algebra $\cW^k(\gg,F)$ can be regarded as a vertex subalgebra of $ V^{\psi_k}(\mathfrak{p})\otimes \Phi_{1/2}$ by the map induced from the injective homomorphism:
\begin{equation}
    V^{\psi_k}(\mathfrak{\mathfrak{p}})\otimes \Phi_{1/2} \to \mathcal{C}^k(\gg,f), \quad a\mapsto J_{a}, \ \Phi_\alpha \mapsto \Phi_{\alpha}.
\end{equation}
Here the element $J_a$ in $\mathcal{C}^k(\gg,f)$ is defined by
\begin{equation}
    J_a = a+ \sum_{\alpha\in S} :\phi^{\bar{\alpha}} \phi_{\overline{\pi_+([u_\alpha,a])}}: \in \mathcal{C}^k(\gg,f),
\end{equation}
for the canonical projection map $\pi_+:\gg \to \mathfrak{n}_+$. When $k\neq -h^\vee$, the $\mathcal{W}$-algebra has a conformal vector $L$ with the central charge 
\begin{equation} \label{eq: W-alg central charge}
    c_k= \frac{k\,\text{sdim}\gg}{k+h^\vee}-12k(x|x) - \sum_{\alpha\in S} (-1)^{p(\alpha)} (12 j_\alpha^2-12 j_\alpha +2) -\frac{1}{2}\text{sdim}\gg_{1/2},
\end{equation}
where $u_\alpha\in \gg_{j_\alpha}$ \cite{DK06}.

\begin{comment}
The conformal weight of $a\in \gg_j$ and $\Phi_m$ are $\Delta_a=1-j$ and $1/2,$ respectively. For $u_\alpha \in \gg_{j_\alpha},$ the conformal weight of $\phi_{u_\alpha}$ and $\phi^{\bar{\alpha}}$ are $1-j_\alpha$ and $j_\alpha.$

 Note that the conformal weight of $J_a$ is $\Delta_a$. The differential subalgebra of $\mathcal{C}^k(\gg,F)$ generated by $J_a$ for $a\in \mathfrak{p}$ is a vertex subalgebra $V^{\psi_k}(J_{\mathfrak{p}})$ with the OPE relation 
\begin{equation}
    J_a(z)J_b(w) \sim \frac{\psi_k(a|b)}{(z-w)^2}+ \frac{J_{[a,b]}}{(z-w)}
\end{equation}
where $\psi_k(a|b)=k(a|b) + \text{str}_{\mathfrak{n}}(\text{ad} a \, \text{ad} b)$. It is proved that $W^k(\gg,F)$ is a subalgebra of $V^{\psi_k}(J_{\mathfrak{p}})\otimes \Phi_{1/2}$ and moreover the following Proposition holds.
\end{comment}

\begin{proposition} [\cite{DK06}] \label{prop: generators of W-alg}
    Let $\{a_i\}_{i\in I^F}$ be a basis of $\gg^F$ and $a_i \in \gg_{1-\Delta_i}.$
    The $\mathcal{W}$-algebra $\cW^k(\gg,F)$ is a freely generated vertex algebra with a free generating set $\{\omega_{a_i}\}_{i\in I^f}\subset V^{\psi_k}(\mathfrak{p}) \otimes \Phi_{1/2}=:V$ such that the conformal weight of $\omega_{a_i}$ is  $\Delta_i$ and the linear term of $\sigma^0(\omega_{a_i})$ is $a_i,$ where $\sigma^0:V \to F^0(V)/F^1(V)$ is the canonical map for the Li's filtration $F^{\bullet}(V).$
\end{proposition}
    
In the vertex algebra $V^{\psi_k}(\mathfrak{p})$, the differential algebra $V^{\psi_k}(\gg_0)$ generated by $\gg_0$ forms a vertex subalgebra. The canonical projection map $V^{\psi_k}(\mathfrak{p})\otimes \Phi_{1/2} \to V^{\psi_k}(\gg_0)\otimes \Phi_{1/2}$ induces the map $\cW^k(\gg,F) \to V^{\psi_k}(\gg_0)\otimes \Phi_{1/2}$ called Miura map, which is injective whenever $k\neq -h^\vee.$

\subsection{SUSY $\mathcal{W}$-algebras}

Suppose there is an $\mathfrak{osp}_{1|2}$ subalgebra $\mathfrak{s}$ of $\gg$ whose even part is spanned by the $\mathfrak{sl}_2$ triple $(E,H=2x,F)$ and let $f\in \mathfrak{s}$ be the odd nilpotent with degree $-1/2$ such that $F=\frac{1}{2}[f,f].$ 
The following two SUSY vertex algebras are the main ingredients of a SUSY $\mathcal{W}$-algebra $\cW^k_{\mathcal{N}=1}(\gg,f)$:
\begin{itemize}
    \item The SUSY affine vertex algebra $V^k_{\mathcal{N}=1}(\gg)$ is the vertex algebra introduced in Example \ref{ex:SUSY affine}.
    \item The SUSY charged free fermion vertex algebra $\mathcal{F}_{\mathcal{N}=1}(\mathfrak{A})$ is freely generated by a basis of $\phi_{\mathfrak{n}_+}\oplus D\phi_{\mathfrak{n}_+} \oplus \phi^{\bar{\mathfrak{n}}_-}\oplus D\phi^{\bar{\mathfrak{n}}_-}$ with the OPE relation 
    \begin{equation}
 D\phi^{\bar{a}}(z)  \phi_{b}(w) \sim (a|b)(z-w)^{-1}, \quad  \phi^{\bar{a}}(z)  D\phi_{b}(w)\sim (-1)^{p(a)}(a|b)(z-w)^{-1}.
    \end{equation}
\end{itemize}
The SUSY BRST complex is 
$   \mathcal{C}^k_{\mathcal{N}=1}(\gg,f)= V^k_{\mathcal{N}=1}(\gg) \otimes \mathcal{F}_{\mathcal{N}=1}(\mathfrak{A})$
and the odd operator $(Dd)_{(0)}$  for 
\begin{equation}
    d=\sum_{\alpha\in S}:(\bar{u}_\alpha-(f|u_{\alpha})) \phi^{\bar{\alpha}}:+\frac{1}{2}\sum_{\alpha,\beta\in S}(-1)^{p(\alpha)p(\beta)+p(\alpha)}:\phi_{[u_\alpha,u_\beta]}\phi^{\bar{\beta}}\phi^{\bar{\alpha}}: \, 
\end{equation}
is the differential on $\mathcal{C}^k_{\mathcal{N}=1}(\gg,f)$. Then the cohomology $H( \mathcal{C}^k_{\mathcal{N}=1}(\gg,f), (Dd)_{(0)})$ is called the SUSY $\cW$-algebra $\cW^k_{\mathcal{N}=1}(\gg,f)$ of level $k$.

Let us denote by $V^{\tau_k}_{\mathcal{N}=1}(\mathfrak{p})$ the SUSY vertex subalgebra of $V^k_{\mathcal{N}=1}(\gg)$ generated by $\bar{\mathfrak{p}}$ and $\mathfrak{p}.$ Here, we denote the bilinear form $\tau_k(a|b)=(k+h^{\vee})(a|b)$ for the dual coxeter number $h^{\vee}$ of $\gg$. The SUSY $\mathcal{W}$-algebra $\cW^k_{\mathcal{N}=1}(\gg,f)$ is an $\mathcal{N}=1$ SUSY vertex subalgebra of $V^{\tau_k}_{\mathcal{N}=1}(\mathfrak{p})$ by the map induced from the injective homomorphism 
\begin{equation}
    V^{\tau_k}_{\mathcal{N}=1}(\mathfrak{p})\to \mathcal{C}^k_{\mathcal{N}=1}(\gg,f), \quad \bar{a}\mapsto  J_{\bar{a}}=\bar{a}+\sum_{\beta\in S}(-1)^{p(\bar{a})p(\bar{\beta})} :\phi^{\bar{\beta}}\phi_{\pi_+([u_\beta,a])}:.
\end{equation}
for $a\in \mathfrak{p}.$ 
\begin{comment}
    The complex has a superconformal vector $\tau= \tau_{KT}+\tau_{\mathcal{F}}+2\partial \bar{x}$ where  $\tau_{KT}$ is the Kac-Todorov superconformal vector 
and \begin{equation}
    \begin{aligned}
    \tau_{\mathcal{F}} =&\sum_{\alpha\in S} (-1)^{p(\alpha)}2j_{\alpha}:(\partial\phi_{\alpha})\phi^{\alpha}:-\sum_{\alpha\in S} (-1)^{p(\alpha)}(1-2j_{\alpha}):\phi_{\alpha}\partial \phi^{\alpha}:+\sum_{\alpha\in S}:(D\phi_{\alpha})(D\phi^{\alpha}):,
    \end{aligned}
  \end{equation}
where $u_\alpha\in \gg_{j_\alpha}.$
\end{comment}
When $k\neq -h^\vee$, the SUSY $\mathcal{W}$-algebra $\cW^k_{\mathcal{N}=1}(\gg,f)$ has a superconformal vector $\tau$ with the central charge 
\begin{equation} \label{eq: susy central charge}
    c_k^{\mathcal{N}=1}=\frac{k\, \textup{sdim}\, \gg}{k+h^{\vee}}+\frac{1}{2}\,\textup{sdim}\, \gg +12\sum_{\alpha\in I_+}(-1)^{p(\alpha)}j_{\alpha}-3\, \textup{sdim}\\gn-12(k+h^{\vee})(x|x).
\end{equation}
For $a\in \gg_{j_a}\in \mathfrak{p},$ its conformal weight with respect to $\tau$ is  $\Delta_{\bar{a}}:=\frac{1}{2}-j_{a}$ and $\cW^k_{\mathcal{N}=1}(\gg,f)$ has a free generating 
set consisting of homogeneous elements \cite{MRS21}. 

\begin{comment}
    \begin{proposition} [\cite{MRS21}]
    Let  $\{a_i\}_{i\in I^f}$ be a basis of $\gg^f$ and $a_i \in \gg_{\frac{1}{2}-\Delta_i}.$
    The SUSY W-algebra $\cW^k_{N=1}(\gg,f)$ is a freely generated vertex algebra with a free generating set $\{\nu_{a_i}, D \nu_{a_i}\}_{i\in I^f}\subset V^{k}_{N=1}(\mathfrak{p})=:V $ such that the conformal weight of $\nu_{a_i}$ is $\Delta_i$ and the linear term in $\sigma^0_D(\nu_{a_i})$ is $\bar{a}_i,$ where $\sigma^0_D:V \to F^0_D(V)/F^1_D(V)$ is the canonical map for the SUSY filtration $F^{\bullet}_D(V).$
\end{proposition}

\end{comment}

As for ordinary $\mathcal{W}$-algebras, when $k\neq -h^\vee$, there is an injective SUSY vertex algebra homomorphism
\begin{equation} \label{eq:SUSY Miura}
    \mu_D: \cW_{\mathcal{N}=1}^k(\gg,f) \to V_{\mathcal{N}=1}^{\tau_k}(\gg_0)
\end{equation}
called the SUSY Miura map, induced from the canonical projection $V_{\mathcal{N}=1}^k(\gg) \to V_{\mathcal{N}=1}^{\tau_k}(\gg_0).$ By comparing the images of the Miura map of $\cW^k(\gg,F)$ and the SUSY Miura map of $\cW^k_{\mathcal{N}=1}(\gg,f)$, one can obtain the relation between the two vertex algebras \cite{GSS25,LSS}.

\begin{theorem}[\cite{GSS25,LSS}] \label{thm:ordinary vs SUSY}
    Let $\gg^{f}_0= \gg^{f} \cap \gg_0$ be the subalgebra of $\gg$, where $\gg^f:=\ker(\textup{ad}f)\subset \gg$. Let $\mathcal{F}(\gg^{f}_0)$ be the free field vertex algebra generated by the parity reversed space $\bar{\gg}^{f}_0$ of $\gg^{f}_0$, where the OPEs between the generators are $\bar{a}(z)\bar{b}(w) \sim \frac{(a|b)}{z-w}$ for $a,b\in \gg^f_0$. When $k\neq -h^\vee$, we have
    \begin{equation} \label{eq: susy vs. nonsusy}
       \cW^k_{\mathcal{N}=1}(\gg,f) \cong \cW^k(\gg,F) \otimes \mathcal{F}(\gg^{f}_0)
    \end{equation}
     as vertex algebras.
\end{theorem}

By Theorem \ref{thm:ordinary vs SUSY}, if $\gg^{f}_0$ is trivial then the ordinary $\mathcal{W}$-algebra $\cW^k(\gg,F)$ itself has a supersymmetry. In particular, when $\gg=\mathfrak{sl}_{n+1|n}$ and $F$ is the even principal nilpotent element, the centralizer of the corresponding $f$ is trivial and the corresponding $\cW$-algebra has a supersymmetry. 

\begin{remark}
 By \cite[Example 1.2]{KRW04}, the free field vertex algebra $\mathcal{F}(\gg_{0}^f)$ is conformal of central charge $\frac{1}{2}\textup{sdim}(\gg_0^f)$. For the central charges $c_k$ and $c^{\mathcal{N}=1}_k$ of $\cW$-algebras in \eqref{eq: W-alg central charge} and \eqref{eq: susy central charge}, one can easily compute that
 \begin{equation*}
     c^{\mathcal{N}=1}_k=c_k+\frac{1}{2}\textup{sdim}(\gg_0^f),
 \end{equation*}
 which is compatible with the isomorphism \eqref{eq: susy vs. nonsusy}.
\end{remark}

\section{$\mathcal{N}=2$ superconformal algebra}\label{sect: N=2}

Recall the OPEs of the universal $N=2$ superconformal algebra at central charge $c$, presented earlier in \eqref{eq:N=2}. It is generated by two even fields: the Virasoro field $L$ and a primary Heisenberg field $H$ of conformal weight $1$, satisfying
\begin{equation}\label{eq:N=2 OPEs 1}
    \begin{aligned}
        H(z)H(w) &\sim \frac{c}{3}(z-w)^{-2}, \\
L(z)H(w) &\sim H(w)(z-w)^{-2} + \partial H(w)(z-w)^{-1}, \\
L(z)L(w) &\sim \frac{c}{2}(z-w)^{-4} + 2L(w)(z-w)^{-2} + \partial L(w)(z-w)^{-1}.
    \end{aligned}
\end{equation}
These are extended by two odd fields $G^{\pm}$ of conformal weight $\tfrac{3}{2}$, which are primary with respect to both the Heisenberg and Virasoro fields:
\begin{equation}
        H(z)G^{\pm}(w) \sim \pm G^{\pm}(w)(z-w)^{-1}, \ 
L(z)G^{\pm}(w) \sim \tfrac{3}{2} G^{\pm}(w)(z-w)^{-2} + \partial G^{\pm}(w)(z-w)^{-1}
\end{equation}
satisfying the following OPE relation
\begin{equation} \label{eq:N=2 OPEs 3} G^+(z)G^-(w)\sim \frac{c}{3}(z-w)^{-3}+H(w)(z-w)^{-2}+\left(L+\frac{1}{2}\partial H\right)(w)(z-w)^{-1}.\end{equation} 
We denote this vertex superalgebra by $\T{Vir}^c_{\cN=2}$, as we regard it as the $\cN=2$ analogue of the universal Virasoro algebra.
For instance, just as the Virasoro algebra can be realized as $\cW^k(\fr{sl}_2)$ for some $k \in \mathbb{C}$, the $\mathcal{N}=2$ superconformal algebra $\T{Vir}^c_{\mathcal{N}=2}$ can likewise be realized as $\cW^k(\fr{sl}_{2|1})$.

The vertex superalgebra $\T{Vir}^c_{\cN=2}$ has the automorphism group $\text{O}_2$.
The subgroup $\text{SO}_2\cong \mathbb{C}^\times$ arises from the exponentiation of the zero mode $H_{(0)}$, and acts by scaling
\begin{equation} \label{eq:N=2 automorphism constant}
G^{+}\mapsto \mu G^+,\quad G^{-}\mapsto \mu^{-1} G^-, \quad \mu\in\mathbb{C}\setminus\{0\},
\end{equation}
and the outer involution $\sigma$ is given by
\begin{equation}\label{eq: N=2 automorphism}
    \sigma: H\mapsto -H,\quad G^{\mp}\mapsto G^{\pm}, \quad L\mapsto L.
\end{equation}

Another perspective on $\T{Vir}^c_{\cN=2}$ is to view it as an extension of \textit{two} $\mathcal{N}=1$ superconformal algebras.
A direct computation shows that
\begin{equation}\label{eq: change of basis 1}
G:=G^++G^- ,\quad \widetilde G:=\sqrt{-1}(G^+ - G^-)
\end{equation}
satisfy the OPE relations \eqref{eq: N=1 superconformal conditions}) of $\cN=1$ superconformal algebra  for the same conformal element $L$.
Setting $J=\sqrt{-1}H$ we have the following OPE relations
\begin{equation} \label{eq: N=2 superconformal conditions}
\begin{split}
      & J(z) J(w) \sim -\frac{c}{3}(z-w)^{-2},\quad L(z)J(w) \sim  J(w) (z-w)^{-2}+\partial J(w) (z-w)^{-1}, \\
     & \widetilde{G}(z) J(w) \sim -G(w) (z-w)^{-1}, \quad  G(z) J(w) \sim \widetilde{G}(w) (z-w)^{-1}, \\ 
     & G(z) \widetilde{G}(w) \sim 2 J(w) (z-w)^{-2}+ \partial J(w) (z-w)^{-1}.
\end{split}
\end{equation} 
In particular, zero modes $(G)_{(0)}=D$ and $(\widetilde G)_{(0)}=\widetilde D$ define \textit{two} commuting odd derivations.
For completeness we record how two bases are related to each other
\begin{equation} \label{eq: Gpm notation for N=2 superconformal}
    L:=-\frac{1}{2}D\widetilde{D}J, \quad G^+:=-\frac{1}{2}\big(\widetilde{D}J+\sqrt{-1}DJ\big), \quad G^-:=-\frac{1}{2}\big(\widetilde{D}J-\sqrt{-1}DJ\big),\quad H:=\sqrt{-1}J.
\end{equation}

In contrast to Lemma \ref{lem: N=1 superconformal conditions} in the $\cN=1$ setting, the two supersymmetries $D$ and $\widetilde{D}$ are not uniquely determined by a given choice of $J$.
Using OPE relations \eqref{eq: N=2 superconformal conditions}, one can see that 
\begin{equation} \label{eq: non-uniqueness of N=2}
    G_{\textup{new}}=p_1 G+p_2 \widetilde{G}, \quad \widetilde{G}_{\textup{new}}=p_1 \widetilde{G}-p_2 G,\quad p_1^2+p_2^2=1 \ \textup{ for }\ p_1, p_2\in \mathbb{C}
\end{equation}
 again make two SUSY $\cN=1$ superconformal algebras for the same $J$ and $L$, and satisfy the same OPE relations in \eqref{eq: N=2 superconformal conditions}. This non-uniqueness can be explained via the automorphism \eqref{eq:N=2 automorphism constant}.
 For the new supersymmetries $D_{\textup{new}}:=(G_{\textup{new}})_{(0)}$ and $\widetilde{D}_{\textup{new}}:=(\widetilde{G}_{\textup{new}})_{(0)}$, consider the corresponding $G^{+}_{\textup{new}}$ and $G^-_{\textup{new}}$ similarly to \eqref{eq: Gpm notation for N=2 superconformal}. Replacing $p_1$ and $p_2$ in \eqref{eq: non-uniqueness of N=2} with $p_1=\frac{1}{2}(\mu+\frac{1}{\mu})$ and $p_2=\frac{1}{2\sqrt{-1}}(\mu-\frac{1}{\mu})$ for $\mu\in \mathbb{C}^\times$, one can show  $G^{\pm}_{\textup{new}}$ and $G^{\pm}$ are related by \eqref{eq:N=2 automorphism constant}.

We call a conformal vertex algebra $V$ to be $\cN=2$ \textit{superconformal} if it is a conformal extension of the $\cN=2$ superconformal algebra, i.e., we have a conformal embedding 
\[\T{Vir}^c_{\cN=2} \hookrightarrow V,\]
where the Virasoro field $L\in \T{Vir}^c_{\cN=2}$ makes $V$ conformal.
In this case, we call the field $H$ the $\cN=2$ \textit{superconformal element}.

\subsection{$\cN=2$ diamonds}

Let $M$ be a $\text{Vir}^c_{\mathcal{N}=2}$-module\footnote{In our examples $M$ will be a vertex algebra containing $\text{Vir}^c_{\mathcal{N}=2}$ as a subVOA.}.
Let $v$ be a homogeneous vector with conformal weight $\Delta(v)=\Delta$ and Heisenberg charge $h(v)=h$, i.e. we have relations
\begin{equation}\label{eq:N=2 label}
    L_{(1)} v= \Delta v,\quad  H_{(0)}v=h v.
\end{equation}
We define the $\cN=2$ \textit{diamond} of weight $(\Delta,h)$ generated by $v$ to be the vector space spanned by the elements 
\begin{equation}\label{eq:N=2 diamond}
v^{\bot}:=v,\quad v^\pm := \mp G^\pm_{(0)}v, \qquad v^{\top} := \pm G^{\pm}_{(0)}G^{\mp}_{(0)}v\mp \frac{1}{2}\partial v,
\end{equation}
and denote it by $v_{\cN=2}^{(\Delta,h)}$. Note that $v^{\top}=\frac{1}{2}\sqrt{-1}D\widetilde{D}v$ in terms of the two supersymmetries $D$ and $\widetilde{D}$ in \eqref{eq: change of basis 1}.

An element $v \in M$ is called $\text{Vir}^c_{\mathcal{N}=2}$ \textit{primary} of weight $(\Delta,h)$ if $L_{(n+1)}v = H_{(n)}v = G^\pm_{(n)} v = 0$ for $n \in \mathbb Z_{>0}$ in addition to (\ref{eq:N=2 label}).
We then call the $\cN=2$ diamond generated by $v$ a \textit{primary} $\cN=2$ diamond of weight $(\Delta,h)$.
In this case, in addition to vectors $v^{\bot},v^{\pm},v^{\top}$ being primary for $L$ of conformal weights $\Delta,\Delta+1/2,\Delta+1$, and $v^{\bot},v^{\pm}$ being primary for $H$ of Heisenberg weights $h,h\pm1$, one derives the following OPE relations:
\begin{equation}\label{eq:N=2 diamond OPEs}
\begin{split}
G^\pm(z)v^{\bot}(w) &\sim \mp v^\pm(w)(z-w)^{-1},\quad H(z)v^{\top}(w) \sim h v^{\bot}(w)(z-w)^{-2},\\
G^\pm(z) v^\mp(w) &\sim \pm {\left(\Delta - \frac{1}{2}h\right) v^{\bot}(w) }(z-w)^{-2} + \left(v^{\top}(w) \pm \frac{1}{2} \partial v^{\bot}(w)\right)(z-w)^{-1}, \\
G^\pm(z) v^{\top}(w) &\sim  {\left(\Delta - \frac{1}{2}h+\frac{1}{2}\right)v^\pm(w) }(z-w)^{-2}+ {\frac{1}{2}  \partial v^\pm(w)}(z-w)^{-1}.
\end{split}
\end{equation}
These OPE relations are conveniently summarized in the Figure \ref{fig:diamond}.
Such $\cN=2$ diamonds feature prominently in all of our examples.

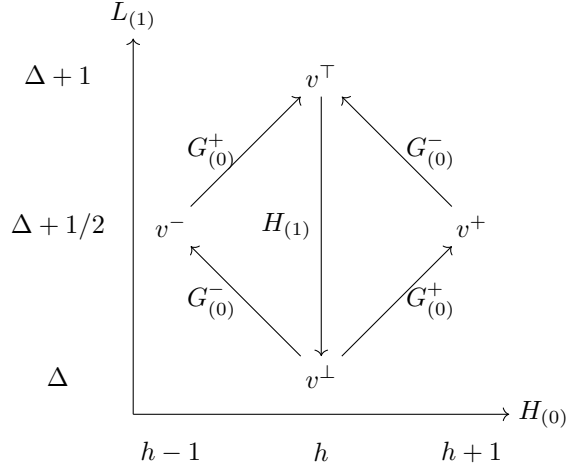
\begin{figure}[h]
\begin{tikzpicture}

% Diamond shape nodes
\node (top) at (0, 2) {$v^{\top}$};
\node (plus) at (2, 0) {$v^{+}$};
\node (minus) at (-2, 0) {$v^{-}$};
\node (bottom) at (0, -2) {$v^{\bot}$};

\node at (-3.5,-2) {$\Delta$};
\node at (-3.5, 0) {$\Delta+1/2$};
\node at (-3.5, 2) {$\Delta+1$};

\node at (-2,-3) {$h-1$};
\node at (0,-3) {$h$};
\node at (2,-3) {$h+1$};

% Arrows and labels
\draw[->] (plus) -- (top) node[midway, right] {$G^-_{(0)}$};
\draw[->] (minus) -- (top) node[midway, left] {$G^+_{(0)}$};
\draw[->] (bottom) -- (minus) node[midway, left] {$G^-_{(0)}$};
\draw[->] (bottom) -- (plus) node[midway, right]{$G^+_{(0)}$};
\draw[->] (top) -- (bottom) node[midway, left] {$H_{(1)}$};

\draw[->] (-2.5,-2.5) -- (-2.5,2.5) node[above]{$L_{(1)}$};
\draw[->] (-2.5,-2.5) -- (2.5,-2.5) node[right]{$H_{(0)}$};
\end{tikzpicture}
\caption{Structure of the $\mathcal{N}=2$ diamond generated by $v$. The coordinate axis correspond the conformal and Heisenberg weights.}\label{fig:diamond}
\end{figure}

\begin{comment}
and the fields corresponding to $v,y^\pm$ and $z$ by
\begin{equation}
\begin{split}
 & v(z) = \sum_{n \in \mathbb Z} v_n z^{-n- h}, \quad  y^\pm(z) =   \sum_{n \in \mathbb Z+\frac{1}{2}} y^\pm_n z^{-n- h-\frac{1}{2}}, \quad x(z) =   \sum_{n \in \mathbb Z} x_n z^{-n- h-1}.
\end{split}
\end{equation}
\end{comment}

\begin{comment}
and hence into the commutation relations
\begin{equation}
\begin{split}
[G^\pm_r, v_n] &= y^\pm_{r+n}, \qquad [G^\pm_r, y^\pm_s] = 0, \\
[G^\pm_r, y^\mp_s] &= \mp x_{r+s} + \left(\left(r+s+h \right) + (2h \pm j) \left(r + \frac{1}{2} \right) \right) v_{r+s}, \\
[G^\pm_r, x_n] &=  \left(\left(r+n+h +\frac{1}{2}\right) \mp (2h +1 \pm ( \pm 1)\right) \left(r + \frac{1}{2} \right)  y^\pm_{r+n}. \\
\end{split}
\end{equation}
\end{comment}

Sometimes a degeneration occurs in (\ref{eq:N=2 diamond}) where we have an additional relation
\begin{equation}\label{eq:N=2 degenerate diamond}
    G^{+}_{(0)}v^{\bot}= 0.
\end{equation}
In this case, the $\cN=2$ diamond degenerates to a bottom left leg in Figure \ref{fig:diamond}, and is spanned by just two\footnote{In fact, $v^{\top}=\pm \partial v^{\bot}$ in this case, however since $\partial v^{\bot}$ is not a new strong generator, we do not consider it as an element of the degenerate diamond.} elements $v^{\bot},v^{-}$.
We call their span the $\cN=2$ \textit{degenerate diamond} generated by $v$, and continue to denote it by  $v^{(\Delta,h)}_{\cN=2}$.
The primary $\cN=2$ degenerate diamond is defined similarly. In this case we find the following OPE relations
\begin{equation}\label{eq:N=2 degenerate diamond OPEs}
    \begin{split}
         G^-(z)v^{\bot}(w)\sim&  v^{-}(w) (z-w)^{-1}, \quad G^+(z)v^{\bot}(w)\sim0,\\
   \quad G^+(z)v^{-}(w)\sim& \left(\Delta-\frac{1}{2}h\right) v^{\bot}(w)(z-w)^{-2}+\frac{1}{2}\partial v^{\bot}(w) (z-w)^{-1}.
    \end{split}
\end{equation}
Note that we do \textit{not} require $v^{\bot}$ and $v^{-}$ to be even or odd elements. In fact, their parities are exchanged in half of our examples.

Lastly, we note that instead of (\ref{eq:N=2 degenerate diamond}) one may consider the relation $G^-_{(0)}v^{\bot}= 0$.
In this case the $\cN=2$ diamond is spanned by $v^{\bot},v^+$ and relations analogous to (\ref{eq:N=2 degenerate diamond OPEs}) can be derived.
If $v^{\bot}$ is primary, then the $\T{Vir}^c_{\cN=2}$-module is isomorphic to the one generated by $v^{\bot},v^-$, thanks to the automorphism (\ref{eq: N=2 automorphism}).
Such $\cN=2$ diamonds and their degenerate versions feature prominently in the $\cN=2$ hook type $\cW$-algebras defined in Section \ref{sect:supersymmetricY}.

\section{$\cN=2$ principal $\cW$-algebras} \label{sec:W-algebra type A(n,n-1)}

Let $\fr{g}=\fr{sl}_{n+1|n}$ and $F=F_{n+1|n}$ be an \textit{even} \textit{principal} nilpotent in $\mathfrak{sl}_{n+1|n}$, that is, the sum of principal nilpotents in $\gs\gl_{n+1}$ and $\gs\gl_n$.
We define the $\cN=2$ \textit{principal} $\cW$-algebra
$$\cW^k(\gs\gl_{n+1|n}):=\cW^k(\gs\gl_{n+1|n}, F),$$
which we regard as an $\cN=2$ analogue of principal $\cW$-algebras in type $A$.
Similarly, we consider an \textit{odd} \textit{principal} nilpotent $f=f_{n+1|n}$ such that $f^2=F$, and define the SUSY $\cN=2$ principal $\cW$-algebra
$$\cW^k_{\cN=1}(\mathfrak{sl}_{n+1|n}):=\cW^k_{\cN=1}(\mathfrak{sl}_{n+1|n},f).$$
By  Theorem \ref{thm:ordinary vs SUSY}, the two $\cW$-superalgebras are isomorphic for noncritical values of the level parameter
\[\cW^k(\gs\gl_{n+1|n})\cong \cW^k_{\cN=1}(\mathfrak{sl}_{n+1|n}),\quad k\neq -h^\vee.\]

For a concrete picture, let us fix the standard basis $\{e_{ij}|i,j=1,\cdots, 2n+1\}$ of $\mathfrak{gl}_{n+1|n}$ with the parity $p(e_{ij})=\bar{\imath}+\bar{\jmath}$, where $\bar{\imath}\equiv i+1 \Mod{2}$ is the parity of each index $i$. 
Since the dual Coxeter number of $\mathfrak{sl}_{n+1|n}$ is $h^{\vee}=1$, the normalized bilinear form is 
\begin{equation}\label{eq:bi form}
    (A|B)=\textup{str}(AB), \quad A,B\in \mathfrak{sl}_{n+1|n}.
\end{equation}

\subsection{$\cN=2$ structure}
First we consider the following low spin content of the centralizer $\fr{g}^F$,
\begin{equation} \label{eq:F and u}
    \begin{split}
        u&=n\sum_{i=1}^{n+1}e_{2i-1\, 2i-1}+(n+1)\sum_{i=1}^{n}e_{2i\, 2i}\in \gg^F_0,\\
          \mathrm{f}^+&= \sum_{i=1}^{n}e_{2i\, 2i-1} \in \gg^F_{-1/2},\quad \mathrm{f}^-= \sum_{i=1}^{n}e_{2i+1\, 2i}\in \gg^F_{-1/2}\\
         F&=\sum_{i=1}^{2n-1} e_{i+2\, i}\in \gg^F_{-1},
           \end{split}
\end{equation}
satisfying the Lie bracket relations 
\begin{equation}\label{eq:gF n=2}
    [u,\T{f}^{\pm}]=\pm \T{f}^{\pm}, \quad [\T{f}^{+},\T{f}^{-}]=F.
\end{equation}

\begin{proposition}\label{prop:sl(n|n-1), N=2}
 Let $k\neq -1$, and let $u$, $\mathrm{f}^{\pm}$, $F$ be as in \eqref{eq:basis of centralizer}.
    Then $\cW^{k}(\fr{sl}_{n+1|n})$ has $\cN=2$ superconformal structure 
    \begin{equation}\label{eq:N=2 in principal}
    H={-{\sqrt{-1}}}\omega_u,\quad c = -3 n (k n + k + n),
    \end{equation} 
    with the remaining fields being uniquely determined by (\ref{eq:N=2 diamond}),
    \begin{equation*}
    G^{+}=\frac{1+\sqrt{-1}}{2\sqrt{k+1}}\omega_{\mathrm{f}^{+}},\quad G^{-}=\frac{-1+\sqrt{-1}}{2\sqrt{k+1}}\omega_{\mathrm{f}^{-}},\quad L=-\frac{1}{k+1}\omega_{F}.
     \end{equation*}
\end{proposition}

\begin{proof}
It is well known that $L=-\frac{1}{k+1}\omega_F$ is a conformal vector of central charge \eqref{eq: W-alg central charge}, which evaluates to (\ref{eq:N=2 in principal}).
Since $\omega_u$ has conformal weight one, it must generate a Heisenberg field, and a computation shows that it is at level $-n (k n + k + n)$.
The first pair of relations in (\ref{eq:gF n=2}) implies that strong generators $\omega_{\T{f}^{\pm}}$ are primary for $\omega_F$ and $\omega_u$ of conformal and Heisenberg weights $\frac{3}{2}$ and $\pm1$, respectively.
Then the second relation (\ref{eq:gF n=2}) implies that the most general OPE compatible with relations (\ref{eq:gF n=2}) is
\[\omega_{\T{f}^{\pm}}(z)\omega_{\T{f}^{\pm}}(w)\sim 0,\quad \omega_{\T{f}^{+}}(z)\omega_{\T{f}^{-}}(w)\sim \alpha_1(z-w)^{-3}+\alpha_H H(w)(z-w)^{-2}+\left(L+\alpha_{\partial H}\partial H\right)(w)(z-w)^{-1}.\]
The Jacobi identities $J(H,G^{\pm},G^{\mp})$, $J(L,G^{\pm},G^{\mp})$, $J(G^{\pm},G^{\mp},G^{\mp})$ uniquely determine the structure constants $\alpha_1, \alpha_H, \alpha_{\partial H}$, and these fields generate a copy of $\T{Vir}^{c}_{\cN=2}$ as in (\ref{eq:N=2 OPEs 1}-\ref{eq:N=2 OPEs 3}).
Finally, since $L$ is the conformal vector for $\cW^k(\fr{sl}_{n+1|n})$, we conclude that $\cW^k(\fr{sl}_{n+1|n})$ has $\cN=2$ structure.
\end{proof}
\begin{remark}
This statement also follows from \cite[Theorem 24]{MTY}. 
While omitting the details of Jacobi computation, our sketch here is included for convenience of the reader.
\end{remark}

Alternatively, we may consider the odd principal nilpotents $f$ and $\tilde{f}$ of $\mathfrak{sl}_{n+1|n}$ given by 
\begin{equation} \label{eq: odd nilpotents}
    f=\sum_{i=1}^{2n}e_{i+1\, i}, \quad \tilde{f}=\sqrt{-1}\sum_{i=1}^{2n}(-1)^{i}e_{i+1\, i},
\end{equation}
satisfying the Lie bracket relations
\begin{equation}\label{eq:nilpotents square}
\begin{gathered}
    [f,f]=2F=[\tilde{f},\tilde{f}],\\
    \tilde{f} =  [-\sqrt{-1}u,f],\quad   f=[\sqrt{-1}u,\tilde{f}].
\end{gathered}
\end{equation}
A change of basis (\ref{eq: change of basis 1}) implies the following.
\begin{corollary} \label{cor: N=2}
    Fields $\omega_{{f}}$ and $\omega_{{\tilde {f}}}$ give rise to two $\cN=1$ structures as in (\ref{eq: N=2 superconformal conditions}), specifically
    \[ G=\frac{\sqrt{-1}}{\sqrt{2(k+1)}}\omega_{{f}}\ , \quad \tilde G=\frac{\sqrt{-1}}{\sqrt{2(k+1)}}\omega_{\tilde{f}}.\]
\end{corollary}

\subsection{$\cN=2$ diamond structure}

The higher weight generators complete to a basis of the centralizer $\gg^F$,
\begin{equation} \label{eq:basis of centralizer}
   \gg^F = \text{Span}_{\mathbb{C}}\{u, \mathrm{f}^+, \mathrm{f}^-, F\}\oplus \T{Span}_{\mathbb{C}}\{ a^{j,\bot}, a^{j,+}, a^{j,-},a^{j,\top} \, | 2\leq j\leq n\}, 
\end{equation}
where elements
\begin{equation} \label{eq: as}
    \begin{split}
        a^{j,\bot}&=\sum_{i=1}^{n-j+1} (-n+j)e_{2i+2j-1\ 2i-1}-\sum_{i=1}^{n-j}(n+j+1)e_{2i+2j\ 2i}\in \fr{g}^{F}_{-j},\\
        a^{j,+}&= \sum_{i=1}^{n-j}e_{2i+2j\ 2i-1}\in \gg^F_{-j-1/2},\quad a^{j,-}=\sum_{i=1}^{n-j}e_{2i+2j+1\ 2i}\in \gg^F_{-j-1/2},\\
        a^{j, \top}&=\sum_{i=1}^{2n-1-2j} e_{i+2j+2\ i} \in \gg^F_{-j-1},
    \end{split}
\end{equation}
satisfy relations 
\begin{equation}\label{eq:c2 u f and a}
[u,a^{j,\bot}]=[u,a^{j,\top}]=0,\quad [u,a^{j,\pm}]=\pm a^{j,\pm},\quad [\mathrm{f}^{\pm}, a^{j,\bot}]=\pm (2j+1) a^{j,\pm}, \quad [\mathrm{f}^{\pm}, a^{j,\mp}]= a^{j,\top}.
\end{equation}
Thanks to Proposition \ref{prop: generators of W-alg} it is immediate that $\cW^k(\gs\gl_{n+1|n})$ has a strong generating set 
\begin{equation}\label{eq:princ gen}
    \omega_{a^{j,\bot}}=J_{a^{j,\bot}}+\dotsb,\quad\omega_{a^{j,\pm}}=J_{a^{j,\pm}}+\dotsb, \quad \omega_{a^{j, \top}}=J_{a^{j, \top}}+\dotsb,
\end{equation}
with omitted terms being the corrections in the BRST complex (\ref{eq:diff_W}), 
and is of strong generating type
\begin{equation} 
\label{princ:gentype} \cW\bigg(1,2^2, 3^2,\dots, n^2, n+1; \bigg(\frac{3}{2}\bigg)^2, \bigg(\frac{5}{2}\bigg)^2,\dots, \bigg(\frac{2n+1}{2}\bigg)^2\bigg).\end{equation}

Let $V$ be a vertex algebra containing $\T{Vir}^{c}_{\cN=2}$ as a subVOA, and let $S$ be a strong generating set for $V$.
We say that $S$ has \textit{$\cN=2$ diamond structure} if there exists a subset of elements $S^{\bot} \subset S$, so that all the $\cN=2$ diamonds generated by elements $S^{\bot}$ are nondegenerate, and they span the same vector space as $S$\footnote{Often one can choose the set $B$ to be \textit{primary} for $\T{Vir}^{c}_{\cN=2}$. This is explained in Appendix \ref{sect:primary structures}.}.

\begin{proposition}\label{cor:diamond}
Let $k\neq -1$.
Then there exists a set of strong generators for $\cW^k(\fr{sl}_{n+1|n})$ with $\cN=2$ diamond structure.
\end{proposition} 

\begin{proof}
Let $H,G^{\pm}, L$ be the generators of $\T{Vir}^{c}_{\cN=2}$ from Proposition \ref{prop:sl(n|n-1), N=2}.
By Proposition \ref{prop: generators of W-alg}, there is a strong generating set consisting of $H,G^{\pm}, L$ and $\{\omega_{a^{i,\bot}},\omega_{a^{i,\pm}},\omega_{a^{i,\top}}:1\leq i\leq n-1\}$.
Consider the $\cN=2$ diamonds (\ref{eq:N=2 diamond}) generated by $a^{i,\bot}$ in (\ref{eq: as}),
\begin{equation}\label{eq: princ w algebra generators}
    \omega^{i,\bot}:= \omega_{a^{i,\bot}}, \quad \omega^{i,\pm}:= \mp G^{\pm}{}_{(0)}\omega^{i,\bot}, \quad \omega^{i,\top}:= \pm G^{\pm}_{(0)}G^{\mp}_{(0)}\omega^{i,\bot}\mp 1/2\partial \omega^{i,\bot},
\end{equation}
and denote the resulting set of strong generators by $S$.
Thanks to the relations (\ref{eq:c2 u f and a}) and Proposition \ref{prop: generators of W-alg} it follows that we have relations
\begin{equation} \label{eq:omega i and omega a}
    \omega^{i,\bot}= \omega_{a^{i,\bot}}, \quad \omega^{i,\pm}=  \pm (2i+1)\omega_{a^{i,\pm}}+\cdots , \quad \omega^{i,\top}=-4 \omega_{a^{i,\top}}+\cdots,
\end{equation}
where the rest of the terms on the RHS only depend on the fields $\omega_{a^{*,*}}$ with conformal weights strictly lower than the LHS.
This establishes the $\cN=2$ diamond structure of $S$.
\end{proof}

\subsection{Automorphism}
There is a natural automorphism of order two on $\cW^k(\fr{sl}_{n+1|n})$ that extends the automorphism (\ref{eq: N=2 automorphism}) of the $\T{Vir}^c_{\cN=2}$ subVOA. It arises naturally from two automorphisms of the Lie superalgebra $\fr{sl}_{n+1|n}$. 
First, recall the automorphism on $\gs\gl_{n+1|n}$
\begin{equation*}
    -\textup{str}: \gs\gl_{n+1|n}\rightarrow \gs\gl_{n+1|n}, \quad e_{i j} \mapsto -(-1)^{(j+1)(i+j)}e_{j i}
\end{equation*}
given by the supertranspose with a minus sign.
Second, let $\tau$ be the inner automorphism on $\gs\gl_{n+1|n}$ given by the conjugation with the anti-diagonal matrix
\begin{equation} \label{eq: conjugation matrix}
    T=\sum_{i=1}^{2n+1}(-\sqrt{-1})^{n-1+i}e_{i\, 2n+2-i},\quad \det T=1,
\end{equation}
and denote the composition of the two automorphisms
\begin{equation} \label{eq:automorphism,sl(n+1|n)}
    \sigma:=\tau \circ (-\textup{str}): e_{i j}\mapsto(-1)^{ij+1}\sqrt{-1}^{i+j}e_{2n+2-j \ 2n+2-i}.
\end{equation}
\begin{comment}
\begin{equation}
    \sigma \ : 
    \left\{ \begin{array}{l} \ e_{i+1\, i} \mapsto (-1)^{i}\sqrt{-1}e_{2n+2-i\, 2n+1-i}, \\ \  e_{i\, i+1} \mapsto (-1)^{i}\sqrt{-1}e_{2n+1-i\, 2n+2-i}.\end{array} \right. 
\end{equation}
\end{comment}

\begin{proposition} \label{prop:automorphism}
Let $\omega^{\bot,i},\omega^{i,\pm},\omega^{i,\top}$ denote strong generating set (\ref{eq: princ w algebra generators}) satisfying $\cN=2$ diamond structure.
Then automorphism $\sigma$ in \eqref{eq:automorphism,sl(n+1|n)} induces an order two automorphism of $\cW^k(\mathfrak{sl}_{n+1|n})$,
\begin{equation}\label{eq:auto principal W}
        \sigma(\omega^{i,\bot})=(-1)^{i}\omega^{i,\bot}, \quad \sigma(\omega^{i,\pm})=(-1)^{i+1}\omega^{i,\mp}, \quad 
        \sigma(\omega^{i,\top})=(-1)^{i+1}\omega^{i,\top}.
     \end{equation}
\end{proposition}
\begin{proof}
    Let $\gg=\mathfrak{sl}_{n+1|n}$. 
    It is easy to see that $\sigma$ is an automorphism of order two on $\gg$.
    Consider its extension to the vertex algebra automorphism on $V^{\psi_k}(\mathfrak{p})\otimes \Phi_{1/2}$,
    \begin{equation} \label{eq:sigma}
    \sigma\ : \ a\mapsto \sigma(a), \quad \Phi_{[m]}\mapsto\Phi_{[\sigma(m)]}.
    \end{equation}
    Since $\sigma$ fixes $F$, it preserves the differential, and hence is an automorphism of the complex $\mathcal{C}^k(\gg,F)$; consequently, it descends to cohomology.
Using the explicit formulas \eqref{eq: as} for the leading terms of the strong generators \eqref{eq:princ gen}, it is straightforward to verify that $\sigma$ acts as in \eqref{eq:auto principal W} on the strong generators arising from the set $S$ constructed in Proposition~ \ref{cor:diamond}.
\end{proof}

\begin{remark} Recall that $\cW^k(\gs\gl_{2|1}) \cong \text{Vir}^{-3(2k+1)}_{\cN=2}$ for all noncritical levels. At the critical level $k = -1$, $\cW^{-1}(\gs\gl_{2|1}) \ncong \text{Vir}^3_{\cN=2}$; instead, $\cW^{-1}(\gs\gl_{2|1})$ has a nontrivial center and its OPE algebra is given in \cite{AN24} right before Proposition 4.1. More generally, it was conjectured by Adamovi\'c, Feigin and Nakatsuka in \cite{AFN26} that $\cW^{-1}(\gs\gl_{n+1|n})$ is isomorphic to the orbifold $(\cH_{\text{deg}}(2n) \otimes \cE(n))^{\text{GL}_n}$. Here $\cH_{\text{deg}}(2n)$ is the degenerate Heisenberg algebra on $2n$ commuting generators which transform under $\text{GL}_n$ as $\mathbb{C}^n \oplus (\mathbb{C}^n)^*$. This clearly holds for $n=1$ from the description in \cite{AN24}, and it was checked for $n=2$ in \cite{AFN26}. In Section 9, we will interpret this conjecture as a limiting case of Ito's conjecture.
\end{remark}

\section{$\mathcal{N}=2$ hook-type $\cW$-algebras and $Y$-algebras} \label{sect:supersymmetricY}
In this section we generalize considerations of the previous one. 
First we introduce a family that generalizes $\cN=2$ principal $\cW$-superalgebras called $\cN=2$ \textit{hook-type} $\cW$-superalgebra; these are associated to even $\cN=2$ \textit{hook-type} nilpotent element 
$$F=F_{n+1,1^r|n,1^s}\in\fr{sl}_{n+r+1|n+s},$$
which is \textit{principal} in $n+1|n$ block and \textit{trivial} in $r|s$ block.
Although we call these algebras $\cN=2$ hook-type, these do \textit{not} have the $\cN=2$ structure.
In addition to the strong generators arising from the $\cN=2$ principal block as in Section \ref{sec:W-algebra type A(n,n-1)}, the trivial block gives rise an affine subVOA, and certain modules over it.
We regard this family as the $\cN=2$ analogue of the hook-type algebras in type $A$.

The $\cN=2$ $Y$-algebras will be defined as the cosets of these $\cW$-superalgebras by the affine subVOA arising from the trivial block. They are a generalization of $\cW^k(\gs\gl_{n+1|n})$ which is the case $r =s = 0$, and they satisfy analogous properties to Propositions \ref{prop:sl(n|n-1), N=2}, \ref{cor:diamond}, \ref{prop:automorphism}; see Theorem \ref{thm: Y alg}.
There is one more family of such algebras which do {\it not} arise as $\cW$-algebras or their cosets, but are the $\cN=2$ analogues of the GKO cosets \eqref{intro:D0m}; see Theorem \ref{thm: GKOs}. The $\cN=2$ $Y$-algebras are generally not simple, but they are simple in the cases $s = 0$ or $r=0$. We will introduce a different notation \eqref{def:DandEfamilies} for these two families, and some of our results such as Proposition \ref{lem:Dnsgenerators} and Theorem \ref{N=2:duality}, will only be proven in these cases.

\subsection{$\cN=2$ hook-type $\cW$-algebras}
Let $\gg = \gs\gl_{n+r+1|n+s}$ which has a natural embedding of $\gs\gl_{n+1|n}$. 
For a concrete picture, consider $\gg$ as composed of block matrices with the index set $I=\{1,2,\dots, 2n+r+s+1\}$ as in Figure \ref{eq: hook type block matrix}. In particular, we set the parity of the matrix by $p(e_{i j})=\bar{\imath}+\bar{\jmath} \Mod{2}$, where $\bar{\imath}$ is the parity of the index $i\in I$ defined as follows:
\begin{enumerate}
    \item The first $n+1|n$ block is as in Section \ref{sec:W-algebra type A(n,n-1)}, i.e., the first $2n+1$ indices have parity $\bar{\imath}=i+1 \Mod{2}$.
    \item The second diagonal $r|s$ block forms $\fr{gl}_{r|s}$, where we let the first $r$ are even, and the last $s$ are odd. To be explicit, the parity of the index is $\overline{2n+1+l}=|l|$, where
    \begin{equation} \label{eq: parity of extension index}
        |l|=\left\{\begin{array} {ll}
            0 & \textup{ if } 1\leq l\leq r,\\
            1 & \textup{ if } r+1 \leq l \leq r+s.
        \end{array}\right.
    \end{equation}
\end{enumerate}

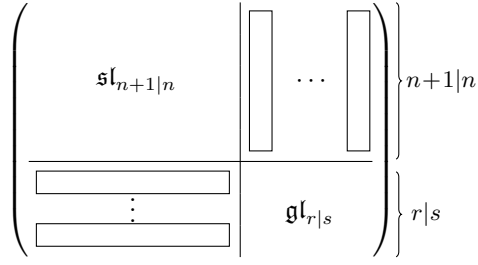
\begin{figure}[h]
    \centering
    \begin{tikzpicture}[baseline=(current  bounding  box.center), B/.style = {decorate,
            decoration={brace, amplitude=2pt,
            pre=moveto,pre length=1pt,post=moveto,post length=1pt,
            raise=1mm,}}]
\node(mat) at (0,0) {$   \left(\begin{array}{cccccccc|ccccc}
         & & & & & & & & & & & & \\
         & & & & & & & & & & & & \\
         & & & & & & & & & & & & \\
         & & & & & & & & & & & & \\
         & & & & & & & & & & & & \\
         \hline
         & & & & & & & & & & & & \\
         & & & & & & & & & & & & \\
         & & & & & & & & & & & &        
    \end{array}\right)$};
\node(AA) at (2.5,1.85) {};
\node(BB) at (2.5,-0.55) {};
\node(BBBB) at (2.5,-0.4) {};
\node(CC) at (2.5,-1.85) {};
\node at(3.2,0.65) {\small{$n\!+\!1|n$}};
\node at(3,-1.125) {\small{$r|s$}};
\draw[B] (AA)--(BB);
\draw[B] (BBBB)--(CC);
\node(sl) at (-0.85,0.65) {$\mathfrak{sl}_{n+1|n}$};
\node(gl) at (1.45,-1.125) {$\mathfrak{gl}_{r|s}$};
\node[draw, rectangle, minimum width=0.3cm, minimum height=5.3em] 
  at (0.8,0.65) {};
  \node at (1.5,0.65) {$\cdots$};
\node[draw, rectangle, minimum width=0.3cm, minimum height=5.3em] 
  at (2.1,0.65) {};
  \node[draw, rectangle, minimum width=7.3em, minimum height=0.3cm] 
  at (-0.9,-0.7) {};
  \node at (-0.9,-0.96) {$\vdots$};
\node[draw, rectangle, minimum width=7.3em, minimum height=0.3cm] 
  at (-0.9,-1.4) {};
  \end{tikzpicture}
    \caption{Block decomposition of $\fr{sl}_{n+r+1|n+s}$.}
    \label{eq: hook type block matrix}
\end{figure}

Recall the even principal nilpotent $F$ in (\ref{eq:F and u}) and odd (\ref{eq: odd nilpotents}) $f,\tilde{f}$ of $\mathfrak{sl}_{n+1|n}$, whose images in $\fr{sl}_{n+r+1|n+s}$ we denote the same. 
To ease notation, in this section we often denote $$F=F_{n+1|n}= F_{n+1,1^r|n,1^s}, \quad f=f_{n+1|n}=f_{n+1,1^r|n,1^s},\quad \tilde{f}=\tilde{f}_{n+1|n}=\tilde f_{n+1,1^r|n,1^s}.$$ 
As before (\ref{eq:nilpotents square}), $F$ is a square of an odd nilpotent $f$ (or $\tilde f$), so we have the corresponding SUSY $\cW$-algebra $\cW^k_{\cN=1}(\gs\gl_{r+n+1|n+s}, f)$.
By Theorem \ref{thm:ordinary vs SUSY} the ordinary and SUSY $\cW$-superalgebras are related by
\begin{equation}\label{eq:hook type SUSY and not}
    \cW^k_{\cN=1}(\gs\gl_{n+r+1|n+s}, f)\cong\cW^k(\gs\gl_{n+r+1|n+s}, F)\otimes \cF(\fr{gl}_{r|s}) \cong \cW^k_{\cN=1}(\gs\gl_{n+r+1|n+s}, \tilde f).
\end{equation}
We call such algebras 
of $\cN=2$ \textit{hook-type} and $\cW^k_{\cN=1}(\gs\gl_{n+r+1|n+s}, f)$ of \textit{SUSY $\cN=2$ hook-type}.
Next we record some structural features of these algebras.

The trivial $r|s$ block gives rise to an affine subVOA $V^{k+1}(\fr{gl}_{r|s})$ associated to the elements 
\begin{equation}\label{eq:affine gen}
\begin{aligned}
    a_{i,j} :=& e_{2n+i+1\ 2n+1+j} \in\gg^F_0,\quad 1\leq i\neq j\leq r+s, \\
    a_{i,i} :=& e_{2n+1+i\ 2n+1+i}-(-1)^{|i|}\sum_{m=1}^{2n+1}e_{m\, m}\in \gg_0^F, \quad 1\leq i \leq r+s,
\end{aligned}
\end{equation}
where $|i|$ is as \eqref{eq: parity of extension index}. The off-diagonal blocks in Figure \ref{eq: hook type block matrix} give rise to the elements
\begin{equation}\label{eq: Q and P}
\begin{array}{ll}
    p^{i,\bot}:&=e_{2n+1+i\ 2} \in \fr{g}^F_{-\frac{n-1}{2}},\quad p^{i,+}:=e_{2n+1+i\ 1}\in \fr{g}^F_{-\frac{n}{2}},\\
    q^{i,\bot}:&=e_{2n\ 2n+1+i}\in \fr{g}^F_{-\frac{n-1}{2}},\quad q^{i,-}:=e_{2n+1\ 2n+1+i}\in \fr{g}^F_{-\frac{n}{2}},
\end{array}\quad1\leq i \leq r+s.
\end{equation}
Each row in the lower off-diagonal block forms an irreducible representation of $\mathfrak{osp}_{1|2}$ subalgebra containing $F$ and $f$, where the adjoint action of $f$ is given by going right in each row. Thus, one obtains the elements $p^{i, \bot}$ and $p^{i,+}$ in \eqref{eq: Q and P} from the two left columns of the lower off-diagonal block. Similarly, from the upper off-diagonal block, one obtains $q^{i,\bot}$ and $q^{i,-}$ in \eqref{eq: Q and P} from the two bottom rows.

It is immediate that the corresponding ordinary $\mathcal{N}=2$ hook-type $\cW$-algebra has a generating type 
        \begin{equation}
\label{hooktype:gentype} \cW\bigg(1^{r^2+s^2+1},2^2, \dots, n^2, n+1, \bigg(\frac{n+1}{2}\bigg)^{2s}, \bigg(\frac{n+2}{2}\bigg)^{2r}; 1^{2rs},\bigg(\frac{3}{2}\bigg)^2,\dots, \bigg(\frac{2n+1}{2}\bigg)^2,\bigg(\frac{n+1}{2}\bigg)^{2r},\bigg(\frac{n+2}{2}\bigg)^{2s}\bigg).\end{equation}

Thanks to Proposition \ref{prop: generators of W-alg}, the following Lie bracket relations in $\fr{g}^F$,
\begin{equation}\label{eq:affine on p and q}
    \begin{aligned}
        [a_{i j},p^{j,\bot}]&=p^{i,\bot},\quad  [a_{i j},q^{i,\bot}]=(-1)^{p(a_{i,j})p(q^{i,\bot})+1}q^{j,\bot}, \\ 
        [a_{i  j},p^{j,+}]&=p^{i,+}, \quad [a_{i j},q^{i,-}]=(-1)^{p(a_{i,j})p(q^{i,-})+1}q^{j,-},
    \end{aligned}
\end{equation}
imply that $\{\omega_{p^{i,\bot}},\omega_{q^{i,\bot}}|\ 1\leq i\leq r+s\}$ transform as $\mathbb{C}^{s|r} \oplus (\mathbb{C}^{s|r})^*$, and $\{\omega_{p^{i,+}},\omega_{q^{i,-}}:1\leq i\leq r+s\}$ transform as $\mathbb{C}^{r|s} \oplus (\mathbb{C}^{r|s})^*$ under $\fr{gl}_{r|s}$\footnote{In particular fields $\{\omega_{p^{i,\bot}},\omega_{q^{i,\bot}}|\ 1\leq i\leq r\}$ are odd and $\{\omega_{p^{i,\bot}},\omega_{q^{i,\bot}}|\ r+1\leq i\leq r+s\}$ are even, and opposite parities for $\{\omega_{p^{i,-}},\omega_{q^{i,+}}|\ 1\leq i\leq r+s\}$.}.
The affine subVOA decomposes as $V^{k+1}(\gg\gl_{r|s}) \cong V^{k+1}(\gs\gl_{r|s}) \otimes \cH^\ell$, where the Heisenberg algebra is generated by $\omega_v$ for
\begin{equation} \label{eq: heisenberg in gl affine}
    v=-\frac{r-s}{r-s+1}\sum_{i=1}^{2n+1}e_{i i}+\frac{1}{{r-s+1}}\sum_{i=2n+2}^{2n+r+s+1}e_{i i},\quad \ell=-\frac{(r-s)(k+r-s+1)}{(r-s+1)},
\end{equation}
and the OPE relations with the fields in \eqref{eq: Q and P} are
\begin{equation}
\begin{split}
      \omega_v(z)\omega_{p^{i,\bot}}(w)&\sim \omega_{p^{i,\bot}}(w)(z-w)^{-1},\quad \omega_v(z)\omega_{q^{i,\bot}}(w)\sim -\omega_{q^{i,\bot}}(w)(z-w)^{-1},\\
       \omega_v(z)\omega_{p^{i,+}}(w)&\sim\omega_{p^{i,+}}(w)(z-w)^{-1},\quad \omega_v(z)\omega_{q^{i,-}}(w)\sim -\omega_{q^{i,-}}(w)(z-w)^{-1}.
\end{split}
\end{equation} 
When $r=s-1$ then the Heisenberg field decouples and commutes with itself and the off-diagonal generators.

By \cite{KW22} and the formula \eqref{eq: W-alg central charge}, one can let $-\frac{1}{k+r-s+1}\omega_F$ be the conformal vector of the $\cW$-algebra $\cW^k(\gg, F)$ with central charge
\begin{equation} \label{eq: cc of hook type}
\begin{split}
c= \frac{(r-s)(k(r-s)+r-s+1)}{k+r-s+1} -\frac{3 (k + n + k n + n r - n s) (n + k n - r + n r + s - n s)}{k+r-s+1},
\end{split}
\end{equation}
where we note that the first term is the central charge of the $\fr{gl}_{r|s}$ Sugawara vector. It is tempting to think that $\cW^k(\gg,F)$ has $\cN=2$ structure with respect to Kac-Roan-Wakimoto conformal vector, however according to our definition at the end of Section \ref{sect: N=2}, it does \textit{not}; the central charge \eqref{eq: cc of hook type} does not appear in the OPE $\omega_u(z)\omega_u(w)$ as expected, while only the second term in \eqref{eq: cc of hook type} appears\footnote{We will soon see in Theorem \ref{thm: Y alg}, that these fields generate a $\cN=2$ superconformal structure using the modified field $H$ for the associated $\cN=2$ $Y$-algebra.}. 

We have the following Lie bracket relations in $\gg^F$,
\begin{equation}\label{eq: f on p and q}
    \begin{gathered}
[\T{f}^{+},q^{i,\bot}]=[\T{f}^{-},p^{i,\bot}]=0,\quad  [\T{f}^{+},p^{i,\bot}]=-(-1)^{p(p^{i, \bot})}p^{i,+},\quad [\T{f}^{-},q^{i,\bot}]=q^{i,-},\\
 [u,p^{i,\bot}]=-(n+1)p^{i,\bot},\quad [u,p^{i,+}]=-np^{i,+}, \quad [u,q^{i, \bot}]=(n+1)q^{i,\bot}, \quad [u,q^{i,-}]=n q^{i,-}.
    \end{gathered}
\end{equation}
By conformal weight and $u$-grading restrictions, above Lie brackets translate into the obvious OPE relations. It is tempting to say that $\{\omega_{p^{i,\bot}},\omega_{p^{i,+}}\}$ and $\{\omega_{q^{i,\bot}},\omega_{q^{i,-}}\}$ transform as $\cN=2$ degenerate diamonds of weights 
$(\frac{n+1}{2}, n)$, and $(\frac{n+1}{2}, -n)$, however they do \textit{not}\footnote{We will soon see that upon considering an affine coset (\ref{eq: Yalg}), these will have a desired transformation property, albeit for different weights}, as $\cW^k(\gg,F)$ does not have $\cN=2$ structure.

Finally, we note that the automorphism $\sigma$ of the $\cN=2$ principal family (\ref{eq:automorphism,sl(n+1|n)}) extends to hook-type one by considering $\sigma=\tau\circ(-\textup{str})$, where $\tau$ is replaced by the conjugation with the block matrix $T\oplus \textup{Id}$ for $T$ in \eqref{eq: conjugation matrix}.
In addition to transforming (\ref{eq:auto principal W}), it exchanges off-diagonal fields up to constant multiple
\begin{equation}\label{eq: auto pq}
    \sigma:\begin{array}{ll}
         \omega_{p^{i,\bot}}\mapsto -(-1)^{|i|}(-\sqrt{-1})^{3n-1}\omega_{q^{i,\bot}}, & \omega_{p^{i,+}}\mapsto (-\sqrt{-1})^{3n}\omega_{q^{i,-}}, \\
         \omega_{q^{i,\bot}}\mapsto (-\sqrt{-1})^{3n-1} \omega_{p^{i,\bot}}, & \omega_{q^{i,-}}\mapsto (-1)^{|i|}(-\sqrt{-1})^{3n} \omega_{p^{i,+}}.
    \end{array}
\end{equation}
Moreover, we have $\sigma(V^{k+1}(\gg\gl_{r|s}))\subset V^{k+1}(\gg\gl_{r|s})$, since $\tau$ acts as an identity on $\gg\gl_{r|s}$. 
{Note that the automorphism $\sigma$ is not of order $2$ in the $\cN=2$ hook-type $\cW$-algebras}.
After restricting $\sigma$ to the subalgebra \eqref{eq: Yalg}, we obtain an automorphism of order $2$.

\subsection{$\cN=2$ $Y$-algebras}
For $n\geq 1$ and $r,s \geq 0$, set $\psi = k+r-s+1$. We define $\cN=2$ $Y$-algebra as the SUSY affine coset
\begin{equation}\label{eq: Yalg}
    \cC^{\psi}_{\cN=2}(n,r|s):=
    \begin{cases}
        \textup{Com}(V^{k+1}_{\mathcal{N}=1}(\mathfrak{gl}_{r|s}),\mathcal{W}^k_{\mathcal{N}=1}(\gs\gl_{n+r+1|n+s},f_{n+1|n})) &r\neq s-1,\\
        \textup{Com}(V^{k+1}_{\mathcal{N}=1}(\mathfrak{sl}_{r|r+1}),\mathcal{W}^k_{\mathcal{N}=1}(\gp\gs\gl_{n+r+1|n+r+1},f_{n+1|n}))^{U(1)} &r= s-1,
    \end{cases}
\end{equation}
by analogy with the $Y$-algebras in type $A$ \cite{GR}.
In the case $r=s-1$, $U(1)$ action arises as outer automorphism of $\fr{psl}_{n+r+1|n+r+1}$, while the Heisenberg field arising from (\ref{eq: heisenberg in gl affine}) degenerates, so the coset would have incorrect generating type\footnote{If we had worked with $\fr{gl}_{n+r+1|n+s}$ instead of $\fr{sl}_{n+r+1|n+s}$, then this distinction would be unnecessary, and the generating type would be $\cW(1^2,2^2,\dotsb;(3/2)^2,(5/2)^2,\dotsb)$. }.
Equivalently, these algebras are isomorphic to the non-SUSY cosets
\[\cC^{\psi}_{\cN=2}(n,r|s)\cong
\begin{cases}
    \textup{Com}(V^{{k+1}}(\mathfrak{gl}_{r|s}),\mathcal{W}^k(\gs\gl_{n+r+1|n+s},F_{n+1|n}))& r \neq s-1,\\
    \textup{Com}(V^{{k+1}}(\mathfrak{sl}_{r|r+1}),\mathcal{W}^k(\fr{psl}_{n+r+1|n+r+1},F_{n+1|n}))^{U(1)}& r = s-1.
\end{cases}\]
We consider this family as the natural generalization of $\cN=2$ principal $\cW$-algebras, since they satisfy analogues of Propositions \ref{prop:sl(n|n-1), N=2}, \ref{cor:diamond}, \ref{prop:automorphism}.

In \cite{CL1}, the first and third authors introduced a method to study cosets of affine vertex algebras $V^k(\gg)$ inside larger vertex algebras $\cA^k$ when $\gg$ is a reductive Lie algebra. Under fairly general hypotheses that hold when $\cA^k$ is any $\cW$-(super)algebra, the coset has the same generating type as its large level limit, which is the $G$-orbifold of a free field algebra, where $G$ is the adjoint group of $\gg$. By \cite[Theorem 5.3]{CL3}, the analogous result holds when $\gg = \go\gs\gp_{1|2n}$. In order to describe the generating type of $\cC^{\psi}_{\cN=2}(n,r|s)$, we will extend this to the case when $\gg$ is any basic Lie superalgebra or $\gg\gl_{n|n}$; see Theorem \ref{thm:orbifoldlimit} in Appendix \ref{appendix:largerstructures}.

Additionally, it was shown in \cite[Corollary 3.6]{CL2} that for any $\cW$-superalgebra $\cW^k(\gg,F)$ with $\mathfrak{g}^{\natural}$ a reductive Lie algebra, the generators of $\cW^k(\gg,F)$ of weight $d>1$ can be assumed primary for the action of $\hat{\gg}^{\natural}$ when $k$ is generic. To describe $\cC^{\psi}_{\cN=2}(n,r|s)$, we will also extend this to the case when $\gg$ is any basic Lie superalgebra or $\gg\gl_{r|s}$; see Theorem \ref{cor:primary} in Appendix \ref{appendix:stronggenerators}.

We introduce the following notation.
\begin{equation}
\cO(r|s,n):=\begin{cases}
    \cS_{\T{even}}(r,n) \otimes \cO_{\T{odd}}(2s,n)& n \T{ odd},\\
    \cO_{\T{even}}(2r,n) \otimes \cS_{\T{odd}}(s,n)& n \T{ even}.
\end{cases}    
\end{equation}
The following generalizes Propositions \ref{prop:sl(n|n-1), N=2}, \ref{cor:diamond}, \ref{prop:automorphism}.

 \begin{theorem}\label{thm: Y alg}
     For $n\geq 1$ and $k\not=r-s+1$, $\cC^{\psi}_{\cN=2}(n,r|s)$ has the following features:
     \begin{enumerate}[$(\B{\textup{Y}}1)$]
         \item Let $u$ be as in \eqref{eq:basis of centralizer} and $v$ as in \eqref{eq: heisenberg in gl affine}.
         Then there is the $\cN=2$ superconformal vector
           \begin{equation} \label{crs}
     H:=-\sqrt{-1}\omega_{u}-\frac{\sqrt{-1}(r-s+1)}{k+r-s+1}\omega_v, \quad c = -\frac{3 (k + n + k n + n r - n s) (n + k n - r + n r + s - n s)}{k+r-s+1},
     \end{equation}
     with the remaining fields uniquely determined by (\ref{eq:N=2 diamond}),
     \[G^{+}=(\omega_{\textup{f}^+})_{(0)}H,\quad G^{-}=-(\omega_{\textup{f}^-})_{(0)}H,\quad L=-\frac{1}{k+1}\omega_{F}-L^{\fr{gl}_{r|s}}.\]
    
     \item The exists a non-minimal strong generating set $\{\omega^{i,\bot},\omega^{i,\pm},\omega^{i,\top}|\ i\geq 1\}$ with $\cN=2$ diamond structure, so that algebra is of type
     \begin{equation}\label{eq: non min Y}
         \cW\bigg(1,2^2,3^2,\dots, n^2, (n+1)^2,\dots ; \bigg(\frac{3}{2}\bigg)^2, \bigg(\frac{5}{2}\bigg)^2,\dots, \left(\frac{2n+1}{2}\right)^2, \left(\frac{2n+3}{2}\right)^2,\dots\bigg).
     \end{equation}
          \item The restriction of $\sigma$ to $\cC^{\psi}_{\cN=2}(n,r|s)$ is an automorphism of order two with the same action as in (\ref{eq:auto principal W}).
     \end{enumerate}

 \end{theorem}
 \begin{proof}
    To prove (1) it is sufficient to find a suitable correction to the generators $\T{Vir}^c_{\cN=2}$ subVOA, so that the modified fields commute with $V^{k+1}(\gg\gl_{r|s})$ and $L$ matches the conformal vector of the coset.
    A calculation shows that the modified Heisenberg field\footnote{Note that when $r=s-1$ then the $\omega_v$ contribution vanishes.}
     \[\omega_{u}+\frac{r-s+1}{k+r-s+1}\omega_v\]
    is contained in the coset, giving rise to the Heisenberg generator $H$ at level $\frac{1}{3}c$ via (\ref{crs}).
    Using the same argument as in Proposition \ref{prop:sl(n|n-1), N=2} one shows that there is an $\cN=2$ structure uniquely determined by (\ref{eq:N=2 diamond}).

    Now we prove (2).
    The remaining generators arising from the principal block (\ref{eq: as}) thanks to Lemma \ref{cor:primary}, can be corrected\footnote{Such corrections are not unique for higher conformal weight generators.} to commute with $V^{k+1}(\gg\gl_{r|s})$ for generic level $k$, which we continue to denote by $\{\omega^{i,\bot},\omega^{i,\pm},\omega^{i,\top}|\ 2\leq i\leq n\}$.
    Since the Lie brackets for the $n+1|n$ block remain the same\footnote{Although $u$ is corrected and is different from the one written in (\ref{eq:F and u}), the desired Lie bracket relations continue to hold, as $\fr{gl}_{r|s}$ commutes with $\fr{sl}_{n+1|n}$.} as in (\ref{eq:c2 u f and a}), these fields organize into $\cN=2$ diamonds by the same argument as in Proposition \ref{cor:diamond}.
    
    In the following we will consider large level limit of $\cW^k(\fr{sl}_{n+r+1|n+s},F)$.
    We note that generators coming from the $n+1|n$ block give rise to a free field limit
    \[\cA:=\cH\otimes\cO_{\T{even}}(1,2n+2) \otimes\bigotimes_{i=2}^{n}\cS_{\T{odd}}\left(1,2i+1\right)\otimes \cO_{\T{even}}(2,2i).\]
    Taken together with the affine $V^{k+1}(\fr{gl}_{r|s})$ and the off-diagonal fields (\ref{eq:affine on p and q}), the large level limit of $\cW^k(\gs\gl_{n+r+1|n+s}, F)$ is therefore
$${\cA \otimes \cO(r^2+s^2|rs,2) \otimes \cO(s|r,n+1) \otimes \cO(r|s,n+2)},
$$
and by Theorem \ref{thm:orbifoldlimit}, the coset $\cC^{\psi}_{\cN=2}(n,r|s)$ has large level limit 
\begin{equation} \label{orbifoldlimit:Dns} 
\cA \otimes \bigg(\cO(s|r,n+1) \otimes \cO(r|s,n+2)\bigg)^{\text{GL}_{r|s}}.
\end{equation}
 The above orbifold has a filtration with associated graded algebra isomorphic to the classical invariant ring
\begin{equation} \label{invariantring} \left(\text{Sym} \bigoplus_{i=0}^{\infty} \left(\mathbb{C}^{r|s}_i \oplus (\mathbb{C}^{r|s}_i)^*\right) \bigotimes \bigwedge_{i= 0}^{\infty}  \left(\mathbb{C}^{s|r}_i \oplus (\mathbb{C}^{s|r}_i)^*\right)\right)^{\text{GL}_{r|s}},\quad \mathbb{C}^{r|s}_i\cong \mathbb{C}^{r|s}.\end{equation} 
Generators of \eqref{invariantring} are given Sergeev's first fundamental theorem of invariant theory for $\text{GL}_{r|s}$ \cite[Theorem 1.1]{SI}, which reduces to Weyl's first fundamental theorem in the case $r = 0$ or $s = 0$.
\begin{equation}\label{eq: Y gen}
    \begin{split}
    \omega^{n+d+1,\bot}_{\infty}:&= \sum_{i=1}^r:\!\partial^{d}\omega_{p^i,\bot}\omega_{q^i,\bot}\!:-\sum_{i=1}^s:\!\partial^{d}\omega_{p^{r+i,\bot}}\omega_{q^{r+i,\bot}}\!:\\
        \omega^{n+d+1,+}_{\infty}:&=  \sum_{i=1}^r:\!\partial^{d}\omega_{p^i,\bot}\omega_{q^i,+}\!:-\sum_{i=1}^s:\!\partial^{d}\omega_{p^{r+i,\bot}}\omega_{q^{r+i,+}}\!:,\\
        \omega^{n+d+1,-}_{\infty}:&=\sum_{i=1}^r:\!\partial^{d}\omega_{p^i,-}\omega_{q^i,\bot}\!:-\sum_{i=1}^s:\!\partial^{d}\omega_{p^{r+i,-}}\omega_{q^{r+i,\bot}}\!:,\\
        \omega^{n+d+1,\top}_{\infty}:&=\sum_{i=1}^r:\!\partial^{d}\omega_{p^i,+}\omega_{q^i,-}\!:-:\!\omega_{p^i,+}\partial^{d}\omega_{q^i,-}\!:-\sum_{i=1}^s:\!\partial^{d}\omega_{p^{r+i,+}}\omega_{q^{r+i,-}}\!:-:\!\omega_{p^{r+i,+}}\partial^{d}\omega_{q^{r+i,-}}\!:,
    \end{split}
\end{equation}
giving rise to strong generators for the orbifolds in \eqref{orbifoldlimit:Dns}, and an infinite strong generating set of type \[\cW\left( (n+1), (n+2)^2,\dots; \left(\frac{2n+3}{2}\right)^2, \left(\frac{2n+5}{2}\right)^2,\dots\right).\] 
Combining with the principal block generators we recover (\ref{eq: non min Y}).
These generators admit modifications, so that the leading term is the same as in (\ref{eq: Y gen}), and modified fields are elements of the $Y$-algebra.
Moreover, thanks to relations (\ref{eq: f on p and q}) and Proposition (\ref{prop: generators of W-alg}) it is easy to see that such modified fields can be chosen with the $\cN=2$ diamonds structure.

Lastly, we show (3).
Note that $\sigma$ fixes $\fr{gl}_{r|s}$ as a subspace, and so it extends to the coset.
Now recall that $\sigma$ when restricted to $\T{Vir}^{c}_{\cN=2}$ acts as (\ref{eq: N=2 automorphism}).
First, we consider the principal block generators.
Inductively, the fields $\{\omega^{i,\bot}|\ 1\leq i\leq n\}$ can be chosen so that $\sigma$ acts as in (\ref{eq:auto principal W}), and the associated $\cN=2$ diamonds will span the rest of principal generators.
Next, we consider the bottoms $\omega^{n+d+i,\bot}$ of higher weight fields.
Since $\sigma$ exchanges off-diagonal fields as in (\ref{eq: auto pq}), we find that the leading term of the quadratic $\omega^{n+d+1,\bot}$, is left invariant up to a sign $(-1)^{n+d+1}$.
Inductively, as in the case of principal generators, one can form the modification of the field $\omega^{n+d+1,\bot}_{\infty}$ so that $\sigma$ acts as expected (\ref{eq:auto principal W}) and $\cN=2$ diamond structure is preserved.
\end{proof}

If $n=0$, \eqref{eq: Yalg} makes sense if we interpret $f_{1|0}$ as the zero (odd) nilpotent. Then $\cW^k_{\cN=1}(\gs\gl_{r+1|s},f_{1|0}) = V^k_{\cN=1}(\fr{sl}_{r+1|s})$ for $r \neq s-1$, and $\cW^k_{\cN=1}(\gp\gs\gl_{r+1|r+1},f_{1|0}) = V^k_{\cN=1}(\fr{psl}_{r+1|r+1})$, so \eqref{eq: Yalg} becomes
\begin{equation}\label{eq:parafermions}
    \cC_{\cN=2}^{\psi}(0,r|s):=\begin{cases}
        \T{Com}(V^{k+1}_{\cN=1}(\fr{gl}_{r|s}),V^{k}_{\cN=1}(\fr{sl}_{r+1|s})) & r\neq s-1,\\
        \T{Com}(V^{k+1}_{\cN=1}(\fr{sl}_{r|r+1}),V^{k}_{\cN=1}(\fr{psl}_{r+1|r+1}))^{U(1)} & r= s-1.
        \end{cases}
\end{equation}
which is the SUSY analogue of the {\it generalized parafermion} algebra $\text{Com}(V^k(\gg\gl_m), V^k(\gs\gl_{m+1}))$ introduced in \cite{Lin}.
Equivalently,
\begin{equation}\label{eq:GKO}
 \cC_{\cN=2}^{\psi}(0,r|s)\cong
 \begin{cases}
     \T{Com}(V^{k+1}(\fr{gl}_{r|s}),V^{k}(\fr{sl}_{r+1|s})\otimes \cE(r)\otimes \cS(s))& r\neq s-1\\
     \T{Com}(V^{k+1}(\fr{sl}_{r|r+1}),V^{k}(\fr{psl}_{r+1|r+1})\otimes \cE(r)\otimes \cS(r+1))^{U(1)} & r=s-1,
 \end{cases}
\end{equation}
which is the analogue of the GKO cosets \eqref{intro:D0m}. For this family, the principal block generators do not exist, and $\cN=2$ structure arises via invariant theory. Observe that if $n=0$, our basis for $\mathfrak{g} = \gs\gl_{n+1+r|n+s}$ still makes sense. The elements $p^{i,\bot}, q^{i,\bot}$ given by \eqref{eq: Q and P} no longer exist, but the elements $\{p^{i,+}, q^{i,-}|\ 1 \leq i \leq r + s\}$ still make sense and lie in $\mathfrak{g}^F_0$. We still use the notation $\omega_{p^{i,+}}, \omega_{q^{i,-}}$ as before, even though these elements have weight $1$. Finally, note that the automorphism $\sigma$ still makes sense on $V^{k}(\fr{sl}_{r+1|s})$ and its action on  $\omega_{p^{i,+}}, \omega_{q^{i,-}}$ is still given by \eqref{eq: auto pq}.

\begin{theorem}\label{thm: GKOs}
Let $k\neq r-s+1$. The family of algebras (\ref{eq:parafermions}) has the following features:
\begin{enumerate}[$(\B{\textup{GKO}}1)$]
    \item There exists $\cN=2$ structure with $\cN=2$ superconformal vector 
\begin{equation}
    H=\frac{\sqrt{k}}{\sqrt{-k-r+s-1}}\left(\sum_{i=1}^{r}:\!b^ic^i\!:-\sum_{i=1}^{s}:\!\beta^i\gamma^i\!:\right),\quad c=\frac{3 k (r-s)}{k+r-s+1},
\end{equation}
with the fields $G^{\pm}$ given up to scaling by
\begin{equation*}
    G^+ := \sum_{i=1}^r:\!\omega_{p^{i,+}} c^i\!: - \sum_{i=1}^{s}:\! \omega_{p^{r+i,+}} \gamma^i\!:,\quad G^-:= \sum_{i=1}^r:\! \omega_{q^{i,-}} b^i\!: - \sum_{i=1}^{s}:\! \omega_{q^{r+i,-}}  \beta^i\!:.
\end{equation*}
    \item There exists a (possibly non-minimal) set of strong generators invariants have $\cN=2$ diamond structure, giving rise to the generating type
    \begin{equation}\label{eq: non min GKO}
        \cW\bigg(1,2^2,3^2,\dotsb; \bigg(\frac{3}{2}\bigg)^2, \bigg(\frac{5}{2}\bigg)^2,\dots\bigg).
    \end{equation}
    \item The automorphism $\sigma$ on $V^{k}(\fr{sl}_{r+1|s})$ extends to an automorphism $\sigma$ on $V^{k}(\fr{sl}_{r+1|s})\otimes \cE(r) \otimes \cS(s)$ by
    \[ \beta^i\mapsto \gamma^i,\quad \gamma^i\mapsto -\beta^i, \quad b^i\mapsto c^i,\quad c^i\mapsto b^i,\] and similarly for $V^{k}(\fr{psl}_{r+1|r+1})\otimes \cE(r) \otimes \cS(r+1)$ in the case $s = r+1$.
 This induces an automorphism of order two on the algebras (\ref{eq:parafermions}), in agreement with (\ref{eq:auto principal W}).
\end{enumerate}
\end{theorem}

\begin{proof}
Proving (1) is a straightforward free field algebra computation that we omit.
Proof of (2) and (3) is the same as in Theorem \ref{thm: Y alg}.
\end{proof}

By \cite[Theorem 4.1]{CL2}, if $\gg$ is a simple Lie superalgebra, $F \in \gg$ is a nilpotent element, and $\ga \subseteq \gg^{\natural}$ is a reductive Lie subalgebra, the coset of $\cW^k(\gg,F)$ by the affine subVOA corresponding to $\ga$ is strongly finitely generated. We expect that the strong generating sets \eqref{eq: non min Y} and \eqref{eq: non min GKO} for $\cC^{\psi}_{\cN=2}(n, r|s)$ are not minimal, and that these vertex algebras are strongly finitely generated. The ideal of relations among the generators of the invariant ring \eqref{invariantring} are given by Sergeev's second fundamental theorem of invariant theory for the standard representation of $\text{GL}_{r|s}$ \cite[Theorem 2.2]{SII}. These relations are certain linear combinations of products of determinants with appropriate signs, and it is straightforward to check that the relation of minimal conformal weight has even parity and conformal weight given by
\begin{equation}\label{sing vectors Y}
\zeta(n,r|s):=\begin{cases}
   (1 + s) (2 + n + r + n r + r s), \quad &r \geq s,\\
    (1 + r) (1 + n + 2 s + n s + r s),\quad &r < s.
\end{cases}
\end{equation}
We expect, but do not prove, that these relations give rise to {\it decoupling relations} for the field $\omega^{\zeta(n,r|s), \bot} \in \cC^{\psi}_{\cN=2}(n, r|s)$; in other words, in the corresponding normally ordered relation, the coefficient of $\omega^{\zeta(n,r|s), \bot}$ is nonzero. If this is the case, it is easy to construct higher decoupling relations inductively for all fields $\{\omega^{i, \bot}, \omega^{i, \pm} ,\omega^{i, \top}|\ i \geq \zeta(n,r|s)\}$. This motivates the following.

\begin{conjecture}\label{conj: gen type}
$\cC^{\psi}_{\cN=2}(n,r|s)$ has minimal strong generating type
\begin{equation}\label{gen type Yalg}\cW\bigg(1,2^2, 3^2,\dots,  (\zeta(n,r|s)-1)^2, \zeta(n,r|s);  \bigg(\frac{3}{2}\bigg)^2, \bigg(\frac{5}{2}\bigg)^2,\dots, \bigg(\zeta(n,r|s)-\frac{1}{2}\bigg)^2\bigg).\end{equation}
\end{conjecture}

Even though $\cW^k(\gs\gl_{n+1+r|n+s}, F)$ is simple for generic $k$ \cite[Theorem 3.6]{CL2}, $\cC^{\psi}_{\cN=2}(n,r|s)$ is {\it not} simple except in the cases $r = 0$ or $s=0$; simplicity in these cases follows from \cite[Theorem 4.1]{ACK}. For this reason, we consider these cases to be most important and we introduce the notation
\begin{equation}\label{def:DandEfamilies}
\cE^{\psi}_{\mathcal{N}=2}(n,r) :=\cC^{\psi}_{\cN=2}(n,r|0),\quad \cD^{\psi}_{\mathcal{N}=2}(n,s):=\cC^{\psi}_{\mathcal{N}=2}(n,0|s+1).
\end{equation}

\begin{proposition}
 \label{lem:Dnsgenerators}
Conjecture \ref{conj: gen type} is true for the families $\cD^{\psi}_{\mathcal{N}=2}(n,s)$ and $\cE^{\psi}_{\mathcal{N}=2}(n,r)$.
\end{proposition}

\begin{proof}
In these cases, Sergeev's first and second fundamental theorems of invariant theory for the standard representation of $\text{GL}_{r|s}$ reduces to Weyl's first and second fundamental theorems for the standard representation of the general linear group \cite{W}. The relations are now ordinary determinants with appropriate signs, and the conformal weight of the relation of minimal weight is given by \eqref{sing vectors Y}. As usual, there is a normally ordered relation among the generators and their derivatives of $\cE^{\psi}_{\mathcal{N}=2}(n,r)$ and $\cD^{\psi}_{\mathcal{N}=2}(n,s)$ whose leading term is this classical relation. The lower order terms are quantum corrections, and one can show that the lowest weight relation is a decoupling relation for $\omega^{\zeta(n,r|0),\bot}$ or $\omega^{\zeta(n,0,s+1),\bot}$, respectively, by finding a recursive formula for the coefficient of this term. This argument was introduced by the third authors in a series of papers \cite{LHeis,Lin2,LinKWY}, and was used extensively in \cite{CL2}, so the details are omitted.
\end{proof}

\section{Universal $2$-parameter vertex algebra $\WNtwo$} \label{sect:main}
In this section we will construct the freely generated \textit{universal 2-parameter} vertex algebra $\WNtwo$ of \textit{infinite} minimal strong generating type \begin{equation} \label{eq: n=2 inf} \cW\bigg(1, 2^2, 3^2, 4^2, \dots;\Big(\frac{3}{2}\Big)^2,\Big(\frac{5}{2}\Big)^2,\Big(\frac{7}{2}\Big)^2,\dots\bigg),\end{equation}
defined over the ring $R$ which is the localization of $\mathbb{C}[c,\lambda]$ along the multiplicatively closed set generated by
\begin{equation}\label{localization set}
   \{c,c-1,c+3\}.
\end{equation}

\subsection{Set-up} \label{setup:Uproperties}
Motivated by Theorems \ref{thm: Y alg} and \ref{thm: GKOs}, we make the following assumptions:
\begin{enumerate}[$(\B{\textup{U}}1)$]\label{list:features of WN2}
	\item $\WNtwo$ has $\cN=2$ structure\footnote{Given the generating type (\ref{eq: n=2 inf}), the existence of $\cN=2$ structure generically is forced from Jacobi identities.} that we denote by
    \begin{equation}
        W^1=\T{Span}_{\mathbb{C}}\{H,G^{\pm},L\},
    \end{equation}
    so that OPEs (\ref{eq:N=2 OPEs 1}-\ref{eq:N=2 OPEs 3}) hold.
      \item The strong generating set arising from type (\ref{eq: n=2 inf}) has $\cN=2$ diamond structure. We denote the diamonds by
      \begin{equation}\label{LCA:generators}
        W^N:=W_{\cN=2}^{(N,0)}=\T{Span}_{\mathbb{C}}\{W^{N,\bot},W^{N,\pm},W^{N,\top}\},
    \end{equation}
    consistent with notation (\ref{eq:N=2 diamond}).
    Moreover, the second diamond $W^2$ is $\cN=2$ primary\footnote{Given the generating type (\ref{eq: n=2 inf}), and feature (\B{U}1) the second diamond must be primary.}.
         \item There exists\footnote{In view of feature (\B{U}4), this property of the automorphism follows from requiring that natural automorphism of $\T{Vir}^c_{\cN=2}$ from (\ref{eq:sigma}) satisfies $\sigma:W^{2,\bot}\mapsto W^{2,\bot}$.} an automorphism of order two
        \begin{equation}\label{eq:aut}
            \sigma: W^{N,\bot}\mapsto (-1)^N W^{N,\bot},\quad W^{N,\pm} \mapsto(-1)^{N+1} W^{N,\mp},\quad W^{N,\top}\mapsto (-1)^{N+1} W^{N,\top}.
        \end{equation}
         \item $\WNtwo$ is weakly generated by the first two diamonds $W^1$ and $W^2$. 
         Specifically, we make the following weak generation hypothesis
        \begin{equation}\label{eq:raise}
            W^{2,\top}_{(1)}W^{N,\bot}=W^{N+1,\bot}, \quad N\geq 2.
        \end{equation}
        \end{enumerate}

\vskip 5mm

    Note that properties $(\B{U}1-\B{U}3)$ are already established for the $Y$-algebras (\ref{eq: Yalg}) and (\ref{eq:parafermions}), while the last property $(\B{U}4)$ is not yet known.
    We prove in Section \ref{sec:one param q} that this property holds for generic values of level parameter $k$ for these families. Presently, it is an essential assumption for us to prove the universality of 2-parameter algebra $\WNtwo$.

	We will construct the OPE algebra of $\WNtwo$ with strong generators \eqref{LCA:generators}, such that identities (\ref{conformal identity}-\ref{quasi-derivation}) are imposed, and the Jacobi identities (\ref{Jacobi}) hold as a consequence of (\ref{conformal identity}-\ref{quasi-derivation}) alone. 
    Equivalently, there exist no nontrivial Jacobiators (Sec. 3, \cite{Lin}).
$\WNtwo$ will be the universal enveloping vertex algebra of a \textit{nonlinear} Lie conformal algebra $\nlcalg$ with generators \eqref{LCA:generators} and conformal weight grading 
\begin{equation} \label{eq:grading}
    \Delta(W^{N,\bot})=N,\quad \Delta(W^{N,\pm})=N+\frac{1}{2},\quad \Delta(W^{N,\top})=N+1,
\end{equation}
 in the sense of DeSole and Kac \cite{DSKI}.
Therefore, except for the above postulates $(\B{U}1-\B{U}4)$, we posit that the remaining OPEs have the \textit{most general} form compatible with the conformal weight and Heisenberg charge gradations,
    \begin{equation}\label{eq:OPEgeneralG}
		W^{N,\alpha}(z)W^{M,\beta}(w)\sim \sum_{r=0}^{\Delta(W^{N,\alpha})+\Delta(W^{M,\beta})-1}\sum_{\Omega \in PBW_r} w^{W^{N,\alpha},W^{M,\beta}}_{\Omega}\Omega (w) (z-w)^{-r-1},\quad \alpha,\beta \in \{\bot,\pm,\top\},
	\end{equation}
    where $\Omega$ denotes a PBW monomial of conformal weight $\Delta(W^{N,\alpha})+\Delta(W^{M,\beta})-r-1$.
    Between any two distinct diamonds $W^N$ and $W^M$ there are 16 OPEs to consider (and 8 between those arising out of the same diamond $W^N(z)W^N(w)$).   
    We use the action of zero modes $G^+_{(0)}$ and $G^{-}_{(0)}$ to express all such OPEs in terms of just 3, which we chose to be of lowest conformal weight possible: $W^{N,\bot}(z)W^{M,\bot}(w)$, $W^{N,\bot}(z)W^{M,+}(w)$, $W^{N,\top}(z)W^{M,\bot}(w)$ with $N\leq M$.
    The remaining interactions can be determined by using automorphism $\sigma$ with the following relations\footnote{It is possible to do better and reduce it to even less free parameters, for instance as in \cite{Blu}, Eqs. 2.26-2.28. However, this redundancy still allows us to solve the problem.},
\begin{equation}\label{eq:G descendend}
    \begin{split}
        W^{n,\bot}_{(r)}W^{m,+}=&-W^{m,\bot}_{(r)}W^{n,+}-G^+_{(0)}W^{n,\bot}_{(r)}W^{m,\bot},\\
        W^{n,\top}_{(r)}W^{m,\bot}=&W^{n,\bot}_{(r)}W^{m,\top}+G^+_{(0)}W^{n,-}_{(r)}W^{m,\bot}+G^-_{(0)}W^{n,\bot}_{(r)}W^{m,+}\\
       &+G^+_{(0)}G^-_{(0)}W^{n,\bot}_{(r)}W^{m,\bot}+\frac{1}{2}\partial W^{n,\bot}_{(r)}W^{m,\bot},\\
       W^{n,+}_{(r)}W^{m,-}=&W^{n,\top}_{(r)}W^{m,\bot}+G^-_{(0)}W^{n,\bot}_{(r)}W^{m,+}-G^+_{(0)}G^-_{(0)}W^{n,\bot}_{(r)}W^{m,\bot}\\
       &+\frac{1}{2}\partial W^{n,\bot}_{(r)}W^{m,\bot} +\frac{r}{2} W^{n,\bot}_{(r-1)}W^{m,\bot},\\
        W^{n,+}_{(r)}W^{m,+}=&-G^+_{(0)}W^{n,\bot}_{(r)}W^{m,+},\\
       W^{n,\top}_{(r)}W^{m,+}=&-G^+_{(0)}W^{n,\top}_{(r)}W^{m,\bot}-\frac{r}{2}G^-_{(0)}W^{n,+}_{(r-1)}W^{m,\bot},\\
       W^{n,\top}_{(r)}W^{m,\top}=&G^+_{(0)}G^-_{(0)}W^{n,\top}_{(r)}W^{m,\bot}-\frac{1}{2}\partial W^{n,\top}_{(r)}W^{m,\bot}\\
       &+\frac{r}{2}W^{n,+}_{(r-1)}W^{m,-}-\frac{r}{2}W^{n,-}_{(r-1)}W^{m,+}-\frac{r(r-1)}{4} W^{n,\bot}_{(r-2)}W^{m,\bot},
    \end{split}
\end{equation}
which can be read from the Figure \ref{fig:diamonds}.

    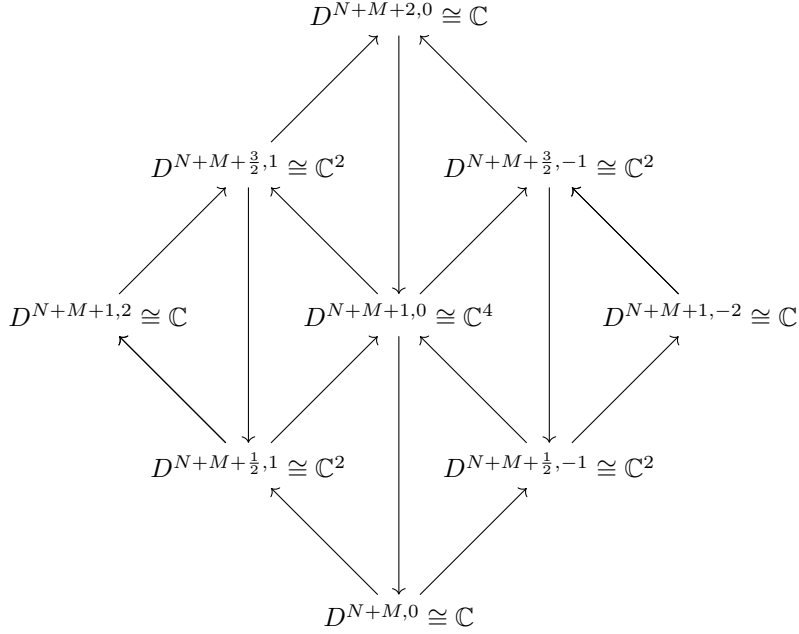
\begin{figure}
\centering
\begin{tikzpicture}
% Diamond shape nodes
\node (top-top) at (0, 4) {$D^{N+M+2,0}\cong \mathbb{C}$};
\node (top-minus) at (2, 2) {$D^{N+M+\frac{3}{2},-1}\cong \mathbb{C}^2$};
\node (top-plus) at (-2, 2) {$D^{N+M+\frac{3}{2},1}\cong \mathbb{C}^2$};
\node (minus-minus) at (4, 0) {$D^{N+M+1,-2}\cong \mathbb{C}$};
\node (center-center) at (0, 0) {$D^{N+M+1,0}\cong \mathbb{C}^4$};
\node (plus-plus) at (-4, 0) {$D^{N+M+1,2}\cong \mathbb{C}$};
\node (bottom-minus) at (2, -2) {$D^{N+M+\frac{1}{2},-1}\cong \mathbb{C}^2$};
\node (bottom-plus) at (-2, -2) {$D^{N+M+\frac{1}{2},1}\cong \mathbb{C}^2$};
\node (bottom-bottom) at (0, -4) {$D^{N+M,0}\cong \mathbb{C}$};
% Arrows and labels
\draw[->] (bottom-bottom) -- (bottom-minus) node[midway, right] {};
\draw[->] (bottom-bottom) -- (bottom-plus)  node[midway, left] {};
\draw[->] (bottom-plus) -- (plus-plus) node[midway, left] {};
\draw[->] (bottom-plus) -- (plus-plus) node[midway, left] {};
\draw[->] (bottom-plus) -- (center-center) node[midway, left]{};
\draw[->] (bottom-minus) -- (minus-minus) node[midway, right]{};
\draw[->] (bottom-minus) -- (center-center) node[midway, right]{};
\draw[->] (minus-minus) -- (top-minus) node[midway, right]{};
\draw[->] (plus-plus) -- (top-plus) node[midway, left]{};
\draw[->] (minus-minus) -- (top-minus) node[midway, right]{};
\draw[->] (center-center) -- (top-plus) node[midway, left]{};
\draw[->] (center-center) -- (top-minus) node[midway, right]{};
\draw[->] (top-plus) -- (top-top) node[midway, left]{};
\draw[->] (top-minus) -- (top-top) node[midway, right]{};
\draw[->] (top-top) -- (center-center) node[midway, right]{};
\draw[->] (top-plus) -- (bottom-plus) node[midway, right]{};
\draw[->] (top-minus) -- (bottom-minus) node[midway, right]{};
\draw[->] (center-center) -- (bottom-bottom) node[midway, right]{};
\end{tikzpicture}
\caption{Structure of the OPEs between two diamonds $W^{N}(z)W^M(w)$. 
The nodes $D^{\Delta,h}$ represent the vector subspaces spanned by products $W^{N,\alpha}(z)W^{M,\beta}(w)$ for $\alpha,\beta \in \{\bot,\pm,\top\}$ with $\Delta=\Delta(W^{N,\alpha})+\Delta(W^{M,\beta})$ and $h=h(W^{N,\alpha})+h(W^{M,\beta})$.
For instance, $D^{N+M,0}$ is spanned only by $W^{N,\bot}(z)W^{M,\bot}(w)$.
The down arrows represent an action of $H_{(1)}$, the lines pointing up-left or up-right represent action of $G^+_{(0)}$ or $G^-_{(0)}$, respectively.} \label{fig:diamonds}
\end{figure}

We have to consider all OPEs \eqref{eq:OPEgeneralG}, which we organize by weight, that we set to be $N+M$. Let $S^n_r$ denote the set of all $r^{ \T{th} }$ products among generators of weight $n$. Further, write $S^n = \bigcup_{r\leq n} S^{n}_r$ and $S^{\leq n} = \bigcup_{m\leq n} S^m$. Lastly, let $J^m$ denote the set of all Jacobi identities $J_{r,s}(a,b,c)$ among generating fields $a,b,c$ of total weight exactly $m$, and $J^{\leq n} = \bigcup_{m\leq n} J^m$.

The outline of this section is as follows.
	\begin{enumerate}
		\item We begin with the base case computation. Here, we write down the most general OPEs in $S^{\leq 9}$ compatible with the $\mathcal{N}=2$ symmetry, and we impose vertex algebra relations (\ref{conformal identity}-\ref{quasi-derivation}) along with Jacobi identities $J^{\leq 9}$ to uniquely determine $S^{\leq 7}$ in terms of two parameters $c$ and $\lambda$. Furthermore, we are able to find two commuting $\cW_{\infty}(c_{\pm},\lambda_{\pm})$ subVOAs, see Proposition \ref{prop:base case prop}.
		\item Next, we compute an infinite set of structure constants in Proposition \ref{prop:structure constants} and describe an induction procedure which expresses all OPEs in $S^{n+1}$ in terms of $S^{\leq n}$ thanks to Jacobi relations in $J^{n+2}$ (\ref{lem:Induction2}).
		\item Lastly, we exhibit a family of vertex algebras with known characters to prove free generation of $\WNtwo$. Specifically, we use the family of $\cW$-superalgebras $\cW^k(\fr{sl}_{n+1|n})$, see Proposition \ref{PrincW:quot}.
	\end{enumerate}
	
    Before we begin, we would like to comment on a closely related work of Candu and Gaberdiel \cite{CanGab}, where authors have done a similar computation to the following subsection under the assumption that \textit{all} diamonds $W^N$ are $\T{Vir}_{\cN=2}^c$ \textit{primary} for $N\geq 2$.
    There, the authors have conjectured the existence and uniqueness of $\WNtwo$ as a $2$-parameter vertex algebra. 
    In this section and Appendix \ref{sect:primary structures}, we prove their conjecture.

\subsection{Base case computation}
In the ansatz (\ref{eq:OPEgeneralG}) we have already imposed the Heisenberg and conformal weight gradation arising from the zero mode $H_{(0)}$ of the Heisenberg generator, and the first mode $L_{(1)}$ of the Virasoro generator, together with relations (\ref{eq:G descendend}); equivalently, the Jacobi identities of the form $J_{1,r}(L,W^2,W^2)$ and $J_{0,r}(H,W^2,W^2)$ vanish.
    We proceed to impose the full superconformal symmetry constraints on OPEs $W^{2}(z)W^{2}(w)$; equivalently, we impose the Jacobi identities of the form $J(W^1, W^{2}, W^{2})$.
    Remarkably, the imposition of $\mathcal{N}=2$ constraints $J(W^1,W^2,W^2)$ expresses OPEs in $W^2(z)W^2(w)$ in terms of the scaling parameter $\omega$ and two parameters $c$ and $\lambda$, defined in the following.
    \begin{equation}\label{eq:22}
    \begin{split}
        W^{2,\bot}(z)W^{2,\bot}(w)\sim& \omega(z-w)^{-4}+\left(\frac{4\omega}{c-1}L-\frac{6\omega}{c(c-1)}:\!HH\!:+\lambda  W^{2,\bot}\right)(w)(z-w)^{-2}\\
        &+\left(\frac{2\omega}{c-1}\partial L-\frac{6\omega}{c(c-1)}:\!\partial HH\!:+\frac{1}{2}\lambda \partial W^{2,\bot}\right)(w)(z-w)^{-1},\\
        W^{2,\bot}(z)W^{2,+}(w)\sim& \frac{6\omega}{c}G^+(w)(z-w)^{-4}+\left(\frac{2\omega}{c-1}\partial G^+-\frac{6\omega}{c(c-1)}:\!HG^+\!:+\frac{1}{2}\lambda  W^{2,+}\right)(w)(z-w)^{-2}\\
        &+\Big(\frac{1}{3}W^{3,+}+\frac{4\lambda}{c+3}:\!G^+W^{2,\bot}\!:-\frac{2\lambda}{c+3}:\!HW^{2,+}\!:+\frac{(c+11)\lambda}{4(c+3)}\partial W^{2,+}\\
        +&\frac{4\omega}{c(c-1)}:\!LG^+\!:+\frac{\omega}{c(c-1)}:\!H\partial G^+\!:-\frac{9\omega}{c(c-1)}:\!\partial H G^+\!:-\frac{(c+1)\omega}{2c(c-1)}\partial^2 G^+\Big)(w)(z-w)^{-1},\\
         W^{2,\top}(z)W^{2,\bot}(w)\sim& \frac{6\omega}{c}H(w)(z-w)^{-4}+\frac{6\omega}{c}\partial H(w)(z-w)^{-3}+W^{3,\bot}(w)(z-w)^{-2}+\Big(\frac{2}{3}\partial W^{3,\bot}\\&+\frac{3\lambda}{c+3}:\!G^-W^{2,+}\!:+\frac{3\lambda}{c+3}:\!G^+W^{2,-}\!:-\frac{\lambda}{c+3}:\!H\partial W^{2,\bot}\!:+\frac{2\lambda}{c+3}:\!\partial HW^{2,\bot}\!:\\
        -&\frac{2\lambda}{c+3}\partial W^{2,\top}+\frac{8\omega}{c(c-1)}:\!L\partial H\!:-\frac{4\omega}{c(c-1)}:\!\partial L H\!:\Big)(w)(z-w)^{-1},
    \end{split}
  \end{equation}
    and the rest follow by applying zero modes $G^{+}_{(0)}$ and $G^{-}_{(0)}$ via (\ref{eq:G descendend}).

    More generally, relations $J(W^1,W^N,W^M)$ give rise to both linear and quadratic constraints among the yet undetermined structure constants over the field of rational functions $\mathbb{C}(c)$. 
    Upon solving these, one finds that the only undetermined constants are those attached to the $\mathcal{N}=2$ primaries arising in a given OPE\footnote{This is a nontrivial statement and we do not prove it in generality, as we only require for $N+M\leq 7$ which we verify by computer.}; importantly, any additional parameters of $\WNtwo$ must arise this way.
    Since our aim is to show that $\WNtwo$ is a 2-parameter VOA, it suffices to restrict our attention to these structure constants. 
    As in \cite{CanGab}, we write
    \begin{equation}\label{OPEs refined}
    \begin{split}
		 W^{N} \times  W^{M} =&\sum_{\Delta(\Omega)\leq N+M +1}w^{N,M}_{\Omega}W^{(\Omega)},
    \end{split}
	\end{equation}
    where $W^{(\Omega)}$ denote $\mathcal{N}=2$ primary fields and
    \begin{equation*}
        \omega^{N,M}_{\Omega}=\left\{\begin{array} {ll}
        \omega^{W^{N, \bot}, W^{M,+}}_{W^{\Omega, \bot}} & \textup{ if } H(W^{\Omega,\bot})\neq 0, \\
        \omega^{W^{N, \bot}, W^{M,\bot}}_{W^{\Omega, \bot}} & \text{ if }H(W^{\Omega,\bot})=0\ \textup{ and } \ \omega^{W^{N, \bot}, W^{M,\bot}}_{W^{\Omega, \bot}}\textup{ is nonzero,}\\
        \omega^{W^{N, \top}, W^{M,\bot}}_{W^{\Omega, \bot}} & \textup{otherwise}
        \end{array}
        \right.
    \end{equation*}
    for the coefficients $\omega$'s in \eqref{eq:OPEgeneralG}.
    For example, OPEs in $W^2(z)W^2(w)$ presented in (\ref{eq:22}) are succinctly expressed as
    \[W^{2} \times W^2 =\omega\B{1}+\lambda W^{2}+W^{(3)}, \quad w^{2,2}_{0}=\omega, \quad w^{2,2}_{2}=\lambda,\quad w^{2,2}_{(3)}=1,\]
    where the diamond $W^3$ is uniquely corrected to become $\mathcal{N}=2$ primary 
    \begin{align*}
        W^{(3),\bot}=&W^{3,\bot}-\frac{3(c-8)\lambda}{2(5c-12)}W^{2,\top}-\frac{21\lambda}{5c-12}:\!HW^{2,\bot}\!:+\frac{18 (4 c+3) \omega }{(c-1) c (c+6) (2 c-3)}:\!H^3\!:\\
        &-\frac{6 \left(8 c^2-9 c+36\right) \omega }{(c-1) c (c+6) (2 c-3)}:\!LH\!:-\frac{18 (5 c-12) \omega }{(c-1) (c+6) (2 c-3)}:\!G^+G^-\!:\\
        &+\frac{9 (5 c-12) \omega }{(c-1) (c+6) (2 c-3)}\partial L-\frac{6 \left(c^3-3 c^2-3 c-9\right) \omega }{(c-1) c (c+6) (2 c-3)}\partial^2 H.
    \end{align*}

    \begin{remark}
         Note that primary field exists generically, except for a finite set of algebraic values, e.g. in the example of $W^{(3),\bot}$, this expression exists $c\neq \frac{12}{5},\frac{3}{2},1,0,-6$. 
         However, upon solving further Jacobi relations, the expressions we obtain for $w^{n,m}_{\Omega}$ be such, that most of these denominators cancel. Specifically, all except $c,c-1,c+3$.
    \end{remark}  
    Thanks to this short computation, we conclude the following proposition, which is proved by direct calculation.
	\begin{proposition}\label{prop:Virasoros}
    Assume that $\omega \neq -\frac{1}{32} (c-1) \lambda ^2$. Then $\WNtwo$ contains two commuting Virasoro subalgebras generated by the fields $L_-$ and $L_+$ with central charges $c_-$ and $c_+$, respectively, which add up to the Virasoro field $T=L-\frac{3}{2c}:\!HH\!:$ of central charge $c_T=c-1$.
    Specifically, these admit the following form.
	\begin{equation}\label{eq:virs}
		\begin{split}
	L_{\mp}=&\frac{1}{2}T\pm\frac{\lambda }{\sqrt{ \lambda ^2+
   \frac{32\omega}{c-1} }} \left(\frac{1}{2} T-\frac{2}{\lambda}W^{2,\bot}\right), \quad c_{\mp}=\frac{1}{2}c_T\left(1\pm\frac{\lambda }{\sqrt{ \lambda ^2+
   \frac{32\omega}{c-1} }}\right).
		\end{split}
	\end{equation}
	\end{proposition}  

    \begin{remark}
    \
    \begin{itemize}
        \item    We may regard the parameter $\lambda$ as related to the conformal weight of diamonds $W^{2}$ with respect to the two commuting Virasoros $L_{\pm}$.
        \item The singularity at $c=1$ corresponds to the case when field $W^{2,\bot}$ becomes singular.
    \end{itemize}
    \end{remark}

    Next, we proceed to evaluate all OPEs in $S^5$, which are $W^2(z)W^3(w)$ and $W^1(z)W^{4}(w)$. 
    Jacobi identities $J(W^{2}, W^{2},W^{2})$ and $J(W^1,W^2,W^3)$ express all of the structure constants arising in OPEs $W^2(z)W^3(w)$ and $W^1(z)W^4(w)$ as rational functions in the central charge $c$, and polynomials in $\lambda$ and $\omega$.
    We have
    \begin{equation*}
        \begin{split}
        W^{1} \times  W^{4}  =&  w^{1,4}_{0} \B{1} + w^{1,4}_{2}  W^{2}+w^{1,4}_{(3)}  W^{(3)}+4  W^{(4)},\\
            W^{2} \times  W^{3}  =&  w^{2,3}_{0} \B{1} + w^{2,3}_{2}  W^{2}+w^{2,3}_{(3)}  W^{(3)}+W^{(4)}+w^{2,3}_{(2,2)}  W^{(2,2)}.
        \end{split}
    \end{equation*}
    In the above, we have two new primaries appearing in weight 4 with Heisenberg charge 0, which are uniquely specified by their bottoms
    \begin{align*}
           W^{(4),\bot}=&W^{4,\bot}-\frac{\left(4 c^2-15 c-126\right) \lambda }{(c+3) (7 c-18)}W^{3,\top}-\frac{3 (23 c+102) \lambda }{(c+3) (7 c-18)}:\!HW^{3,\bot}\!:+\dotsb,\\
     W^{(2,2),\bot}=&:\!W^{2,\bot}W^{2,\bot}\!:-\frac{2 c}{3 (7 c-18)}W^{3,\top}+\frac{6}{7 c-18}:\!HW^{3,\bot}\!:+\dotsb,
    \end{align*}
       where the omitted terms involve fields arising from diamonds $W^1,W^2$; the structure constants are as follows
    \begin{align*}
        w^{2,3}_{0} =&\frac{3}{2}\lambda\omega,\quad w^{2,3}_{2}=\frac{3 (3 c-23) \lambda ^2}{4 (c+3)}+\frac{6 (21 c-26) \omega }{(c-1) c},\quad  w^{2,3}_{(3)}=\frac{3 (c+5) \lambda }{2 (c+3)},\quad w^{2,3}_{(2,2)}=0,\\
        w^{1,4}_{0} =&\frac{27}{2}\lambda\omega,\quad  w^{1,4}_{2}=\frac{15 (3 c-23) \lambda ^2}{4 (c+3)}+\frac{30 (23 c-28) \omega }{(c-1) c},\quad \quad w^{1,4}_{(3)}=\frac{(5 c+21) \lambda }{c+3},\quad w^{1,4}_{(2,2)}=0.
    \end{align*}
            \begin{remark}
            Notation $W^{(2,2)}$ reflects the presence of monomial $:\!W^{2,\bot}W^{2,\bot}\!:$ in the above primary field.
        \end{remark}
     From this computation we conclude the following proposition, which is again proved by direct computation.
     \begin{proposition}\label{prop:weight3}
     Assume that $c\neq 0$ and $\omega\neq \frac{(c-1) (c+5) \lambda ^2}{4 (c+3)^2}$.
    Then $\WNtwo$ contains two fields $W^3_{\pm}$ of conformal weight three which are primary for the Heisenberg subVOA and respective Virasoro generators $L_{\pm}$, that is, 
    \[L_{\pm}(z)W_{\pm}^3(w) \sim 3 W^3_{\pm}(w)(z-w)^{-2} + \partial W^3_{\pm}(w)(z-w)^{-1},\quad L_{\pm}(z)W_{\mp}^3(w)\sim 0\sim H(z)W^3_{\pm}(w).\]
    Moreover, the above conditions specify fields $W^3_{\pm}$ up to a scaling constants $\omega_{\pm}$. Specifically, these admit the following form.
	\begin{equation}\label{eq:W3s}
		\begin{split}
	W^3_{\mp}=&{\omega_{\mp}}\left(W^{2,\top}+T^3\pm \sqrt{\lambda ^2+
   \frac{32\omega}{c-1}} Q^3\right),\\
   T^3=&-\frac{(c-1) (c+3) (c+7) \lambda }{6 \left((c-1) (c+5) \lambda ^2-4 (c+3)^2 \omega \right)}W^{3,\bot}+\frac{(c-1) \left(c^2+4 c-9\right) \lambda ^2+24 (c+3)^2 \omega }{c \left((c-1) (c+5) \lambda ^2-4 (c+3)^2
   \omega \right)}:\!HW^{2,\bot}\!:\\
   &+\frac{3 \lambda  \left((1-c) \lambda ^2+8 (c+3) \omega \right)}{2 \left((c-1) (c+5) \lambda ^2-4 (c+3)^2 \omega
   \right)}:\!G^{+}G^-\!:+\frac{\lambda  \left(9 (c-1) \lambda ^2+8 (c-2) (c+3) \omega \right)}{2 c \left((c-1) (c+5) \lambda ^2-4
   (c+3)^2 \omega \right)}:\!LH\!:\\
   &-\frac{3 \lambda  \left(3 (c-1) \lambda ^2+4 (c+1) (c+3) \omega \right)}{2 c^2 \left((c-1) (c+5) \lambda ^2-4
   (c+3)^2 \omega \right)}:\!H^3\!:+\frac{3 \lambda  \left((c-1) \lambda ^2-8 (c+3) \omega \right)}{4 \left((c-1) (c+5) \lambda ^2-4 (c+3)^2 \omega
   \right)}\partial L\\
   &+\frac{(c-1) \lambda  \left((c-9) \lambda ^2+2 (c-1) (c+3) \omega \right)}{4 c \left((c-1) (c+5) \lambda ^2-4
   (c+3)^2 \omega \right)}\partial^2 H,\\
  Q^3=&\frac{(c-1) (c+3)^2}{6 \left((c-1) (c+5) \lambda ^2-4 (c+3)^2 \omega \right)}W^{3,\bot}-\frac{(c-1) (c+3)^2 \lambda }{c \left((c-1) (c+5) \lambda ^2-4 (c+3)^2 \omega \right)}:\!HW^{2,\bot}\!:\\
   &-\frac{3 (c-1) \lambda ^2}{2 \left((c-1) (c+5) \lambda ^2-4 (c+3)^2 \omega \right)}:\!G^{+}G^-\!:+\frac{9 (c-1) \lambda ^2-8 (c+3)^2 \omega }{2 c \left((c-1) (c+5) \lambda ^2-4 (c+3)^2 \omega \right)}:\!LH\!:\\
   &\frac{3 \left(4 (c+3)^2 \omega -3 (c-1) \lambda ^2\right)}{2 c^2 \left((c-1) (c+5) \lambda ^2-4 (c+3)^2 \omega
   \right)}:\!H^3\!:+\frac{3 (c-1) \lambda ^2}{4 \left((c-1) (c+5) \lambda ^2-4 (c+3)^2 \omega \right)}\partial L\\
   &+\frac{(c-9) (c-1) \lambda ^2-2 (c-5) (c+3)^2 \omega }{4 c \left((c-1) (c+5) \lambda ^2-4 (c+3)^2 \omega
   \right)}\partial^2 H.
			\end{split}
	\end{equation}
	\end{proposition}

    Next, we proceed to evaluate all OPEs in $S^6$, which are $W^3(z)W^3(w)$, $W^2(z)W^{4}(w)$ and $W^1(z)W^{5}(w)$.
    As before, imposing the Jacobi identities $J(W^{2}, W^{2},W^{3})$ and $J(W^1,W^2,W^4)$ allows us to express all structure constants arising in $W^2(z)W^4(w)$ and $W^3(z)W^3(w)$ as rational functions in the central charge $c$, and polynomials in $\lambda$ and $\omega$. 
     We have
	\begin{equation}\label{eq:D5}
		\begin{split}
			\\
			 W^{3} \times  W^{3}  =&  w^{3,3}_{0} \B{1} + w^{3,3}_{2}  W^{(2)}+w^{3,3}_{(3)}  W^{(3)}+w^{3,3}_{(4)}  W^{(4)}+w^{3,3}_{(2,2)}  W^{(2,2)}\\
             &+w^{3,3}_{(5)}  W^{(5)}+w^{3,3}_{(\frac{5}{2},\frac{5}{2})}W^{(\frac{5}{2},\frac{5}{2})}+w^{3,3}_{(2,3)}W^{(2,3)}+w^{3,3}_{(2,\frac{7}{2})}W^{(2,\frac{7}{2})}+w^{3,3}_{(3,\frac{5}{2})}W^{(3,\frac{5}{2})},\\
             W^{2} \times  W^{4}  =&  w^{2,4}_{0}\B{1}+w^{2,4}_{2}  W^{(2)}+w^{2,4}_{(3)}  W^{(3)}+w^{2,4}_{(4)}  W^{(4)}+w^{2,4}_{(2,2)}  W^{(2,2)}\\
             &+W^{(5)}+w^{2,4}_{(\frac{5}{2},\frac{5}{2})}W^{(\frac{5}{2},\frac{5}{2})}+w^{2,4}_{(2,3)}W^{(2,3)}+w^{2,4}_{(2,\frac{7}{2})}W^{(2,\frac{7}{2})}+w^{2,4}_{(2,\frac{5}{2})}W^{(2,\frac{5}{2})},\\
              W^{1} \times  W^{5}  =&w^{1,5}_{0}  \B{1}+ w^{1,5}_{2}  W^{(2)}+w^{1,5}_{(3)}  W^{(3)}+w^{1,5}_{(4)}  W^{(4)}+w^{1,5}_{(2,2)} W^{(2,2)}+5W^{(5)},
		\end{split}
	\end{equation}
       In the above, we have two new primaries appearing in weight 5 with Heisenberg charge 0, and two more in weight $5\frac{1}{2}$ with Heisenberg charge $1$\footnote{In fact, there are 4 more primaries in weight $5\frac{1}{2}$, where 2 are Heisenberg charge $\pm1$. However, these are related by automorphism $\sigma$, so omit them in the description \eqref{eq:D5}}
       \begin{align*}
                W^{(5),\bot}=&W^{5,\bot}-\frac{\left(10 c^2-29 c-376\right) \lambda }{4 (c+3) (3 c-8)}W^{4,\top}-\frac{3 (17 c+98) \lambda }{(c+3) (3 c-8)}:\!HW^{4,\bot}\!:+\dotsb,\\
                W^{(2,3),\bot}=&:\!W^{2,\bot}W^{3,\bot}\!:-\frac{2}{3 c-8}:\!HW^{4,\bot}\!:-\frac{c}{6 (3 c-8)}W^{4,\top}+\dotsb,\\
                W^{(\frac{5}{2},\frac{5}{2}),\bot}=&:\!W^{2,+}W^{2,-}\!:+\dotsb,\quad       W^{(2,\frac{5}{2}),\bot}=:\!W^{2,\top}W^{2,+}\!:+\dotsb,\\
                W^{(2,\frac{7}{2}),\bot}=&:\!W^{2,\bot}W^{3,+}\!:-\frac{3}{2}:\!W^{3,\bot}W^{2,+}\!:-\frac{3}{8 (c+9)}:\!HW^{4,+}\!:+\frac{c+12}{8 (c+9)}\partial W^{4,+}+\dotsb,
       \end{align*}
       where the omitted terms involve fields arising in diamonds $W^1,W^2,W^3$; the structure constants are
        \begin{align*}
            w^{1,5}_{0}=&\frac{45 (3 c-23) \lambda ^2 \omega }{2 (c+3)}+\frac{900 (7 c-8) \omega ^2}{(c-1) c},\quad w^{1,5}_{2}=\frac{9 \left(15 c^2-86 c-585\right) \lambda ^3}{4 (c+3)^2}+\frac{36 (69 c-110) \lambda  \omega }{(c-1) c},\\
            w^{1,5}_{(3)}=&\frac{15 \left(37 c^2-188 c-1281\right) \lambda ^2}{8 (c+3)^2}+\frac{195 (17 c-24) \omega }{(c-1) c},\quad w^{1,5}_{(4)}=\frac{(9 c+47) \lambda }{c+3},\\
            w^{1,5}_{(2,2)}=&-\frac{12 (c+5) \lambda ^2}{(c+3)^2}-\frac{720 \omega }{(c-1) c},\quad w^{1,5}_{(2,\frac{7}{2})}=w^{1,5}_{(2,\frac{5}{2})}=0,\\
            w^{2,4}_{0}=&\frac{9 (3 c-23) \lambda ^2 \omega }{4 (c+3)}+\frac{90 (7 c-8) \omega ^2}{(c-1) c},\quad w^{2,4}_{2}=\frac{3 \left(15 c^2-86 c-585\right) \lambda ^3}{8 (c+3)^2}+\frac{6 (67 c-108) \lambda  \omega }{(c-1) c},\\
            w^{2,4}_{(3)}=&\frac{3 \left(31 c^2-160 c-1143\right) \lambda ^2}{8 (c+3)^2}+\frac{3 (175 c-256) \omega }{(c-1) c},\quad w^{2,4}_{(4)}=\frac{(2 c+13) \lambda }{c+3},\\
             w^{2,4}_{(2,2)}=&-\frac{12 \lambda ^2}{(c+3)^2}-\frac{192 \omega }{(c-1) c},\quad
            w^{2,4}_{(3,\frac{5}{2})}=\frac{9 \lambda ^2}{(c+3)^2}+\frac{36 \omega }{(c-1) c},\quad w^{2,4}_{(2,\frac{7}{2})}=\frac{2 \lambda }{c+3},\\
            w^{3,3}_{0}=&\frac{3 (3 c-23) \lambda ^2 \omega }{4 (c+3)}+\frac{30 (15 c-16) \omega ^2}{(c-1) c},\quad w^{3,3}_{2}=\frac{3 \left(9 c^2-58 c-447\right) \lambda ^3}{8 (c+3)^2}+\frac{6 (39 c-70) \lambda  \omega }{(c-1) c},\\
            w^{3,3}_{(3)}=&\frac{3 \left(21 c^2-106 c-867\right) \lambda ^2}{8 (c+3)^2}+\frac{180 (2 c-3) \omega }{(c-1) c},\quad w^{3,3}_{(4)}=\frac{3 (c+7) \lambda }{2 (c+3)},\quad w^{3,3}_{(5)}=\frac{3}{4},\\
            w^{3,3}_{(2,2)}=&-\frac{6 (c+9) \lambda ^2}{(c+3)^2}-\frac{216 \omega }{(c-1) c},\quad
            w^{3,3}_{(3,\frac{5}{2})}=\frac{54 \omega }{(c-1) c}-\frac{27 (c+13) \lambda ^2}{16 (c+3)^2},\quad w^{3,3}_{(2,\frac{7}{2})}=\frac{21 \lambda }{2 (c+3)}.
        \end{align*}

        Finally, we proceed to evaluate all OPEs in $S^7$, which are $W^3(z)W^4(w)$, $W^2(z)W^{5}(w)$, and $W^1(z)W^{6}(w)$.
        Again, imposing the Jacobi identities $J(W^{2}, W^{2},W^{4})$, $J(W^{2}, W^{3},W^{3})$, and $J(W^{1}, W^{2},W^{5})$ determines all structure constants arising in $W^2(z)W^5(w)$, $W^3(z)W^4(w)$, $W^1(z)W^6(w)$ as rational functions in the central charge $c$, and polynomials in $\lambda$ and $\omega$. Remarkably, knowledge of OPEs in $S^{\leq 7}$ is sufficient to conclude the following.

         \begin{proposition}\label{prop:lambdas}
         Recall the weight $3$ fields $W^3_{\pm}\in \cW^{\cN=2}_{\infty}$ in Proposition \ref{prop:weight3}. Let $W^n_{\pm}$ be defined recursively as $(W^3_{\pm})_{(1)}W^{n}_{\pm} = W^{n+1}_{\pm}$.
         Then fields $W^n_{\pm}$ satisfy the OPE relations of the $\cW_{\infty}(c_{\pm},\lambda_{\pm})$, where the parameters $\lambda_{\pm}$ are as follows.
          \begin{equation}\label{eq:lambdas}
        \begin{split}
            \lambda_{\mp}(\omega,c,\lambda)=&\frac{2 (3 + c)\left(A(\omega,c,\lambda)\mp\sqrt{\lambda ^2+\frac{32 \omega }{c-1}}B(\omega,c,\lambda)\right)}{(\left(c^2-1\right) \lambda ^2+4 (c+3)^2 \omega)\left(C(\omega,c,\lambda)\mp\sqrt{\lambda ^2+\frac{32 \omega }{c-1}}D(\omega,c,\lambda)\right)},\\
            A(\omega,c,\lambda)=&(c+3) \lambda  \left((c-1) \lambda ^2+32 \omega \right) \left((c-1) \left(c^2+8 c-1\right) \lambda ^2+4 (c+3)^3 \omega \right),\\
            B(\omega,c,\lambda)=&4 (c-1) \left(c^2+14 c+9\right) (c+3)^2 \lambda ^2 \omega +(c-1)^2 \left(c^3+9 c^2+27 c+3\right) \lambda ^4+32 (c+3)^4 \omega ^2,\\
            C(\omega,c,\lambda)=&8 (c+3)^2 \left(7 c^2+13 c-132\right) \lambda  \omega +(c-1) \left(c^4-5 c^3-141 c^2-447 c+144\right) \lambda ^3,\\
            D(\omega,c,\lambda)=&(c+3) \left((c-1) \left(c^3-8 c^2-85 c-48\right) \lambda ^2+8 (5 c-12) (c+3)^2 \omega \right).
        \end{split}
    \end{equation}
	\end{proposition}
    \begin{proof}
        At this stage of our base case computation, we have enough OPEs in the algebra up to weight 7. 
        By Theorem 5.1 in \cite{Lin}, it is enough to verify that OPEs among fields $L_{\pm},W^{3}_{\pm},W^{4}_{\pm},W^{5}_{\pm}$ satisfy the OPE relations of $\cW_{\infty}(c_{\pm},\lambda_{\pm})$ algebra. 
        The normalization constants $\omega_{\pm}$ of the weight three generators are determined by imposition of the following OPE relations
        \begin{equation*}
        \begin{split}
         L_{\pm}(z)W^{4}_{\pm}(w)\sim &3c_{\pm}(z-w)^{-6}+10L_{\pm}(w)(z-w)^{-4}+3L_{\pm}(w)(z-w)^{-3}\\
        &+4W_{\pm}(w)(z-w)^{-2}+\partial W_{\pm}(w)(z-w)^{-1},\\
            W^{3}_{\pm}(z)W^{3}_{\pm}(w) \sim& \frac{c_{\pm}}{3}(z-w)^{-6}+2L_{\pm}(w)(z-w)^{-4}+\partial L_{\pm}(w)(z-w)^{-3}+W^4_{\pm}(w)(z-w)^{-2}\\
            &+\left(\frac{1}{2}\partial W^{4}_{\pm}-\frac{1}{12}\partial^3 L_{\pm}\right)(w)(z-w)^{-1},
            \end{split}
        \end{equation*}
        which yield the following results
        \begin{equation*}
        \begin{split}
            \omega^2_{\mp}(\omega,c,\lambda)=&\frac{64 (c-1) c \left(\left(c^2+4 c-5\right) \lambda ^2-4 (c+3)^2 \omega \right)^2}{3 \left((c-1) \lambda
   ^2+32 \omega \right) \left(X(\omega,c,\lambda)\mp\sqrt{\lambda ^2+\frac{32 \omega }{c-1}} Y(\omega,c,\lambda)\right)},\\
        X(\omega,c,\lambda)= &128 (5 c-12) (c+3)^4 \omega ^2+32 (c-1) \left(2 c^3-c^2-64 c-21\right) (c+3)^2 \lambda ^2 \omega\\
   &+(c-1)^2 \left(c^5-4 c^4-130 c^3-468 c^2-231c-288\right) \lambda ^4,\\
        Y(\omega,c,\lambda)=&(c-1) (c+3)^2 \lambda  \left((c-1) \left(c^3-10 c^2-79 c+32\right) \lambda ^2+16 (3 c-10) (c+3)^2 \omega
   \right).
        \end{split}
        \end{equation*}
    Next, to determine the $\lambda_{\pm}$ parameters we impose the following OPE relations of $\Winf(c_{\pm},\lambda_{\pm})$ algebras
    \begin{equation}
        \begin{split}
        L_{\pm}(z)W^5_{\pm}(w)\sim& (185-80\lambda_{\pm}(c_{\pm}+2))W^3_{\pm}(w)(z-w)^{-4}+(55-16\lambda_{\pm}(c_{\pm}+2))\partial W^3_{\pm}(z-w)^{-3}\\
        &+5W^5_{\pm}(w)(z-w)^{-2}+\partial W^5_{\pm}(w)(z-w)^{-1},\\
           W^3_{\pm}(z)W^4_{\pm}(w)\sim& (31-16\lambda_{\pm}(c_{\pm}+2))W^3_{\pm}(w)(z-w)^{-4}+\frac{8}{3}(5-\lambda_{\pm}(c_{\pm}+2))\partial W^3_{\pm}(w)(z-w)^{-3}+W^{5}_{\pm}(w)(z-w)^{-2}\\
           &+\Big(\frac{2}{5}\partial W^5_{\pm}+\frac{32}{5}\lambda_{\pm}:\!L\partial W^3_{\pm}\!:-\frac{48}{5}\lambda_{\pm}:\!\partial  LW^3_{\pm}\!:+\frac{2}{15}(2\lambda_{\pm}(c_{\pm}-1)-5)\partial^3 W^3_{\pm}\Big)(w)(z-w)^{-1}.
        \end{split}
    \end{equation}
    which yields us the desired solutions (\ref{eq:lambdas}).
    A direct computer computation verifies that all Jacobi identities in $J^{\leq 9}$ hold, and therefore the subalgebras strongly generated by fields $\{L_{\pm},W_{\pm}^n|\ n\geq 3\}$ are respective quotients of $\Winf(c_{\pm},\lambda_{\pm})$.
    \end{proof}

    \begin{remark}
    The automorphism $\sigma$ on $\cW^{\cN=2}_{\infty}$ restricts to the nontrivial automorphism of each $\cW_{\infty}$ factor described in \cite[Corollary 5.3]{Lin}.
    \end{remark}
    
    We summarize this subsection with the following.
    \begin{proposition}\label{prop:base case prop}
    The Jacobi relations in $J_9$ express all OPEs in $D_7$ in terms of two parameters, the central charge $c$ and $\lambda$ defined in (\ref{eq:22}). 
    All denominators of structure constants are powers of elements in the set $\{(c),(c-1),(c+3)\}$.
    Moreover, there exists two commuting quotients of $\cW_{\infty}(c_{\pm},\lambda_{\pm})$ with values $c_{\pm}$ and $\lambda_{\pm}$ given in Propositions \ref{prop:Virasoros} and \ref{prop:lambdas}.
    \end{proposition}

    \begin{remark}
    The two $\cW_{\infty}(c_{\pm},\lambda_{\pm})$ are defined over the localization of a quadratic extension of the polynomial ring $\mathbb{C}[c,\lambda](\gamma)$, where 
        $$\gamma=\sqrt{\lambda ^2+\frac{32 \omega }{c-1}}.$$
        There is an automorphism of the base ring mapping
        $$\gamma\mapsto -\gamma,$$
        and acting as identity on the strong generators of $\WNtwo$, which exchanges two $\cW_{\infty}$-subalgebras.
    \end{remark}

     \begin{remark}
    Using the renormalization 
    \[\lambda = (c+3)\tilde \lambda, \quad \omega= c(c-1),\]
    all structure constants in $D_7$ become polynomial in $\tilde\lambda$ and $c$.
     \end{remark}

        \subsection{Induction}
        To motivate this section we note the following,
				\begin{equation}\label{eq:C3}
                    \begin{split}
				W^{2,\bot}\times W^{2,\bot} \equiv & 0 ,\quad  W^{2,\bot}\times W^{2,\pm} \equiv \frac{1}{3}W^{3,\pm},\\
				    W^{2,\top}\times W^{2,\bot} \equiv&W^{3,\bot}+\frac{2}{3} \partial W^{3,\bot}, \quad W^{2,+}\times W^{2,-} \equiv W^{3,\bot}+\frac{1}{2} \partial W^{3,\bot}+\frac{1}{3}W^{3,\top},\\
                    W^{2,\top}\times W^{2,\pm} \equiv&\frac{7}{6}W^{3,\pm}+\frac{2}{3} \partial W^{3,\pm},\quad W^{2,\top}\times W^{2,\top} \equiv \frac{4}{3}W^{3,\top}+\frac{2}{3} \partial W^{3,\top},
                    \end{split}
				\end{equation}
        where $\equiv$ means projecting onto the differential subspace spanned by the third diamond $W^3$.
        Specifically, the structure constants displayed above in (\ref{eq:C3}) are remarkably simple, giving one hope that they have a simple origin. 
        Thankfully, they do, and we will be able to obtain their explicit form in Proposition \ref{prop:structure constants} for an arbitrary interaction of $i^{\T{th}}$ with $j^{\T{th}}$ diamonds. Next, we use such expressions to produce enough linear relations involving OPEs in $S^{i+j}$, allowing us to solve them explicitly in terms of OPE data computed in previous subsection, see Proposition \ref{lem:Induction2}.

        \begin{remark}
         This feature continues to hold for primary generators also. So we regard this property as an important feature the algebra $\WNtwo$, independent of the choice of strong generators.
        \end{remark}

        First, we introduce some notation describing the interaction of fields in diamonds $W^i$ and $W^j$.
	\begin{itemize}
            \item Bottom with Bottom. The conformal weight and Heisenberg charge only permit $W^{i+j-1}$ to appear in the first order pole. 
            However, this would be inconsistent with existence of automorphism $\sigma$. So we may write
            \begin{equation}\label{ansatzBottomBottom}
			\begin{split}
				W^{i,\bot}_{(0)}W^{j,\bot}=&W_{0}^{i,\bot;j,\bot}(L,H,G^+,G^-,\dots,W^{i+j-3,\top},W^{i+j-3,+},W^{i+j-3,-},W^{i+j-2,\bot}),
			\end{split}
		\end{equation}
                where the right side is some normally ordered polynomial in the above generators and their derivatives. 
		\item Bottom with Side. The first order pole of $W^{i,\bot}$ and $W^{j,+}$ has conformal weight $i+j-\frac{1}{2}$ and Heisenberg charge $1$. Therefore we may write
			\begin{equation}\label{ansatzWW}
			\begin{split}
				W^{i,\bot}_{(0)}W^{j,\pm}=&\pm w_{W^{i+j-1,\pm}}^{W^i,W^j}W^{i+j-1,\pm}+W_{0}^{i,\bot;j,\pm},\\
                W^{i,\bot;j,+}_{0} =& W^{i,\bot;j,+}_{0} (L,H,G^+,G^-,\dots,W^{i+j-2,\top},W^{i+j-2,\bot},W^{i+j-2,+},W^{i+j-2,-}).
			\end{split}
		\end{equation}
            \item Bottom with Top.  The first order pole of $W^{i,\top}$ and $W^{j,\bot}$ has conformal weight $i+j$ and Heisenberg charge 0, and transforms with sign $(-1)^{i+j+1}$ under $\sigma$. This permits only $\partial W^{i+j-1,\bot}$ to appear, so we write
		\begin{equation}\label{ansatzXH}
			\begin{split}
				W^{i,\top}_{(0)}W^{j,\bot}=& w^{W^{i,\top},W^{j,\bot}}_{\partial W^{i+j-1,\bot}} \partial W^{i+j-1,\bot}+ W_{0}^{i,\top;j,\bot},\\
				W^{i,\top}_{(1)}W^{j,\bot}=& w^{W^{i,\top},W^{j,\bot}}_{W^{i+j-1,\bot}} W^{i+j-2,\bot}+W_{1}^{i,\top;j,\bot},\\
                W^{i,\top;j,\bot}_{r} =& W^{i\top,j\bot}_{r} (L,H,G^+,G^-,\dots,W^{i+j-2,\top},W^{i+j-2,\bot},W^{i+j-2,+},W^{i+j-2,-}).
			\end{split}
		\end{equation}
	\end{itemize}
	The remaining OPEs are determined via relations (\ref{eq:G descendend}).
	
	\begin{proposition} \label{prop:structure constants} Let the notation be fixed as in (\ref{ansatzWW}-\ref{ansatzXH}), and $i,j\geq 2$.
    			\begin{equation*}
					\begin{split}
                    W^{i,\bot}_{(0)}W^{j,\pm} &= \frac{i!j!}{2(i+j-1)!}W^{i+j-1,\pm}+W^{i,\bot;j,\pm}_{0}, \\
				    W^{i,\top}_{(1)}W^{j,\bot} &= \frac{i!j!}{2(i+j-2)!}W^{i+j-1,\bot}+W^{i,\top;j,\bot}_{1},\quad W^{i,\top}_{(0)}W^{j,\bot} =\frac{i i!j!}{2(i+j-1)!}\partial W^{i+j-1,\bot}+W^{i,\top;j,\bot}_{0}.
                    \end{split}
			\end{equation*}
        Using relations all the remaining structure constants follow, which are recorded here for completeness.
                 \begin{equation*}
					\begin{split}
                    W^{i,+}_{(1)} W^{j,-} \equiv&\frac{i!j!}{2(i+j-2)!}W^{i+j-1,\bot},\quad  W^{i,+}_{(0)} W^{j,-}\equiv\frac{(i+\frac{1}{2})i!j!}{2(i+j-1)!}\partial W^{i+j-1,\bot}+\frac{i!j!}{2(i+j-1)!}W^{i+j-1,\top},\\
                    W^{i,\top}_{(1)} W^{j,\pm} \equiv& \frac{(i+j-\frac{1}{2})i!j!}{2(i+j-1)!}W^{i+j-1,\pm},\quad W^{i,\top}_{(0)} W^{j,\pm}\equiv\frac{ii!j!}{2(i+j-1)!} \partial W^{i+j-1,\pm},\\
                    W^{i,\top}_{(1)} W^{j,\top} \equiv&\frac{(i+j-1)i!j!}{2(i+j-2)(i+j-2)!}W^{i+j-1,\top},\quad W^{i,\top}_{(0)}W^{j,\top} \equiv \frac{i i!j!}{2(i+j-1)!}\partial W^{i+j-1,\top}.
					\end{split}
			\end{equation*}
	\end{proposition}
	
	\begin{proof}
        Extracting the coefficient of fields $W^{i+j+k-2,\bot}$ in identities $J_{2,0}(W^{k,\top},W^{i,\top},W^{j,\bot})$ gives a relation
		\begin{equation}\label{20identity}
			 w^{W^{i,\top}, W^{j,\bot}}_{\partial W^{i+j-1,\bot}}w^{W^{k,\top},W^{i+j-1,\bot}}_{W^{i+j+k-2,\bot}}+ \left(w^{W^{k,\top}, W^{i,\top}}_{\partial W^{i+k-1,\top}} - w^{W^{k,\top}, W^{i,\top}}_{W^{i+k-1,\top}}\right) w^{W^{i+k-1,\top}, W^{j,\bot}}_{W^{i+j+k-2,\bot}}=0.
		\end{equation}
		First, setting $k=1$ in(\ref{20identity}) we find the following relation.
		\begin{equation}\label{tempEq}
			w^{W^{i,\top}, W^{j,\bot}}_{\partial W^{i+j-1,\bot}}=\frac{i}{i+j-1}w^{W^{i,\top}, W^{j,\bot}}_{W^{i+j-1,\bot}}.
		\end{equation}
        Next, extracting the coefficient of $W^{i+j-1,\bot}$ arising in identities $J_{1,0}(G^+,W^{i,\bot},W^{j,-})$ yields the relation
		\begin{equation*}\label{xyw}
        w^{W^{i,\bot},W^{j,+}}_{W^{i+j-1,+}}=\frac{1}{i+j-1}w^{W^{i,\top},W^{j,\bot}}_{W^{i+j-1,\top}}.
		\end{equation*}
		Lastly, set $k=2$ in (\ref{20identity}), and recall the weak generation property $w^{W^{2,\top},W^{i,\bot}}_{W^{i+1,\bot}}=1$; thanks to (\ref{eq:G descendend}), it implies $w^{W^{2,\top},W^{i,\top}}_{W^{i+1,\top}}=\frac{i+2}{i+1}$, and (\ref{tempEq}) gives $w^{W^{2,\top},W^{i,\top}}_{\partial W^{i+1,\top}}=\frac{2}{i+1}$. Using (\ref{tempEq}) we obtain
		\begin{equation}\label{inductLeft}
			w^{W^{i+1,\top}, W^{j,\bot}}_{W^{i+j,\bot}}=-\frac{i w^{W^{2,\top},W^{i+j-1,\bot}}_{W^{i+j,\bot}}w^{W^{i,\top}, W^{j,\bot}}_{W^{i+j-1,\bot}} }{(i+j-1)(w^{W^{2,\top},W^{i,\top}}_{\partial W^{i+1,\top}} - w^{W^{2,\top},W^{i,\top}}_{W^{i+1,\top}})}=\frac{i+1}{i+j-1}w^{W^{i,\top}, W^{j,\bot}}_{W^{i+j-1,\bot}}.
		\end{equation}
            This recursion implies the desired formula.
            Finally, the remaining structure constants may be obtained by applying (\ref{eq:G descendend}).
	\end{proof}

	Now we proceed with the induction argument.
	By inductive data we mean the set of OPEs in $S^{\leq n}$, where we recall
	\[S^{\leq n} = \bigcup_{k=2}^n S^{k},\quad S^k=S^k_{\bot,\bot}\cup S^k_{\bot,+}\cup S^k_{\top,\bot},\] and that they are fully expressed in terms of parameters $c$ and $\lambda$.
    Our base case can be taken $n=7$ as in Proposition \ref{prop:base case prop}, but one can do better and set it to $n=4$\footnote{We went up higher than necessary to conclude the existence of two $\cW_{\infty}$ subVOAs.}.
	At this stage, the OPEs in $S^{n+1}$ are yet undetermined.
	We will use a subset of Jacobi identities in $J^{n+2}$ to express $S^{n+1}$ in terms of inductive data $S^{\leq n}$.
	We write $A\equiv 0$ to denote that $A$ is computable from inductive data.
    We remark that this notation differs from the one used in (\ref{eq:W3s}) and in Proposition \ref{prop:structure constants}.
	We use Proposition \ref{prop:structure constants} to establish the following Lemma.
	\begin{lemma}\label{lem:Induction1}
 $W^2(z)W^{n-1}(w)$ and $W^3(z)W^{n-2}(w)$ together with inductive data determine OPEs $S^{n+1}$.
	\end{lemma}

    \begin{proof}
    This proof consists of producing many linear relations relating OPEs in $S^{n+1}$.
    The manner in which such relations arise is similar, so we present (1) in detail, while only highlighting relevant features for the rest.
    We must proceed by the order of the pole in OPEs; this is because the recurrence relations will be different, and depend on our choice of imposed Jacobi identities. 
    In broad strokes, the imposition of Jacobi identities $J(W^2,W^i,W^j)$ with $i+j=n$ is sufficient to obtain the desired conclusion. Specifically, we have the following.
		\begin{enumerate}
			\item  $W^{2,\top}_{(0)}W^{n-1,\bot}$  determine $\{W^{i+1,\top}_{(0)}W^{n-i,\bot} | i\geq 2\}$ thanks to $J_{1,0}(W^{2,\top},W^{i,\top},W^{j,\bot})$;
            indeed, expanding this Jacobi relation yields
		\begin{equation*}\label{Recurion0Full}
			\left(1-w^{W^{2,\top},W^{i,\top}}_{\partial W^{i+2,\top}}\right)W^{i+1,\top;j,\bot}_0 =
	w^{W^{i,\top},W^{j,\bot}}_{\partial W^{i+j-1,\top}} W^{2,\top;i+j-1,\bot}_0-W^{i,\top;j+1,\bot}_0+R^{i,\top;j,\bot}_0,
		\end{equation*}
		where 
		\begin{equation*}
		R^{i,\top;j,\bot}_{0}=W^{2,\top}_{(1)}W^{i,\top;j,\bot}_{0} - (W^{2,\top;i,\top}_{0})_{(1)}W^{j,\bot}.
		\end{equation*}
		Term $R^{i,\top;j,\bot}_{0}$ is known from inductive data, so we may write
		\begin{equation*}\label{Recurion0Partial}
			W^{i+1,\top;j,\bot}_{0}\equiv \frac{w^{W^{i,\top},W^{j,\bot}}_{W^{i+j-1,\bot}}}{1-w^{W^{2,\top},W^{i,\top}}_{\partial W^{i+1,\top}}} W^{2,\top;i+j-1,\bot}_{0}-\frac{1}{1-w^{W^{2,\top},W^{i,\top}}_{\partial W^{i+1,\top}}}W^{i,\top;j+1,\bot}_0.
		\end{equation*}
		We iterate above recursion to obtain desired relations.
			\item $W^{2,\top}_{(0)}W^{n-1,+}$  determine $\{W^{i+1,\top}_{(0)}W^{n-i,+} | i\geq 2\}$ thanks to $J_{1,0}(W^{2,\top},W^{i,\bot},W^{j,+})$.\\
            \item $W^{2,\top}_{(1)}W^{n-1,\bot}$  determine $\{W^{i+1,\top}_{(1)}W^{n-i,\bot} | i\geq 3\}$ thanks to $J_{1,1}(W^{2,\top},W^{i,\top},W^{j,\bot})$.\\
            \item Let $r>1$. $W^{2,\top}_{(r)}W^{n-1,\bot}$  determine $\{W^{i+1,\top}_{(r)}W^{n-i,\bot} | i\geq 2\}$ thanks to $J_{r,1}(W^{2,\top},W^{i},W^{j})$.
            \item Let $r>0$. $W^{2,\bot}_{(r)}W^{n-1,+}$  determine $\{W^{i+1,\bot}_{(r)}W^{n-i,+} | i\geq 2\}$ thanks to $J_{r,1}(W^{2,\top},W^{i,\bot},W^{j,+})$.
            \item Let $r\geq 0$. $W^{2,\bot}_{(r)}W^{n-1,\bot}$  determine $\{W^{i+1,\bot}_{(r)}W^{n-i,\bot} | i\geq 2\}$ thanks to $J_{r,1}(W^{2,\top},W^{i,\bot},W^{j,\bot})$.
            \item OPEs in $W^{2}(z)W^{n-1}(w)$ determine $W^1(z)W^{n}(w)$ thanks to $J_{1,r}(W^{2,\top},W^{1},W^{n})$.
    
		\end{enumerate}
    \end{proof}
	
	 By Lemma \ref{lem:Induction1}, it suffices to determine OPEs $W^{2,\bot}(z)W^{n-1,\bot}(w)$, $W^{2,\bot}(z)W^{n-1,+}(w)$, $W^{2,\top}(z)W^{n-1,\bot}(w)$, and $W^{3,\top}(z)W^{n-2,\bot}(w)$ to determine all OPEs in $S^{n+1}$.
	In the following Lemma we write down a small set of Jacobi identities to obtain linear relations among the desired products expressing the OPEs in $S^{n+1}$ in terms of the inductive data $S^{\leq n}$.
	\begin{lemma}\label{lem:Induction2}
            Let $n\geq 5$. Then all OPEs in $D^{n+1}$ are determined modulo inductive data $D_n$. 
            Specifically, we have the following.
		\begin{enumerate}
			\item $J_{1,0}(W^{2,\top},W^{2,\top},W^{n-2,\bot})$, $J_{1,0}(W^{2,\top},W^{3,\top},W^{n-3,\bot})$, and $J_{1,0}(W^{3,\top},W^{2,\top},W^{n-3,\bot})$ express $W^{3}_{(0)}W^{2n-1}$ in terms of inductive data, and $W^{3,\top;n-2,\bot}_1$. 
                \[W^{2,\top;n-1,\bot}_0\equiv \frac{2 (n-1)}{n (n+1)}\partial  W^{3,\top;n-2,\bot}_1  .\]
            \item $J_{1,0}(W^{2,\top},W^{n-2,\bot},W^{2,+})$, $J_{1,0}(W^{2,\top},W^{n-3,\bot},W^{3,+})$, and $J_{1,0}(W^{3,\top},W^{3,\bot},W^{n-2,\bot})$ express $W^{n-1,\bot}_{(0)}W^{2,+}$ in terms of inductive data, and $W^{2,\bot;n-2,\bot}_1$.
			\begin{equation*}
				\begin{split}
			W^{n-1,\bot;2,+}_0 \equiv & -\frac{7 (n-1)}{3 n (2 n-7)} G^+_{(0)} W^{3,\top;n-2,\bot}_1 .
				\end{split}
                    \end{equation*}
            \item $J_{2,0}(W^{2,\top},W^{2,\top},W^{n-2,\bot})$ and $J_{1,1}(W^{2,\top},W^{3,\top},W^{n-3,\bot})$ express $W^{3,\top}_{(1)}W^{n-3,\bot}$ in terms of inductive data and $W^{2,\top}_{(2)}W^{n-1,\bot}$.
			\begin{equation*}
				\begin{split}
					W^{3,\top;n-2,\bot}_1\equiv & \frac{3}{2 (n-1)} \partial (W^{2,\top}_{(2)}W^{n-1}).\\
				\end{split}
			\end{equation*}
            \item $J_{r+1,0}(W^{2,\top},W^{n-2,\bot},W^{2,\top})$, $J_{r,1}(W^{2,\top},W^{2,\top},W^{n-2,\bot})$ express $W^{2,\top}_{(r)}W^{n-1,bot}$ in terms of inductive data.
                \begin{equation*}
				\begin{split}
					W^{2,\top}_{(r)}W^{n-1,\bot}\equiv & \frac{1}{r+1}\partial (W^{2,\top}_{(r+1)}W^{n-1,\bot}).
				\end{split}
			\end{equation*}
            \item  $J_{r,1}(W^{2,\top},W^{n-2,\bot},W^{2,+})$ express $W^{2,\bot}_{(r)}W^{n-1,+}$ in terms of inductive data.
            \item  $J_{r+1,1}(W^{2,\bot},W^{2,\bot},W^{n-2,\bot})$ express $W^{2,\bot}_{(r)}W^{n-1,\bot}$ in terms of inductive data.
          \end{enumerate}
	\end{lemma}
\begin{proof}We prove only part (1), and the rest is similar. 
    We proceed to expand the identities in part (1), which are as follows
    \begin{align*}
        J_{1,0}(W^{2,\top},W^{2,\top},W^{n-2,\bot})\equiv&\frac{n-3}{n-1}W^{2,\top,n-1;\bot}_0-\frac{2}{3}W^{3,\top,n-2;\bot}_0\equiv 0,\\
         J_{1,0}(W^{2,\top},W^{3,\top},W^{n-3,\bot})\equiv&\frac{9}{(n-2) (n-1)}W^{2,\top,n-1;\bot}_0-W^{3,\top,n-2;\bot}_0-\frac{3}{4}W^{4,\top,n-3;\bot}_0\equiv 0,\\
         J_{1,0}(W^{3,\top},W^{2,\top},W^{n-3,\bot})\equiv&-\frac{3}{n-2}W^{2,\top,n-1;\bot}_0+\frac{2}{n-2}W^{3,\top,n-2;\bot}_0-\frac{1}{2}W^{4,\top,n-3;\bot}_0\\
         &+\frac{2}{n-2}\partial W^{3,\top,n-2;\bot}_1 \equiv 0.
    \end{align*}
    These form a linear system in the desired unknown, yielding
    \begin{align*}
        W^{2,\top,n-1;\bot}_0\equiv& \frac{2 (n-1)}{n (n+1)}\partial W^{3,\top,n-2;\bot}_1,\quad W^{3,\top,n-2;\bot}_0\equiv -\frac{3 (n-3)}{n (n+1)}\partial W^{3,\top,n-2;\bot}_1,\\
        W^{4,\top,n-3;\bot}_0\equiv& \frac{4 \left(n^2-5 n+12\right)}{(n-2) n (n+1)}\partial W^{3,\top,n-2;\bot}_1,
    \end{align*}
    and looking at first relation we obtain claim (1). 
\end{proof}

	\begin{theorem}\label{thm:induction}
    Let {$\nlcalg$} be a nonlinear conformal algebra over the ring $R$ defined as the localization $\mathbb{C}[c,\lambda]$ along \eqref{localization set},  generators as in (\ref{LCA:generators}) with the conformal grading $\eqref{eq:grading}$, 
    and satisfying features in list $\textup{(\textbf{U}1)}-\textup{(\textbf{U}4)}$. The universal enveloping vertex algebra $\WNtwo$ of $\nlcalg$ has the following properties:
	\begin{enumerate}\label{WN2 2 parameters}
		\item It has conformal weight grading \[\WNtwo=\bigoplus_{N\in\frac{1}{2}\mathbb{N}}\WNtwo[N],\quad \WNtwo[0]=R^{-1}\mathbb{C}[c,\lambda].\]

		\item  It is strongly generated by $\cN=2$ generators $\{H,G^{\pm},L\}$ together with higher weight fields $$\{W^{i,\top},W^{i,\bot};W^{i,+},W^{i,-}| i\geq 2\},$$ and satisfies the OPE relations in $S^{\leq 5}$, which is a subset of the OPE relations in Proposition \ref{prop:base case prop}, and those Jacobi identities which appear in Lemmas \ref{lem:Induction1} and \ref{lem:Induction2}.

        \item It is the unique initial object in the category of vertex algebras with the above properties.
        
	\end{enumerate} 
	\end{theorem}

	\begin{proof} As we have seen, all OPEs in $D_5$ follows from our assumptions together with Jacobi identities. By Lemmas \ref{lem:Induction1} and \ref{lem:Induction2}, all remaining OPEs among the generators are determined uniquely from the subset of Jacobi identities imposed in the proof of these lemmas.
\end{proof}
	
	\subsection{One-parameter quotients}
	There are more Jacobi identities than those imposed in Lemmas \ref{lem:Induction1} and \ref{lem:Induction2}.
    Therefore, it is not yet clear that all Jacobi identities among the strong generators hold as a consequence of (\ref{conformal identity}-\ref{quasi-derivation}) alone.
Let $I\subseteq R \cong \WNtwo [0]$
	be an ideal, and let $I\cdot \WNtwo$ denote the vertex algebra ideal generated by $I$. 
	The quotient 
	\begin{equation}
		\cW^{\cN=2,I}_{\infty} = \WNtwo / I\cdot \WNtwo
	\end{equation}
	has strong generators $\{W^{N} | N\geq 1\}$ satisfying the same OPE relations as the corresponding  generators of $\WNtwo$ where all structure constants in $R$ are replaced by their images in $R/I$. Further, we may consider a localization of $\cW^{\cN=2,I}_{\infty}$.
	Let $E \subseteq R/I$ be a multiplicatively closed set, and let $S= E^{-1}(R/I)$ denote the localization of $R/I$ along $S$.
	Thus we have the localization of $R/I$-modules
	\[\cW^{\cN=2,I}_{\infty, S}= S\otimes_{R/I}\cW^{\cN=2,I}_{\infty},\]
	which is a vertex algebra over $S$.
	
	\begin{theorem}\label{one-parameter quotients theorem}
		Let $R$, $I$, $E$, and $S$ be as above, and let $\cW$ be a vertex algebra over $S$ with the following properties:
		\begin{enumerate}[(i)]
			 \item $\cW$ is weakly generated by $\mathcal{N}=2$ fields $\bar {H}$,$\bar{G}^{\pm}$, $\bar{L}$ of central charge $c$ and $\mathcal{N}=2$ primary field $\bar {W}^{2,\bot}$.
       
            \item Setting $\bar{W}^{N+1,\bot} =\bar{W}^{2,\top}_{(1)} \bar{W}^{N,\bot}$ for all $N\geq 2$, the fields $\{\bar{W}^{N}|\ N\geq 2\}$ together with the $\mathcal{N}=2$ fields, satisfy conditions $\textup{(\textbf{U}1)}-\textup{(\textbf{U}4)}$.
            \item OPE relations for $\bar{W}^{N}(z)\bar{W}^{M}(w)$ for $N+M \leq 5$ are the same as in $\WNtwo$ if the structure constants are replaced with their images in $S$.
		\end{enumerate}
		Then $\cW$ is a quotient of $\cW^{\cN=2,I}_{\infty,S}$ by some graded ideal $\cI \subseteq \cW^{\cN=2,I}_{\infty,S}$. If $\cW$ is simple, $\cI$ is the maximal graded ideal of $\cW^{\cN=2,I}_{\infty,S}$.
	\end{theorem}

\begin{proof} As we have seen, all OPE relations among the generators of $\cW^{\cN=2,I}_{\infty,S}$ must hold among the corresponding fields $\{ \bar{L}, \bar {H}, \bar{G}^{\pm}, \bar{W}^{N,\bot},\bar{W}^{N,\pm}, \bar{W}^{N,\top}|\ N \geq 2\}$, since they are formal consequences of the OPEs $\bar{W}^N(z) \bar{W}^M(w)$ for $N+M \leq 5$, together with Jacobi identities, which hold in $\cW$. It follows that 
$$\{ \bar {H}, \bar{G}^{\pm},\bar{L},  \bar{W}^{N,\bot},\bar{W}^{N,\pm}, \bar{W}^{N,\top}|\ N \geq 2\}$$ strongly generates a vertex subalgebra $\cW' \subseteq \cW$, which must coincide with $\cW$ since $\cW$ is assumed to be generated by $\{  \bar {H}, \bar{G}^{\pm},\bar{L}, \bar{W}^{2,\bot},\bar{W}^{2,\pm}, \bar{W}^{2,\top}\}$ as a vertex algebra. Hence $\cW$ has the same strong generating set and OPE algebra as $\cW^{\cN=2,I}_{\infty,S}$. Since $\cW^{\cN=2,I}_{\infty,S}$ is freely generated by the above generators, it is the initial object in the category of such vertex algebras, so $\cW$ is a quotient of $\cW^{\cN=2,I}_{\infty,S}$ by some graded ideal $\cI$. Finally, if $\cW$ is simple, since the category of vertex algebras over $S$ with this strong generating set and OPE algebra has a unique simple graded object, $\cW$ must be the quotient of $\cW^{\cN=2,I}_{\infty,S}$ by its maximal graded ideal. \end{proof}

The following corollary gives a useful criterion for when a vertex algebra of strong generating type \eqref{princ:gentype} for some $n$, is a quotient of $\cW^{\cN=2,I}_{S,\infty}$ for some $I$ and $S$.

\begin{corollary} \label{cor:gradedcharacter} Let $\cW$ be a vertex algebra of type \eqref{princ:gentype} satisfying conditions (\textup{\textbf{U}1})--(\textup{\textbf{U}4}), which is defined over some localization $S$ of $\mathbb{C}[c,\lambda] / I$, for some prime ideal $I$. Suppose that $\cW$ is generated by $\{H,G^{\pm}, L, W^{2,\bot}\}$. If in addition, the graded character of $\cW$ agrees with that of $\cW^{\cN=2}_{\infty}$ up to weight $6$, then $\cW$ is a (possibly non-simple) $1$-parameter quotient of $\cW^{\cN=2,I}_{\infty,S}$.
\end{corollary}

\begin{proof} By Theorem \ref{one-parameter quotients theorem}, it suffices to prove that the OPEs $W^{N,*}(z) W^{M,*}(w)$ for $N+M \leq 5$ in $\cW$ are the same as the corresponding OPEs in $\cW^{\cN=2}_{\infty}$ if the structure constants are replaced with their images in $S$. But this is automatic because the graded character assumption implies that there are no null vectors of weight less than $7$ in the (possibly degenerate) nonlinear conformal algebra corresponding to $\{H,G^{\pm}, L, W^{N,\bot},W^{N,\pm},W^{N,\top}|\ N \geq 2\}$. \end{proof}

\subsection{$\cW^k(\fr{sl}_{n+1|n})$ revisited} 
 We will now realize $\cW^k(\fr{sl}_{n+1|n})$ explicitly as a $1$-parameter quotient of the $2$-parameter algebra $\WNtwo$. 
Recall the generators in (\ref{eq: as}), and note that we have the following nonzero brackets,
\begin{equation} \label{eq:lem_weak}
\begin{gathered}
        [a^{2,\bot},a^{j,\pm}]=\mp a^{j+1,\pm}, \quad [a^{2,+},a^{j,-}]= a^{j+1,\top}.
\end{gathered}
\end{equation}

\begin{proposition} \label{lem:weakgen} For all noncritical values of $k$, $\cW^k(\gs\gl_{n+1|n})$ is weakly generated by $\{H, G^{\pm}, L, \omega^{2,\bot}\}$.
\end{proposition}

\begin{proof} 
Set $W^{2,\bot} = \omega^{2,\bot}$, $W^{2,\pm} = \pm G^{\mp}_{(0)} W^{2,\bot}$, and $W^{2,\top} = G^{+}_{(0)} W^{2,-}-\frac{1}{2}\partial W^{2,\bot}$, and define the remaining fields $W^{i,\bot}, W^{i,\pm}, W^{i,\top}$ for $i = 3,\dots, n$ as in the universal object. For $i = 3,\dots, n$, we have
$$W^{i,\bot} = \lambda_{i,\bot} \omega^{i,\bot} + \cdots,\qquad 	W^{i,\pm} = \lambda_{i,\pm} \omega^{i,\pm} + \cdots, \qquad W^{i,\top} = \lambda_{i,\top} \omega^{i,\top} + \cdots,$$ where the remaining terms depend only on the fields $\omega^{j,*}$ for $j < i$, and their derivatives. It suffices to show that $ \lambda_{i,\bot}, \lambda_{i,\pm}, \lambda_{i,\top}$ are nonzero for all the above values of $k$. In fact, they are all equal since the actions of $G^{\mp}$ on both $W^{i,*}$ and $\omega^{i,*}$ are compatible, so it suffices to show that any of them is nonzero.

Next, we have
\begin{equation} 
\omega^{2,\bot}_{(0)} \omega^{j,\pm}  = c^{2,\bot}_{j,\pm} \omega^{j+1,\pm} + \cdots, \qquad \omega^{2,+}_{(0)} \omega^{j,-} = c^{2,+}_{j,-} \omega^{j+1,\top} + \cdots
\end{equation}
Here the remaining terms depend only of the fields $\omega^{i,*}$ for $i \leq j$, and $c^{2,\bot}_{j,\pm}$ and $c^{2,+}_{j,-}$ are independent of $k$ since they are structure constants for the Poisson vertex algebra structure on $\text{gr}(\cW^k(\gs\gl_{n+1|n}))$. Moreover, we know from \eqref{eq:omega i and omega a} and \eqref{eq:lem_weak} that 
$c^{2,\bot}_{j,\pm}$ and $c^{2,+}_{j,-}$ are nonzero constants, since they are proportional to the corresponding structure constants in the Lie superalgebra $\gg^{f}$.

By Proposition \ref{prop:structure constants}, for $j\geq 2$ we have
$$W^{2,\bot} _{(0)} W^{j,\pm} = \frac{2j!}{2(j+1)!} W^{j+1,\pm} + \cdots,$$ On the other hand,
$$ W^{2,\bot} _{(0)} W^{j,\pm} = \omega^{2,\bot}_{(0)}( \lambda_{j,\pm} \omega^{j,\pm} + \cdots ) =  \lambda_{j,\pm} c^{2,\bot}_{j,\pm} \omega^{j+1,\pm} + \cdots  = \frac{\lambda_{j,\pm} c^{2,\bot}_{j,\pm}}{\lambda_{j+1,\pm} }W^{j+1,\pm} + \cdots .$$
It follows that
$$\frac{\lambda_{j,\pm} c^{2,\bot}_{j,\pm}}{\lambda_{j+1,\pm} } =  \frac{2j!}{2(j+1)!}.$$ Since $\lambda_{2,\pm} = 1$ and $c^{2,\bot}_{j,\pm}$ is a nonzero constant, it follows by induction on $j$ that $\lambda_{j,\pm}$ is a nonzero constant for all $j = 3,\dots,n$. This completes the proof since  $ \lambda_{i,\bot}, \lambda_{i,\pm}, \lambda_{i,\top}$ are all equal. \end{proof}

\begin{proposition} \label{PrincW:quot} For all $n \geq 1$, $\cW^k(\fr{sl}_{n+1|n})$ arise as $1$-parameter quotient of $\WNtwo$.
\end{proposition}

\begin{proof} 
This is immediate from Propositions \ref{prop:sl(n|n-1), N=2}, \ref{cor:diamond}, \ref{prop:automorphism}, and \ref{lem:weakgen}.
\end{proof}

	\begin{theorem}\label{Wn=2 freely generated}
		All Jacobi identities among generators $\{W^{N} | N\geq 1\}$ holds as consequences of (\ref{conformal identity}-\ref{quasi-derivation}) alone, so $\nlcalg$ is a nonlinear Lie conformal algebra with generators $\{W^{N} | N\geq 1\}$. Equivalently, $\WNtwo$ is freely generated by $\{W^{N} | N\geq 1\}$ and has graded character
		\begin{equation}\label{character}
			\chi (\WNtwo,q) =\sum_{n=0}^{\infty}rank_{X}(\WNtwo [n])q^n = \prod_{i,j,l=1}^{\infty}\frac{(1+q^{i+l+\frac{1}{2}})(1+q^{j+l+\frac{1}{2}})}{(1-q^{2i+2j+1})(1-q^{2i+2l+2})},
		\end{equation}
        where $X$ is any localization $S=(E^{-1}R)/I$ with respect to a multiplicatively closed set $E$, and for some prime ideal $I \subset R$. 
       	\end{theorem}

	\begin{corollary} 
		The vertex algebra $\WNtwo$ is simple.
	\end{corollary}
   \begin{proof}
		If $\cW^{\cN=2}_{\infty}$ were not simple, it would have a singular vector $\omega$ in some weight conformal $N$. Let $p\in R$ be an irreducible polynomial, and let $I$ be the ideal $(p)\subseteq R$. By rescaling $\omega$ if necessary, we may assume that $\omega$ is not divisible by $p$, and hence descends to a nontrivial singular vector in $\cW^{\cN=2,I}_{\infty}$. For any localization $S$ of $R/I$, the simple quotient of $\cW^{\cN=2,I}_{\infty,S}$ would then have a smaller weight $N$ submodule than $\cW^{\cN=2,I}_{\infty,S}$. This contradicts Proposition \ref{PrincW:quot}, since $\cE^{\psi}_{\mathcal{N}=2}(n,0)$ is a quotient of $\cW^{\cN=2,I}_{\infty,S}$ for some $I$ and $S$, and has the same character as $\cW^{\cN=2}_{\infty}$ up to weight $n+1$.
	\end{proof}

Recall that if $\cA$ is an $\mathcal{N}=2$ supersymmetric vertex algebra and $\cV$ embeds in $\cA$ (not necessarily conformally), $\cA$ is called an {\it $\mathcal{N}=2$ supersymmetric extension} of $\cV$ \cite{SY25}. We call $\cA$ a {\it minimal} $\mathcal{N}=2$ supersymmetric extension if no proper $\mathcal{N}=2$-superconformal subalgebra of $\cA$ is an $\mathcal{N}=2$ supersymmetric extension of $\cV$. 

 \begin{corollary} \label{minextension} $\cW^{\cN=2}_{\infty}$ is a minimal $\mathcal{N}=2$ supersymmetric extension of either $\cW^+_{\infty}$ or $\cW^-_{\infty}$.
      \end{corollary}
      
    \begin{proof} Suppose that $\cA \subseteq \cW^{\cN=2}_{\infty}$ is an $\mathcal{N}=2$ superconformal extension of $\cW^+_{\infty}$. In particular, $\cA$ contains $H, G^{\pm}, L$, and $L_+$. Therefore $\cA$ also contains $L_-$, hence $W^{2,\bot}$, and hence $\cA = \cW^{\cN=2}_{\infty}$. The proof for $\cW^-_{\infty}$ is the same.
      \end{proof}
      
	Next we show that there are no automorphisms besides those arising as automorphisms of $\mathcal{N}=2$ superconformal algebra.
	\begin{proposition}\label{prop:autos}
	The vertex algebra $\WNtwo$ has full automorphism group $O_2$. 
	\end{proposition}
	\begin{proof}
		Let $g$ be any automorphism of $\WNtwo$ that is not contained in $\T{O}_2$-subgroup of automorphisms, so it fixes the generators $H,L, G^+,G^-$. 
        The most general deformation of weight two field $\tilde W^{2,\bot}$ compatible with
		\[H(z) \tilde W^{2,\bot}(w) \sim 0, \quad \tilde W^{2,\bot}(z)\tilde W^{2,\bot}(w)\sim \dotsb\] is given by the identity map.
        The weak generation property then forces this automorphism to be identity on $\WNtwo$.
	\end{proof}

	\subsection{Quotients by maximal ideals of $\WNtwo$}
	So far, we have considered quotients of the form $\cW^{\cN=2,I}_R$ which are 1-parameter vertex algebras in the sense that $R$ has Krull dimension 1. Here, we consider simple quotients of $\cW^{\cN=2,I}$ where $I \subseteq R$ is a maximal ideal.
	Such an ideal always has the form $I=(c-c_0,\lambda-\lambda_0)$ for $c_0,\lambda_0 \in\mathbb{C}$, and $\cW^{\cN=2,I}$ is a vertex algebra over $\mathbb{C}$.
	We first need a criterion for when the simple quotients of two such vertex algebras are isomorphic.
	\begin{theorem}\label{max quotients}
		Let $c_0,c_1,\lambda_0,\lambda_1$ be complex numbers and let \[I_0 = (c-c_0,\lambda-\lambda_0),\quad I_1 =(c-c_1,\lambda-\lambda_1)\]
		be the corresponding maximal ideals in $R$.
		Let $\cW_0$ and $\cW_1$ be the simple quotients of $\cW^{\cN=2,I_0}_{\infty}$ and $\cW^{\cN=2,I_1}_{\infty}$, respectively. 
		Then $\cW_0 \cong \cW_1$ if and only if:

        \begin{enumerate}
            \item $c_0 = c_1$ and $\lambda_0=\lambda_1$.
            \item $c_0=0=c_1$, then $\cW_0 \cong \mathbb{C} \cong \cW_1$, for all $\lambda$.
            \item $c_0=1=c_1$, then $\cW_0 \cong \text{Vir}_{\cN=2, 1} \cong \cW_1$, for all $\lambda$.
            \item $c_0=-3=c_1$, then $\cW_0 \cong (\cS(1) \otimes \cA(1))^{\textup{U}(1)} \cong \cW_1$, for all $\lambda$.
        \end{enumerate}
	\end{theorem}
	
	\begin{corollary}\label{max quotients2}
		Let $I=(p)$ and $J=(q)$ be prime ideals in $R$ such that $\cW^{\cN=2,I}_{\infty}$ and $\cW^{\cN=2,J}_{\infty}$ are not simple. Then any pointwise coincidences between the simple quotients of $\cW^{\cN=2,I}_{\infty}$ and $\cW^{\cN=2,J}_{\infty}$ must correspond to intersection points of the truncation curves $V(I)\cap V(J)$.
	\end{corollary}
	
	\begin{corollary} \label{cor:uniquenessofcurve}
		Suppose that $\cA$ is a simple, 1-parameter vertex algebra which is isomorphic to the simple quotient of $\cW^{\cN=2,I}_{\infty}$ for some prime ideal $I\subseteq R$, possibly after localization. 
		Then if $\cA$ is the quotient of $\cW^{\cN=2,J}_{\infty}$ for some prime ideal $J$, possibly localized, we must have $I=J$.
	\end{corollary}
	\begin{proof}
		This is immediate from Theorem (\ref{max quotients}) and Corollary (\ref{max quotients2}), since if $I$ and $J$ are distinct prime ideals, their truncation curves $V(I)$ and $V(J)$ can intersect in at most finitely many points. The simple quotients of $\cW^{\cN=2,I}_{\infty}$ and $\cW^{\cN=2,J}_{\infty}$ therefore cannot coincide as $1$-parameter families. \end{proof}

 \subsection{Reconstruction} As we shall see in the next section, all the vertex algebras $\cC^{\psi}_{\mathcal{N}=2}(n,r|s)$ arise is $1$-parameter quotients of $\cW^{\cN=2}_{\infty}$, and we need to compute their explicit truncation curves. The GKO cosets $\cC^{\psi}_{\mathcal{N}=2}(0,r|s)$ are the easiest to handle because they do not involve quantum Hamiltonian reduction. The case $\cC^{\psi}_{\mathcal{N}=2}(n,0|0) = \cW^k(\gs\gl_{n+1|n})$ is the most difficult and requires an explicit formula for the fields in weights up to $2$ in the BRST complex. However, the truncation curves in all other cases can be computed uniformly by a {\it reconstruction theorem}.

 \begin{theorem}\label{recon:first} Let $n\geq 0$ and $r+s >0$, and let $\cA^{\psi}(n,r|s)$ be a simple $1$-parameter vertex algebra with the following properties, which are shared with 
 \begin{itemize}
 \item $\cW^{k}(\gs\gl_{n+r+1|n+s}, F_{n+1|n})$ for $n\geq 1$ and $s \neq r+1$,
  \item $\cW^{k}(\gp\gs\gl_{n+r+1|n+r+1}, F_{n+1|n})$ for $n\geq 1$ and $r\geq 1$ and $s = r+1$,
   \item $V^{k}(\gs\gl_{r+1|s}) \otimes \cE(r) \otimes \cS(s)$ for $n = 0$ and $s \neq r+1$,
     \item $V^{k}(\gp\gs\gl_{r+1|r+1}) \otimes \cE(r) \otimes \cS(r+1)$ for $n = 0$ and $s = r+1$:
     \end{itemize}
\begin{enumerate}
			\item 
			$\cA^{\psi}(n,r|s)$ is a conformal extension of $\cW\otimes V^{k+1}(\ga)$, where 
			$ \ga \cong 
 \begin{cases}
\gg\gl_{r|s}, & s \neq r+1,\\
\gs\gl_{r|s}, & s = r+1,
 \end{cases}$
and $\cW$ is some 1-parameter quotient of $\cW^{\cN=2}_{\infty}$.
		
			\item The extension is generated by even fields $\{P^{i,\bot}|\ i = i,\dots, r+s\}$ and $\{Q^{i,\bot}|\ i = 1,\dots, r+s\}$ in weight $\frac{n+1}{2}$ which transform under $\ga$ as $\mathbb{C}^{r|s}$ and $(\mathbb{C}^{r|s})^*$, and odd fields $\{P^{i,-}|\ i = i,\dots, r+s\}$ and $\{Q^{i,+}|\ i = 1,\dots, r+s\}$ in weight $\frac{n+2}{2}$ which transform under $\ga$ as $\mathbb{C}^{r|s}$ and $(\mathbb{C}^{r|s})^*$.
			
			\item The extension fields are primary for the $\mathcal{N}=2$ superconformal algebra and primary for $V^{k+1}(\ga)$.
						
			\item $\cA^{\psi}(n,r|s)$ is strongly generated by these fields, together with the generators of  $\cW\otimes V^{k+1}(\ga)$.
			\item The restriction of the Shapovalov form to the extension fields is nondegenerate.
		\end{enumerate}
		Then we have isomorphisms of $1$-parameter vertex superalgebras 
		\begin{equation}\label{eq:reconsio} \cA^{\psi}(n,r|s) \cong 
 \begin{cases}
 \cW^{k}(\gs\gl_{n+r+1|n+s}, F_{n+1|n}), & n\geq 1,\  s \neq r+1,\\
 \cW^{k}(\gp\gs\gl_{n+r+1|n+r+1}, F_{n+1|n}), & n\geq 1,\  r\geq 1,\  s = r+1,\\
V^{k}(\gs\gl_{r+1|s}) \otimes \cE(r) \otimes \cS(s), & n = 0,\  s \neq r+1,\\
 V^{k}(\gp\gs\gl_{r+1|r+1}) \otimes \cE(r) \otimes \cS(r+1), & n = 0,\  s = r+1.
 \end{cases}
\end{equation}
\end{theorem}

\begin{proof}
We recall the $\mathcal{N}=2$ degenerate diamonds 
$$P^{\frac{n+1}{2}}:=P^{(\alpha,2\alpha)}_{\cN=2}=\T{Span}\{P^{i,\bot},P^{i,-}|1\leq i\leq r+s\},\quad Q^{\frac{n+1}{2}}:=Q^{(\alpha,-2\alpha)}_{\cN=2}=\T{Span}\{Q^{i,\bot},Q^{i,+}|1\leq i\leq r+s\}$$ so that
        \[G^-(z)P^{i,\bot}(w)\sim 0 \sim G^+(z)Q^{i,\bot}(w),\]
and $\alpha$ is the constant
\begin{equation}\label{eq:alpha}
 \alpha = \frac{n+1}{2}-\frac{r-s+1}{2(k+r-s+1)}.
\end{equation}
By conformal and Heisenberg weight considerations, the $\cN=2$ degenerate diamonds $P^{\frac{n+1}{2}},Q^{\frac{n+1}{2}}$ are in fact primary over the universal $\mathcal{N}=2$ subVOA.

      Consider $W^2(z)P^{\frac{n+1}{2}}(w)$. 
        The most general form compatible with $\mathcal{N}=2$ symmetry is
        \begin{equation}
            W^2\times P^{\frac{n+1}{2}} = p_1 P^{\frac{n+1}{2}}+\dotsb,\quad W^3\times P^{\frac{n+1}{2}} = p_2 P^{\frac{n+1}{2}}+p_3M_3+ p_4M_4+\dotsb,
        \end{equation}
         where the omitted terms are normally ordered monomials in 
         $W^1$, generators of $\fr{gl}_m$, and their derivatives, and the fields $M_3$, $M_4$ denote the $\mathcal{N}=2$ primaries determined by the following bottom elements
        \begin{align*}
            M^{i,\bot}_3 =& :\!W^{2,\bot}P^{i,\bot}\!:+\dotsb,\\
            M^{i,\bot}_4 = &:\!W^{2,\bot}P^{i,-}\!:-2(n+1):\!W^{2,-}P^{i,\bot}\!:+\dotsb.
        \end{align*}
         Imposition of Jacobi identities $J(W^{2}, W^{2},P^{\frac{n+1}{2}})$ allows us to express all of the structure constants arising in $W^2(z)P^{\frac{n+1}{2}}(w)$ and $W^3(z)P^{\frac{n+1}{2}}(w)$ as algebraic functions in the central charge $c$, and polynomials in $\sqrt{\omega}$,
            \begin{equation}\label{eq:pretrunc}
            \begin{split}
                p_1^2=&-\frac{\alpha ^2 (2 \alpha +1) \omega  (c-6 \alpha ) (12 \alpha +c-3)}{(c-1) c (-3 \alpha +\alpha  c-c)},\\
              p_2 =&\frac{3 \alpha  \omega  \left(96 \alpha ^3+24 \alpha ^2+2 \alpha  c^2-5 c^2-16 \alpha ^2 c-18 \alpha  c+7 c\right)}{(c-1) c (-3 \alpha +\alpha  c-c)}\\
             p_3^2=&-\frac{144 \omega  \left(12 \alpha ^2+4 \alpha  c-c\right)^2}{(2 \alpha +1) (c-1) c (c-6 \alpha ) (12 \alpha +c-3) (-3 \alpha +\alpha  c-c)},\\
             p_4^2=&-\frac{144 \alpha ^2 \omega  \left(-48 \alpha ^2-18 \alpha +2 \alpha  c-5 c+9\right)^2}{(2 \alpha +1) (c-1) c (c-6 \alpha ) (12 \alpha +c-3) (-3 \alpha +\alpha  c-c)}.
              \end{split}
            \end{equation}
            Moreover, we obtain the following truncation curve
           \begin{equation}\label{eq:trunc}
              \lambda^2=-\frac{4 (c+3)^2 \omega  \left(12 \alpha ^2+4 \alpha  c-c\right)^2}{(2 \alpha +1) (c-1) c (c-6 \alpha ) (12 \alpha +c-3) (-3 \alpha +\alpha  c-c)}.
        \end{equation}
        There is an analogous system for $W^2(z)Q^{\frac{n+1}{2}}(w)$.

        Note  $J(W^2,W^2,P^{\frac{n+1}{2}})$ have uniquely determined OPEs in $W^3(z)P^{\frac{n+1}{2}}(w)$.
        More generally, using the weak generation (\ref{eq:raise}) of $\WNtwo$, together with Proposition \ref{eq:pretrunc}, it is not hard to see that Jacobi identities $J(W^2,W^q,P^{\frac{n+1}{2}})$ determine OPEs $W^{q+1}(z)P^{\frac{n+1}{2}}(w)$ in terms of $W^{q}(z)P^{\frac{n+1}{2}}(w)$; specifically, one can make use of identity $J_{r,1}(W^{2,\top},W^{q,\bot},P^{\frac{n+1}{2}})$. 
        Similar consideration allow us to determine $W^{q}(z)Q^{\frac{n+1}{2}}(w)$.
        Furthermore, using Proposition \ref{eq:pretrunc} and Jacobi relations $J(W^2,P^{\frac{n+1}{2}},Q^{\frac{n+1}{2}})$ one can uniquely determine OPEs among the extension fields $P^{\frac{n+1}{2}}(z)Q^{\frac{n+1}{2}}(w)$; specifically, the identity $J_{r,1}(W^{2,\top},P^{\frac{n+1}{2}},Q^{\frac{n+1}{2}})$ determines $P^{\frac{n+1}{2}}_{(r-2)}Q^{\frac{n+1}{2}}$ inductively in terms of higher order poles, with base case determined by condition (5) of Theorem \ref{recon:first}.
    \end{proof}
\begin{remark}
    The details of the evaluation of $J_{r,1}(W^2,P^{\frac{n+1}{2}},Q^{\frac{n+1}{2}})$ are tedious but straightforward, and can be worked out from VOA axioms together with the representation theory of $\fr{gl}_{r|s}$ or $\fr{sl}_{r|s}$. 
    However, the existence of recursive determination is sufficient for our purposes, so we omit the details.
\end{remark}

\begin{remark} \label{recon:cased0n} Theorem \ref{recon:first} also applies in the case $n\geq 1$, $r= 0$ and $s = 1$. The action of $V^{k+1}(\gs\gl_{r|r+1})$ on $\cW^k(\gp\gs\gl_{n+r+1|n+r+1}, F_{n+1|n})$ is replaced with the outer action of $U(1)$, and the argument is otherwise unchanged.
\end{remark}

\section{One-parameter quotients of $\WNtwo$}\label{sec:one param q} Our main result in this section is the following.
\begin{theorem}\label{trunc:main} For all $n, r, s$, $\cC^{\psi}_{\cN=2}(n,r|s)$ arises as $1$-parameter quotient of $\cW^{\cN=2}_{\infty}$ after a suitable finite localization.
Recall that $\lambda$ is the structure constant first appearing in the OPEs (\ref{eq:22}), and $\omega$ is a scaling parameter associated with the second diamond $W^2$. Setting $d = r-s$, the truncation curve is given by 
\begin{equation}\label{gentruncationcurve}
\begin{split}
    \frac{\lambda^2}{\omega}&=\frac{8 \psi  (d-n \psi +1)^2 (d-n \psi -\psi )^2 (2 d-2 n \psi -\psi +1)^2}{(d-n \psi -1) (d-n \psi ) (d-n \psi +\psi ) (d-(n+1) \psi +1) (d-(n+1) \psi +2) (d-(n+2) \psi +1)Q},\\
     Q&=\left(3 d^2-3 d (2 n \psi +\psi -1)+3 n \psi  (n \psi +\psi -1)+\psi \right).
   \end{split}
\end{equation}
Evaluating \eqref{eq:lambdas} and \eqref{eq:virs} we determine the associated parameters of the $\cW_{\infty}(c_{\pm},\lambda_{\pm})$ subVOAs are given by 
\begin{equation}
\begin{split}
    \lambda_+ &= \frac{-(\psi -1) \psi }{(d -n \psi -3 \psi +1) (d -n \psi -\psi +3) (d -n \psi +\psi -1)},\quad c_+= \frac{(n \psi -d ) (d -n \psi -\psi +2) (d -n \psi -2 \psi +1)}{-(1-\psi ) \psi },\\
\lambda_-&=  \frac{-(1-\psi ) \psi }{(d -n \psi -2) (d -n \psi -2 \psi +2) (d -n \psi +2 \psi )},\quad c_-= \frac{(d-n \psi -1) (d -n \psi +\psi ) (d -n \psi -\psi +1)}{-(\psi -1) \psi }.
\end{split}
\end{equation}

\end{theorem}
The proof will be divided into three cases (see Theorems \ref{trunc:c0rs}, \ref{trunc:en0}, and \ref{trunc:Cnmthird}), and has the following precise meaning. Recall the localization $R$ of $\mathbb{C}[c,\lambda]$ along \eqref{localization set}, and let $A_{n,r|s} \subseteq R$ be the ideal generated by \eqref{gentruncationcurve} after clearing denominators. There exists a localization  $E^{-1} (R/A_{n,r|s})$, and a (possibly non-maximal) vertex algebra ideal $\cA_{n,r|s} \subseteq \cW^{\cN=2,  A_{n,r|s}}_{\infty}$, such that we have isomorphisms of $1$-parameter vertex algebras
\begin{equation} \label{eq:localizationgeneral} E^{-1} (R/A_{n,r|s}) \otimes_{R/A_{n,r|s}} \cW^{\cN=2, A_{n,r|s}}_{\infty} / (\cA_{n,r|s}  \cdot \cW^{\cN=2, A_{n,r|s}}_{\infty}) \cong \cC^{\psi}_{\cN=2}(n,r|s).\end{equation}
Here $E$ is the (finite) set corresponding to the denominators of structure constants of $\cW^{\cN=2}_{\infty}$ after replacing $c, \lambda$ with the corresponding functions of $\psi$. In the case when $\cC^{\psi}_{\cN=2}(n,r|s)$ is one of the simple vertex algebras $\cE^{\psi}_{\cN=2}(n,r)$ or $\cD^{\psi}_{\cN=2}(n,s)$, we denote $A_{n,r|s}$ by $I^{\cN=2}_{n,r}$ and $J^{\cN=2}_{n,s}$, respectively. In this case, the vertex algebra ideal $\cA_{n,r|s}$ is maximal, so that $\cW^{\cN=2, A_{n,r|s}}_{\infty} / (\cA_{n,r|s}  \cdot \cW^{\cN=2, A_{n,r|s}}_{\infty})$ is in fact the simple quotient $\cW^{\cN=2}_{\infty, I^{\cN=2}_{n,r}}$ or $\cW^{\cN=2}_{\infty, J^{\cN=2}_{n,s}}$, respectively. Then \eqref{eq:localizationgeneral} becomes 
\begin{equation} E^{-1} (R/I^{\cN=2}_{n,r}) \otimes_{R/I^{\cN=2}_{n,r}} \cW^{\cN=2}_{\infty, I^{\cN=2}_{n,r}} \cong \cE^{\psi}_{\cN=2}(n,r),\quad F^{-1} (R/J^{\cN=2}_{n,s}) \otimes_{R/J^{\cN=2}_{n,s}} \cW^{\cN=2}_{\infty, J^{\cN=2}_{n,s}} \cong \cD^{\psi}_{\cN=2}(n,s),
\end{equation} where $E$ and $F$ are finite sets as above. Throughout this section, we suppress these localizations from our notation.

The fact that $\cW^{\cN=2}_{\infty}$ is a conformal extension of $\cH \otimes \cW^+_{\infty} \otimes \cW^-_{\infty}$ implies that $\cC^{\psi}_{\cN=2}(n,r|s)$ is a conformal extension of $\cH \otimes \cA^+ \otimes \cA^-$, where $\cA^{\pm}$ are $1$-parameter quotients of $\cW^{\pm}_{\infty}$, respectively. By Proposition \ref{prop:lambdas}, the truncation curve for $\cC^{\psi}_{\cN=2}(n,r|s)$, whose defining ideal is $A_{n,r|s}$, completely determine the parameters $\lambda_{\pm}$, or equivalently, the truncation curves for $\cA^{\pm}$. 

%The most useful way to formulate this information will be to give the truncation curves for $\cA^{\pm}$. In the case when $\cC^{\psi}_{\cN=2}(n,r|s)$ is one of the simple vertex algebras $\cE^{\psi}_{\cN=2}(n,m)$ and $\cD^{\psi}_{\cN=2}(n,m)$, we will see that $\cA^{\pm}$ are {\it simple} quotients of $\cW^{\pm}_{\infty}$. 

\begin{remark} \label{rem:honestcoset} We will always regard $\cC^{\psi}_{\cN=2}(n,r|s)$ as a $1$-parameter vertex algebra with formal parameter $\psi$. At a given point $\psi_0 \in \mathbb{C}$, the specialization
$$\cC^{\psi_0}_{\cN=2}(n,r|s) := \cC^{\psi}_{\cN=2}(n,r|s)   / (\psi-\psi_0) \cC^{\psi}_{\cN=2}(n,r|s)$$ makes sense as long as $\psi_0$ is not in the (finite) set of poles of denominators of structure constants. Here $(\psi - \psi_0)  \cC^{\psi}_{\cN=2}(n,r|s)$ is the vertex algebra ideal generated by $\psi - \psi_0$. For generic values of $\psi_0$, $\cC^{\psi_0}_{\cN=2}(n,r|s)$ coincides with the \lq\lq honest" coset 
\begin{equation} \label{eq:honestcoset} \text{Com}(V^{k_0+1}_{\mathcal{N}=1}(\gg\gl_{r|s}), \cW^{k_0}_{\mathcal{N}=1}(\gs\gl_{r+n+1|n+s}, f_{n+1|n})),\qquad \psi _0= k_0+r-s+1,\end{equation} 
which is obtained by first specializing $\cW^{k}_{\mathcal{N}=1}(\gs\gl_{r+n+1|n+s}, f_{n+1|n})$ to $k = k_0$, and then taking the coset by $V^{k_0+1}(\gg\gl_{r|s})$. However, $\cC^{\psi_0}_{\cN=2}(n,r|s)$ can be a proper subalgebra of \eqref{eq:honestcoset} for special values of $\psi_0$. The reason is that $\cW^{\cN=2}_{\infty}$ is {\it weakly generated} by the fields in weight at most $2$, so the same holds for the specialization $\cC^{\psi_0}_{\cN=2}(n,r|s)$, but this weak generation property of the honest coset can fail at special values $\psi_0$. Throughout this section, we always use the notation $\cC^{\psi}_{\cN=2}(n,r|s)$ for the $1$-parameter vertex algebra, $\cC^{\psi_0}_{\cN=2}(n,r|s)$ for its specialization to $\psi = \psi_0$, and \eqref{eq:honestcoset} for the honest coset. In the case of $\cC^{\psi}_{\cN=2}(n,0|0) = \cW^k(\gs\gl_{n+1|n})$ for $\psi = k+1$, we have already proven that this weak generation property holds for all noncritical levels $k$ in Lemma \ref{lem:weakgen}, and this fact will be useful to us.  %We will make use of this fact several times throughout this section, mainly to extract information from the isomorphisms \eqref{tautological} coming from intersection points on truncation curves.
\end{remark}

We begin with the easiest case $\cC^{\psi}_{\cN=2}(0,r|s)$. The most difficult case is $\cC^{\psi}_{\cN=2}(n,0|0) = \cW^k(\gs\gl_{n+1|n})$. We have already shown that $\cW^k(\gs\gl_{n+1|n})$ is a $1$-parameter quotient of $\cW^{\cN=2}_{\infty}$, but the truncation curve for $\cW^k(\gs\gl_{n+1|n})$ cannot be obtained using \eqref{eq:trunc}. Instead, using the isomorphism $\cW^k(\gs\gl_{n+1|n}) \cong \cW^k_{\cN=1}(\gs\gl_{n+1|n})$, we will compute it directly from the generators of $\cW^k_{\cN=1}(\gs\gl_{n+1|n})$ in the SUSY BRST complex. In the remaining cases $\cC^{\psi}_{\cN=2}(n,r|s)$ for $n > 0$ and $r+s >0$, the truncation curve can be deduced by specializing \eqref{eq:trunc}, but it requires some further effort to prove that $\cC^{\psi}_{\cN=2}(n,r|s)$ is weakly generated by the fields in weights $1, \frac{3}{2}, 2$. Finally, we will prove that in the cases $\cC^{\psi}_{\cN=2}(n,r|0)$ and $\cC^{\psi}_{\cN=2}(n,0|s)$ which are simple $1$-parameter vertex algebras, $\cA^{\pm}$ are also simple.

\subsection{The case $\cC^{\psi}_{\cN=2}(0,r|s)$} 
\begin{theorem} \label{trunc:c0rs} $\cC^{\psi}_{\mathcal{N}=2}(0,r|s)$ is a $1$-parameter quotient of $\cW^{\cN=2}_{\infty}$ with truncation curve given by \eqref{gentruncationcurve}.  \end{theorem} 
 \begin{proof} For all $r,s \geq 0$, $\cC^{\psi}_{\mathcal{N}=2}(0,r|s) = \text{Com}(V^{\ell+1}(\mathfrak{gl}_{r|s}), V^{\ell}(\mathfrak{sl}_{r+1|s}) \otimes  \cE(r) \otimes \cS(s))$ for $\psi = \ell + 1 + r - s$. This has large level limit
$$\big(\cO(s|r,1) \otimes \cO(r|s,2)\big)^{GL_{r|s}},$$ which is easily seen to be weakly generated by the fields in weights $1,\frac{3}{2}, 2$. It follows that for generic levels, $\cC^{\psi}_{\mathcal{N}=2}(0,r|s)$ is also weakly generated by the fields in weights $1,\frac{3}{2}, 2$, and hence by the $\mathcal{N}=2$ fields together with $W^{2,\bot}, W^{2,\pm}, W^{2,\top}$. Moreover, the automorphism $\sigma$ given by \eqref{eq:sigma} acts on this subalgebra, and hence on all of $\cC^{\psi}_{\mathcal{N}=2}(0,r|s)$. By Corollary \ref{cor:diamond}, $\cC^{\psi}_{\cN=2}(0,r|s)$ satisfies the hypotheses of Theorem \ref{one-parameter quotients theorem}, and is therefore a $1$-parameter quotient of $\cW^{\cN=2}_{\infty}$. Finally, the truncation curve for $\cC^{\psi}_{\mathcal{N}=2}(0,r|s)$ can be computed directly using Theorem \ref{recon:first}, and it agrees with \eqref{gentruncationcurve} for $n=0$. 
\end{proof}

Recall from Proposition \ref{prop:lambdas} that the truncation curve for $\cC^{\psi}_{\mathcal{N}=2}(0,r|s)$ as a $1$-parameter quotient of $\cW^{\cN=2}_{\infty}$, determines the truncation curves for the $\cW_{\infty}$-quotients $\cA^{\pm}$, and vice versa. We can deduce these independently as follows. For $s \neq r+1$, $V^{\ell}(\gs\gl_{r+1|s}) \otimes \cE(r) \otimes \cS(s)$ has a subVOA $V^{\ell}(\gg\gl_{r|s}) \otimes  \cE(r) \otimes \cS(s)$, and 
$$\text{Com}(V^{\ell}(\gg\gl_{r|s}) \otimes  \cE(r) \otimes \cS(s),  V^{\ell}(\gs\gl_{r+1|s}) \otimes  \cE(r) \otimes \cS(s)) \cong \text{Com}(V^{\ell}(\gg\gl_{r|s}), V^{\ell}(\gs\gl_{r+1|s})).$$ 
Next, $V^{\ell}(\gg\gl_{r|s}) \otimes  \cE(r) \otimes \cS(s)$ is an extension of 
$$V^{\ell+1}(\gg\gl_{r|s}) \otimes \text{Com}(V^{\ell+1}(\gg\gl_{r|s}), V^{\ell}(\gg\gl_{r|s}) \otimes  \cE(r) \otimes \cS(s)).$$

Therefore $\cC^{\psi}_{\mathcal{N}=2}(0,r|s)$ contains a subalgebra 
\begin{equation} \label{embeddedsubVOA} \text{Com}(V^{\ell}(\gg\gl_{r|s}), V^{\ell}(\gs\gl_{r+1|s})) \otimes \text{Com}(V^{\ell+1}(\gg\gl_{r|s}), V^{\ell}(\gg\gl_{r|s}) \otimes \ \cE(r) \otimes \cS(s)),\end{equation}
 which is conformally embedded because it has the same Virasoro element 
 $$(L^{V^{\ell}(\gs\gl_{r+1|s})} - L^{V^{\ell}(\gg\gl_{r|s})}) + (L^{V^{\ell}(\gg\gl_{r|s})} + L^{ \cE(r) \otimes \cS(s)} - L^{V^{\ell+1}(\gg\gl_{r|s})}) =L^{V^{\ell}(\gs\gl_{r+1|s})} + L^{ \cE(r) \otimes \cS(s)} - L^{V^{\ell+1}(\gg\gl_{r|s})}.$$ 
 Similarly, if $s = r+1$, $\cC^{\psi}_{\mathcal{N}=2}(0,r|r+1)$ has a conformally embedded subalgebra
 \begin{equation} \label{embeddedsubVOApsl}  \text{Com}(V^{\ell}(\gs\gl_{r|r+1}), V^{\ell}(\gp\gs\gl_{r+1|r+1}))^{U(1)}  \otimes \text{Com}(V^{\ell+1}(\gs\gl_{r|r+1}), V^{\ell}(\gs\gl_{r|r+1}) \otimes  \cE(r) \otimes \cS(r+1))^{U(1)}.\end{equation}
 
For $r\geq s$, we have
 \begin{equation}  \text{Com}(V^{\ell}(\gg\gl_{r|s}), V^{\ell}(\gs\gl_{r+1|s})) \cong \cA^+,\qquad  \text{Com}(V^{\ell+1}(\gg\gl_{r|s}), V^{\ell}(\gg\gl_{r|s}) \otimes  \cE(r) \otimes \cS(s)) \cong \cH \otimes \cA^-.
 \end{equation}
 It is straightforward to check using \cite[Theorem 9.1]{CL2} that $\cA^+$ and $\cA^-$ are $1$-parameter quotients of $\cW_{\infty}$ along the truncation curves of $\cC^{\psi}(1,r-s)$ and $\cD^{-\psi+1}(0,r-s)$, respectively.
 
 Similarly, for $s - r > 1$, we have
  \begin{equation}  \text{Com}(V^{\ell}(\gg\gl_{r|s}), V^{\ell}(\gs\gl_{r+1|s})) \cong \cA^+,\qquad  \text{Com}(V^{\ell+1}(\gg\gl_{r|s}), V^{\ell}(\gg\gl_{r|s}) \otimes  \cE(r) \otimes \cS(s)) \cong \cH \otimes \cA^-.
 \end{equation}
 Here $\cA^+$ and $\cA^-$ are $1$-parameter quotients of $\cW_{\infty}$ along the truncation curves of $\cD^{\psi}(1,s-r)$ and $\cC^{-\psi+1}(0,s-r)$, respectively.

 Finally, when $s = r+1$,
  \begin{equation}  \text{Com}(V^{\ell}(\gs\gl_{r|r+1}), V^{\ell}(\gp\gs\gl_{r+1|r+1}))^{U(1)} \cong \cA^+,\quad \text{Com}(V^{\ell+1}(\gs\gl_{r|r+1}), V^{\ell}(\gs\gl_{r|r+1}) \otimes  \cE(r) \otimes \cS(r+1))^{U(1)}   \cong \cH \otimes \cA^-  \end{equation}
 Here $\cA^+$ and $\cA^-$ are $1$-parameter quotients of $\cW_{\infty}$ along the truncation curves of $\cD^{\psi}(1,1)$ and $\cC^{-\psi+1}(0,1)$, respectively.

 \begin{corollary} \label{trunc:e0md0m} Let $r,s \geq 1$. Then
 \begin{enumerate}
 \item $\cE^{\psi}_{\mathcal{N}=2}(0,r)$ is a conformal extension of $\cH \otimes \cC^{\psi}(1,r) \otimes \cD^{-\psi+1}(0,r)$.
  \item $\cD^{\psi}_{\mathcal{N}=2}(0,s)$ is a conformal extension of $\cH \otimes \cD^{\psi}(1,s+1) \otimes \cC^{-\psi+1}(0,s+1)$.
 \end{enumerate}
 In particular, the subalgebras $\cA^{\pm}$ are both simple in these cases.
 \end{corollary} 

 \begin{proof} 
 Specializing \eqref{embeddedsubVOA} to the case $s=0$, we see that $\cE^{\psi}_{\mathcal{N}=2}(0,r)$ is a conformal extension of
 $$\text{Com}(V^{\ell}(\gg\gl_{r}), V^{\ell}(\gs\gl_{r+1})) \otimes \text{Com}(V^{\ell+1}(\gg\gl_{r}), V^{\ell}(\gg\gl_{r}) \otimes \cE(r)),\qquad \psi = \ell + r+1.$$
We obtain
 \begin{equation} 
\begin{split} \cA^+ & \cong \text{Com}(V^{\ell}(\gg\gl_{r}), V^{\ell}(\gs\gl_{r+1})) = \cC^{\psi}(1,r),
\\ \cA^- & \cong \text{Com}(\cH, \text{Com}(V^{\ell+1}(\gg\gl_{r}), V^{\ell}(\gg\gl_{r}) \otimes \cE(r)) )
\\ & \cong \text{Com}(\cH, \text{Com}(V^{\ell+1}(\gg\gl_{r}), \cH \otimes V^{\ell}(\gs\gl_{r}) \otimes \cE(r)))
\\ & \cong \text{Com}(V^{\ell+1}(\gg\gl_{r}), V^{\ell}(\gs\gl_{r}) \otimes \cE(r))
\\ & =  \cD^{-\psi+1}(0,r).
\end{split}
\end{equation}

Similarly, $\cD^{\psi}_{\mathcal{N}=2}(0,s)$ is an extension of
$$\text{Com}(V^{-\ell}(\gg\gl_{s+1}), V^{\ell}(\gs\gl_{1|s+1})) \otimes \text{Com}(V^{-\ell-1}(\gg\gl_{s+1}), V^{-\ell}(\gg\gl_{s+1}) \otimes \cS(s+1), \qquad \psi = \ell - s.$$ 
We have
 \begin{equation} 
\begin{split} \cA^+ \cong \text{Com}(V^{-\ell}(\gg\gl_{s+1}), V^{\ell}(\gs\gl_{1|s+1})) & = \cD^{\psi}(1,s+1), 
\\ \cA^- \cong  \text{Com}(V^{-\ell-1}(\gg\gl_{s+1}), V^{-\ell}(\gs\gl_{s+1}) \otimes \cS(s+1)) & = \cC^{-\psi+1}(0,s+1).
\end{split}
\end{equation}
\end{proof}

\subsection{The case $\cC^{\psi}_{\cN=2}(n,0|0)$}  \label{subsec:truncation W(sl)}
We have already proven in Proposition \ref{PrincW:quot} that $\cC^{\psi}_{\cN=2}(n,0|0) = \cW^k(\mathfrak{sl}_{n+1|n})$ is a $1$-parameter quotient of $\WNtwo$. Hence $\cW^k(\mathfrak{sl}_{n+1|n})$ has two commuting Virasoro fields $L_{\pm}$ which are in the centralizer of the weight 1 field $H$. Using the isomorphism $\cW^k(\mathfrak{sl}_{n+1|n}) \cong \cW^k_{\cN=1}(\mathfrak{sl}_{n+1|n})$, which we regard as a subalgebra of $\cW^k_{\cN=1}(\mathfrak{gl}_{n+1|n})$, we will compute the central charge of one of these Virasoro fields, namely $L_-$. Together with the central charge of $\cW^k(\mathfrak{sl}_{n+1|n})$, this specifies the central charge of $L_+$, and hence the parameter $\lambda$ that specifies the truncation curve for $\cW^k(\mathfrak{sl}_{n+1|n})$.

In \cite{MRS21}, it was shown that a strong generating set of $\cW^k_{\cN=1}(\mathfrak{gl}_{n+1|n})$ can be obtained as the coefficients of the column determinant of the following $(2n+1)\times (2n+1)$ matrix 
\begin{equation}
    \mathcal{A}_n(\mathfrak{gl}) := \left[ \mathcal{A}^{\mathfrak{gl}}_{ij} \right]_{i,j=1}^{2n+1}\, , \quad \text{ where } \quad \mathcal{A}^{\mathfrak{gl}}_{ij}=\left\{
    \begin{array}{ll}
    \delta_{ij}\psi D +(-1)^{i+1} J_{\bar{e}_{ij}} & \textup{if }i \leq j,\\
    1 & \textup{if }i=j+1,\\
    0 & \textup{otherwise.}
    \end{array}
    \right.
\end{equation}
In the above, $\psi=k+r-s+1=k+1$, and $J_{\bar{e}_{ij}}$ are the building blocks of the SUSY BRST complex for the elements $e_{ij}\in \mathfrak{gl}_{n+1|n}$ introduced in Section \ref{sec:W-algebra type A(n,n-1)}. Let $I:= \sum_{i=1}^{2n+1} e_{ii}$ and $E_{ij}:= e_{ij}+(-1)^i \delta_{i,j} I$. Then $E_{ij}$ is in $\mathfrak{sl}_{n+1|n}$ and the coefficients of the column determinant of 
\begin{equation}
    \mathcal{A}_n := \left[ \mathcal{A}_{ij} \right]_{i,j=1}^{2n+1}\, , \quad \text{ where } \quad \mathcal{A}_{ij}= \left\{ \begin{array}{ll} \delta_{ij}\psi D +(-1)^{i+1} J_{\bar{E}_{ij}} & \textup{if } i \leq j, \\
    1& \textup{if }i=j+1, \\
    0& \textup{otherwise}\end{array}\right.
\end{equation}
generates $\cW^k_{\cN=1}(\gs\gl_{n+1|n}).$ The column determinant of $\mathcal{A}_n$ is 
\begin{equation} \label{eq:cdet}
\begin{aligned}
        \text{cdet}\mathcal{A}_n & = \sum_{\substack{i_0=0<i_1< \cdots < i_N=2n+1\\ N\geq 1}} \mathcal{A}_{i_{1}\,  i_0+1} \mathcal{A}_{i_2\, i_1+1}\cdots \mathcal{A}_{i_N\, i_{N-1}+1} \\
        & = \psi^{2n+1} D^{2n+1} + W[1] D^{2n-1} + W\big[\frac{3}{2}\big] D^{2n-2}+W[2] D^{2n-3} + \cdots + W\big[n+\frac{1}{2}\big],
\end{aligned}
\end{equation}
and $W[m]$ is a weight $m$ free generator of $\cW^k_{\cN=1}(\mathfrak{sl}_{n+1|n}).$
In this section, we denote $\bar{E}_{ij}:=J_{\bar{E}_{ij}}$, $D(\bar{E}_{ij}):=D(J_{\bar{E}_{ij}})$,  $\bar{e}_{ij}:=J_{\bar{e}_{ij}}$ and $D(\bar{e}_{ij}):=D(J_{\bar{e}_{ij}})$ for the simplicity of notations. Similarly, we denote $\partial(\bar{e}_{i j}):=\partial(J_{\bar{e}_{i j}})$. In the rest of this section, we intensively use the following OPE 
\begin{equation}
\begin{aligned}
    & \bar{E}_{ii}(z)\bar{E}_{jj}(w)\sim \frac{\delta_{ij}\psi(-1)^{i+1}+\psi(-1)^{i+j+1}}{z-w}, \\
    & D(\bar{E}_{ii})(z)D(\bar{E}_{jj})(w)\sim \frac{\delta_{ij}\psi(-1)^{i+1}+\psi(-1)^{i+j+1}}{(z-w)^2},
\end{aligned}
\end{equation}
which are induced from $\bar{e}_{ii}(z)\bar{e}_{jj}(w)\sim \delta_{ij}\frac{\psi(-1)^{i+1}}{z-w}$ and  $D(\bar{e}_{ii})(z)D(\bar{e}_{jj})(w)\sim \delta_{ij}\frac{\psi(-1)^{i+1}}{(z-w)^2.}$
\begin{lemma}\label{lem:commuting virasoro-1}
    The OPE between $W[1]$ and itself is 
    \begin{equation}
        W[1](z)W[1](w) \sim \frac{n \psi^{4n}(1-(n+1)\psi)}{(z-w)^2}.
    \end{equation}
\end{lemma}
\begin{proof}
    The image of the SUSY Miura map $\mu_D$ of $W[1]$ is 
    \begin{equation}
    \mu_D(W[1])= -\psi^{2n}\sum_{i=1}^n  D(\bar{E}_{2i\, 2i}) \ -\ \psi^{2n-1}\sum_{1\leq i<j \leq 2n+1}:\bar{E}_{ii}\bar{E}_{jj}:.
    \end{equation}
    We can easily see that the first order pole of $\mu_D(W[1])(z) \mu_D(W[1])(w)$ is  trivial. For the second order pole, it is enough to show 
   \begin{equation}\label{eq:commuting virasoro-1-1}
   \begin{aligned}
          \sum_{i,j=1}^n\quad \ &  \big(D(\bar{E}_{2i\, 2i})\big)(z)\big(  D(\bar{E}_{2j\, 2j})\big)(w) \sim -\frac{\psi n(n+1)}{(z-w)^2}, \\
          \sum_{\substack{1\leq i<j \leq 2n+1\\ 1\leq s<t \leq 2n+1}}& \big(:\bar{E}_{ii}\bar{E}_{jj}:\big)(z)\big(:\bar{E}_{ss}\bar{E}_{tt}:\big)(w) \sim \frac{\psi^2 n }{(z-w)^2}.
   \end{aligned}
   \end{equation}
   The first equality in \eqref{eq:commuting virasoro-1-1} is obtained directly from the fact  that $D(\bar{E}_{2i\, 2i})(z)D(\bar{E}_{2j\, 2j})(w)\sim -\frac{2 \psi}{(z-w)^2}$ if $i=j$ and $-\frac{\psi}{(z-w)^2}$ otherwise. For the second equality in \eqref{eq:commuting virasoro-1-1}, observe that 
   \begin{equation}
       \sum_{1\leq i<j\leq 2n+1}:\bar{E}_{ii}\bar{E}_{jj}:= \sum_{1\leq i<j\leq 2n+1}:\big(\bar{e}_{ii}+(-1)^i \bar{I}\big) \big(\bar{e}_{jj}+(-1)^j \bar{I}\big):\  = \sum_{1\leq i<j\leq 2n+1}:\bar{e}_{ii} \bar{e}_{jj}:\, . 
   \end{equation}
   Since $\bar{e}_{ii}(z)\bar{e}_{jj}(z) \sim \delta_{ij} \frac{\psi(-1)^{i+1}}{(z-w)}, $ we have 
   \begin{equation}
       :\bar{e}_{ii}\bar{e}_{jj}:(z):\bar{e}_{ii}\bar{e}_{jj}:(w) \sim -\frac{\psi^2(-1)^{i+1}(-1)^{j+1}}{(z-w)^2}.
   \end{equation}
   By comparing the number of ways to choose two distinct numbers from  $\{1,2, \dots, 2n+1\}$ with the same parity and with opposite parities, we obtain the second equality in \eqref{eq:commuting virasoro-1-1}.
\end{proof}

\begin{lemma} \label{lem:commuting virasoro-2}
 The weight 2 subspace of $\cW^k_{\cN=1}(\mathfrak{sl}_{n+1|n})$ commuting with $W[1]$ is spanned by the following two elements 
 \begin{equation} 
     \widetilde{W}[2]:= W[2] + C(W[2]), \quad \widetilde{V}\big[2\big]:= D\Big(W\big[\frac{3}{2}\big]\Big) + C(V[2]),
 \end{equation}
 where $C(W[2])$ and $C(V[2])$ are in the vertex subalgebra generated by $W[1].$  Precisely, we have 
 \begin{equation}
     \widetilde{W}[2]=W[2]-\frac{n-1}{(2n)\,  \psi^{2n+1}}: W[1]W[1]: -\frac{n-1}{2}\partial(W[1]).
 \end{equation}
\end{lemma}
\begin{proof}
    Let $\gg=\mathfrak{sl}_{n+1|n}$. From the fact that $[a,b]=0$ for any $a\in \gg^f_0$ and $b\in \gg^f_{-1}$, we can conclude that the first and the second order pole of $W[2](z)W[1](w)$ are in the vertex subalgebra $\left< W[1]\right>$ generated by $W[1].$ Hence to determine the first and second order poles of $W[2](z)W[1](w)$, it suffices to track those terms in $W[2](z)W[1](w)$ that involve only $\bar{E}_{21}(w)$ and its derivatives. Now let us show
    \begin{equation} \label{eq:W[2](m)W[1]}
        W[2]_{(m)}W[1]= \frac{1}{\psi^{2n-1}}\cdot \frac{n-1}{2n\psi^2}  :W[1]W[1]:_{(m)}W[1]
    \end{equation}
    for $m=0,1.$ By examining the column determinant formula \eqref{eq:cdet}, to identify the terms in the LHS of \eqref{eq:W[2](m)W[1]} that involve only $\bar{E}_{21}(w)$ and its derivatives, it is enough to consider the component \[:(-\bar{E}_{21}) \big(\mu_D(W_{n-1}[1])\big) :\] in $W[2]$, where $\mu_D(W_{n-1}[1]) = -\psi^{2n-2}\sum_{i=2}^n D(\tilde{E}_{2i\, 2i}) -\psi^{2n-3}\sum_{3 \leq i<j\leq 2n+1}:\bar{E}_{ii}\bar{E}_{jj}:=\frac{1}{\psi^2}\big(\mu_D(W[1])-D(\tilde{E}_{22})-\sum_{i=1,2, \, j\geq 3} \bar{E}_{ii}\bar{E}_{jj}\big)$ together with the free field part $\mu_D(W[1])$ of $W[1].$
    For the RHS of \eqref{eq:W[2](m)W[1]}, notice that the linear part of $W[1]$ is $-\psi^{2n-1} \bar{E}_{21}$. Finally,  one can get \eqref{eq:W[2](m)W[1]} by direct computations. 

    For the third order pole, we need to compute $W[2]_{(3)}W[1]$ and compare with $\partial(W[1])_{(3)}W[1]$. This process is similar but simpler than Lemma \ref{lem:commuting virasoro-4}. Hence we omit the detailed proof.
    For $\widetilde{V}[2]$, let us first note that the weight $1$ and $\frac{3}{2}$ fields in $W^{k}_{\cN=1}(\gg)$ and the weight 2 field $D\Big(W \big[\frac{3}{2}\big]\Big)$ comprise the initial diamond part. Consequently, $D\Big(W \big[\frac{3}{2}\big]\Big)$ is a constant multiple of the total Virasoro of $\cW^k_{\cN=1}(\gg)$. Hence one can find an element $C(V[2])$ in $\left< W[1]\right>$ which let $\widetilde{V}[2]$ commutes with $W[1].$ Indeed, for 
    \begin{equation}
    C(V[2])=-\frac{1}{(2n \psi^{2n} (-1+(n+1)\psi)} :W[1]W[1]:-\frac{1}{2} \partial(W[1]),
    \end{equation}
    $\widetilde{V}[2]$ commutes with $W[1].$ Now, since there is no center in $\left< W[1]\right>$, it is obvious that the weight 2 space commuting with $W[1]$ has at most dimension $2.$ Hence we proved the lemma.
\end{proof}

\begin{lemma}\label{lem:commuting virasoro-3}
 In Lemma \ref{lem:commuting virasoro-2}, $\widetilde{W}[2]$ is a constant multiple of one of the commuting Virasoros. More precisely, 
 \begin{equation}
Vir:=\frac{1}{\psi^{2n-1}(\psi-1)} \widetilde{W}[2]
 \end{equation}
 is one of the commuting Virasoros.
\end{lemma}

\begin{proof}
   To prove that $Vir$ is a Virasoro, it suffices to show that the second order pole of the OPE of $Vir$ and itself is $2\cdot\, Vir$. The first  and third order poles are then determined by skew-symmetry.
   On the other hand, $\widetilde{W}[2]$ commutes with $W[1]$and by Lemma \ref{lem:commuting virasoro-2}, the weight‑2 subspace commuting with  $W[1]$ is spanned by $\widetilde{W}[2]$ and $\widetilde{V}[2]$. Hence, the second-order pole of the OPE of $\widetilde{W}[2]$ and itself is a linear combination of $\widetilde{W}[2]$ and $\widetilde{V}[2].$ Therefore, to prove that $Vir$ is a Virasoro field, it suffices to show that the linear part of $W[2]_{(1)}W[2]$ contains the term $-2\psi^{4n-4}(\psi-1) \bar{E}_{41}$ but does not contain $D(\bar{E}_{31}).$ Here we note that by the formula of $cdet \mathcal{A}_n$, one can see that $\bar{E}_{41}$ appears as a linear term $-\psi^{2n-3}\bar{E}_{41}$ in $W[2]$ but not in $DW\big[\frac{3}{2}\big]$ or $W[1]$. Similarly, $D(\bar{E}_{31})$ only appears in $DW\big[\frac{3}{2}\big]$ but not in $W[2]$ or $W[1].$
   
   By the column determinant formula of $\mathcal{A}_n$, one can show the parts with $\bar{E}_{41}$ in $W[2]_{(1)}W[2]$ can arise in the following terms in $W[2]_{(1)}W[2]$:
   \begin{equation}
   \begin{aligned}
       &  \big(-\psi^{2n-2}D(\bar{E}_{42})\big)_{(1)}\big(\psi^{2n-2}\bar{E}_{21}D (\bar{E}_{44})\big)+ \big(\psi^{2n-2}\bar{E}_{21}D (\bar{E}_{44})\big)_{(1)}\big(-\psi^{2n-2}D(\bar{E}_{42})\big)\\
       & + \big( -\psi^{2n-3}\bar{E}_{41}\big)_{(1)}\big( -\psi^{2n}D^3(\bar{E}_{44}\big)\big)+ \big( -\psi^{2n}D^3(\bar{E}_{44})\big)_{(1)}\big( -\psi^{2n-3}\bar{E}_{41}\big)\\
        & = \psi^{4n-4}2(1-\psi)\bar{E}_{41}.    \end{aligned}
   \end{equation}
   In a similar way, one can show that there is no linear term with $D(\bar{E}_{31})$ in $W[2]_{(1)}W[2]$.
   Hence we can conclude $Vir$ is a Virasoro.   Similarly, one can show that $\frac{1}{\psi^{2n}}\big(\widetilde{V}[2]+\frac{\psi}{\psi-1}\widetilde{W}[2]\big)$ is another Virasoro, commuting with $Vir.$
\end{proof}

In order to specify the point $\cW^k(\mathfrak{sl}_{n+1|n})\cong \cW^k_{\cN=1}(\mathfrak{sl}_{n+1|n})$ in $\WNtwo$, we need to find the central charge of $Vir$. To this end, let us find the 4th order pole of $W[2](z)W[2](w)$ which is completely determined by 
the free field part $\mu_D(W[2])$ of $W[2].$ From the column determinant formula of $\mathcal{A}_n,$ we have 
\begin{align}
    & \label{eq:commuting virasoro-4-1} \mu_D(W[2]) = -\psi^{2n}\sum_{i=1}^{n} (i-1)D(\partial(\bar{E}_{2i\, 2i})) \\
    &\label{eq:commuting virasoro-4-2} +\psi^{2n-1}\sum_{1\leq i<j\leq n}:D(\tilde{E}_{2i\, 2i})D(\tilde{E}_{2j\, 2j}):\\
    & \label{eq:commuting virasoro-4-3}-\psi^{2n-1} \Big( \sum_{1\leq i <2j \leq 2n-1} (j-1):\bar{E}_{ii} \, \partial(\bar{E}_{2j\, 2j}):+\sum_{1\leq i <2j+1 \leq 2n-1} (j-1):\bar{E}_{ii} \, \partial(\bar{E}_{2j+1\, 2j+1}):\\
    & \label{eq:commuting virasoro-4-4}\hskip 1cm + \sum_{\substack{i\geq 1\\ 2i+1<j\leq 2n-1}}(i-1):\partial(\bar{E}_{2i-1\, 2i-1})\,\bar{E}_{jj}:+\sum_{\substack{i\geq 1\\ 2i<j\leq 2n-1}}(i-1):\partial(\bar{E}_{2i\, 2i})\,\bar{E}_{jj}:\Big)\\
    &\label{eq:commuting virasoro-4-5} +\psi^{2n-2}\Big(\sum_{\substack{i<j<k\\ i:\text{even}}}:D(\bar{E}_{ii})\bar{E}_{jj}\bar{E}_{kk}:-\sum_{\substack{i<j<k\\ j:\text{odd}}}:D(\bar{E}_{jj})\bar{E}_{ii} \bar{E}_{kk}:+ \sum_{\substack{i<j<k\\ k:\text{even}}}:D(\bar{E}_{kk})\bar{E}_{ii} \bar{E}_{jj}:  \Big)\\
    & \label{eq:commuting virasoro-4-6}+ \psi^{2n-3} \sum_{i<j<s<t}:\bar{E}_{ii}\bar{E}_{jj}\bar{E}_{ss}\bar{E}_{tt}:\, .
\end{align}

\begin{lemma}\label{lem:commuting virasoro-4}
The 4th order pole of OPE between $W[2]$ and itself is $p=p_1+p_2+p_3+p_4$, where 
\begin{equation}
\begin{aligned}
    & p_1= \Big((n - 1) n (2 n - 1) + \frac{3(n (n - 1))^2}{2} \Big) \psi^{4n+1}\, , \\
    & p_2= \Big(-(n - 1) n (2 n - 1) + \frac{n (n - 1)}{
    2} (\, 5 + 6 (n - 2) + (n - 2) (n - 3)\, )\Big)\,    \psi^{4n}\, ,  \\
    & p_3= -(n - 1) n^2  \psi^{4n-1}\, ,  \\
    & p_4= \frac{n (n - 1)}{2}  \psi^{4n-2}\, .
\end{aligned}
\end{equation}
\end{lemma}
\begin{proof}
    The value $p_1$ is the 4th order pole of \eqref{eq:commuting virasoro-4-1} and itself. To be clear, note that 
    \begin{equation}
        D(\partial (\bar{E}_{2i\, 2i}))\ {}_{(3)}\, D(\partial(\bar{E}_{2i\, 2i}))=12\psi, \quad   D(\partial(\bar{E}_{2i\, 2i}))\ {}_{(3)}\, D(\partial(\bar{E}_{2j\, 2j}))=6\psi
    \end{equation}
    for distinct $i$ and $j.$ Hence 
    \begin{equation}
        \eqref{eq:commuting virasoro-4-1}_{(3)}\eqref{eq:commuting virasoro-4-1}= \psi^{4n+1} \big(\, 6 (1^2+2^2 +\cdots + (n-1)^2) + 6 \big( \frac{n(n-1)}{2}\big)^2\, \big)=p_1.
    \end{equation}
    Let us show the value $p_2$ is the 4th order pole of \eqref{eq:commuting virasoro-4-2}+ \eqref{eq:commuting virasoro-4-3}+\eqref{eq:commuting virasoro-4-4} and itself. In addition, it is the same as $\eqref{eq:commuting virasoro-4-2}_{(3)}\eqref{eq:commuting virasoro-4-2}+ (\eqref{eq:commuting virasoro-4-3}+\eqref{eq:commuting virasoro-4-4})_{(3)}(\eqref{eq:commuting virasoro-4-3}+\eqref{eq:commuting virasoro-4-4})$. In order to find $\eqref{eq:commuting virasoro-4-2}_{(3)}\eqref{eq:commuting virasoro-4-2}$, note that for $i\neq j$ and $s\neq t$
    \begin{equation}
        (:D(\bar{E}_{2i\, 2i})D(\bar{E}_{2j\, 2j}):)_{(3)}(:D(\bar{E}_{2s\, 2s})D(\bar{E}_{2t\, 2t}):)=\left\{ \begin{array}{ll} 5 \psi^2 & \text{ if } i=s \text{ and } j=t, \\
        3 \psi^2 & \text{ if exactly two of $i,j,s,t$ are the same,}  \\
        2 \psi^2 & \text{ if all $i,j,s,t$ are distinct.} \end{array} \right.
    \end{equation}
    Hence 
    \begin{equation} \label{eq:p2-part1}
        \eqref{eq:commuting virasoro-4-2}_{(3)}\eqref{eq:commuting virasoro-4-2}= \psi^{4n} \Big(5\cdot \frac{n(n-1)}{2}+3\cdot n(n-1)(n-2)+ 2 \cdot 2 \cdot \frac{n(n-1)}{2}\frac{(n-2)(n-3)}{2}  \Big).
    \end{equation} 
    Now, for $\eqref{eq:commuting virasoro-4-3}+\eqref{eq:commuting virasoro-4-4})_{(3)}(\eqref{eq:commuting virasoro-4-3}+\eqref{eq:commuting virasoro-4-4}$, observe that 
    \begin{equation} 
    \begin{aligned}
        & \eqref{eq:commuting virasoro-4-3}+\eqref{eq:commuting virasoro-4-4} \\
        & =-\psi^{2n-1} \Big( \sum_{1\leq i <2j \leq 2n-1} (j-1):\bar{e}_{ii} \,\partial(\bar{e}_{2j\, 2j})+\sum_{1\leq i <2j+1 \leq 2n-1} (j-1):\bar{e}_{ii}\, \partial(\bar{e}_{2j+1\, 2j+1}):\\
    & \hskip 1cm + \sum_{\substack{i\geq 1\\ 2i+1<j\leq 2n-1}}(i-1):\partial(\bar{e}_{2i-1\, 2i-1})\,\bar{e}_{jj}:+\sum_{\substack{i\geq 1\\ 2i<j\leq 2n-1}}(i-1):\partial(\bar{e}_{2i\, 2i}),\bar{e}_{jj}:\Big).
    \end{aligned}
    \end{equation}
    This equality can be obtained by replacing $E_{ii}$ by $e_{ii}+(-1)^i I$ in \eqref{eq:commuting virasoro-4-3}+\eqref{eq:commuting virasoro-4-4} and see all the coefficients of $:\bar{e}_{ii} \,\partial(\bar{I}):$, $:\partial(\bar{e}_{ii})\, \bar{I}:$ and $:\bar{I}\,\partial(\bar{I}):$ are trivial. Moreover, observe that 
    \begin{equation}
    \begin{aligned}
        & \big(a:\bar{e}_{2i-1\, 2i-1}\,\partial(\bar{e}_{jj}):+\, b :\partial(\bar{e}_{2i-1\, 2i-1})\,\bar{e}_{jj}:\big)_{(3)}\big(a:\bar{e}_{2i-1\, 2i-1}\,\partial(\bar{e}_{jj}):+\, b :\partial(\bar{e}_{2i-1\, 2i-1})\,\bar{e}_{jj}:\big)\\
        & \quad  =- \big(a:\bar{e}_{2i\, 2i}\,\partial(\bar{e}_{jj}):+\, b :\partial(\bar{e}_{2i\, 2i})\,\bar{e}_{jj}:\big)_{(3)}\big(a:\bar{e}_{2i\, 2i}\,\partial(\bar{e}_{jj}):+\, b :\partial(\bar{e}_{2i\, 2i})\,\bar{e}_{jj}:\big)
    \end{aligned}
    \end{equation}
    for any constant $a,b$ and $j \geq 2i.$ In conclusion, we have 
    \begin{equation} \label{eq:p2-part2}
    \begin{aligned}
         & \big(\eqref{eq:commuting virasoro-4-3}+\eqref{eq:commuting virasoro-4-4}\big)_{(3)} \big(\eqref{eq:commuting virasoro-4-3}+\eqref{eq:commuting virasoro-4-4}\big)\\
         & = \psi^{4n-2}\sum_{i=2}^{n}\, (i-1)^2 \big(:\bar{e}_{2i-1\, 2i-1}\,\partial(\bar{e}_{2i\, 2i}):+:\partial(\bar{e}_{2i-1\, 2i-1})\,\bar{e}_{2i\, 2i}:\big){}_{(3)}\\
         & \hskip 5cm \big(:\bar{e}_{2i-1\, 2i-1}\,\partial(\bar{e}_{2i\, 2i}):+:\partial(\bar{e}_{2i-1\, 2i-1})\,\bar{e}_{2i\, 2i}:\big)\\
         & = \ -6\psi^{4n}(1^2+2^2+\cdots (n-1)^2) \ =\ -\psi^{4n}(n-1)n(2n-1).
    \end{aligned}        
    \end{equation}
    By \eqref{eq:p2-part1} and \eqref{eq:p2-part2}, we get $p_2= \big(\eqref{eq:commuting virasoro-4-2}+\eqref{eq:commuting virasoro-4-3}+\eqref{eq:commuting virasoro-4-4}\big)_{(3)} \big(\eqref{eq:commuting virasoro-4-2}+\eqref{eq:commuting virasoro-4-3}+\eqref{eq:commuting virasoro-4-4}\big).$

    Next, let us show that $p_3= \eqref{eq:commuting virasoro-4-5}_{(3)}\eqref{eq:commuting virasoro-4-5}.$ One can see that 
    \begin{equation}
        \eqref{eq:commuting virasoro-4-5}= \psi^{2n-2}\Big(\sum_{\substack{i<j<k\\ i:\text{even}}}:D(\bar{E}_{ii})\bar{e}_{jj}\bar{e}_{kk}:-\sum_{\substack{i<j<k\\ j:\text{odd}}}:D(\bar{E}_{jj})\bar{e}_{ii} \bar{e}_{kk}:+ \sum_{\substack{i<j<k\\ k:\text{even}}}:D(\bar{E}_{kk})\bar{e}_{ii} \bar{e}_{jj}:  \Big).
    \end{equation}
    In addition, we have \begin{equation}
    \begin{aligned}
        & (:D(\bar{E}_{2s\, 2s})\bar{e}_{2i-1\, 2i-1}\bar{e}_{j\, j}:+:D(\bar{E}_{2s\, 2s})\bar{e}_{2i\,2i}\bar{e}_{j\, j}:)_{(3)}\eqref{eq:commuting virasoro-4-5}=0, \quad j>2i \text{ and } 2s<2i-1\text{ or }2s>j, \\
        & (:D(\bar{E}_{2s\, 2s})\bar{e}_{ii}\bar{e}_{2j\, 2j}:+:D(\bar{E}_{2s\, 2s})\bar{e}_{ii}\bar{e}_{2j+1\, 2j+1}:)_{(3)}\eqref{eq:commuting virasoro-4-5}=0, \quad i<2j \text{ and }2s<i \text{ or }2s>2j+1, \\
        & (:D(\bar{E}_{2s+1\, 2s+1})\bar{e}_{2i-1\, 2i-1}\bar{e}_{j\, j}):+:D(\bar{E}_{2s+1\, 2s+1})\bar{e}_{2i\,2i}\bar{e}_{j\, j}:)_{(3)}\eqref{eq:commuting virasoro-4-5}=0, \ i<s\text{ and } 2s+1<j.
     \end{aligned}
    \end{equation}
    This implies $\eqref{eq:commuting virasoro-4-5}_{(3)}\eqref{eq:commuting virasoro-4-5}$ is the sum of $n(n-1)$ terms of the form $\big(:D(\bar{E}_{2i \, 2i})\bar{e}_{ss}\bar{e}_{tt}:\big)_{(3)}\eqref{eq:commuting virasoro-4-5}$ with $s-t \equiv 1$ (mod 2) and each of the term has the same value. In other words,  
    \begin{equation} \label{eq:p3-1}
    \begin{aligned}
         & \eqref{eq:commuting virasoro-4-5}_{(3)}\eqref{eq:commuting virasoro-4-5}= \psi^{4n-4}n(n-1)\Big(:D(\bar{E}_{22})\bar{e}_{33}\bar{e}_{44}:\Big)_{(3)}\eqref{eq:commuting virasoro-4-5}.
    \end{aligned}
    \end{equation}
    We can compute \eqref{eq:p3-1} directly since 
    \begin{equation}
    \begin{aligned}
    & \psi^{4n-4}(:D(\bar{E}_{22})\bar{e}_{33}\bar{e}_{44}:)_{(3)}\eqref{eq:commuting virasoro-4-5}\\
    & = (:D(\bar{E}_{22})\bar{e}_{33}\bar{e}_{44}:)_{(3)}\Big(:\big(D(\bar{E}_{22})+\sum_{i=3}^{n}D(\bar{E}_{2i\, 2i})\, \big)\bar{e}_{33}\bar{e}_{44}:\Big)= -\psi^{4n-1}n.
    \end{aligned}
    \end{equation}
    Here we used $D(\bar{E}_{22})_{(3)}D(\bar{E}_{2i\, 2i})=-\psi$ if $i\neq 1$ and $D(\bar{E}_{22})_{(3)}D(\bar{E}_{22})=-2\psi$. Hence we can conclude $p_3=\eqref{eq:commuting virasoro-4-5}_{(3)}\eqref{eq:commuting virasoro-4-5}.$

    The last part is $\eqref{eq:commuting virasoro-4-6}_{(3)} \eqref{eq:commuting virasoro-4-6}.$ First, we can see show that 
    \begin{equation}
        \eqref{eq:commuting virasoro-4-6}= \psi^{2n-3}\sum_{i<j< s< t} :\bar{e}_{ii}\bar{e}_{jj}\bar{e}_{ss}\bar{e}_{tt}:
    \end{equation}
    and $:\bar{e}_{ii}\bar{e}_{jj}\bar{e}_{ss}\bar{e}_{tt}:_{(3)}:\bar{e}_{i'i'}\bar{e}_{j'j'}\bar{e}_{s's'}\bar{e}_{t't'}:\, =\, \delta_{(i,j,s,t)(i',j',s',t')} \psi^4(-1)^{i+j+k+l}.$ Hence $ \eqref{eq:commuting virasoro-4-6}_{(0)}\eqref{eq:commuting virasoro-4-6}= \psi^{4n-2}(N_{ev}-N_{odd})$,
    where $N_{ev}$ (resp. $N_{odd}$) is the number of quadruples with even (resp. odd) number of odd parity indices. One can find 
    \begin{equation}
    \begin{aligned}
        & N_{ev}= \frac{(n+1)n(n-1)(n-2)}{4!}+ \frac{(n+1)n}{2}\frac{n(n-1)}{2}+ \frac{n(n-1)(n-2)(n-3)}{4!},\\
        & N_{odd}= \frac{(n+1)n(n-1)}{3!}\cdot n + \frac{n(n-1)(n-2)}{3!}\cdot (n+1),
    \end{aligned}
    \end{equation}
    and hence $\eqref{eq:commuting virasoro-4-6}_{(0)}\eqref{eq:commuting virasoro-4-6}= \psi^{4n-2}(N_{ev}-N_{odd})= \psi^{4n-2}\frac{n(n-1)}{2}.$

    As a conclusion, $p=p_1+p_2+p_3+p_4$ is the 4th order pole of $W[2](z)W[2](w).$
    \end{proof}

\begin{lemma}\label{lem:commuting virasoro-5}
    Recall the correction term $C(W[2])=-\frac{n-1}{(2n)\,  \psi^{2n+1}}: W[1]W[1]: -\frac{n-1}{2}\partial(W[1])$ in Lemma \ref{lem:commuting virasoro-2}. The 4th order pole of $C(W[2])$ and itself is 
    \begin{equation}
        p_c=\frac{3}{2}  \ n\, (1 - n)^2  \, \big(\psi\, (1 + n)-1\big)\, \psi^{4 n} + \frac{1}{2}  \, (1 - n)^2 \, \big(\, \psi (1 + n))^2-1\, \big)\psi^{4n-2}.
    \end{equation}
\end{lemma}
\begin{proof}
    This lemma directly follows from Lemma \ref{lem:commuting virasoro-1}.
\end{proof}

\begin{proposition}
    The central charge of $Vir$ is 
    \begin{equation} \label{eq:ccofVir}
        \frac{2(p-p_c)}{\psi^{4n-2}(\psi-1)^2}=\frac{(n-1) (-1 + \psi + \psi^2 n (1 + n))}{\psi-1}= \frac{(n-1)(1 + n + k n) (k + n + k n)}{k}\ 
    \end{equation}
    where $k=\psi-1$ is the level of $\cW^k(\gg,F).$
\end{proposition}
\begin{proof}
    We know that $\widetilde{W}[2]$ commutes with $C(W[2]).$ Hence we have 
    \begin{equation}
        \widetilde{W}[2]_{(3)}\widetilde{W}[2]= W[2]_{(3)}W[2]-C(W[2])_{(3)}C(W[2])=p-p_c.
    \end{equation}
    Since $Vir=\frac{1}{\psi^{2n-1}(\psi-1)}\widetilde{W}[2],$ the central charge of $Vir$ is $\frac{2(p-p_c)}{\psi^{4n-2}(\psi-1)^2}.$ The rest of the proposition follows from direct computations and the fact that $\psi= k+1.$
\end{proof}

\begin{theorem} \label{trunc:en0} $\cC^{\psi}_{\cN=2}(n,0|0)=\cW^k(\gs\gl_{n+1|n})$ for $\psi = k+1$, is a conformal extension of $\cH \otimes \cD^{\psi}(n+1,n) \otimes \cD^{-\psi + 1}(n,0)$.
\end{theorem}
\begin{proof} We have shown that the central charge of one of the Virasoro fields $Vir$ commuting with $H$ is given by \eqref{eq:ccofVir}. Since $\cW^k(\gs\gl_{n+1|n})$ has central charge $-3 n (k + n + k n)$, it follows that the central charge of the other Virasoro field commuting with $H$ is $-\frac{n (-1 + k + n + k n) (1 + 2 k + n + k n)}{k}$. Therefore $Vir = L_-$ and the other one is $L_+$, and by Proposition \ref{prop:Virasoros}, this completely determines the parameter $\lambda$, or equivalently the truncation curve for $\cE^{\psi}(n,0)$. 

By Proposition \ref{prop:lambdas}, this determines the parameters $\lambda_{\pm}$, or equivalently, the truncation curves for the quotients $\cA^{\pm}$ of $\cW^{\pm}_{\infty}$, respectively. In particular, $\cA^+$ and $\cA^-$ are possibly non-simple quotients of $\cW^{\pm}_{\infty}$ along the truncation curves for $\cD^{\psi}(n+1,n)$ and $\cD^{-\psi + 1}(n,0)$, respectively. Finally, the fact that $\cA^+$ and $\cA^-$ are the {\it simple} quotients $\cD^{\psi}(n+1,n)$ and $\cD^{-\psi + 1}(n,0)$ is immediate from Corollary \ref{trunc:e0md0m} together with the duality $\cE^{\psi}_{\mathcal{N}=2}(n,0) \cong \cE^{\psi^{-1}}_{\cN=2}(0,n)$ given by Theorem \ref{N=2:duality}, which will be proved in the next section.
\end{proof}

\begin{corollary} \label{minextension:wsln} In the terminology of \cite{SY25}, $\cW^k(\gs\gl_{n+1|n})$ is a minimal $\mathcal{N}=2$ supersymmetric extension of $\cW^{-k-n}(\gs\gl_n)$.
\end{corollary}

\begin{proof} Since $\cD^{-\psi+1}(n,0) \cong \cW^{-k-n}(\gs\gl_n)$ for $\psi = k+1$, it is immediate that $\cW^k(\gs\gl_{n+1|n})$ is an $\mathcal{N}=2$ supersymmetric extension of $\cW^{-k-n}(\gs\gl_n)$. Minimality follows from the same argument as Corollary \ref{minextension}.
\end{proof}

\subsection{The cases $\cC^{\psi}_{\mathcal{N}=2}(n, r|s)$ with $n\geq 1$ and $r+s > 0$}

We need a different argument to show that in the remaining cases, $\cC^{\psi}_{\mathcal{N}=2}(n, r|s)$ is a $1$-parameter quotient of $\cW^{\mathcal{N}=2}_{\infty}$ because even though it has the correct strong generating type, it is not yet obvious that it is generated by the fields in weight $1, \frac{3}{2}, 2$. 

\begin{lemma} \label{lem:genofcpsi} For $n\geq 1$ and $r+ s > 0$, $\cC^{\psi}_{\mathcal{N}=2}(n, r|s)$ is weakly generated by the fields 
$$\{H,G^{\pm},L, W^{i,\bot}, W^{i,\pm}, W^{i,\top}|\ i = 2,3,\dots, n+2\},$$ for generic levels.
\end{lemma}

\begin{proof} It suffices to show that the free field limit has this property. This is given in the proof of Theorem \ref{thm: Y alg}, and it suffices to show that  
$(\cO(s|r, n+1) \otimes \cO(r|s, n+2))^{\text{GL}_{r|s}}$ is generated by the fields in weights $n+1, \frac{2n+3}{2}, n+2$. This is easy to verify directly.
\end{proof}

Next, let $$\tilde{\cC}^{\psi}_{\mathcal{N}=2}(n, r|s) \subseteq \cC^{\psi}_{\mathcal{N}=2}(n, r|s)$$ be the subalgebra generated by the fields in weights $1, \frac{3}{2}, 2$, which clearly has an action of the automorphism $\sigma$ given by \eqref{eq:sigma}. A priori, $\tilde{\cC}^{\psi}_{\mathcal{N}=2}(n, r|s)$ can be a proper subalgebra of $\cC^{\psi}_{\mathcal{N}=2}(n, r|s)$ and therefore need not be simple even if $r = 0$ or $s = 0$. Let $\{\omega^{i, \bot}, \omega^{i,\pm}, \omega^{i,\top}| \ r \geq 1 \}$ be the strong generators of $\cC^{\psi}_{\mathcal{N}=2}(n, r|s)$ corresponding to the large level limits which are given by Theorem \ref{thm: Y alg}. For $r \geq 3$, we can write 
\begin{equation} \label{lambdaterm} W^{r,*} = \lambda_r \omega^{r,*} + \cdots,\qquad \lambda_r \in \mathbb{C},\end{equation} where the remaining terms are normally ordered monomials in lower fields. If $\lambda_r \neq 0$ for all $r$, then $\tilde{\cC}^{\psi}_{\mathcal{N}=2}(n, r|s) = \cC^{\psi}_{\mathcal{N}=2}(n, r|s)$. Otherwise, let $M\geq 2$ be the first integer such that $\lambda_{M+1} = 0$.

\begin{lemma} \label{lem:tildecpsi} With $M$ as above, $\{H, G^{\pm}, L, \omega^{i,\bot}, \omega^{i,\pm}, \omega^{i,\top}|\ 2\leq i \leq M\}$ close under OPE, so that $\tilde{\cC}^{\psi}_{\mathcal{N}=2}(n, r|s)$ is of type $\cW\big(1,2^2,3^3,\dots, M^2, M+1; \big(\frac{3}{2}\big)^2,  \big(\frac{5}{2}\big)^2,\dots, \big(\frac{2M+1}{2}\big)^2 \big)$. 
\end{lemma}

\begin{proof} This is the same as the proof of \cite[Lemma 5.10]{CL3}. \end{proof}

\begin{lemma} \label{thm:tildednm} For $n \geq 1$ and $r+s > 0$, $\tilde{\cC}^{\psi}_{\mathcal{N}=2}(n, r|s)$ is a $1$-parameter quotient of $\cW^{\mathcal{N}=2}_{\infty}$ with truncation curve given by \eqref{gentruncationcurve}.
 \end{lemma}

\begin{proof} 
 We first exclude the case $\tilde{\cC}^{\psi}_{\mathcal{N}=2}(1, 1|0)$ which must be handled separately. In all other cases, let $\{H, G^{\pm}, L,\omega^{i,\bot}, \omega^{i,\pm}, \omega^{i,\top}|\ 2\leq i \leq M\}$ be the strong generating set for $\tilde{\cC}^{\psi}_{\mathcal{N}=2}(n, r|s)$ given by Lemma \ref{lem:tildecpsi}. We need a slightly different argument in the cases $M \geq 6$ and $M < 6$.

Suppose first that $M \geq 6$. By Corollary \ref{cor:diamond} and Theorem \ref{one-parameter quotients theorem}, to show that $\tilde{\cC}^{\psi}_{\mathcal{N}=2}(n, r|s)$ is a $1$-parameter quotient of $\cW^{\mathcal{N}=2}_{\infty}$, it suffices to prove that $\{H,G^{\pm}, L, W^{i,\bot},W^{i,\pm},W^{i,\top}|\ 2\leq i \leq 4\}$ satisfy the OPE relations $D_5$ in $\cW^{\mathcal{N}=2}_{\infty}$; equivalently, all Jacobi identities in $J_7$ hold as a consequence of \eqref{skew-symmetry}-\eqref{quasi-derivation}.

By Corollary \ref{cor:gradedcharacter} this condition is automatic if the graded character of $\tilde{\cC}^{\psi}_{\mathcal{N}=2}(n, r|s)$ coincides with that of $\cW^{\mathcal{N}=2}_{\infty}$ up to weight $6$. Recall from \eqref{sing vectors Y} that the first relation among the generators $$\{H,G^{\pm}, L, W^{i,\bot},W^{i,\pm},W^{i,\top}|\ i \geq 2\}$$ and their derivatives occurs in weight $(1 + s) (2 + n + r + n r + r s)$ if $r \geq s$, and in weight $(1 + r) (1 + n +2s + ns +rs)$ if $r < s$. Aside from the case $\tilde{\cC}^{\psi}_{\mathcal{N}=2}(1, 1|0)$, there are no normally ordered relations in $\cC^{\psi}_{\mathcal{N}=2}(n, r|s)$ among these fields in weight below $8$. Therefore the character of $\cD^{\psi}_{\mathcal{N}=2}(n,m)$ coincides with that of $\cW^{\mathcal{N}=2}_{\infty}$ in weight up to $6$. Since $M \geq 6$, $\tilde{\cC}^{\psi}_{\mathcal{N}=2}(n, r|s)$ and $\cC^{\psi}_{\mathcal{N}=2}(n, r|s)$ have the same graded character up to weight $6$, so the conclusion holds.

Next, suppose that $M < 6$. Since the constants in \eqref{lambdaterm} satisfy $\lambda_{M+1} = 0$ and $\lambda_r \neq 0$ for $2\leq r \leq M$, there can be no nontrivial normally ordered relations among the generators $\{J,G^{\pm}, L, W^{i,\bot}, W^{i,\pm}, W^{i,\top}|\ 2 \leq i \leq M\}$ of $\tilde{\cC}^{\psi}_{\mathcal{N}=2}(n, r|s)$ in weight $w \leq M$, since this property holds for the corresponding fields in the large level limit. Equivalently, all Jacobi relations among $\{H,G^{\pm}, L, W^{i,\bot}, W^{i,\pm}, W^{i,\top}|\ 2 \leq i \leq M\}$ in $D_{M+2}$ must hold as a consequence of \eqref{skew-symmetry}-\eqref{quasi-derivation} alone. Therefore the OPEs $W^{i,*}(z) W^{j,*}(w)$ for $i,j \leq M$ and $i + j \leq M$, are the same as those of $\cW^{\mathcal{N}=2,I}_{\infty}$ for some ideal $I \subseteq \mathbb{C}[c,\lambda]$.

If we use the same procedure as our construction of $\cW^{\mathcal{N}=2}_{\infty}$ beginning with the fields in $\tilde{\cC}^{\psi}_{\mathcal{N}=2}(n, r|s)$,
 $$\{H,G^{\pm}, L, W^{i,\bot}, W^{i,\pm}, W^{i,\top}|\ 2 \leq i \leq M\},$$ and the OPEs $W^{i,*}(z) W^{j,*}(w)$ for $i+j \leq M$, we can inductively define new fields in $\tilde{\cC}^{\psi}_{\mathcal{N}=2}(n, r|s)$ for all $r\geq 1$ by 
$$W^{M+r,\bot} = W^{2,\top}_{(0)} W^{M+r-1,\bot},\quad W^{M+r,\mp} = G^{\pm}_{(0)} W^{M+r,\bot},\quad  {W^{M+r,\top} = \pm G^{\pm}_{(0)}G^{\mp}_{(0)} W^{M+r,\bot}}\mp 1/2 \partial W^{M+r,\bot},$$ and then define the OPE algebra of all fields $\{H,G^{\pm}, L, W^{N,\bot},W^{N,\pm},W^{N,\top}|\ N \geq 2\}$ recursively so that they are the same as the OPEs in $\cW^{\mathcal{N}=2,I}_{\infty}$. In particular, this realizes $\tilde{\cC}^{\psi}_{\mathcal{N}=2}(n, r|s)$ as a one-parameter quotient of $\cW^{\mathcal{N}=2, I}_{\infty}$ by some vertex algebra ideal $\cI$ containing a field in weight $M+1$ of the form $W^{M+1,\bot} - P(H,G^{\pm}, L, W^{i,\bot},W^{i,\pm},W^{i,\top})$, where $P$ is a normally ordered polynomial and $i \leq M$. The truncation curve for $\tilde{\cC}^{\psi}_{\mathcal{N}=2}(n, r|s)$ can be computed using Theorem \ref{recon:first}, and it agrees with \eqref{gentruncationcurve}. 

Finally, we consider the case $\tilde{\cC}^{\psi}_{\mathcal{N}=2}(1, 1|0)$. The first normally ordered relation among the generators $\{H, G^{\pm}, L,\omega^{i,\bot}, \omega^{i,\pm}, \omega^{i,\top}|\ 2\leq i \leq M\}$ given by Lemma \ref{lem:tildecpsi} occurs at weight $5$, and it can checked by passing to the large level limit that $\cC^{\psi}_{\mathcal{N}=2}(1, 1|0)$ is generically of type $$\cW\bigg(1,2^2, 3^2, 4^2, 5; \bigg(\frac{3}{2}\bigg)^2,  \bigg(\frac{5}{2}\bigg)^2,  \bigg(\frac{7}{2}\bigg)^2,  \bigg(\frac{9}{2}\bigg)^2\bigg).$$ By imposing Jacobi identities, one finds that there are exactly two $1$-parameter vertex algebras with this structure satisfying the properties $(\B{U}1-\B{U}3)$ in Subsection \ref{setup:Uproperties}. One of them is $\cW^k(\gs\gl_{4|3})$, which is freely generated of this type, and the other must therefore be $\cC^{\psi}_{\mathcal{N}=2}(1, 1|0)$, which is not freely generated since its large level limit $(\cO(1|0,1) \otimes \cO(0|1,2))^{\text{GL}_1}$ is not freely generated. One can then check that this OPE algebra is a specialization of the OPE algebra $\cW^{\cN=2}_{\infty}$ along the ideal \eqref{gentruncationcurve} for $n=1$, $r = 1$, and $s = 0$. \end{proof}

\begin{theorem} \label{trunc:Cnmthird} For $n \geq 1$ and $r+s > 0$, $\tilde{\cC}^{\psi}_{\mathcal{N}=2}(n, r|s) = \cC^{\psi}_{\mathcal{N}=2}(n, r|s)$. Therefore $\cC^{\psi}_{\mathcal{N}=2}(n, r|s)$ is a $1$-parameter quotient of $\cW^{\mathcal{N}=2}_{\infty}$, with truncation curve given by \eqref{gentruncationcurve}.
\end{theorem}

\begin{proof} Suppose first that $s - r \geq 1$. By Proposition \ref{prop:lambdas}, the truncation curve for $\tilde{\cC}^{\psi}_{\mathcal{N}=2}(n, r|s)$ determines the truncation curves for the $\cW_{\infty}$-quotients $\cA^{\pm}$. A computation shows that $\cA^+$ and $\cA^-$ have the same truncation curves as the simple $Y$-algebras $\cD^{\psi}(n+1,n+s-r)$ and $\cC^{-\psi +1}(n,s-r)$, respectively. Therefore $\tilde{\cC}^{\psi}_{\mathcal{N}=2}(n, r|s)$ is a conformal extension of 
$$\cH \otimes \overline{\cD^{\psi}(n+1,n+s-r)} \otimes \overline{\cC^{-\psi +1}(n,s-r)},$$ where $\overline{\cD^{\psi}(n+1,n+s-r)}$ and $\overline{\cC^{-\psi +1}(n,s-r)}$ are possibly non-simple $1$-parameter quotients of $\cW_{\infty}$ which have simple quotients $\cD^{\psi}(n+1,n+s-r)$ and $\cC^{-\psi +1}(n,s-r)$, respectively.

By \cite[Theorems 6.1 and 7.1]{CL2}, $\cD^{\psi}(n+1,n+s-r)$ and $\cC^{-\psi +1}(n,s-r)$ have strong generating types $\cW(2,3,\dots, (n + s-r +1) (n + 2) - 1)$ and $\cW(2,3,\dots, (s-r + 1) (n+ s-r +1) - 1)$ respectively, for generic $\psi$. So there are no normally ordered relations among the generator $\{L_+, W^d_+|\ 3 \leq d \leq (n + s-r+1) (n + 2) - 1\}$ of $\overline{\cD^{\psi}(n+1,n+s-r)}$, or among the generators $\{L_-, W^d_-|\ 3 \leq d \leq (s-r + 1) (n+s-r +1) - 1)\}$ of $\overline{\cC^{-\psi+1}(n,s-r)}$. 

Up to the scaling constant $\omega_{\pm}$, it follows from the explicit form of the weight $3$ fields $W^3_{\pm}$ given by \eqref{eq:W3s} and the truncation curve \eqref{eq:trunc} that these fields have the form
$$W^3_{\pm} \equiv W^{2,\top} +  a_{\pm}(\psi) W^{3,\bot},$$ modulo terms which depend only on $H, G^{\pm}, L,W^{2,\bot}$ and their derivatives. Moreover, $a_{\pm}(\psi)$ are non-constant rational functions of $\psi$. 
Next, using the fact that $(W^3_{\pm})_{(1)} W^{d-1}_{\pm} = W^{d}_{\pm}$ for $d\geq 4$, together with Proposition \ref{prop:structure constants}, we see that for all $d \geq 4$, $W^{d}_{\pm}$ has the form
$$W^d_{\pm} \equiv p_{\pm}(a_{\pm}(\psi)) W^{d-1,\top} + q_{\pm}(a_{\pm}(\psi))W^{d,\bot},$$ modulo terms which depend only on $W^{n,\bot}, W^{n,\top}, W^{n,\pm}$ for $n < d-1$ as well as $W^{d-1,\bot}$, and their derivatives. 
Here $p_{\pm}(a_{\pm}(\psi))$ and $q_{\pm}(a_{\pm}(\psi))$ are nonzero polynomials in $a_{\pm}(\psi)$ with rational coefficients. Since $a_{\pm}(\psi)$ are not constants, it follows that $p_{\pm}(a_{\pm}(\psi))$ and $q_{\pm}(a_{\pm}(\psi))$ are also non-constant rational functions of $\psi$, and hence are nonzero for generic values of $\psi$.
Therefore if $\psi$ is generic, the subalgebra $\tilde{\cC}^{\psi}_{\mathcal{N}=2}(n, r|s)$ generated by the fields in weights $1, \frac{3}{2}, 2$, which contains $W^3_{\pm}$, also contains $\{L_+, W^d_+|\ 3 \leq d \leq (n + s-r+1) (n + 2) - 1\}$ as well as $\{L_-, W^d_-|\ 3 \leq d \leq (s-r + 1) ( n+s-r+1) - 1)\}$. Therefore $\tilde{\cC}^{\psi}_{\mathcal{N}=2}(n, r|s)$ contains $W^{d,\bot}$ for all $d \leq \text{max}\{(n + s-r+1) (n + 2) - 1),  (s-r + 1) (n + s-r+1) - 1\}$,  which is greater than $n+2$. Hence $\tilde{\cC}^{\psi}_{\mathcal{N}=2}(n, r|s)$ contains all the fields $W^{d,\pm}, W^{d,\top}$ for $d \leq n+2$, so by Lemma \ref{lem:genofcpsi}, $\tilde{\cC}^{\psi}_{\mathcal{N}=2}(n, r|s) = \cC^{\psi}_{\mathcal{N}=2}(n, r|s)$.

In the case $r \geq s$, one checks similarly that $\tilde{\cC}^{\psi}_{\mathcal{N}=2}(n, r|s)$ is a conformal extension of 
$$
\left\{
\begin{array}{ll}
\cH \otimes \overline{\cD^{\psi}(n+1,n-(r-s))} \otimes \overline{\cD^{-\psi+1}(n,r-s)}, & n- (r-s) \geq 0,
\smallskip
\\  \cH \otimes \overline{\cC^{\psi}(n+1, r-s-n)} \otimes \overline{\cD^{-\psi+1}(n, r-s)} & n -(r-s) < 0, \\
\end{array} 
\right.
$$
where as above, $\overline{\cC^{\psi}(a,b)}$ and $\overline{\cD^{\psi}(a,b)}$ are possibly non-simple $\cW_{\infty}$-quotients with the same truncation curves as $\cC^{\psi}(a,b)$ and $\cD^{\psi}(a,b)$, respectively. 
The proof that $\tilde{\cC}^{\psi}_{\mathcal{N}=2}(n, r|s) = \cC^{\psi}_{\mathcal{N}=2}(n, r|s)$ in this case is similar to the case $s-r \geq 1$, and is omitted.
\end{proof}   

In the cases $r = 0$ and $s = 0$, recall the notation $$\cD^{\psi}_{\mathcal{N}=2}(n,s) := \cC^{\psi}_{\cN=2}(n,0|s+1),\qquad \cE^{\psi}_{\mathcal{N}=2}(n,r) := \cC^{\psi}_{\cN=2}(n,r|0),$$ which are both simple as a $1$-parameter vertex algebras.

\begin{corollary} \label{trunc:DandE} \begin{enumerate}
\item $\cD^{\psi}_{\mathcal{N}=2}(n,s)$ is a conformal extension of $\cH \otimes \overline{\cD^{\psi}(n+1,n+s+1)} \otimes \overline{\cC^{-\psi + 1}(n,s+1)}$.
\item $\cE^{\psi}_{\mathcal{N}=2}(n,r)$ is a conformal extension of $$
\left\{
\begin{array}{ll}
\cH \otimes \overline{\cD^{\psi}(n+1,n-r)} \otimes \overline{\cD^{-\psi+1}(n,r)}, & n- r \geq 0,
\smallskip
\\  \cH \otimes \overline{\cC^{\psi}(n+1, r-n)} \otimes \overline{\cD^{-\psi+1}(n, r)} & n -r < 0. \\
\end{array} 
\right.
$$
\end{enumerate}
\end{corollary} 

\begin{remark} It is an interesting problem to determine precisely which $\cW_{\infty}$-quotients $\cA^{\pm}$ appear in $\cC^{\psi}_{\mathcal{N}=2}(n, r|s)$. Based on reduction by stages, we expect both algebras $\overline{\cD^{\psi}(n+1,n+s+1)}$ and $\overline{\cC^{-\psi + 1}(n,s+1)}$ appearing in $\cD^{\psi}_{\mathcal{N}=2}(n,s)$ to be simple. Similarly, we expect the second algebra $\overline{\cD^{-\psi+1}(n,r)}$ appearing in $\cE^{\psi}_{\mathcal{N}=2}(n,r)$, but not the first one, to be simple. \end{remark}

\section{Dualities of Feigin-Frenkel type}\label{sect:FFduality}

The following conjectured dualities of Feigin-Frenkel type appeared in the work of Proch\'azka and Rap\v{c}\'ak \cite{PR}, which we specialize to the case of $\WNtwo$.
\begin{conjecture}[Eq. 4.4, \cite{PR}]\label{dualities:generalcase} Let $n, r, s$ be non-negative integers.
$$\cC^{\psi}_{\cN=2}(n,r|s)\cong\begin{cases}
  \cC_{\cN=2}^{\psi^{-1}}(r-s, n+s|s), & r\geq s,  \\
\cC_{\cN=2}^{\psi^{-1}}(s-r-1, r|n+r+1), & r < s. 
\end{cases}$$
\end{conjecture}
Our main application of the construction of $\cW^{\cN=2}_{\infty}$ is to prove Conjecture \ref{dualities:generalcase} in the case when either $r=0$ or $s=0$.

\begin{theorem} \label{N=2:duality} For all integers $n,m \geq 0$, we have isomorphisms of $1$-parameter vertex superalgebras
\begin{equation} \cD^{\psi}_{\mathcal{N}=2}(n,m) \cong \cD^{\psi^{-1}}_{\mathcal{N}=2}(m,n),\qquad \cE^{\psi}_{\mathcal{N}=2}(n,m) \cong \cE^{\psi^{-1}}_{\mathcal{N}=2}(m,n). \end{equation}
The case $\cE^{\psi}_{\mathcal{N}=2}(n,0) \cong \cE^{\psi^{-1}}_{\mathcal{N}=2}(0,n)$ is just Ito's conjecture. The case $\cD^{\psi}_{\mathcal{N}=2}(n,0) \cong \cD^{\psi^{-1}}_{\mathcal{N}=2}(0,n)$ is a similar statement which provides a coset realization of $\cW^{k}(\gp\gs\gl_{n+1|n+1}, F_{n+1|n})^{U(1)}$.
\end{theorem}

\begin{proof} By Corollary \ref{cor:uniquenessofcurve}, simple $1$-parameter quotients of $\cW^{\cN=2}_{\infty}$ are uniquely determined by their truncation curves. By Propositions \ref{prop:Virasoros} and \ref{prop:weight3}, the truncation curve for such a $\cW^{\cN=2}_{\infty}$-quotient is uniquely determined by the truncation curves for the two $\cW_{\infty}$-quotients $\cA^{\pm}$ that commute with the Heisenberg field. For $\cD^{\psi}_{\mathcal{N}=2}(n,m)$ and $\cE^{\psi}_{\mathcal{N}=2}(n,m)$, these are given by Corollary \ref{trunc:DandE}. The claim then follows from the triality, Theorem \ref{intro:trialitytheorem}. Note that this argument only makes use of the truncation curves for $\cA^{\pm}$ and not their simplicity, so the simplicity of $\cA^{\pm}$ in the case of $\cE^{\psi}_{\mathcal{N}=2}(n,0)$ which we claimed in Theorem \ref{trunc:en0}, is immediate from $\cE^{\psi}_{\mathcal{N}=2}(n,0) \cong \cE^{\psi^{-1}}_{\cN=2}(0,n)$ together with Corollary \ref{trunc:e0md0m}. \end{proof}

Next, we observe that the isomorphism $\cD^{\psi}_{\mathcal{N}=2}(n,0) \cong \cD^{\psi^{-1}}_{\mathcal{N}=2}(0,n)$ can be enhanced to a coset realization of $\cW^{k}(\gp\gs\gl_{n+1|n+1}, F_{n+1|n})$.
Consider the lattice $L = \sqrt{-n} \mathbb{Z}$, with generator $\alpha$ satisfying  $\langle \alpha, \alpha \rangle = -n$, and let $V_L$ be the corresponding lattice vertex (super)algebra. It is generated by $e^{\pm \alpha}$ which have conformal weights $-\frac{n}{2}$ and has a Heisenberg field $\alpha$ satisfying $\alpha(z) \alpha(w) \sim -n (z-w)^{-2}$ and $\alpha(z) e^{\pm \alpha}(w) \sim \pm n e^{\pm \alpha}(w)(z-w)^{-1}$.

Consider $V^k(\gs\gl_{1|n+1}) \otimes \cS(n+1) \otimes V_L$, which has an action of $V^{-k-1}(\gg\gl_{n+1}) := \cH \otimes V^{-k-1}(\gs\gl_{n+1})$ where $V^{-k-1}(\gs\gl_{n+1})$ acts diagonally on $V^k(\gs\gl_{1|n+1}) \otimes \cS(n+1)$ and the generator of $\cH$ is $J^1 + J^2 + \frac{n+1}{n}\alpha$. Here $J^1$ is the Heisenberg field in $V^k(\gs\gl_{1|n+1})$ normalized so the odd fields $\{\psi^{\pm,i}|\ i = 1,\dots, n+1\}$ have $J^1_{(0)}$-eigenvalues $\pm 1$, and $J^2$ is the Heisenberg field in $ \cS(n+1)$ normalized so that $\beta^i, \gamma^i$ have $J^2_{(0)}$-eigenvalues $\mp 1$, and $\alpha$ is the above Heisenberg field in $V_L$. Consider the coset 
\begin{equation} \label{GKO:psl} \text{Com}(V^{-k-1}(\gg\gl_{n+1}), V^k(\gs\gl_{1|n+1}) \otimes \cS(n+1) \otimes V_L),\end{equation} which contains $\cD^{\psi}_{\mathcal{N}=2}(0,n)$ as a subalgebra. To see where the additional elements of \eqref{GKO:psl} beyond $\cD^{\psi}_{\mathcal{N}=2}(0,n)$ come from, consider the coset
\begin{equation} \label{GKO:pslsecond} \text{Com}(V^{-k-1}(\gs\gl_{n+1}), V^k(\gs\gl_{1|n+1}) \otimes \cS(n+1)),\end{equation} which also contains $\cD^{\psi}_{\mathcal{N}=2}(0,n)$ as a subalgebra. This has large level limit
\begin{equation} \label{GKO:pslsecondlarge} \big(\cS_{\text{odd}}(n+1,2) \otimes \cS_{\text{ev}}(n+1,1)\big)^{\text{SL}_{n+1}},\qquad \cS_{\text{ev}}(n+1,1) \cong \cS(n+1).\end{equation} 
In addition to the quadratic generators which lie in $\big(\cS_{\text{odd}}(n+1,2) \otimes \cS_{\text{ev}}(n+1,1)\big)^{\text{GL}_{n+1}}$ and correspond to the generators of $\cD^{\psi}_{\mathcal{N}=2}(0,n)$, \eqref{GKO:pslsecondlarge} also contains determinants of degree $n+1$ in the generators of $\cS_{\text{odd}}(n+1,2) \otimes \cS_{\text{ev}}(n+1,1)$, with appropriate signs. Each of these determinants corresponds to a choice of $n+1$ distinct copies of the standard representation $\mathbb{C}^{n+1}$, or $n+1$ distinct copies of the dual representation $(\mathbb{C}^{n+1})^*$. These determinants do not lie in the large level limit of \eqref{GKO:psl} because the zero mode of the generator of $\cH$ acts nontrivially, but this can be corrected by multiplying them by $e^{\pm \alpha}$. The generators of this kind of lowest weight have the form
$$:\big(\sum_{i=1}^{n+1} :\phi^{+,1} \cdots \widehat{\phi^{+,i}} \cdots \phi^{+,n+1} \gamma^i: \big) e^{-\alpha}:,\qquad :\big(\sum_{i=1}^{n+1} :\phi^{-,1} \cdots \widehat{\phi^{-,i}} \cdots \phi^{-,n+1} \beta^i: \big) e^{\alpha}:$$ These are even of conformal weight $\frac{n+1}{2}$, and have $H_{(0)}$-eigenvalues $\pm (n+1)$.
We also have odd weight $\frac{n+2}{2}$ fields 
$$ :\phi^{+,1} \cdots \phi^{+,n+1} e^{-\alpha}:,\qquad :\phi^{-,1} \cdots \phi^{-,n+1} e^{\alpha}:$$ which have $H_{(0)}$-eigenvalues $\pm n$. It is straightforward to check that the large level limit of \eqref{GKO:psl} is strongly generated by the generators of $\lim_{\psi \rightarrow \infty}\cD^{\psi}_{\mathcal{N}=2}(0,n)$, together with the above $4$ fields. Therefore \eqref{GKO:psl} is strongly generated by the corresponding fields at generic levels. By Theorem \ref{recon:first} and Remark \ref{recon:cased0n}, we obtain

\begin{corollary} \label{GKO:pslnn}
For generic values of $k$, the coset \eqref{GKO:psl} is isomorphic to $\cW^{\ell}(\gp\gs\gl_{n+1|n+1}, F_{n+1|n})$ for $\psi = k-n$ and $\ell = \psi^{-1}$.
\end{corollary}
\begin{remark} In the case $n=1$, $F_{2|1}$ is the minimal nilpotent in $\gp\gs\gl_{2|2}$, so this is a coset construction of the small $\cN=4$ superconformal algebra \cite{KW}. This extends the coset realization of the large $\cN=4$ superconformal algebra of \cite{CFL} to the small one. \end{remark}

We remind the reader that the dualities given by Theorem \ref{N=2:duality} are isomorphisms of $1$-parameter vertex algebras, and they can fail to hold for the honest cosets at special points even though they hold for the honest cosets generically. We now prove that in the case of Ito's conjecture, the duality holds for the simple honest cosets at the levels where they are rational. Recall that $\cE^{\psi}_{\mathcal{N}=2}(0,n) = \text{Com}(V^{\ell+1}(\gg\gl_n), V^{\ell}(\gs\gl_{n+1}) \otimes \cE(n))$ for $\psi = \ell + n + 1$, as $1$-parameter vertex algebras. If $\ell_0 + n > 0$, we have a homomorphism
$$\tilde{V}^{\ell_0+1}(\gg\gl_n) \hookrightarrow L_{\ell_0}(\gs\gl_{n+1}) \otimes \cE(n),$$ where $\tilde{V}^{\ell_0+1}(\gg\gl_n)$ is a (possibly non-simple) homomorphic image of $V^{\ell_0+1}(\gg\gl_n)$. By \cite[Theorem 4.1]{ACK}, the simple quotient $\cE_{N=2,\psi_0}(0,n)$ of $\cE^{\psi_0}_{\cN=2}(0,n)$ for $\psi_0 = \ell_0 + n + 1$, is isomorphic to 
\begin{equation}\label{coset:admissible}  \text{Com}(\tilde{V}^{\ell_0+1}(\gg\gl_n), L_{\ell_0}(\gs\gl_{n+1}) \otimes \cE(n)).\end{equation} In particular, the specialization of the $1$-parameter VOA $\cE^{\psi}_{\mathcal{N}=2}(0,n)$ (where $\psi$ is regarded as a formal parameter) to $\psi = \psi_0$, maps surjectively to \eqref{coset:admissible}. Moreover, Lemma \ref{lem:weakgen}, which says that $\cE^{\psi}(n,0)$ is weakly generated by the fields in weight at most $2$ for all noncritical levels, together with the isomorphism $\cE^{\psi}(n,0) \cong \cE^{\psi^{-1}}(0,n)$ of $1$-parameter VOAs, implies that the simple coset \eqref{coset:admissible} is weakly generated by the fields in weight at most $2$ for all such levels $\ell_0$. It follows that in the induced isomorphism of simple vertex algebras $\cE_{N=2,\psi_0}(n,0) \cong \cE_{N=2,\psi^{-1}_0}(0,n)$, the right hand side can be replaced by the honest coset \eqref{coset:admissible} when $\ell_0= \psi^{-1}_0 - n-1$ satisfies $\ell_0 + n > 0$. In particular, we have the isomorphism
\begin{equation}\label{W:admissible}\cW_k(\gs\gl_{n+1|n}) \cong \text{Com}(\tilde{V}^{\ell_0+1}(\gg\gl_n), L_{\ell_0}(\gs\gl_{n+1}) \otimes \cE(n)), \qquad k = -1  + \frac{1}{n+\ell_0+1}.\end{equation}

We now specialize to the case when $\ell_0$ is a positive integer. In this case, $\tilde{V}^{\ell_0+1}(\gg\gl_n) = L_{\ell_0+1}(\gg\gl_n)$, so \eqref{W:admissible} is replaced with $\cW_k(\gs\gl_{n+1|n}) \cong \text{Com}(L_{\ell_0+1}(\gg\gl_n), L_{\ell_0}(\gs\gl_{n+1}) \otimes \cE(n))$. We obtain

\begin{theorem}\label{thm:rational} For $\ell_0 \in \mathbb{N}$ and $k = -1 + \frac{1}{n+\ell_0+1}$, the simple quotient $\cW_k(\gs\gl_{n+1|n})$ is strongly rational.
\end{theorem}

\begin{proof} This follows from \eqref{W:admissible} together with  \cite[Corollary 14.1]{ACL}, which says that $\text{Com}(L_{\ell_0+1}(\gg\gl_n), L_{\ell_0}(\gs\gl_{n+1}) \otimes \cE(n))$ is strongly rational for $\ell_0 \in \mathbb{N}$.
\end{proof}

%By \cite[Corollary 14.1]{ACL}, \eqref{coset:admissible} is strongly rational for $\ell_0 \in \mathbb{N}$. %Combining this with the isomorphism$$\cW_k(\gs\gl_{n+1|n}) \cong \text{Com}(L_{\ell+1}(\gg\gl_n), L_{\ell}(\gs\gl_{n+1}) \otimes \cE(n)),$$ where $k = -1  + \frac{1}{n+\ell+1}$ and $\ell \in \mathbb{N}$, 

\subsection{$\cW^k(\gs\gl_{n+1|n})$ at the critical level} Recall the conjecture of Adamovi\'c, Feigin, and Nakatsuka \cite{AFN26} that $\cW^{-1}(\gs\gl_{n+1|n})$ is isomorphic to the orbifold $(\cH_{\text{deg}}(2n) \otimes \cE(n))^{\text{GL}_m}$, where $\cH_{\text{deg}}(2n)$ is the degenerate Heisenberg algebra of rank $2n$, whose generators transform under $\text{GL}_n$ as $\mathbb{C}^n \oplus (\mathbb{C}^n)^*$. We can interpret this conjecture as a limiting case of Ito's conjecture \eqref{intro:itocoset}. Note that the critical level $k = -1$ corresponds to the large level limit $\lim_{\ell \rightarrow \infty} \cC^{\ell}(n)$ in \eqref{intro:itocoset}. However, there are several different ways to define this limit which lead to non-isomorphic algebras. For example, we can rescale the generators of $V^{\ell}(\gs\gl_{n+1})$ by $\frac{1}{\sqrt{\ell}}$, so that $\lim_{\ell \rightarrow \infty} V^{\ell}(\gs\gl_{n+1}) \cong \cH((n+1)^2 - 1)$, i.e., the Heisenberg algebra of rank $(n+1)^2 -1$. By  \cite[Theorem 6.10]{CL1}, we have $$\lim_{\ell \rightarrow \infty} \cC^{\ell}(n) \cong (\cH(2n) \otimes \cE(n))^{\text{GL}_n}$$ which is a simple VOA of central charge $c = 3n$ with the correct strong generating type. This algebra corresponds to specializing the OPE algebra to central charge $c = 3n$, and is {\it not} isomorphic to $\cW^{-1}(\gs\gl_{n+1|n})$ which has a nontrivial center.

There is another way to take the large level limit of $\cC^{\ell}(n)$ where we rescale the generators of $V^{\ell}(\gs\gl_{n+1})$ by $\frac{1}{\ell}$. With this rescaling, $\lim_{\ell \rightarrow \infty}V^{\ell}(\gs\gl_{n+1})$ is isomorphic to the degenerate Heisenberg algebra $\cH_{\text{deg}}((n+1)^2 - 1)$ of rank $(n+1)^2 -1$, and we have
$$\lim_{\ell \rightarrow \infty} \cC^{\ell}(n) \cong \cH_{\text{deg}}(n^2) \otimes (\cH_{\text{deg}}(2n) \otimes \cE(n))^{\text{GL}_n}.$$ This is too big because of the extra factor $\cH_{\text{deg}}(n^2) \cong \lim_{\ell\rightarrow \infty} V^{\ell+1}(\gg\gl_n)$. But this factor lies in the center, and the quotient
$$\cH_{\text{deg}}(n^2) \otimes (\cH_{\text{deg}}(2n) \otimes \cE(n))^{\text{GL}_n} / \langle \cH_{\text{deg}}(n^2) \rangle \cong (\cH_{\text{deg}}(2n) \otimes \cE(n))^{GL_n}$$ has the same graded character as $\cW^{-1}(\gs\gl_{n+1|n})$. According to the conjecture of \cite{AFN26}, this is the correct way to take the large level limit of $\cC^{\ell}(n)$ so that Ito's conjecture holds at the critical level.

\section{Level rank duality via mirror equivalence} \label{level-rankrationality} 
In this section, we study the representation theory of strongly rational vertex superalgebras, and we apply our results to the $\cW$-algebras $\cW_k(\mathfrak{sl}_{n+1|n})$ in Theorem \ref{thm:rational}.

\subsection{Mirror equivalence for vertex superalgebras}\label{sec:mirror}
Let $V$ be a simple vertex superalgebra with the decomposition $V=V_0\oplus V_1$ according to the parity. While for vertex algebras one usually is only interested in the category of local modules, in the super case one often also likes to consider so-called Ramond-twisted modules. Both categories together are best described via the even subalgebra $V_0$ of $V$. 

Let $\mathcal V_0$ be a vertex tensor category of $V_0$-modules that contains $V_1$ as an object. 
Let $\mathcal{V}$ be the supercategory of local and also Ramond twisted modules of $V$.
This category can be studied using the theory of vertex superalgebra extensions \cite{CKM}. In particular, $V$ can be identified with a commutative superalgebra in $\mathcal V_0$ \cite{CKL}. Let $\mathcal{SV}_0 = \mathcal V_0 \boxtimes \text{sVect}$ be the supercategory associated to $\mathcal V_0$. The underlying category $\underline{\mathcal{C}}$  of a supercategory $\mathcal{C}$ has the same objects as $\mathcal C$, but morphisms are required to be even.
Then $V$ can also be identified with a commutative algebra in $\mathcal{SV}_0$ whose underlying object is $(V_0, V_1)$ that is the object whose even part is $V_0$ and whose odd part is $V_1$. It comes with a multiplication morphism $m: V \boxtimes V \rightarrow V$. The category $\underline{\mathcal V}$ is then precisely the underlying category of modules for the commutative algebra $V$ that lie in $\underline{\mathcal{SV}}_0$.
Then there is an induction functor from $\underline{\mathcal{SV}}_0$ to $\underline{\mathcal V}$
\begin{equation} \label{eq: induction1}
\text{Ind}: \underline{\mathcal{SV}}_0 \rightarrow \underline{\mathcal V}, \qquad X \mapsto (V \boxtimes X, m \boxtimes \text{id}_X), \qquad f \mapsto \text{Id}_V \boxtimes f, 
\end{equation}
This functor is monoidal and by composing it with the embedding of $\mathcal V_0$ in $\underline{\mathcal{SV}}_0$, \eqref{eq: induction1} lifts to a functor
\begin{equation} \label{eq: induction2}
\text{Ind}^V: \mathcal V_0 \rightarrow \underline{\mathcal V}, \qquad X \mapsto (V \boxtimes (X, 0), m \boxtimes \text{id}_{(X,0)}), \qquad f \mapsto \text{Id}_V \boxtimes (f, 0),
\end{equation}
which we also call induction. According to \cite[Theorem A.1]{CFKLN}, the induction \eqref{eq: induction2} is an equivalence as tensor categories.
\begin{proposition}\textup{\cite[Prop. 4.22]{CKM}}
Under the above set-up, that is in particular $V = V_0 \oplus V_1$ is a vertex superalgebra with even  and odd part $V_0$ and $V_1$ and $V$ is an object in a vertex tensor category $\mathcal V_0$ of $V_0$-modules,  then $V_1$ is a simple current with $V_1 \boxtimes V_1 \cong V_0$ and 
    for $X$ a simple object in $\underline{\mathcal V}$ with $X = X_0 \oplus X_1$  its decomposition into even and odd part, then $X_0$ is simple and $X_1 \cong V_1 \boxtimes X_0$.
\end{proposition}
In the case of this proposition, we call $X$ a local (or a Neveu-Schwarz) module if it is a module for the vertex superalgebra $V$, and otherwise we call it a (Ramond) twisted module. 

Let $W$ be another simple vertex superalgebra with even subalgebra $W_0$ and vertex tensor category $\mathcal W_0$ of $W_0$-modules. Let $n: W \boxtimes W \rightarrow W$ denote the multiplication of the corresponding superalgebra object in $\mathcal W_0$. Let $\mathcal H: \mathcal V_0 \rightarrow \mathcal W_0$ be a functor such that $\mathcal H(V) \cong W$. If $\mathcal H$ is a functor of vertex tensor categories, then the multiplication on $V$ induces one on $\mathcal H(V)$, that is, $(\mathcal H(V), \mathcal H(m))$ is a commutative superalgebra in $\mathcal W_0$. The resulting algebra is simple if $\mathcal H$ is fully faithful. By \cite[Remark 3.11]{CKL}, the superalgebra structure is unique up to isomorphism and hence $(\mathcal H(V), \mathcal H(m)) \cong (W, n)$. It follows that if $\mathcal H$ is an equivalence, then it induces an equivalence $\underline V \cong \underline W$ of the corresponding objects in $\underline{\mathcal{V}}$ with the functor $\text{Ind}^W \circ \mathcal H \circ (\text{Ind}^V)^{-1}$.
This means that equivalences of categories $\mathcal{V}_0$ and $\mathcal{W}_0$ of even subalgebras can be lifted to equivalences of categories $\underline{\mathcal{V}}$ and $\underline{\mathcal{W}}$ for the vertex superalgebras that include the Ramond twisted modules. We apply this observation to the mirror equivalence.

The following theorem is called mirror equivalence for vertex algebras. A first version was proven for unitary vertex algebras \cite{Lin:2016hsa} and was generalized in \cite{CKM2, McRae:2021urf}.
  \begin{theorem} \label{thm:duality}
 Let $U$ and $V$ be simple self-contragredient vertex operator algebras such that:
\begin{itemize}
\item  There is an injective conformal vertex algebra homomorphism $\iota_A : U \otimes V \rightarrow A$, where
$A$ is a simple conformal vertex algebra.
\item $A$ is a $U \otimes V$-module, and
\[
A  = \bigoplus_{i \in I}  U^i \otimes V^i
\]
where the $U^i$ are distinct simple $U$-modules and the $V^i$ are semisimple $V$-modules. Let
$0  \in I$ denote the index such that $U^0 = U$.
\item  The vertex subalgebras $U$ and $V$ form a dual pair in $A$ in the sense that $V^0 = V$ and
$\text{dim Hom}_{V}(V,V^i) = \delta_{i,0}$
for all $i \in I$.
\item  The $U$-modules $U^i$ for $i \in I$ are objects of a locally-finite semisimple braided ribbon
tensor category $\mathcal U$ of $U$-modules that is closed under contragredients.
\item The $V$-modules $V^i$ for $i \in I$ are objects of a braided tensor category $\mathcal V$ of grading-
restricted generalized $V$-modules which is closed under submodules, quotients, and
contragredients.
\end{itemize}
Let $\mathcal U_A$ (respectively, $\mathcal V_A$) denote the category of $U$-modules (respectively, $V$-modules) whose
objects are finite direct sums of the $U^i$ (respectively, of the $V^i$) for $i \in I$. Then:
\begin{enumerate}
    \item  The category $\mathcal U_A$ of $U$-modules is a tensor subcategory of $\mathcal U$.
\item The category $\mathcal V_A$ of $V$-modules is a semisimple ribbon tensor subcategory of $\mathcal V$ with
distinct simple objects $\{V^i | i \in I\}$. In particular, $\mathcal V_A$ is rigid.
\item There is a braid-reversed tensor equivalence $\tau : \mathcal U_A \rightarrow  \mathcal V_A$ such that $\tau(U_i) \cong  (V^{i})'$ for $i \in I$, where $(V^i)'$ is the contragredient dual of $V^i$.
\end{enumerate}
\end{theorem}
This theorem generalizes to $A$ being a vertex superalgebra \cite{McRae:2021urf}. The purpose of this section is to adapt  this theorem to the case where $U$ and $V$ are allowed to be vertex superalgebras. 

Before doing so, we note that in the special case that $U, V$ are strongly rational and that  $A$ is holomorphic, that is, every module of $A$ is isomorphic to a direct sum of copies of $A$, then the categories $\mathcal U_A, \mathcal V_A$ in Theorem \ref{thm:duality} are in fact
$\mathcal U \cong \mathcal U_A$ and $\mathcal V \cong \mathcal V_A$. The reason is that their Frobenius-Perron dimensions coincide, since the category of modules of the vertex algebra $A$ is the category of local modules of the corresponding commutative algebra object $A$ \cite{HKL, CKM} and since in this setting the following relation holds \cite[Cor.3.32]{DMNO} 
\[
1 = FP(\mathcal (\mathcal U \boxtimes \mathcal V)_A^{\text{loc}}) = \frac{FP(\mathcal U \boxtimes \mathcal V)}{FP(A)^2} = \frac{FP(\mathcal U)FP(\mathcal V)}{FP(\mathcal U_A)FP(\mathcal V_A)}.
\]

Our set-up is as follows:
\begin{enumerate}
    \item By a category of modules of a vertex superalgebra, we mean the underlying category, that is, morphisms are required to be even, see \cite{CKM}. 
    \item Let $F$ be a simple vertex superalgebra with $F = F_0 \oplus F_1$ its decomposition into even and odd part.
    \item Let $F^R$ be a simple Ramond twisted module with parity decomposition $F^R = F_0^R \oplus  F_1^R$.
    \item Assume that there is a conformal embedding $V \otimes W \hookrightarrow F$ with $V, W$ a mutually commuting pair of vertex subsuperalgebras and let
    \begin{equation} \label{eq: decomposition ramond}
    F = \bigoplus_{i\in I} V^i \otimes W^i, \qquad F^R = \bigoplus_{i \in I^R} V^{i,R} \otimes W^{i,R}
    \end{equation}
    be the decompositions into $V \otimes W$-modules. Here, $W^i, W^{i,R}$ are assumed to be distinct simple $W$-modules and $W \cong W^0$.
    \item Assume that $F_0$ has a vertex tensor category $\mathcal F_0$, that is, semisimple with simple objects $F_0, F_1, F_0^R, F_1^R$, allowing that possibly $F_0^R \cong F_1^R$. As $F_0$ is a $\mathbb Z/2\mathbb Z$-orbifold of $F$, it follows that $F_1$ is a simple current of order two: $F_1 \boxtimes F_1 \cong F_0$ \cite{McRae-orbifold}. We require that $F_1 \boxtimes F_0^R \cong F_1^R$.
    \item Let $V = V_0 \oplus V_1$, $W = W_0 \oplus W_1$ be the decompositions into even and odd parts. 
    Assume that there are vertex tensor categories $\mathcal V_0, \mathcal W_0$ of $V_0$ and $W_0$-modules such that $F$ is an object in the Deligne product $\mathcal V_0 \boxtimes \mathcal W_0$.
    Let $\underline{\mathcal V}$ be the underlying supercategory of local and also Ramond twisted modules of $V$.
    \item Let $\widetilde F = \widetilde F_0 \oplus \widetilde F_1$ be another vertex superalgebra with a vertex tensor category $\mathcal {\widetilde F}_0$ of modules of its even subalgebra, such that there is a braid-reversed equivalence $\tau: \mathcal F_0 \rightarrow \mathcal {\widetilde F}_0$ with the property that $\tau(F_1) \cong \widetilde F_1$.    
    \item Let $A$ be the corresponding canonical algebra, see \cite{CKM2}, that is
    \begin{equation} \label{eq: decomposition A}
    A = F_0 \otimes \widetilde F_0 \oplus F_1 \otimes \widetilde F_1 \oplus F_0^R \otimes \tau(F_0^R)' \oplus \left[ F_1^R \otimes \tau( F_1^R)' \right]
    \end{equation}
    with the convention here that terms in square brackets only appear if $F_0^R \not\cong F_1^R$:
    \[
    [X] := \begin{cases}
        X & \text{if} \ F_0^R \not\cong F_1^R \\ 0 & \text{if} \ F_0^R \cong F_1^R
    \end{cases}
    \]
     We set $\widetilde F_0^R = \tau(F_0^R)'$, $\widetilde F_1^R = \tau(F_1^R)'$.
    Let $V^i = V^i_0 \oplus V^i_1$ et cetera be the parity decompositions of modules. Combining \eqref{eq: decomposition ramond} and \eqref{eq: decomposition A}, we get
    \begin{equation}
        \begin{split}
            A \cong &\bigoplus_{i \in I} \left( V^i_0 \otimes W^i_0 \otimes \widetilde F_0 \oplus 
            V^i_1 \otimes W^i_1 \otimes \widetilde F_0 \oplus
            V^i_0 \otimes W^i_1 \otimes \widetilde F_1 \oplus 
            V^i_1 \otimes W^i_0 \otimes \widetilde F_1  \right)\ \ \oplus \\
            &\bigoplus_{i \in I^R} \left( V_0^{i, R} \otimes W_0^{i, R} \otimes \widetilde F_0^R \oplus 
            V_1^{i, R} \otimes W_1^{i, R} \otimes \widetilde F_0^R \oplus
            \left[V_0^{i, R} \otimes W_1^{i, R} \otimes \widetilde F_1^R \oplus 
            V_1^{i, R} \otimes W_0^{i, R} \otimes \widetilde F_1^R\right]  \right) \\
            \cong &\bigoplus_{i \in I} \left( V^i_0 \otimes X^i_0 \oplus V^i_1 \otimes X^i_1 \right) \ \ \oplus \ \ 
            \bigoplus_{i \in I^R} \left( V_0^{i, R} \otimes X_0^{i, R}  \oplus  V_1^{i, R} \otimes X_1^{i, R}  \right), \\
\text{where} \ \     &      X_0^i := W_0^i \otimes \widetilde F_0 \oplus  W_1^i \otimes \widetilde F_1, \qquad 
            X_1^i := W_1^i \otimes \widetilde F_0 \oplus  W_0^i \otimes \widetilde F_1  \\
          &   X_0^{i, R} := W_0^{i, R} \otimes \widetilde F^R_0 \oplus  \left[W_1^{i, R} \otimes \widetilde F_1^R\right] , \qquad X_1^{i, R} := W_1^{i, R} \otimes \widetilde F_0^R \oplus  \left[W_0^{i, R} \otimes \widetilde F^R_1\right] .
        \end{split}
    \end{equation}
\end{enumerate}
Under this set-up, the mirror equivalence for vertex algebras applies to the pair $(V_0, X_0:=X_0^0)$ inside $A$, that is, if the assumptions of the mirror equivalence theorem are satisfied for $V_0$ and $X_0$, then there is a corresponding braid-reversed equivalence. In particular, if $\mathcal F_0$ and $\mathcal {\widetilde F}_0$  are the categories of all $F_0$ and $\widetilde F_0$-modules, respectively, and if $V_0$ and $X_0$ are strongly rational, then there is a braid-reversed equivalence between $\mathcal V_0$ and the category $\mathcal X_0$ of $X_0$-modules sending 
\begin{equation} \label{eq: braidequiv even module}
V^i_0 \mapsto (X^i_0)', \qquad V^i_1 \mapsto (X^i_1)', \qquad V_0^{i, R} \mapsto (X_0^{i, R})', \qquad V_1^i \mapsto (X_1^{i, R})'.
\end{equation}
Note that since $V^0_1 \boxtimes V^0_1 \cong V_0$, it follows that its image under $\tau$ has the same property, 
$(X^0_1)' \boxtimes (X^0_1)' \cong X_0$, and hence both $V^0_1$ and $(X^0_1)'$ are their own contragredient duals. It follows that the equivalence $\tau$ sends $V$ to $X := X_0 \oplus X^0_1$. Hence, by the previous observation in \eqref{eq: induction2}, the functor \eqref{eq: braidequiv even module} induces an equivalence 
between $\underline{\mathcal V}$ and $\underline{\mathcal X}$ sending
\[
V^i :=(V^i_0, V^i_1) \mapsto X^i:= (X^i_0, X^i_1), \qquad
V^{i,R}:= (V_0^{i, R}, V_1^{i, R}) \mapsto X^{i,R}:= (X_0^{i, R}, X_1^{i, R}).
\]
Note that the modules for the superalgebras $X$ and $W$ are related, namely
 $X^i = W^i \otimes \widetilde{F}$  and $X^{i,R} = W^{i,R} \otimes \widetilde F^R$
with $\widetilde F^R = \widetilde F^R_0 \oplus [\widetilde F^R_1]$. 
Since $\widetilde F \boxtimes \widetilde F \cong \widetilde F$, we can immediately conclude that the Grothendieck rings of the local modules of $X$ and $W$ coincide. On the other hand, the relation between Ramond twisted modules has an obstruction. In the following example, we introduce vertex superalgebras $\widetilde F$ of type $B$ and $D$. The case $D$ is then the type such that the Grothendieck rings of $\underline{\mathcal W}$ and $\underline{\mathcal X}$ coincide, while for type $B$ there is a multiplicity two issue. 
\begin{example} 
Assume that $\widetilde F$ has a category that is semisimple with a uniqe simple local module $\widetilde{F}$ and a unique simple Ramond twisted module $\widetilde{F}^R$. Let $\widetilde F = \widetilde F_0 \oplus \widetilde F_1, \widetilde F^R = \widetilde F^R_0 \oplus \widetilde F^R_1$ be their decompositions into even and odd part. 
Assume that $\widetilde F_0$ admits a semisimple vertex tensor category that is closed under contragredient duals with simple objects $\{\widetilde F_0, \widetilde F_1, \widetilde F_0^R, \widetilde F_1^R\}$ allowing only possibly $\widetilde F^R_0\cong \widetilde F_1^R$. Since $\widetilde F$ is a simple vertex superalgebra, $\widetilde F_1$ is a simple current, in other words,  $\widetilde F_1 \boxtimes \widetilde F_1 \cong \widetilde F_0$. Moreover, since
$\widetilde F^R$ is a simple $\widetilde F$-module, we have $\widetilde F_1 \boxtimes \widetilde F^R_0 \cong \widetilde F^R_1$ and $\widetilde F_1 \boxtimes \widetilde F^R_1 \cong \widetilde F^R_0$. In particular, $\textup{Ind}^{\widetilde F}(\widetilde F_0^R) \cong \widetilde F^R$ and $\textup{Ind}^{\widetilde F}(\widetilde F_1^R) \cong  \pi(\widetilde F^R)$ with $\pi$ the parity reversal. As the tensor product of two Ramond twisted modules is local, it follows that $\widetilde F^R \boxtimes \widetilde F^R \cong a\,\widetilde F \oplus\, b \,\pi(\widetilde F)$ for some $a,b\in \mathbb{Z}_{\geq 0}$. Since the induction is an equivalence, we only need to compute the corresponding multiplicities $a$ and $b$ in $\widetilde F^R_0\boxtimes \widetilde F^R_0\cong a \widetilde F_0\oplus b \widetilde F_1$. There are three cases:

\noindent {\bf Type B:} If  $\widetilde F^R_0 \cong \widetilde F^R_1$, then $\widetilde F^R_0$ has to be its own contragredient dual and hence $\text{Hom}(\widetilde F^R_0 \boxtimes \widetilde F^R_0, \widetilde F_0) \cong \text{Hom}(\widetilde F^R_0, \widetilde F^R_0) = \mathbb C$, that is $a=1$. Since $\widetilde F_1 \boxtimes \widetilde F_0^R \cong \widetilde F_0^R$ it follows that
\[
\widetilde F_0 \oplus b \widetilde F_1 \cong \widetilde F^R_0 \boxtimes \widetilde F^R_0 \cong \widetilde F_1 \boxtimes (\widetilde F^R_0 \boxtimes \widetilde F^R_0) \cong \widetilde F_1 \boxtimes (\widetilde F_0 \oplus b \widetilde F_1) \cong \widetilde F_1 \oplus  b\widetilde F_0,
\]
that is $b=1$ as well. Via induction we get that 
\[
\widetilde F^R \boxtimes \widetilde F^R \cong \widetilde F \oplus \pi(\widetilde F).
\]
\noindent {\bf Type D$_{2n}$:} If  $\widetilde F^R_0 \not\cong \widetilde F^R_1$ and $(\widetilde F_0^R)' \cong \widetilde F_0^R$, then 
a similar  argument gives that
\[
\widetilde F_0^R \boxtimes (\widetilde F_0^R)' \cong \widetilde F_0 \oplus b \widetilde F_1, \qquad \widetilde F_1^R \boxtimes (\widetilde F_0^R)' \cong \widetilde F_1 \oplus b \widetilde F_0
\]
and since $ (\widetilde F_0^R)'$ is not the contragredient dual of $\widetilde F_1^R$ it follows that $b=0$, that is 
\[
\widetilde F^R \boxtimes \widetilde F^R \cong \widetilde F.
\]

\noindent {\bf Type D$_{2n+1}$:}
If  $\widetilde F^R_0 \not\cong \widetilde F^R_1$  and $(\widetilde F_0^R)' \cong \widetilde F_1^R$. Then the argument is the same as the type $D_{2n}$ case with  only a parity difference that is
\[
\widetilde F^R \boxtimes \widetilde F^R \cong 
\pi(\widetilde F).
\]  

We note that for $n \in \mathbb Z_{\geq 2}$, $\widetilde F_0 = L_1(\gs\go_n)$ gives an example of type $B$ if $n$ is odd and of type $D_{n/2}$ if $n$ is even, hence the naming of the cases.  
\end{example}

\subsection{Level-rank duality}

Let $\cH_{n, m}$ be the Heisenberg vertex algebra corresponding to $X^{\alpha_m} + \sum_{i=1}^m :b_ic_i:$ in $L_{n}(\gs\gl_{m+1}) \otimes \cE(m)$. In this section, we discuss
\[
C(n, m) :=
\text{Com}(L_{n+1}(\gs\gl_m)\otimes \cH_{n,m}, L_{n}(\gs\gl_{m+1}) \otimes \cE(m)), \quad n,m\in \mathbb{Z}_{\geq 2}.
\]
When $n=1$ or $m=1$, one can handle these via the conventions $L_{k}(\gs\gl_1) = \mathbb C = W_k(\gs\gl_1), A_0 = \{ 0 \}$ et cetera.

\begin{theorem}
For $n, m \in \mathbb Z_{\geq 2}$, $C(n, m) \cong C(m, n)$. 
\end{theorem}
\begin{proof}
The coset $C(n, m)$ has been discussed in \cite[Section 14]{ACL}. In particular it was shown that $C(n, m)$ is a conformal extension of $\cH \otimes W_{ 1 -m + \frac{1}{n+m}}(\gs\gl_m) \otimes
W_{ 1 -n + \frac{1}{n+m}}(\gs\gl_n)$. Also, note that $\cH_{n, m}$ extends to a lattice $L_{n, m}\cong \sqrt{m(m+1)(n+m+1)}\mathbb Z$ such that
$C(n, m) =
\text{Com}(L_{n+1}(\gs\gl_m)\otimes V_{L_{n,m}}, L_{n}(\gs\gl_{m+1}) \otimes \cE(m))$. This is in fact always the case for any Heisenberg coset $C = \text{Com}(\cH, V)$ as long as $V$ is simple as vertex algebra and an object in a completion of a vertex tensor category of $C \otimes \cH$-modules on which $\cH$ acts semisimply. In this case $\text{Com}(C, V)$ 
is either $\cH$ or a lattice vertex algebra $V_L$ \cite{CKLR}. 

Let $L_{N} = \epsilon_N\mathbb Z$ with $\epsilon_N^2=N$, that is $L_N \cong \sqrt{N}\mathbb Z$.
Let $n, m$ be positive integers and  let $\ell(n, m) = 1- m +\frac{1}{n+m}$. Recall the following cosets \cite{ACL, ABI}
\begin{equation}
\begin{split}
\text{Com}(L_n(\gs\gl_m), L_{n+1}(\gs\gl_m) \otimes \cE(m)) &\cong W_{ \ell(n, m) }(\gs\gl_m) \otimes V_{L_m} \\
\text{Com}(L_n(\gs\gl_m), \cE(nm)) &\cong L_m(\gs\gl_n) \otimes V_{L_{nm}} \\
\text{Com}(L_n(\gs\gl_m)  \otimes V_{L_{nm}} , \cE(nm)) &\cong L_m(\gs\gl_n)
\end{split}
\end{equation} 
This can be described via conformal embeddings
\begin{equation}\label{eq:embedding}
\begin{split}
\cE(nm+n+m) &\hookleftarrow L_m(\gs\gl_{n+1}) \otimes \cE(n) \otimes L_{n+1}(\gs\gl_{m}) \otimes V_{m(n+1)} \\
\cE(nm+n+m) &\hookleftarrow L_n(\gs\gl_{m+1}) \otimes \cE(m) \otimes L_{m+1}(\gs\gl_{n}) \otimes V_{n(m+1)} \\
\cE(nm+n+m) &\hookleftarrow L_m(\gs\gl_{n}) \otimes    L_{n}(\gs\gl_{m})  \otimes V_{mn}       \otimes   \cE(n) \otimes \cE(m) \\
&\hookleftarrow  L_{m+1}(\gs\gl_{n}) \otimes    L_{n+1}(\gs\gl_{m})  \otimes V_{L_{mn}} \otimes V_{L_n} \otimes V_{L_m} \otimes W_{ \ell(n, m) }(\gs\gl_m) \otimes
W_{ \ell(m, n) }(\gs\gl_n)
\end{split}
\end{equation}
 Via the embedding in $\cE(nm+n+m)$ the lattice $L_{n, m}$ is generated by $\epsilon_{nm}  + \epsilon_m + m(\epsilon_m-\epsilon_n)$. 
We verify that 
\begin{equation}\nonumber
\begin{split}\text{Span}_{\mathbb Q}( \epsilon_{nm}  + \epsilon_m + m(\epsilon_m-\epsilon_n), \epsilon_{nm} +\epsilon_n) &= 
\text{Span}_{\mathbb Q}(\epsilon_{nm} +\epsilon_m, \epsilon_{nm} +\epsilon_n) \\ &= 
\text{Span}_{\mathbb Q}( \epsilon_{nm}  + \epsilon_n + n(\epsilon_n-\epsilon_m), \epsilon_{nm} +\epsilon_m)
\end{split}
\end{equation}
and hence the intersection of the lattices $L_{n, m} + L_{(m+1)n}$ and $L_{m, n} + L_{(n+1)m}$ has rank two.

As a consequence we get
\begin{equation} \label{eq:levelrankduality}
\begin{split} 
C(n, m) &=
\text{Com}(L_{n+1}(\gs\gl_m)\otimes V_{L_{n,m}}, L_{n}(\gs\gl_{m+1}) \otimes \cE(m))  \\
&\cong  \text{Com}( L_{n+1}(\gs\gl_m) \otimes L_{n, m}, \text{Com}(L_{m+1}(\gs\gl_{n})  \otimes V_{L_{(m+1)n}}, \cE((m+1)n))) \otimes \cE(m))   \\
&\cong  \text{Com}( L_{n+1}(\gs\gl_m) \otimes L_{m+1}(\gs\gl_{n})  \otimes V_{L_{n, m}} \otimes  V_{L_{(m+1)n}},  \cE((m+1)n)) \otimes \cE(m))   \\
&\cong  \text{Com}( L_{n+1}(\gs\gl_m) \otimes L_{m+1}(\gs\gl_{n})  \otimes V_{L_{n, m}} \otimes  V_{L_{(m+1)n}},  \cE(mn+m+n) )   \\
&\cong  \text{Com}( L_{n+1}(\gs\gl_m) \otimes L_{m+1}(\gs\gl_{n})  \otimes V_{L_{m, n}} \otimes  V_{L_{(n+1)m}},  \cE(mn+m+n) )   \\
&\cong  \text{Com}( L_{m+1}(\gs\gl_n) \otimes L_{n+1}(\gs\gl_{m})  \otimes V_{L_{m, n}} \otimes  V_{L_{(n+1)m}},  \cE((n+1)m)) \otimes \cE(n))   \\
&\cong  \text{Com}( L_{m+1}(\gs\gl_n) \otimes L_{m, n}, \text{Com}(L_{n+1}(\gs\gl_{m})  \otimes V_{L_{(n+1)m}}, \cE((n+1)m))) \otimes \cE(n))   \\
& \cong \text{Com}(L_{m+1}(\gs\gl_n)\otimes V_{L_{m,n}}, L_{m}(\gs\gl_{n+1}) \otimes \cE(n))  = C(m, n).
\end{split}
\end{equation} 
\end{proof}

As a consequence of \eqref{eq:embedding} and \eqref{eq:levelrankduality} we get that 
$C(n, m)$ is a subalgebra of $\cE(nm+n+m)$  and its commutant is a conformal extension of $L_{n+1}(\gs\gl_{m}) \otimes L_{m+1}(\gs\gl_{n})$ times a rank two Heisenberg vertex algebra. 
We now want to determine this commutant more explicitly and then derive the branching rules, that is explain how $\cE(nm+n+m)$ decomposes into modules for the tensor product of these two commuting subalgebras and in particular via applying the mirror equivalence of the previous section we will then get the classification of simple modules and simple Ramond twisted modules of $C(n, m)$ as well as their fusion rules.

\subsection{Branching Rules}

We derive the branching rules. For this we continue to use \cite{ACL, ABI}. We list some notation needed for $n, m \in \mathbb Z_{\geq 2}$. 
\begin{itemize}
\item $\ell(n, m) = 1- m +\frac{1}{n+m}$
\item $Q_m = A_{m-1}$ the root lattice of $\gs\gl_m$
\item $P^+_{m, n}$ the set of admissible weights of $L_n(\gs\gl_m)$.
\item $\Lambda^N_i$ the $i^{\T{th}}$ fundamental weight of $\gs\gl_N$ (and $\Lambda_0^N = 0$) and define maps
\[
k(\nu) = i \qquad \text{iff} \qquad \nu = \Lambda_i^m \mod Q_m, \qquad \text{and} \qquad 
j(\lambda) = i \qquad \text{iff} \qquad \lambda = \Lambda_i^n \mod Q_n
\]
\item $\beta: P^+_{n, m} \rightarrow P^+_{m, n}$ a bijection of sets of admissible weights, that is explicitly given in Section 2 of \cite{ABI}. 
\item $\mathbb Z/m\mathbb Z$ acts on $P^+_{m,  n}$ via $\sigma(\Lambda^m_i) = \Lambda^m_{\sigma(i)}$. 
\item For $\mu = \sum\limits_{i=0}^{n-1} \mu_i \Lambda^n_i \in P^+_{n, m}$ and $\sigma \in \mathbb Z/n\mathbb Z$, set $\delta(\mu, \sigma) = \sum\limits_{i=0}^{n-1} i \mu_i + m \sigma \mod nm$, see Lemma 2 of \cite{ABI}.
\item Consider $\mathbb Z\beta_1$ and $N = \mathbb Z \beta_2 \oplus \mathbb Z \beta_3$ with $\beta_1^2 = nm(n+m+1),\beta_2^2 = n(m+1), \beta_2\beta_3= nm, \beta_3^2 = m(n+1)$. Let $\omega_1,\omega_2, \omega_3$ be the duals of $\beta_1,\beta_2, \beta_3$. 
\item Define  an equivalence  relation on $P^+_{m, n+1} \times P^+_{n, m+1} \times N'/N$ as follows  
let $\sigma^m_a$ be the cyclic permutation $(1, 2, \dots, m) \mapsto (m+1-a, m+2 - a, \dots, m-a)$ and similarly $\sigma_b^n$. Then
$(\nu, \lambda, \omega ) \sim (\nu', \lambda', \omega')$ if there exists  $a \in \mathbb Z/m\mathbb Z,\  b\in \mathbb Z/n\mathbb Z$ with 
\begin{itemize}
\item  $\nu' = \sigma^m_a(\nu)$, 
\item $ \lambda' = \sigma^n_b(\lambda)$, 
\item $\omega' = \omega + ((m+1)b+ na)\omega_2 + ((n+1)a+mb)\omega_3 \mod  N$.
\end{itemize}
\item Let $I$ be the set of equivalence classes. 
\item Let $S(\nu, \lambda, \omega) = \left\{ (a, \sigma, \mu) |  a \in \{ 0, \dots, n+m\},  \ \  \sigma \in \mathbb Z/m\mathbb Z,  \ \ \mu \in P^+_{n, m},\right.$\newline  $\left.\omega = (anm +\delta(\mu, \sigma)   +j(\nu-\mu))\omega_2+  (anm + \delta(\mu, \sigma)  +k(\lambda-\sigma(\beta(\mu))))\omega_3 \right\}$ and set
\begin{equation}\nonumber
\begin{split}
M_{\nu, \lambda, \omega} &=  \bigoplus_{(a, \sigma, \mu) \in S(\nu, \lambda, \omega) }   W_{\ell(m, n)}(\nu, \mu) \otimes 
   W_{\ell(n, m)}(\lambda, \sigma(\beta(\mu))) 
  \otimes V_{\mathbb Z \beta_1 + (anm + \delta(\mu, \sigma)  -j(\nu-\mu)m -k(\lambda- \sigma(\beta(\mu)))n)\omega_1} \\
  M^R_{\nu, \lambda, \omega} &=  \bigoplus_{(a, \sigma, \mu) \in S(\nu, \lambda, \omega) }   W_{\ell(m, n)}(\nu, \mu) \otimes 
   W_{\ell(n, m)}(\lambda, \sigma(\beta(\mu))) 
  \otimes V_{\mathbb Z \beta_1 + ((a-1/2)nm + \delta(\mu, \sigma)  -j(\nu-\mu)m -k(\lambda- \sigma(\beta(\mu)))n)\omega_1}.
  \end{split}
  \end{equation}
  Here $W_{k}(\lambda, \mu)$ is the $W_k(\gs\gl_n)$-module obtained from $L_k(\lambda)$ via the $\mu$-twisted quantum Hamiltonian reduction, see \cite{AF, AFC}.
\item For $x \in I$, set 
\begin{equation}\nonumber
\begin{split}
U_{x} 
&= \bigoplus_{ (\nu', \lambda', \omega') \in x}
L_{m+1}(\nu') \otimes L_{n+1}(\lambda') \otimes V_{N +\omega'}.
\end{split}
\end{equation}
\end{itemize}

\begin{theorem} \label{thm:fusionrulesW}
For $n, m \in \mathbb Z_{\geq 2}$, we have 
\begin{enumerate}
\item $C(n, m) \cong M_{0, 0, 0}$ and $U_{0, 0, 0}$ form a mutually commuting pair in $\cE(nm+n+m)$.
\item For $(\nu, \lambda, \omega) \in P^+_{n, m+1} \times P^+_{m, n+1} \times N'/N$ the module $M_{\nu, \lambda, \omega}$ is a simple $C(n, m)$-module with  $M_{\nu, \lambda, \omega} \cong M_{\nu', \lambda', \omega'}$  if and only if $(\nu, \lambda, \omega) \sim (\nu', \lambda', \omega')$.
\item Set $M_x = M_y$ for $x\in I$ and $y \in x$. The set $\{ M_x, \pi(M_x)| x \in I\} \cup \{ M^R_x, \pi(M^R_x)| x \in I\}$ is a complete set of inequivalent simple local and Ramond twisted modules for $C(n, m)$.
\item  $\cE(mn+m+n) \cong   \bigoplus\limits_{x\in I} M_x \otimes U_x$
\item  The fusion rules are 
  \begin{equation}
  \begin{split}
  M_{\nu, \lambda, \omega} \boxtimes M_{\nu', \lambda', \omega'} &\cong \bigoplus_{(\nu'', \lambda'') \in P^+_{n, m+1} \times P^+_{m, n+1}} N_{\nu, \nu'}^{ \ \ \nu''}  \widetilde N_{\lambda, \lambda'}^{ \ \ \lambda''}    M_{\nu'', \lambda'', \omega+\omega'}\\
  M_{\nu, \lambda, \omega} \boxtimes M^R_{\nu', \lambda', \omega'} &\cong \bigoplus_{(\nu'', \lambda'') \in P^+_{n, m+1} \times P^+_{m, n+1}} N_{\nu, \nu'}^{ \ \ \nu''}  \widetilde N_{\lambda, \lambda'}^{ \ \ \lambda''}    M^R_{\nu'', \lambda'', \omega+\omega'}\\
  M^R_{\nu, \lambda, \omega} \boxtimes M^R_{\nu', \lambda', \omega'} &\cong \bigoplus_{(\nu'', \lambda'') \in P^+_{n, m+1} \times P^+_{m, n+1}} N_{\nu, \nu'}^{ \ \ \nu''}  \widetilde N_{\lambda, \lambda'}^{ \ \ \lambda''}   \pi^{mn+m+n} \left(M_{\nu'', \lambda'', \omega+\omega'- nm(\omega_2+\omega_3)}\right)
  \end{split}
  \end{equation}
  with the fusion rules of $L_{n+1}(\gs\gl_m)$ and modules $L_{m+1}(\gs\gl_n)$
  \[
  L_{n+1}(\nu) \boxtimes L_{n+1}(\nu') = \bigoplus_{\nu'' \in P^+_{m, n+1}}  N_{\nu, \nu'}^{ \ \ \nu''} L_{n+1}(\nu''), \qquad
   L_{m+1}(\lambda) \boxtimes L_{m+1}(\lambda') = \bigoplus_{\lambda'' \in P^+_{n, m+1}}  \widetilde N_{\lambda, \lambda'}^{ \ \ \lambda''} L_{m+1}(\lambda'')
    \]
\end{enumerate}

\end{theorem}
\begin{proof}
Choose $r \in \mathbb Z_{\geq 3}$ such that $mn+m+n+r = 0 \mod 8$. We then apply the mirror equivalence with $F = \cE(mn+m+n)$ and $\widetilde F = \cE(r)$. In this case $F_0 = L_1(\gs\go_{2(mn+m+n)}) = V_{D_{mn+m+n}}$ and $\widetilde F_0 =
L_1(\gs\go_{2r}) = V_{D_r}$. 
Note that $D_s$ for any $s \in \mathbb Z_{\geq 3}$ has four simple modules corresponding to the fundamental weights $\nu_1, \nu_{s-1}, \nu_s$ and to $0$, that is $V_{D_r}, V_{D_r+ \nu_1}, V_{D_r+ \nu_{s-1}}, V_{D_r+ \nu_s}$. The top levels of these modules are the corresponding integrable modules and the conformal weights of these top levels are $\frac{1}{2}, \frac{s}{4}, \frac{s}{4}$. The module $V_{D_r+ \nu_1}$ corresponds to $F_1,\widetilde F_1$ for $s = nm+n+m, r$ and the other two modules correspond to the two Ramond twisted modules and we set $F_0^R = V_{D_{nm+n+m} + \nu_{nm+n+m}}$ and $\widetilde F_{0}^R = V_{D_r + \nu_r}$. In particular this is an example of type $D_r$. Moreover in order to pass from local modules to Ramond twisted modules we can simply translate lattice vectors by $\nu_{nm+n+m}$.
The lattice $D_s + (D_s+ \nu_1)$ is isomorphic to $\mathbb Z^s$. Let $\epsilon_1, \dots, \epsilon_s$  be an orthonormal basis of $\mathbb Z^s$ identifying $\nu_1$ with $\epsilon_1$. Then $\nu_{s-1}, \nu_s$ are identified with $\frac{1}{2}(\epsilon_1 + \dots + \epsilon_{s-1} - \epsilon_s), \frac{1}{2}(\epsilon_1 + \dots + \epsilon_{s-1} + \epsilon_s)$. We will use this identification to pass from local modules to Ramond twisted modules. 

Let now $s = mn+m+r$. Then the lattice $D_s^+ = D_s \cup  (D_s+ \nu_s)$ is even and self-dual, i.e. the lattice vertex operator algebra $V_{D_s^+}$ is holomorphic and as a module for 
$V_{D_{nm+n+m}} \otimes D_{r}$ it is
\begin{equation}\nonumber
    \begin{split}
V_{D_s^+} \cong\ &V_{D_{nm+n+m}} \otimes V_{D_{r}} \oplus V_{D_{nm+n+m}+\nu_1} \otimes V_{D_{r+\nu_1}} \oplus \\ &V_{D_{nm+n+m}+ \nu_{nm+n+m-1}} \otimes V_{D_{r+ \nu_{r-1}}} \oplus V_{D_{nm+n+m}+\nu_{nm+n+m}} \otimes V_{D_{r+\nu_r}},         
    \end{split}
\end{equation}
that is the mirror equivalence as discussed above applies and all we have to do is to determine the branching rules for $\cE(mn+m+n)$.

We need the decompositions 
\begin{equation}
\begin{split}
L_n(\mu) \otimes L_1(\Lambda^m_i) &\cong \bigoplus_{\substack{\lambda \in P^+_{m, n+1} \\ \mu + \Lambda^m_i = \lambda \mod Q_m  }} L_{n+1}(\lambda) \otimes W_{\ell(n, m)}(\lambda, \mu) \\
\cE(n) &\cong \bigoplus_{i=0}^{n-1} L_1(\Lambda_i^n) \otimes V_{L_n + \frac{i \epsilon_n}{n}} \\ 
L_1(\Lambda^{nm}_i) & \cong \bigoplus_{\sigma \in \mathbb Z/m\mathbb Z} \bigoplus_{\substack{\mu \in P^+_{n, m} \\ \delta(\mu, \sigma) = i \mod nm}} L_m(\mu) \otimes L_n(\sigma(\beta((\mu))).
\end{split}
\end{equation}
The first one is from \cite{ACL}, the second one is obvious and the third one is the main result of \cite{ABI}.
Firstly the first two decompositions combine into the decomposition
\[
L_n(\mu) \otimes  \cE(m) \cong \bigoplus_{\lambda \in P^+_{m, n+1}} L_{n+1}(\lambda) \otimes W_{\ell(n, m)}(\lambda, \mu) \otimes V_{L_m + \frac{\epsilon_m}{m}k(\lambda-\mu)}
\]
with
\[
k(\nu) = i \qquad \text{iff} \qquad \nu = \Lambda_i^m \mod Q_m, \qquad \text{and} \qquad 
j(\lambda) = i \qquad \text{iff} \qquad \lambda = \Lambda_i^n \mod Q_n
\]
as defined in the set-up.
Putting these together with the third decomposition we get
\begin{equation}
\begin{split}
 \cE(mn+m+n) &\cong \bigoplus_{i=0}^{nm- 1} L_1(\Lambda^{nm}_i) \otimes V_{L_{nm} + \frac{i \epsilon_{nm}}{nm}} \otimes \cE(n) \otimes \cE(m) \\ 
 &\cong \bigoplus_{\sigma \in \mathbb Z/m\mathbb Z} \bigoplus_{\mu \in P^+_{n, m}} L_m(\mu) \otimes L_n(\sigma(\beta(\mu))) \otimes V_{L_{nm} + \frac{\delta(\mu, \sigma)\epsilon_{nm}}{nm}} \otimes \cE(n) \otimes \cE(m) \\ 
 &\cong \bigoplus_{\sigma \in \mathbb Z/m\mathbb Z} \bigoplus_{\mu \in P^+_{n, m}} \bigoplus_{\nu \in P^+_{n, m+1}} \bigoplus_{\lambda \in P^+_{m, n+1}}  L_{m+1}(\nu) \otimes L_{n+1}(\lambda) 
  \otimes W_{\ell(m, n)}(\nu, \mu) \otimes \\
  &\qquad 
  \otimes W_{\ell(n, m)}(\lambda, \sigma(\beta(\mu))) 
  \otimes V_{L_{nm} + \frac{\delta(\mu, \sigma) \epsilon_{nm}}{nm}} V_{L_m + \frac{\epsilon_m}{m}j(\nu-\mu)} V_{L_n + \frac{\epsilon_n}{n}k(\lambda-\sigma(\beta(\mu)))} \\ 
\end{split}
\end{equation}
We need to decompose the dual lattice of $L:= L_{mn} \oplus L_m \oplus L_n$.
Set $\epsilon_1 = \epsilon_{nm}, \epsilon_2 = \epsilon_n, \epsilon_3 = \epsilon_m$ and $\beta_1 = \epsilon_1 - m\epsilon_2 - n \epsilon_3, \beta_2 = \epsilon_1 + \epsilon_2, \beta_3 = \epsilon_1 + \epsilon_3$. So that $\mathbb Z \beta_2 = L_{(m+1)n}, \mathbb Z \beta_3 = L_{(n+1)m}$ and $\beta_1$ is orthogonal on both. 
Let $\omega_1, \omega_2, \omega_3 \in L'$ dual to $\beta_1, \beta_2, \beta_3$, explicitly.
\[
\omega_1 = \frac{\beta_1}{nm(n+m+1)}, \qquad
\omega_2 = \frac{\epsilon_1 + (n+1)\epsilon_2 - n \epsilon_3}{n(n+m+1)}, \qquad 
\omega_3 = \frac{\epsilon_1 -m\epsilon_2 + (m+1) \epsilon_3}{m(n+m+1)}.
\] 
Let $a, a' \in \mathbb Z$ and 
\[
v = \left( a + \frac{i}{nm}\right) \epsilon_1 +  \frac{j}{n} \epsilon_2 +   \frac{k}{m} \epsilon_3, \qquad
v' = \left( a' + \frac{i}{nm}\right) \epsilon_1 +   \frac{j}{n} \epsilon_2 +  \frac{k}{m} \epsilon_3,
\]
Then we check that $v = v' \mod L$ if and only if $a -a'\in (n+m+1)\mathbb Z$. We compute
\[
\beta_1 v = a nm + i  - jm  - kn, \qquad
\beta_2 v = a nm + i + j, \qquad
\beta_3 v =  a nm + i +k 
\]
so that
\[
v = (a nm + i - jm - kn)\omega_1  + (anm + i  +j )\omega_2 + (anm + i +k)\omega_3.
\]
Thus we have the decomposition of cosets
\begin{equation}
\begin{split}
L + \frac{i}{nm} \epsilon_1 + \frac{j}{n}\epsilon_2 + \frac{k}{m}\epsilon_3 =
\bigcup_{a =0}^{n+m}& (\mathbb Z \beta_1 + (anm + i -jm -kn)\omega_1) \oplus \\
& 
((\mathbb Z \beta_2 \oplus \mathbb Z \beta_3)+ (anm +i  +j)\omega_2 +(anm + i  +k)\omega_3)
\end{split}
\end{equation}
Hence 
\begin{equation}\label{eq:decomp}
\begin{split}
 \cE(mn+m+n) &\cong   \bigoplus_{\sigma \in \mathbb Z/m\mathbb Z} \bigoplus_{\mu \in P^+_{n, m}}\bigoplus_{\nu \in P^+_{n, m+1}} \bigoplus_{\lambda \in P^+_{m, n+1}}  L_{m+1}(\nu) \otimes L_{n+1}(\lambda) 
  \otimes W_{\ell(m, n)}(\nu, \mu) \otimes \\
  &\qquad 
  \otimes W_{\ell(n, m)}(\lambda, \sigma(\beta(\mu))) 
  \otimes V_{L_{nm} + \frac{i \epsilon_{nm}}{nm}} V_{L_m + \frac{\epsilon_m}{m}j(\nu-\mu)} V_{L_n + \frac{\epsilon_n}{n}k(\lambda-\sigma(\beta(\mu)))} \\ 
  &\cong  \bigoplus_{\sigma \in \mathbb Z/m\mathbb Z} \bigoplus_{\mu \in P^+_{n, m}} \bigoplus_{\nu \in P^+_{n, m+1}} \bigoplus_{\lambda \in P^+_{m, n+1}} \bigoplus_{a=0 }^{n+m} L_{m+1}(\nu) \otimes L_{n+1}(\lambda) 
  \otimes W_{\ell(m, n)}(\nu, \mu) \otimes \\
  &\qquad 
  \otimes W_{\ell(n, m)}(\lambda, \sigma(\beta(\mu))) 
  \otimes V_{\mathbb Z \beta_1 + (anm + \delta(\mu, \sigma)  -j(\nu-\mu)m -k(\lambda- \sigma(\beta(\mu)))n)\omega_1} \\
&\qquad  \otimes  V_{(\mathbb Z \beta_2 \oplus \mathbb Z \beta_3)+ (anm +\delta(\mu, \sigma)   +j(\nu-\mu))\omega_2 +(anm + \delta(\mu, \sigma)  +k(\lambda-\sigma(\beta(\mu))))\omega_3}
   \\  
\end{split}
\end{equation}
Define
\begin{equation}\nonumber
\begin{split}
S := \left\{ (\sigma, \mu, a) | \sigma \in \mathbb Z/m\mathbb Z , a \in \{ 0, \dots, n+m\}, \mu \in P^+_{n, m};  \right. \\ \left.  (anm +\delta(\mu, \sigma)   +j(-\mu))\omega_2 + (anm + \delta(\mu, \sigma)   +k(-\sigma(\beta(\mu))))\omega_3 \in \mathbb Z\beta_2 \oplus \mathbb Z \beta_3\right\} 
\end{split}
\end{equation}
Note that this condition is equivalent to 
\[
\omega_i ((anm +\delta(\mu, \sigma)   +j(-\mu))\omega_2 + (anm + \delta(\mu, \sigma)   +k(-\sigma(\beta(\mu))))\omega_3) \ \in \ \mathbb Z \ \text{for} \ i \in \{ 2, 3 \}
\]
which can be made explicit using
\[
\omega_2^2 = \frac{n+1}{n(n+m+1)}, \qquad 
\omega_3^2 = \frac{m+1}{m(n+m+1)}, \qquad 
\omega_2\omega_3 = -\frac{1}{(n+m+1)}, \qquad 
\]
It follows that 
\[
C(n, m) \cong \bigoplus_{(\sigma, \mu, a) \in S}
  W_{\ell(m, n)}(0, \mu) \otimes  W_{\ell(n, m)}(0, \sigma(\beta(\mu))) 
  \otimes V_{\mathbb Z \beta_1 + (anm + i -j(-\mu)m -k(- \sigma(\beta(\mu)))n)\omega_1}.
\]
Let $\mathcal D(n, m)$ be the category of modules of $L_m(\gs\gl_n)$ and $\mathcal C$ the one of $V:=W_{\ell(m, n)} \otimes W_{\ell(n, m)} \otimes V_{\mathbb Z \beta_1}$
Hence 
\[
\text{Hom}_{\mathcal D(n, m+1) \boxtimes \mathcal D(m, n+1) \boxtimes \mathcal C}\left(L_{m+1}(\mathfrak{sl_{n}}) \otimes L_{n+1}(\mathfrak{sl_{m}}) \otimes V, \cE(mn+m+n) \right) \cong V_{\mathbb Z \beta_2 \oplus \mathbb Z \beta_3}.
\]
Conversely let us consider the weights corresponding to the simple currents, that is
$\lambda = (n+1)\Lambda^m_j, \nu = (m+1) \Lambda^n_k, \mu = m\Lambda^n_k$ and we choose $\sigma, \beta$ such that $\sigma(\beta(\mu)) = n \Lambda^m_j$. This happens precisely if $\sigma = j \mod m$ and this sets $i = \delta(\lambda, \sigma)= mk+nj \mod mn$. 
We have $W_{\ell(m, n)}(\nu, \mu) \cong W_{\ell(m, n)}(0, 0)$ if and only if $(\mu, \nu) \in \{ ((m+1) \Lambda^n_k, m \Lambda^n_k), k =0, \dots, n-1 \}$ and similarly  
$W_{\ell(n, m)}(\lambda, \sigma(\beta(\mu))) \cong W_{\ell(m, n)}(0, 0)$ if and only if $(\lambda, \sigma(\beta(\nu))) \in \{ ((n+1) \Lambda^m_j, n \Lambda^m_j), j =0, \dots, m-1 \}$.
Let $\sigma^j = j \mod m$, Then $\delta(m \Lambda^n_k, \sigma^j) = mk +nj \mod nm$. 
It follows that 
\begin{equation}\nonumber
\begin{split}
U :&=\text{Hom}_{\mathcal C}\left(V, \cE(mn+m+n) \right) \\
&\cong \bigoplus_{j=0}^{m-1} \bigoplus_{k=0}^{n-1} L_{m+1}((m+1)\Lambda^n_k) \otimes L_{n+1}((n+1)\Lambda^m_j) \otimes V_{\mathbb Z \beta_2 + ((m+1)k+nj  )\omega_2 \oplus \mathbb Z \beta_3 +((n+1)j+mk)\omega_3}
\end{split}
\end{equation}
and thus by Frobenius reciprocity that $U$ is the commutant of $C(n, m)$\footnote{$C(n, m)$ corresponds to a commutative algebra $A$ in $\mathcal C$ \cite{HKL} and the category of local $A$-modules in $\mathcal C$ coincides with the category $\mathcal D$ of modules of the vertex algebra $C(n, m)$ \cite{CKM} and Frobenius reciprocity is referred to that the restriction functor is right adjoint to the induction functor, it says especially that 
$\text{Hom}_{\mathcal C}(V, M) \cong \text{Hom}_{\mathcal D}(C(n, m), M)$ for any $C(n, m)$-module $M$.}.
As $U$ is a double commutant, $U = \text{Com}\left( C(n, m),  \cE(mn+m+n) \right)$, it follows that $U$ and $C(n, m)$ are a mutually commuting pair.  
Note that the $ L_{m+1}((m+1)\Lambda^n_k), L_{n+1}((n+1)\Lambda^m_j)  $ are simple currents with fusion rules (see e.g. \cite[Section 2]{Ga}),
\[
 L_{m+1}((m+1)\Lambda^n_k) \boxtimes  L_{m+1}(\nu) \cong L_{m+1}(\sigma^n_k(\nu)), \qquad
  L_{n+1}((n+1)\Lambda^m_j) \boxtimes L_{n+1}(\lambda) \cong L_{n+1}(\sigma_j^m(\lambda))
\]
with $\sigma^m_0$ the identity and $\sigma^m_a$  the cyclic permutation $(1, 2, \dots, m) \mapsto (m+1-a, m+2 - a, \dots, m-a)$ for $a = 1, \dots, m-1$; and similarly $\sigma_b^n$. 
$U$ is a simple current extension of $L_{m+1}(\gs\gl_n) \otimes L_{n+1}(\gs\gl_m) \otimes V_N$ and since the action of $V_N$-modules is fixed-point free, that is $V_{N+\omega} \otimes V_{N+\omega'} \cong V_{N+\omega'}$ if and only if $\omega = 0$ the same is in particularly true for the simple currents appearing in the extension $U$. Thus Proposition 4.5 of \cite{CKM} applies, saying that every simple $U$-module is the induction of a simple $L_{m+1}(\gs\gl_n) \otimes L_{n+1}(\gs\gl_m) \otimes V_N$-module $L_{m+1}(\nu) \otimes L_{n+1}(\lambda) \otimes V_{N +\omega}$ for certain triples $(\nu,  \lambda, \omega) \in P^+_{n, m+1} \times P^+_{m, n+1} \times N'/N$. Denote the corresponding module $U_{\nu, \lambda, \omega}$ it decomposes as
\begin{equation}\nonumber
\begin{split}
U_{\nu, \lambda, \omega} &=\bigoplus_{j=0}^{m-1} \bigoplus_{k=0}^{n-1}   L_{m+1}(\sigma^n_k(\nu)) \otimes L_{n+1}(\sigma^m_j(\lambda)) \otimes V_{N +\omega+ + ((m+1)k+nj  )\omega_2    +((n+1)j+mk)\omega_3} \\
&= \bigoplus_{\substack{(\nu', \lambda', \omega') \in P^+_{n, m+1} \times P^+_{m, n+1} \times N'/N \\ (\nu', \lambda', \omega') \sim (\nu, \lambda, \omega)}}
L_{m+1}(\nu') \otimes L_{n+1}(\lambda') \otimes V_{N +\omega'}
\end{split}
\end{equation}
 in particular 
$ U_{\nu, \lambda, \omega}  \cong U_{\nu', \lambda', \omega'}$ if and only if  $(\nu', \lambda', \omega') \sim (\nu, \lambda, \omega)$. From \eqref{eq:decomp} we get that
  \[
  \cE(mn+m+n) \cong   \bigoplus_{\omega \in N'/N } \bigoplus_{\nu \in P^+_{n, m+1}} \bigoplus_{\lambda \in P^+_{m, n+1}} M_{\nu, \lambda, \omega} \otimes  L_{m+1}(\nu) \otimes L_{n+1}(\lambda) \otimes V_{N +\omega}. 
  \]
  Since $\cE(mn+m+n)$ is an extension of $C(n, m) \otimes U$ it follows that 
  $M_{\nu, \lambda, \omega} \cong M_{\nu', \lambda', \omega'}$  if $(\nu, \lambda, \omega) \sim (\nu', \lambda', \omega')$. Hence we can define for $x \in I$, $M_x := M_y$ and $U_x := U_y$ for any $y \in x$. Thus   
  \[
  \cE(mn+m+n) \cong   \bigoplus_{x\in I} M_x \otimes U_x.
  \]
  By mirror equivalence, $\{M_x | x \in I\}$ is a complete list of inequivalent simple local modules up to parity, and 
  the fusion rules of the $M_x$ are the same as the ones of the $U_x$. As the latter is a simple current extension of $L_{m+1}(\gs\gl_n) \otimes L_{n+1}(\gs\gl_m) \otimes V_N$, those can be read off easily, i.e.
  \[
  M_{\nu, \lambda, \omega} \boxtimes M_{\nu', \lambda', \omega'} \cong \bigoplus_{(\nu'', \lambda'')} N_{\nu, \nu'}^{ \ \ \nu''}  \widetilde N_{\lambda, \lambda'}^{ \ \ \lambda''}    M_{\nu'', \lambda'', \omega+\omega'}
  \]
  with the fusion rules
  \[
  L_{n+1}(\nu) \boxtimes L_{n+1}(\nu') = \bigoplus_{\nu'' \in P^+_{m, n+1}}  N_{\nu, \nu'}^{ \ \ \nu''} L_{n+1}(\nu''), \qquad
   L_{m+1}(\lambda) \boxtimes L_{m+1}(\lambda') = \bigoplus_{\lambda'' \in P^+_{n, m+1}}  \widetilde N_{\lambda, \lambda'}^{ \ \ \lambda''} L_{m+1}(\lambda'')
    \]
    The formulas for the Ramond twisted modules follow immediately as $\nu_{nm+n+m}$ via our identification coincides with the vector in $\frac{1}{2}(\mathbb Z\beta_1 \oplus N')$ satisfying $\omega_{nm+n+m} \beta_1 = -\frac{mn}{2}, \omega_{nm+n+m}\beta_2 = \frac{n(m+1)}{2}, \omega_{nm+n+m} \beta_3 = \frac{m(n+1)}{2}$. In particular the Ramond-twisted module associated to $M_{\nu, \lambda, \omega}$ is obtained via shifting the lattice vectors by $-\frac{mn}{2}\omega_1$.
   \end{proof}

\subsection{Results at admissible levels}

Recall that a level $\ell$ of $\gs\gl_n$ is called admissible if $\ell = - n + \frac{u}{v}$ with $u, v$ coprime positive integers and $u \geq n$. The representation theory of the simple affine vertex algebra $L_\ell(\gs\gl_n)$ at a non-integral admissible level is expected to be still quite nice. In general it is known that the category of $L_\ell(\gs\gl_n)$-modules that lie in the category $\mathcal O$ is semisimple \cite{Ar3} and that the category of $L_\ell(\gs\gl_n)$-modules that are integrable as $\gs\gl_n$-modules (also called the Kazhdan-Lusztig category or the category of ordinary modules) is a ribbon tensor category \cite{CHY2, C}. 
If the admissible level $k$ is integral then every $L_\ell(\gs\gl_n)$module is integrable as an $\gs\gl_n$-module. At non-integral admissible levels there are continuous families of modules that are not integrable as a $\gs\gl_n$-module and not even in the category $\mathcal O$. These modules have infinite dimensional conformal weight spaces and usually the conformal weight is not lower-bounded, see e.g. \cite{KR} and see \cite{ACK2} for descriptions of categories appearing. 
Hence the category of the dual $\cW$-algebra $\cW_k(\gs\gl_{n|n-1}) \cong \text{Com}(\tilde{V}^{\ell+1}(\gg\gl_{n-1}), L_\ell(\gs\gl_n) \otimes \cE(n-1))$ also have continuous families of modules. Note that here $k = -1 + \frac{v}{u}$. 

As not much is known beyond $\gs\gl_2$, we now restrict to this case. The category of weight modules of 
	$L_\ell(\gs\gl_2)$ at admissible level has been studied and there is a full classification \cite{ACK2}. This category is a vertex tensor category \cite{C2}, it is ribbon \cite{CMY, NORW}, fusion rules are determined \cite{C3, NORW} and the logarithmic Verlinde formula of \cite{CR} holds \cite{C3}. In other words, this category is under great control. The Kazama-Suzuki duality to $\cW_k(\gs\gl_{2|1})$ is such that all these results transport to the category of weight modules of $\cW_k(\gs\gl_{2|1})$ for $k = -1 + \frac{v}{u}$ and this has been worked out in great detail by many \cite{NORW, CLRW, S}. In particular the category of weight modules  of $\cW_k(\gs\gl_{2|1})$ for $k = -1 + \frac{v}{u}$ for $u, v$ coprime positive integers and $v \geq 2$ is a locally finite ribbon supercategory, modules are classified and fusion rules are known \cite{NORW}.

    We expect that similar results hold for higher rank, but likely only for certain good admissible levels. This is exemplified in the case of $L_\ell(\gs\gl_3)$ at the admissible level $\ell = -3 + \frac{u}{v}$ with $u, v$ coprime positive integers and $u\geq 3$. In this case the representation theory is much nicer for $v=2$ then for $v>2$, see \cite{ACG, CRR, FRR, KRW}.  Beyond the cases of $\gs\gl_2$ and $\gs\gl_3$ there are only few results \cite{KR}, and it is a major goal of the community to improve the situation.

\appendices

\section{Appendix}\label{sect:primary structures}

\subsection{Singular vectors and extensions}\label{sec:sing vectors and extensions}
Here we recall some notions of $\T{Vir}^c_{\cN=2}$-modules. Let $\mathcal C$ be the category of ordinary real weight $\text{Vir}^c_{\mathcal{N}=2}$-modules, i.e., an object $M$ of $\mathcal C$ carries an action of $\T{Vir}^c_{\cN=2}$ so that: 
\begin{enumerate}
    \item $M$ admits a grading by conformal weights $M = \bigoplus_{n \in \mathbb R} M_n$, with $M_n=0$ for $n<\!<0$,
    \item weight spaces $M_n$ are finite-dimensional,
    \item weight spaces $M_n$ decompose into Heisenberg weight spaces, on which $H_{(0)}$ acts semisimply.
\end{enumerate}
Let $V(c, h, j)$ be the Verma module whose lowest conformal weight is $h$ and the lowest conformal weight space has Heisenberg weight $j$, and $L(c, h, j)$ be its simple quotient.

\begin{lemma} 
For irrational $c$, $h \in \frac{1}{2} \mathbb Z_{\geq 0}$ and $j \in \mathbb Z$, $V(c, h, j)$ has
\begin{itemize} 
\item no singular vector if $ j \neq \pm 2h$,
\item a single singular vector of conformal weight $h + 1/2$ and $H_{(0)}$-weight $\pm(2h-1)$ if $j = \pm2h$.
\end{itemize}
\end{lemma}
\begin{proof}
This follows immediately from the Shapovalov form \cite[Section 8.3]{KW}.
\end{proof}

\begin{corollary} \label{cor:nonsplitexact}
For irrational $c$, $V(c, h, j)$ is simple for $h \in \frac{1}{2}\mathbb Z_{\geq 0}$ and $|j| < 2h$. 
Moreover there are the following non-split exact sequences
\begin{equation}
\begin{split}
& \ 0 \rightarrow L\Big(c, \frac{1}{2},  1 \Big) \oplus L\Big(c, \frac{1}{2}, - 1 \Big)  \rightarrow  V(c, 0, 0) \rightarrow L(c, 0, 0) \rightarrow  0 , \\
& 0 \rightarrow L(c, 1, 0) \rightarrow  V\Big(c, \frac{1}{2}, \pm 1\Big) \rightarrow L\Big(c, \frac{1}{2}, \pm 0\Big) \rightarrow 0.
\end{split}
\end{equation}
\end{corollary}
\begin{proof}
    From the previous Lemma we know that the kernel of the map $V(c, 0, 0) \rightarrow L(c, 0, 0)$ is a quotient of $V\Big(c, \frac{1}{2},  1 \Big) \oplus V\Big(c, \frac{1}{2}, - 1 \Big)$. The Kazama-Suzuki duality provides a blockwise equivalence between weight modules of $V^k(\gs\gl_2)$  and $\T{Vir}^c_{\cN=2}$-modules, where $c$ and $k$ are related via $c = \frac{3k}{k+2}$ \cite{FST, CGNS}.
    Under this duality Verma modules become the so-called {\it relaxed} highest-weight modules and the relaxed highest-weight module $R$ corresponding to $V(c, 0, 0)$ is the one satisfying the short exact sequence
    \[
    0 \rightarrow D_{k, 2}^- \oplus D_{k, -2}^+ \rightarrow R \rightarrow V^k(\mathfrak{sl}_2) \rightarrow 0    \]
if $k$ is irrational. Here $D_{k, 2}^\pm$ are the Weyl modules induced from the highest/lowest weight modules of $\gs\gl_2$ of highest/lowest weight $\lambda$; see \cite{CLRW, FST} for details on the duality. 
The second short exact sequence is seen similarly.
\end{proof}
Set 
\[
\mathcal S := \{ L(c, 0, 0) \} \cup \Big\{ \, V(c, h, j) \, \Big|\  h \in \frac{1}{2} \mathbb Z_{\geq 0},\, j \in \mathbb R,\, |j| < 2h \, \Big\}.
\]
\begin{corollary} \label{cor:ext} For irrational $c$, 
$\text{Ext}^1_{\mathcal C}(M, N) = 0$ for $M, N \in \mathcal S$. 
\end{corollary}

\begin{proof}
Let $N, M$ in $\mathcal S$ and consider a short exact sequence
\[
0 \rightarrow N \rightarrow R \rightarrow M \rightarrow 0.
\]
Let $h_N, h_M$ be the conformal weight of their top levels and $j_N, j_M$ be their Heisenberg weights. 

If $h_M < h_N$ then the top level of $R$ coincides with the top level of $M$. Since $M$ is generated by the top level it follows that either the short exact sequence splits or $R$ is also generated by the top level. Assume that the short exact sequence is non-split. By the universal property of Verma modules, $R$ needs to be a homomorphic image of $V(c, h_M, j_M)$.
This is impossible if $M \cong  V(c, h_M, j_M)$. If $M \cong L(c, 0, 0)$ then by the previous corollary the only possibilities would be $N\cong V(c, \frac{1}{2}, \pm 1)$ but these are not in $\mathcal S$. 

The contragredient dual functor is covariant and strong exact (maps non-split short exact sequences to non-split short exact sequences). The contragredient dual of $V(c, h, j)$ is $V(c, h, -j)$ and the one of $L(c, 0, 0)$ is $L(c, 0, 0)$ itself. That is, $\mathcal S$ is closed under the contragredient dual, and hence the corollary is also true for $h_M > h_N$.

If $h_M = h_N$, then necessarily $N \cong M$ and the top level of $R$ is $2$-dimensional and a module for the commutative algebra $\mathbb C L_{(1)} \oplus \mathbb C H_{(0)}$. 
As we assume that $L_0$ and $H_0$ act semisimply, the top level is a direct sum of two simple modules for this zero-mode subalgebra. By the universal property of Verma modules, $R$ is thus a homomorphic image of $V(c, h_M, j_M) \oplus V(c, h_M, j_M)$. In particular if $M$ is itself a Verma module then $R \cong  V(c, h_M, j_M) \oplus V(c, h_M, j_M)$ and if $M \cong L(c, 0, 0)$ then by the previous corollary the only possibility is $R \cong  L(c, 0, 0) \oplus L(c, 0, 0)$.
\end{proof}

\subsection{Diamond structure}
We now wish to prove some further structural properties possessed by the $\cN=2$ $\cW$-superalgebras $\cW(\fr{sl}_{n+1|n})$. The generating type \eqref{princ:gentype} suggests that the higher conformal weight fields should organize into $\cN=2$ diamonds. To be precise, we mean that generators $\omega_{a^{j,\bot}},\omega_{a^{j,\pm}},\omega_{a^{j,\top}}$ satisfy the relations in (\ref{eq:N=2 diamond}).
Here we will prove a more general claim, which will imply the former statement.

Suppose that $\cA^c$ is $1$-parameter vertex algebra with the following properties:
\begin{enumerate}[(P1)]
\item $\cA^c$ has strong generating type \eqref{princ:gentype}.
\item The strong generating fields $H, G^{\pm}, L$ of weights $1,\frac{3}{2}, \frac{3}{2}, 2$ generate $\text{Vir}^c_{\mathcal{N}=2}$ of central charge $c$, as in (\ref{eq:N=2}).
\item All even strong generators have Heisenberg weight zero.
\item For each half-integer conformal weight, the two odd generators have Heisenberg weights $1$ and $-1$.
\end{enumerate}

Recall the categories $\cC$ and $\cS$ defined in Section \ref{sec:sing vectors and extensions}.
Let $\cD$ be the subcategory of $\cC$ whose objects have finite composition series consisting only  of elements in $\mathcal S$. By Corollaries \ref{cor:nonsplitexact} and \ref{cor:ext}, this category is closed under kernels and cokernels and it is semisimple.

\begin{proposition} \label{primarydiamonds} Suppose $\cA^c$ satisfies the assumptions (P1)--(P5). Then there exists a set of even vectors $\{ v_1, \dots, v_{n-1}\} \subset \cA^c$ with the following properties:
\begin{itemize}
\item $v_1 = H$,
\item $v_d$ is $\text{Vir}^c_{\mathcal{N}=2}$-primary of conformal weight $d$ and Heisenberg weight 0 for $d =2, \dots, n-1$. 
\item The set $\{ \, v_d, \ G^+_{(0)}v_d, \ G^-_{(0)}v_d, \ G^+_{(0)}G^-_{(0)}v_d \, |\,  d =1, \dots, n-1 \, \}$ is a minimal strong generating set of $\cA^c$ as a $1$-parameter vertex algebra.
\end{itemize}
When the above three properties hold, we say the vertex algebra has a primary diamond structure. 
\end{proposition}

\begin{proof}
Consider the decomposition of $\cA^c$ in conformal weight spaces,
\[
\cA^c = \bigoplus_{n \in \frac{1}{2}\mathbb Z_{\geq 0}} \cA_n.
\]
Define
\[
\cA^N = \bigoplus_{\substack{n \in \frac{1}{2}\mathbb Z_{\geq 0} \\ n \leq N}} \cA_n,\]
and
$M^N := \text{Vir}^c_{\mathcal{N}=2} \cA^N$.
By definition of $\cD$, $M^N$ is an object in $\cD$ for every $N \in \mathbb Z_{\geq 0}$ and since $\cD$ is semisimple each $M^N$ is completely reducible. 
Since $M^{N-1} \hookrightarrow M^N$ there exists an object $R^N$ in $\cD$, such that
\[
M^N \cong M^{N-1} \oplus R^N.
\] 
We prove the theorem by induction for $m \in \{ 0 , \dots, n-1\}$. The induction hypothesis is that there exists a minimal strong generating set $I_m$ consisting of vectors $$\{ v_d, G^+_{(0)}v_d, G^-_{(0)}v_d, G^+_{(0)}G^-_{(0)}v_d | \ d =1, \dots, m\} \cup \{ x_i, y^+_{i+\frac{1}{2}}, y^-_{i+\frac{1}{2}}, z_{i+1}|\ i = m+1, \dots, n-1 \}$$
where the $v_d$ satisfy the properties of the theorem and  $x_i, z_{i+1}$ are even of conformal weight $i$ resp. $i+1$ and Heisenberg weight zero, while $y^\pm_{i+\frac{1}{2}}$ are odd of conformal weight $i+\frac{1}{2}$ and Heisenberg weight $\pm 1$.  
For $m=0$, this set $I_0$ is just the minimal strong generating set that our set-up starts with. 
If $m=1$, then 
$ \{ H, G^+_{(0)}H, G^-_{(0)}H, G^+_{(0)}G^-_{(0)}H\}$ and $ \{H, \partial G^+, \partial G^-, 2L+\partial H\}$ span the same vector space.
Replacing $x_2$ if necessary by a suitable linear combination of $x_2$ and $z_2$ then gives the set $I_1$. 
Now assume that we have constructed the set $I_m$ for $m \in \{ 1 , \dots, n-2\}$. The vector $x_{m+1}$ is in $M^{m+1} \cong M^m \oplus R^{m+1}$. 
If $x_{m+1}$ were in $M^m$, then $I_m \setminus \{ x_{m+1}\}$ would already be a strong generating set, contradicting the minimality of $I_m$. 
So by a suitable redefinition we can assume that $x_{m+1}$ is in $R^{m+1}$. $x_{m+1}$ has to be a $\text{Vir}^c_{\mathcal{N}=2}$-primary: since if not, then there exists $X_n$ with $X_n x_{m+1} \neq 0$ and $n \in  \mathbb Z_{>0}$ and $X \in \{ H, L, G^+, G^-\}$. 
 Since all strong generators have conformal weight greater or equal to $m+1$, except for the ones in $\{ v_d, G^+_{(0)}v_d, G^-_{(0)}v_d, G^+_{(0)}G^-_{(0)}v_d | d =1, \dots, m\}$ it follows that $X_n x_{m+1}$ must be a polynomial in the negative modes of the fields of these strong generators, that is $X_n x_{m+1}$ must be in $M^m$, that is $X_n x_{m+1} \in M^m \cap R^{m+1} = \{ 0 \}$. Hence $x_{m+1}$ is the desired primary vector, i.e. we set $v_{m+1 } = x_{m+1}$. since the $\{ v_d | d =1, \dots, m\}$ 
are $\text{Vir}^c_{\cN=2}$-primary and since $\{ v_{m+1},  G^+_{(0)}v_{m+1}, G^-_{(0)}v_{m+1}, G^+_{(0)}G^-_{(0)}v_{m+1}\}$ belong to the same irreducible $\text{Vir}^c_{\cN=2}$-module none of these four can be a polynomial in  the negative modes of the fields of the strong generators $\{ v_d, G^+_{(0)}v_d, G^-_{(0)}v_d, G^+_{(0)}G^-_{(0)}v_d | d =1, \dots, m\}$. It follows that we can obtain the minimal strong generating set $I_{m+1}$ by replacing $x_{m+1}, y^\pm_{m+\frac{3}{2}},  z_{m+2}$ with these vectors and if we in addition replace $x_{m+2}$ by a suitable linear combination of $x_{m+2}$ and $z_{m+2}$.
\end{proof}

    \begin{corollary} \label{univeral:primarydiamond} For generic values of $c, \lambda$, $\cW^{\cN=2}_{\infty}$ has a primary diamond structure, which proves the conjecture of Candu and Gaberdiel in \cite{CG}.
    \end{corollary}

    \begin{proof} For any fixed conformal weight $N$, the fact that the fields $\{W^{d,\bot}, W^{d,\pm}, W^{d,\top}|\ 1\leq N\}$ can be replaced by primary diamonds follows immediately from Theorem \ref{primarydiamonds}. Since $N$ is arbitrary, the claim follows.
     \end{proof}

\section{Properties of strong generators}\label{app:generators} \label{appendix:stronggenerators}

The purpose of this section is to generalize some results of \cite[Section 3]{CL2} to Lie superalgebras, especially \cite[Lemma 3.1]{CL2}.
Let $\gg$ be a finite-dimensional basic, classical Lie superalgebra, or $\gg\gl_{n|n}$. In particular $\gg$ admits a non-degenerate, invariant, consistent, supersymmetric bilinear form $B$. Let $\gg = \gg_0 \oplus \gg_1$ be the decomposition into its even and odd parts. Let $\ga$ be a direct summand of $\gg$. Then $\ga_0$ is a direct sum of abelian and simple Lie algebras. If $\ga_0$ has a simple summand then we choose a simple summand and normalize $B$ on $\ga$ such that long roots of this simple summand have norm two, otherwise $\ga \in \{\gg\gl_1, \gg\gl_{1|1}, \gp\gs\gl_{1|1}\}$ and we normalize $B$ on $\ga$ such that the odd positive root has norm one for $\ga \in \{\gg\gl_{1|1}, \gp\gs\gl_{1|1}\}$ (for $\ga = \gg\gl_1$ we can choose any non-trivial normalization). We fix this normalization throughout the remainder of this text and call it the {\it normalized bilinear form}. 

Let $\gh$ be a Cartan subalgebra of $\gg$ which by definition is a Cartan subalgebra of $\gg_0$. Then $B$ restricts non-degenerately to $\gh$ and thus induces a bilinear form on $\gh^*$. Let $P_{\mathbb Q} = \{\lambda \in \gh^* | B(\lambda, \lambda) \in \mathbb Q\}$, and let $\gg\text{-mod}_{\mathbb Q}^{fin}$ be the category of finite-dimensional weight modules for $\gg$ whose weight support is in $P_{\mathbb Q}$. Given a $\gg$-module $M$, the corresponding Weyl module at level $k$ is 
\[
V^k(M) := U(\widehat\gg) \otimes_{U(\widehat\gg_{\geq 0})} M_k.
\]
In particular, $V^k(\mathbb C)$ has a vertex superalgebra structure that we denote as usual by $V^k(\gg)$.
\begin{definition}
Let $V$ be a vertex operator superalgebra and $V$-mod a category of $V$-modules. 
\begin{enumerate}
    \item An object $M$ in $V$-mod is called almost semisimple if it is lower-bounded by conformal weight and if every proper submodule $N$ intersects the top level of $M$ nontrivially.
    \item $V$-mod is called almost semisimple if every object is a direct sum of almost semisimple objects.
\end{enumerate}
\end{definition}

Let $\gg$ be a simple Lie superalgebra. We recall the definition of the Kazhdan-Lusztig category $KL_k(\gg)$.
\begin{definition}\label{def:KL}
The category $KL_k= KL_k(\gg)$ is the full subcategory of $\widehat \gg$-modules that satisfy the following properties.
\begin{enumerate}
\item The central element $K \in \widehat{\gg}$ acts by multiplication by the scalar $k$. 
\item An object $M$ of $KL_k$ is graded by conformal weight with finite-dimensional weight spaces
\[
M = \bigoplus_{n \in \mathbb C}M_n, \qquad  \text{dim} \ M_n < \infty,
\] 
such that conformal weight is lower bounded, i.e., $M_n = 0$ unless Re$(n) > N$ for some bound $N$. 
\item There exists a finite set $\{h_1, \dots, h_s\}$, such that $M_n= 0$ unless $n\in \mathbb Z_{\geq 0} +h_i$ for some $i \in \{1, \dots, s\}$. 
\item
A module in $KL_k$ is called {\it almost simple} if every submodule intersects the top level nontrivially, \item $KL_k$ is called {\it almost semisimple} if every indecomposable module is almost simple. 
    \item $KL^k_{\mathbb Q}(\gg)$ is the category of lower-bounded $V^k(\gg)$-modules with finite dimensional generalized conformal weight spaces and weight support in $P_{\mathbb Q}$, i.e., each generalized weight subspace is in $\gg\text{-mod}_{\mathbb Q}^{fin}$.
\end{enumerate}    
\end{definition}

\begin{definition}\label{def:goodLSA}
Let $\cS$ be the set of Lie superalgebras, such that every element is a finite direct sum of the Lie superalgebras 
\begin{enumerate}
\item $\gg$ is an abelian or simple Lie algebra,
    \item $\gg$ is a simple basic classical Lie superalgebra, excluding $\gd(2, 1; \alpha)$,
    \item $\gg = \gd(2, 1; \alpha)$ and $\alpha \in \mathbb Q$,
\item $\gg = \gg\gl_{n|n}$ for some $n \in \mathbb Z_{>0}$.
\end{enumerate}    
\end{definition}

\begin{proposition}\label{prop:almost}
    For $\gg \in \cS$  and  $k \notin \mathbb Q$, the category $KL^k_{\mathbb Q}(\gg)$ is almost semisimple.
\end{proposition}
\begin{proof}
    Let $M$ be an indecomposable object in $KL^k_{\mathbb Q}(\gg)$ and let $M^{top}$ be its top space, that is the subspace of lowest conformal weight. Then $M^{top}$ is in $\gg\text{-mod}_{\mathbb Q}^{fin}$ and in particular must have a highest-weight vector with respect to any choice of positive root system. Let $\lambda \in P_{\mathbb Q}$ be this highest-weight. Then the Casimir of $\gg$ acts on this vector by multiplication with $C_\lambda := B(\lambda, \lambda + 2\rho)$ with $\rho$ the Weyl vector. In particular $C_\lambda \in \mathbb Q$. 
    By the Sugawara construction, the conformal weight of this highest-weight vector is $\frac{C_\lambda}{2(k+h^\vee)}$ with $h^\vee$ the dual Coxeter number of $\gg$ (this is a rational number). In particular, since $k\notin\mathbb Q$ it follows that neither is $\frac{C_\lambda}{2(k+h^\vee)}$. Now assume that $M$ has a proper submodule $N$, and let $N^{top}$ be its top level. The conformal weight of this top level must be in $\frac{C_\lambda}{2(k+h^\vee)} + \mathbb Z_{\geq 0}$, since $M$ is indecomposable and hence is conformal weight graded by a single $\mathbb Z$-coset.
    On the other hand, $N^{top}$ must also have a highest-weight vector for $\gg$, say of highest-weight $\mu$. As above, the conformal weight must be $\frac{C_\mu}{2(k+h^\vee)}$, so that 
    \[
    \frac{C_\lambda}{2(k+h^\vee)} + n = \frac{C_\mu}{2(k+h^\vee)}
    \]
    must be satisfied for some $n \in \mathbb Z_{\geq 0}$. But for $n \neq 0$,
    \[
   \mathbb Q \not\ni 2(k+h^\vee)n = C_\mu - C_\nu \in \mathbb Q
    \]
    so that $n=0$ and hence $N \cap M^{top}$ is nonzero.
\end{proof}

\begin{definition} \textup{(\cite[Definition 3.5]{CL2})}
Let $\ga$ be a Lie superalgebra which is the sum of a reductive Lie algebra and finitely many simple Lie superalgebras, such that the bilinear form on each simple summand is normalized in the standard way. Let $M$ be an indecomposable module for the corresponding affine vertex superalgebra $V^k(\ga)$, such that
\[
M = \bigoplus_{n \in \mathbb Z_+ + h} M_n 
\]
for some $h \in \mathbb{C}$ and $M_h$ nonzero. 
A vector in $M_h$ is called $\ga$-primary, and a vector in $\bigoplus\limits_{n \in \mathbb Z_+ + h} M_{n+1}$ is called an $\ga$-descendant. Fields corresponding to primary or descendant vectors are called primary or descendant fields, respectively. 
\end{definition}
The following is \cite[Lemma 3.1]{CL1}, but under weaker assumptions. 
\begin{lemma}\label{lemma:primary}
Let $V = \bigoplus_{n \in \frac{1}{2}\mathbb Z_+} V_n$ be a vertex superalgebra graded by conformal weight, with
$\text{dim}\ V_0 = 1$, $V_{\frac{1}{2}}=0$, and $\text{dim} \ V_n < \infty$. Let $I_1, \dots, I_d$ be finite sets, such that 
$$\bigcup_{i=1}^d S_i, \qquad S_i = \{ X_j |\  j \in I_i, \ X_j \in V_i \},$$ 
is a minimal strong generating set for $V$, with $S_1$ a basis of $V_1$ generating an affine vertex superalgebra $V^k(\ga)$, where $\ga$ is as above. Suppose that either $KL_k(\ga)$ is almost semisimple, or that each $V_n$ for $n \in \frac{1}{2}\mathbb Z_+$ is an object in $\gg\text{-mod}_{\mathbb Q}^{fin}$ and that $KL^k_{\mathbb Q}(\ga)$ is almost semisimple. Then there exists a minimal strong generating set 
\[
\bigcup_{i=1}^d   \widetilde S_i, \qquad \widetilde S_i = \{ \widetilde X_j |\  j \in I_i, \ \widetilde X_j \in V_i \},  
\]
such that $\widetilde S_1 = S_1$, the fields $\widetilde X_j$  in $\widetilde S_i$ for $i\geq 2$ are all $\ga$-primary fields, and their linear span is an $\ga$-module. 
\end{lemma}
\begin{proof}
Let $n\geq 1$ and let  
\[
 S_{\leq n} := \bigcup_{i=1}^n    S_i, \qquad  S_i = \{ X_j | \ j \in I_i, \ X_j \in V_i \}.  
\]
Assume inductively that the $X_j$ in $S_i$ for $1< i\leq n$ can be replaced by primary fields $\widetilde X_j$, such that their span is an $\ga$-module. Then there exists a set
\[
\widetilde S_{\leq n} := \bigcup_{i=1}^n   \widetilde S_i, \qquad \widetilde S_i = \{ \widetilde X_j |\  j \in I_i,\ \widetilde X_j \in V_i \},  
\]
such that the subspace spanned by normally ordered words in iterated derivatives of the elements in $S_{\leq n}$ and $\widetilde S_{\leq n}$ coincide. 
We denote this subspace by $W$. 
Since the $V_n$ for $n>1$ and also the $\widetilde X_j$ form $\ga$-modules, $W$ is a possibly infinite direct sum of $V^k(\ga)$-modules. Let $W_m$ be the subspace of $W$ of elements of conformal weight $m$, which is also an $\ga$-module. Note that $W_m = V_m$ for $m\leq n$. Since $W_{n+1}$ is a $\ga$-submodule of $V_{n+1}$, there is a short exact sequence of $\ga$-modules
\[
0 \rightarrow W_{n+1} \rightarrow V_{n+1} \xrightarrow{\pi_{n+1}} U_{n+1} \rightarrow 0,
\]
for some  $\ga$-module $U_{n+1}$. By construction, every element in $W_{n+1}$ is an $\ga$-descendant, and in particular a normally ordered polynomial in the iterated derivatives of the elements of $S_{\leq n}$. It follows that none of the $X_j$ for $j\in S_{n+1}$ is in $W_{n+1}$. Therefore the projection $\pi_{n+1}(X_j)$ of $X_j$ onto $U_{n+1}$ is nonzero, and in fact, $\{\pi_{n+1}(X_j)|\ j\in S_{n+1}\}$ must form a basis of $U_{n+1}$.

For each $j \in S_{n+1}$ we now construct $\widetilde X_j$ with the desired property. 
Let $M$ be the $V^k(\ga)$-module that is generated by $X_j$ for $j \in S_{n+1}$. 
Let $M = M_1 \oplus \dots \oplus M_r$ be its decomposition into indecomposable $V^k(\ga)$-modules. This means that we can write $X_j = X_j^1 + \dots + X_j^r$, such that $M_i$ is generated by $X_j^i$. If $x_{(1)}X_j^i \neq 0$ for some $x \in \ga$ and $i \in \{ 1, \dots, r\}$, then $X_j^i$ is by definition a $\ga$-descendent and in particular the top level of $M_i$ is less than $n+1$. By the almost semisimplicity of $KL_k(\ga)$, this implies that $X_j^i$ is already in a module that is generated by a vector of lower conformal weight, that is by a vector that must be in $W$, i.e. $X_j^i \in W_{n+1}$. Since $X_j$ is not in $W$ there must exist at least one index $i \in \{ 1, \dots, r\}$, such that $X_j^i \not\in W$ and hence in particular $x_{(1)}X_j^i = 0$ for all $x \in \ga$, i.e., $X_j^i$ must be $\ga$-primary. Let $i_1, \dots, i_s$
be the indices with this property and set $\widetilde X_j = X_j^{i_1} +  X_j^{i_2} + \dots +  X_j^{i_s}$.
\end{proof}

As mentioned, Lemma \ref{lemma:primary} is \cite[Lemma 3.5]{CL2} under weaker assumptions. This lemma has several consequences which now hold under our weaker assumptions, whose proofs are word by word the same as in \cite{CL2}. Let $V = \bigoplus_{n \in \frac{1}{2}\mathbb Z_+} V_n$ be a vertex superalgebra graded by conformal weight, satisfying the same conditions as Lemma \ref{lemma:primary}, except that $k$ is now regarded as a formal parameter. Assume that all structure constants appearing in the OPEs of the strong generators $\bigcup_{i=1}^d   S_i$ for $V$ are rational functions of $k$. Since there are only finitely many structure constants, the set $D$ of possible poles of the structure constants is finite, and we can regard $V$ as a vertex superalgebra over the ring of rational functions in $k$ with poles along $D$.

\begin{corollary}\label{cor:primarygeneralcase} \textup{(\cite[Corollary 3.5]{CL2})}
Let $V = \bigoplus_{n \in \frac{1}{2}\mathbb Z_+} V_n$ be a vertex superalgebra defined over the ring of rational functions in $k$ with poles along $D$, where $V_1$ generates $V^k(\ga)$, as above. If $\ga$ is in $\cS$,  then there exists a finite set $D'$ containing $D$ such that over the ring of rational functions in $k$ with poles along $D'$, we can replace the minimal strong generating set $\bigcup_{i=1}^d   S_i$ with a minimal strong generating set $\bigcup_{i=1}^d   \widetilde S_i$
such that $\widetilde S_1 = S_1$ and the fields $\widetilde X_j$  in $\widetilde S_i$ for $i\geq 2$ are $\ga$-primary, and their linear span is an $\ga$-module. 
\end{corollary}

Let $\gg$ be a Lie superalgebra with invariant nondegenerate bilinear form $(\ |\ )$, and let $F \in \gg$ be a nilpotent such that $\cW^k(\gg, F)$ has affine subalgebra $V^{\ell}(\ga)$, and $\ga$ is in $\cS$. Recall that the strong generators of $\cW^k(\gg, F)$ are indexed by $J^F$, a basis of $\gg^F$. Moreover, those strong generators that have conformal weight $r+1$ correspond to $\gg^F_{-r}$ which is ad$(\ga)$-invariant and hence an $\ga$-module that we denote by $M_r$.
 
\begin{corollary}\label{cor:primary}  \textup{(\cite[Corollary 3.6]{CL2})} Let $\cW^k(\gg, F)$ and $\ga$ be as above.
\begin{enumerate}
\item If we regard $\cW^k(\gg, F)$ as a $1$-parameter vertex algebra, the strong generators of conformal weight $r+1$ for $r>0$ can be chosen to be primary for the affine subalgebra $V^{\ell}(\ga)$. In particular, this holds for all but finitely values of $k$.
\item If the trivial representation appears as a direct summand of $M_r$, then the corresponding strong generator is a field of ${\rm Com}(V^\ell(\ga), \cW^k(\gg, F))$. 
\end{enumerate}
\end{corollary}
The family 
$\cW^1_{\sigma}(\gg,F)$ has been constructed in \cite[Section 3]{CL2}. It is a family of vertex superalgebras, that is OPE coefficients are continuous in $\sigma$ and $\cW^1_{\sigma}(\gg,F)/(\sigma^2k-1) \cong \cW^k(\gg,F)$. In the large level limit it becomes a free field algebra:

\begin{lemma} \label{lemma:norm} \textup{(\cite[Lemma 3.2]{CL2})} Let $\gg$ be a Lie superalgebra with invariant nondegenerate bilinear form $(\ |\ )$, and let $F$ be a nilpotent element in $\gg$ such that $\ga$ is in $\cS$. Let $\{X^{\alpha}|\ \alpha \in J^F\}$ denote the strong generating set for $\cW^1_{\sigma}(\gg,F)$ which satisfies the $\lambda$-brackets of \cite[Theorem 3.5]{CL2} in the $\sigma \rightarrow 0$ limit. If we replace the fields $X^{\alpha}$ with corrected fields $\tilde{X}^{\alpha}$ which are $\ga$-primary as above, we again have $\lambda$-brackets 
$$
[\tilde{X}^\alpha {}_\lambda \tilde{X}^\beta] = \delta_{j, k}\lambda^{2k+1} B_k(q^\alpha, q^\beta)
$$ in the $\sigma \rightarrow 0$ limit.
\end{lemma}

\section{Cosets of affine vertex superalgebras inside larger structures} \label{appendix:largerstructures}
Let $\gg$ be a Lie superalgebra admitting a non-degenerate, supersymmetric, invariant and consistent bilinear form $B$, and let $V^k(\gg, B)$ denote the affine vertex superalgebra of $\gg$ associated to $B$ at level $k$. Let $\cA^k$ be a vertex superalgebra whose structure constants depend continuously on $k$, which admits a homomorphism 
\begin{equation} \label{eq:affineaction} V^k(\gg,B) \rightarrow \cA^k. \end{equation} In \cite{CL1,CL4}, the first and third authors studied the structure of the coset $\text{Com}(V^k(\gg,B), \cA^k)$ when $\gg$ is a reductive Lie algebra, using the language of {\it deformable families}. In \cite{CL3} this was adapted to the case when $\gg = \go\gs\gp_{1|2n}$. We now generalize this to the case where $\gg$ is a simple, basic Lie superalgebra or $\gg\gl_{n|n}$.

Let $K \subseteq \mathbb{C}$ be a subset containing $0$ which is at most countable, and let $F_K$ denote the $\mathbb{C}$-algebra of rational functions in a formal variable $\kappa$ of the form $\frac{p(\kappa)}{q(\kappa)}$ where $\text{deg}(p) \leq \text{deg}(q)$ and the roots of $q$ lie in $K$. A deformable family is a free $F_K$-module $\cB$ with the structure of a vertex algebra over $F_K$. We assume that $\cB$ has a $\frac{1}{2}\mathbb{N}$-grading $\cB = \bigoplus_{d \in \frac{1}{2} \mathbb{N}} \cB[d]$ by weight, where each $\cB[d]$ is free $F_K$-module of finite rank, and $\cB[0] \cong F_K$. Under these assumptions, the limit $\cB^{\infty} :=\lim_{\kappa \rightarrow \infty} \cB$ is a well-defined vertex algebra with the same graded character as $\cB$. While $\cB^{\infty}$ has a simpler structure than $\cB$, a strong generating set for $\cB^{\infty}$ gives rise to a strong generating set for $\cB$ after a suitable localization; see \cite[Lemma 3.2]{CL1}.

By \cite[Example 3.1]{CL1}, $V^k(\gg, B)$ comes from a deformable family $\cV$ with $K = \{0\}$ in the following sense: for all $k \neq 0$, $\cV/(\kappa - \sqrt{k})\cdot \cV \cong V^k(\gg, B)$, where $(\kappa - \sqrt{k})\cdot \cV$ denotes the ideal generated by $\kappa - \sqrt{k}$. Let $\gg = \gg_0 \oplus \gg_1$ be the decomposition into even and odd parts, and let $d_i = \text{dim}\ \gg_i$. Then $\cV^{\infty}$ is the free field vertex superalgebra $\cH(d_0|d_1)$ generated by fields $X^i$ for $\{x^i\ i =1, \dots, d_0 + d_1\}$ a homogeneous basis of $\gg$ and OPEs
\[
X^i(z)X^j(w) = \frac{B(x^i, x^j)}{(z-w)^2}.
\] Note that $d_1$ is even and $\cH(d_0|d_1) \cong \cO_{\text{ev}}(d_0, 2) \otimes \cS_{\text{odd}}(\frac{d_1}{2}, 2)$ in our earlier notation. Since $B$ is nondegenerate $\cH(d_0|d_1)$ has a natural Virasoro field $L^\cH$ of central charge $\text{sdim}(\gg) = d_0 -d_1$.

\subsection{Supergroup action} Suppose that $\cA^k$ admits the homomorphism \eqref{eq:affineaction}, and that under the zero mode action of $\gg$, $\cA^k$ decomposes into finite-dimensional $\gg$-modules. 
 The restriction of this action to the reductive Lie algebra $\gg_0$ then integrates to action of a connected Lie group $G_0$ by inner automorphisms of $\cA$, having $\gg_0$ as Lie algebra. Furthermore, the adjoint action of $\gg_0$ on $\gg_1$ lifts to a compatible action of $G_0$ on $\gg$, and this action is compatible with the action of $G_0$ on $\cA$ in the sense that 
\begin{equation} \label{eq:supergroupdef} g\cdot(\xi_{(0)}(a)) =  (g\cdot \xi)_{(0)} (g\cdot a),\ \ \forall g \in G_0,\ \xi \in \gg, \ a \in \cA.\end{equation}
The compatible actions of $G_0$ and $\gg$ on $\cA$ may be regarded as the action of a Lie supergroup $G$ with Lie superalgebra $\gg$, and we say that the action of $\gg$ {\it integrates} to the action of $G$. We define the $G$-orbifold
\begin{equation} \label{def:supergrouporbifold} \cA^G: = \cA^{G_0} \cap \text{Ker}_{\cA}(\gg_1) = \text{Ker}_{\cA}(\gg).\end{equation}

Suppose that $\cA^k$ comes from a deformable family $\cA$, meaning that $\cA^k \cong \cA / (\kappa - \sqrt{k})\cdot \cA$, for all $k$ with $\sqrt{k} \notin K$. Letting $\cV$ be as above, suppose that there is a homomorphism $\cV \ra \cA$ inducing \eqref{eq:affineaction}, when $\sqrt{k} \notin K$. There is an induced action of $G_0$ on the limit $\cA^{\infty}$, but this action is often by {\it outer} automorphisms. Similarly, the derivations $\xi_{(0)}$ for $\xi \in\gg$ induce derivations on $\cA^{\infty}$ which may now be outer derivations, and are still compatible with the $G_0$-action as in \eqref{eq:supergroupdef}. This defines the action of $G$ on $\cA^{\infty}$, and we define the orbifold $(\cA^{\infty})^G: = (\cA^{\infty})^{G_0} \cap \text{Ker}_{\cA^{\infty}}(\gg_1) = \text{Ker}_{\cA^{\infty}}(\gg)$ in the same way.

\begin{definition} (cf. \cite[Definition 6.1]{CL1})\label{def:good} A family of vertex algebra $\cA^k$ with structure constants depending continuously on $k$, which admits a homomorphism \eqref{eq:affineaction} will be called {\it good} if the following conditions hold.

\begin{enumerate} 
\item There exists a deformable family $\cA$ such that $\cA^k = \cA / (\kappa - \sqrt{k})\cdot \cA$, for all $k$ with $\sqrt{k} \notin K$, and there is a homomorphism $\cV \ra \cA$ inducing \eqref{eq:affineaction}, when $\sqrt{k} \notin K$.
\item For generic $k$, $\cA^k$ has a Virasoro element $L^{\cA}$ and a conformal weight grading $\cA^k = \bigoplus_{d \in \frac{1]}{2}\mathbb{N}} \cA^k[d]$, where each $\text{dim}(\cA^k[d])$ is finite and independent of $k$.
\item For generic $k$, $\cA^k$ decomposes into finite-dimensional $\gg$-modules, and the action of $\gg$ integrates to an action of a Lie supergroup $G$ as above.
\item We have a vertex algebra isomorphism 
\begin{equation}\label{splittingoflimit} \cA^{\infty} = \lim_{\kappa \ra \infty} \cA \cong \cH(d_0|d_1) \otimes \tilde{\cA},\qquad d_0 =\text{dim}(\gg_0),\qquad d_0 =\text{dim}(\gg_1).\end{equation} 
Here  $\tilde{\cA}$ is a vertex subalgebra of $\cA^{\infty}$ with Virasoro element $L^{\tilde{\cA}}$ and $\frac{1}{2}\mathbb{N}$-grading by conformal weight, with finite-dimensional graded components. 

\item The action of $G$ on $\cA^{\infty}$ preserves $\tilde{\cA}$, that is, $g \cdot \tilde{\cA} \subseteq \tilde{\cA}$ and $\xi(\tilde{\cA}) \subseteq \tilde{\cA}$ for all $g \in G_0$ and $\xi \in \gg$. Note that \eqref{splittingoflimit} implies that $G_0$ acts on $\tilde{\cA}$ by outer automorphisms and $\gg$ acts on $\tilde{\cA}$ by outer derivations.
\item Although $L^{\gg}_0$ need not act diagonalizably on $\cA^k$, it induces a decomposition of $\cA^k$ into generalized eigenspaces corresponding to the Jordan blocks of each eigenvalue. These generalized eigenspaces can be infinite-dimensional, but any highest-weight $V^k(\gg,B)$-submodule of $\cA^k$ has finite-dimensional components with respect to this grading, for generic $k$.
\end{enumerate}
\end{definition}

Let $\cA^k$ be a good $F_K$-vertex algebra, $\cA$ and $\cV$ be the deformable families with $\cA^k = \cA / (\kappa - \sqrt{k})\cdot \cA^k$ for $\sqrt{k} \notin K$, and $V^k(\gg,B) = \cV / (\kappa - \sqrt{k})
\cV$. If $G$ is as above, it is clear that $\cA^G$ is also a deformable family and $\cA^G / (\kappa - \sqrt{k}) \cdot \cA^G \cong (\cA^k)^G$ for $\sqrt{k} \notin K$.

We now consider a new $F_K$-vertex algebra 
$$\cC := \text{Com}(\cV, \cA),$$ with Virasoro element $L^{\cC} = L^{\cA} - L^{\gg}$, where $L^{\cA}$ and $L^{\gg}$ are the Virasoro elements in $\cA$ and $\cV$, respectively. For $\sqrt{k}\notin K$, define $\cC^k = \cC / (\kappa - \sqrt{k})$. We use the same symbol $L^{\cC} = L^{\cA} - L^{\gg}$ to denote the Virasoro element of $\cC^k$, where $L^{\cA}$ and $L^{\gg}$ are now the Virasoro elements in $\cA^k$ and $V^k(\gg,B)$, respectively. Under the isomorphism \eqref{splittingoflimit}, we have $$\lim_{\kappa \ra \infty}  L^{\gg} = L^{\cH},\qquad \lim_{k\ra \infty} L^{\cA} = L^{\cH} + L^{\tilde{\cA}}.$$ 
It is not clear a priori that $\cC$ is a deformable family (i.e., $\cC[n]$ is a free $F_K$-submodule of $\cA[n]$ for all $n\geq 0$), or that for $\sqrt{k} \notin K$, we have \begin{equation} \label{defcommutant} \cC^k := \text{Com}(V^k(\gg,B), \cA^k).\end{equation} Clearly $\cC^k \subseteq \text{Com}(V^k(\gg,B), \cA^k)$ and \eqref{defcommutant} holds generically.
We will determine the possible values of $k$ for which $\cC^k \neq \text{Com}(V^k(\gg,B), \cA^k)$, and in the process we will show that $\cC$ is indeed a deformable family. For the moment, we assume that $\gg$ is indecomposable and in $\cS$. We also assume that $B$ is normalized in the canonical way (see the first paragraph of section \ref{app:generators}), so $V^k(\gg,B)=V^k(\gg)$. Next, we introduce another $F_K$-vertex algebra $$\cK := \text{Ker}_\cA(L^{\gg}_0) \cap \cA^G.$$ We clearly have $\cC \subseteq \cK$, and we let $$\cK^k := \cK / (\kappa - \sqrt{k})\cdot \cK.$$

\begin{lemma} \textup{(cf. \cite[Lemma 6.3]{CL1})} If $\sqrt{k}\notin K$ and $k+h^{\vee}\notin \mathbb{Q}$, then $\cK^k = \text{Ker}_{\cA^k}(L^{\gg}_0) \cap (\cA^k)^G$. In particular, the graded character of $\cK^k$ is independent of $k$, for all $k$ such that $k+h^{\vee}\notin \mathbb{Q}$.\end{lemma}

\begin{proof} Observe that $\cK^k \subseteq \text{Ker}_{\cA^k}(L^{\gg}_0) \cap (\cA^k)^G$, and that $\text{Ker}_{\cA}(L^{\gg}_0)$ is the eigenspace of eigenvalue zero. The action of $G$ clearly preserves the graded subspaces of $\cA$ and the $G$-invariant space is again a free $F_K$-module. Let $v \in  \text{Ker}_{\cA^k}(L^{\gg}_0) \cap (\cA^k)^G$. In particular $v$ has conformal weight $0$ for $L^{\gg}_0$.
Let $\langle v \rangle$ be the $V^k(\gg)$-module generated by $v$. As in the proof of Proposition \ref{prop:almost} we have that $v$
must be a top level vector if $k \notin\mathbb Q$. Since $v \in (\cA^k)^G$, $\langle v \rangle$ must be the trivial representation
for $k \notin\mathbb Q$. Hence for $k+h^{\vee} \notin \mathbb{Q}$, we have $\cK^k \supseteq \text{Ker}_{\cA^k}(L^{\gg}_0) \cap (\cA^k)^G$. \end{proof}

\begin{corollary}\label{cor:kdeffamily} \textup{(cf. \cite[Corollary 6.5]{CL1})}
$\cK$ is a deformable family, that is, each weight-graded space $\cK[n]$ is a free $F_K$-submodule of $\cA[n]$.
\end{corollary}

\begin{lemma} \label{lem:characterizationofck}  \textup{(cf. \cite[Lemma 6.6]{CL1})} If $\sqrt{k}\notin K$ and $k+ h^{\vee}\notin \mathbb{Q}$, then $\cC^k = \text{Com}(V^k(\gg), \cA^k) = \cK^k$. \end{lemma}

\begin{proof} Let $\omega \in \text{Com}(V^k(\gg), \cA^k)$, then $\omega$ is in particular annihilated by 
 $L^{\gg}_0$ as well as by all the zero modes $\{X^{\xi}_{(0)}|\ \xi\in\gg\}$, that is $\omega \in \text{Ker}_{\cA^k}(L^{\gg}_0) \cap (\cA^k)^G = \cK^k.$
 
Conversely, suppose that $\omega \in  \text{Ker}_{\cA^k}(L^{\gg}_0) \cap (\cA^k)^G$. 
Assume that there is a $\xi \in \gg$ with 
 $X^{\xi}_{(1)} \omega \neq 0$. We will show that this cannot happen if $k + h^{\vee} \notin \mathbb{Q}$.
 
Recall that $\cA^k$ is $\frac{1}{2}\mathbb{N}$-graded by conformal weight (i.e., generalized $L_0^{\cA}$-eigenvalue). Write $\omega$ as a sum of terms of homogeneous weight, and let $m$ be the maximum value which appears. Let $\gg_+\subseteq \hat{\gg}$ be the Lie subalgebra generated by the positive modes $\{X^{\xi}_{(k)}|\ \xi\in\gg,\ k>0\}$. Note that each element of $U(\gg_+)$ lowers the weight by some $k>0$, and the conformal weight grading on $U(\gg_+)$ is the same as the grading by $L^{\gg}_0$-eigenvalue. An element $x\in U(\gg_+)$ of weight $-k$ satisfies $x (\omega) \in \cA_{m-k}$, where $\cA_{m-k}$ is the filtered component of weight at most $m-k$. Also, $x(\omega)$ lies in the generalized eigenspace of $L^{\gg}_0$ of eigenvalue $-k$, and $x(\omega) = 0$ if $k>m$.

It follows that $U(\gg_+)\omega$ is a finite-dimensional vector space graded by conformal weight. In particular, the subspace $M\subseteq U(\gg_+)\omega$ of minimal conformal weight is finite-dimensional. Hence it is a finite-dimensional $\gg$-module. Moreover, $U(\gg_+)$ acts trivially on $M$.  Since $\gg$ is indecomposable and $L^{\gg}$ is the Sugawara vector at level $k$, the eigenvalue of $L^{\gg}_0$ on $M$ is given by
\begin{equation} \label{cas} L^{\gg}_0|_{M} = \frac{\text{Cas}(M)}{k+h^{\vee}}. \end{equation}
In fact, each indecomposable summand of $M$ must have the same $L_0^{\gg}$ eigenvalue and hence the same Casimir eigenvalue. This is a rational number and hence our assumption can only be true for $k+h^\vee \in \mathbb Q$.

Thus for $k+h^\vee \notin \mathbb Q$ we have just shown that if $\omega \in  \text{Ker}_{\cA^k}(L^{\gg}_0)$, then $\omega$ is annihilated by all positive modes, that is elements in $\{X^{\xi}_{(k)}|\ \xi\in\gg,\ k>0\}$. Being in $(\cA^k)^G$ precisely means to be annihilated by all zero-modes, so that any $\omega \in  \text{Ker}_{\cA^k}(L^{\gg}_0) \cap (\cA^k)^G$ is annihilated by all non-negative modes, that is $\omega$ is a vacuum vector for $V^k(\gg)$. These are by definition the elements of $\text{Com}(V^k(\gg), \cA^k)$.
\end{proof}

\begin{corollary}  \label{cor:grchar} \textup{(cf. \cite[Corollary 6.7]{CL1})} $\cC$ is a deformable family over $F_K$, and if $\sqrt{k}\notin K$ and $k+h^{\vee}\notin \mathbb{Q}$, we have $\cC^k = \text{Com}(V^k(\gg), \cA^k)$.
\end{corollary}

\begin{proof} Since $\cC^k = \text{Com}(V^k(\gg), \cA^k) = \text{Ker}_{\cA^k}(L^{\gg}_0) = \cK^k$ for generic values of $k$, we obtain $\cC = \cK$ as vertex algebras over $F_K$. Since $\cK$ is a deformable family over $F_K$, so is $\cC$. Therefore $\cC^k = \cK^k$ for all $k$ with $\sqrt{k}\notin K$. Finally, since $$\cK^k = \text{Ker}_{\cA^k}(L^{\gg}_0) \cap (\cA^k)^G =  \text{Com}(V^k(\gg), \cA^k)$$ for $k+h^{\vee}\notin \mathbb Q$, the claim follows. \end{proof}

So far, we have reproven statements of \cite[Section 6]{CL1} in the setting where $\gg$ is a simple basic classical Lie superalgebra or $\gg = \gg\gl_{r|s}$. The proofs of the following statements are now  the same as the corresponding statements in \cite[Section 6]{CL1}. From now on $\gg$ is an arbitrary element in $\cS$, and $B$ is nondegenerate.

\begin{corollary} \textup{(cf. \cite[Corollary 6.8]{CL1})}  Suppose that $\gg$ is in $\cS$, $B$ is nondegenerate and the restriction of $B$ to each indecomposable direct summand $\gg_i \subseteq \gg$ is  the normalized bilinear form on $\gg_i$. Then $\cC$ is a deformable family over $F_K$ and $\cC^k = \text{Com}(V^k(\gg,B), \cA^k)$ for all $k$ such that $k+h_i^{\vee} \notin \mathbb{Q}$, for all $i$. Here $h_i^{\vee}$ is the dual Coxeter number of the $\gg_i$. \end{corollary}

\begin{remark} \label{rem:failureofspec} \textup{(cf. \cite[Remark 6.9]{CL1})} 
If the restriction of $B$ to $\gg_i$ is a nonzero constant $\lambda_i$ times the normalized bilinear form on $\gg_i$, the above statement must be modified as follows: $\cC$ is a deformable family over $F_K$ and $\cC^k = \text{Com}(V^k(\gg,B), \cA^k)$ for all $k$ such that $k \lambda _i +h_i^{\vee} \notin \mathbb{Q}$, for all $i$. \end{remark}

Finally, we give a precise description of the limit $\cC^{\infty} = \lim_{\kappa \ra \infty} \cC$.

\begin{theorem}\label{thm:orbifoldlimit} \textup{(cf. \cite[Theorem 6.10]{CL1})}  Let $\gg$, $B$, and $\cA$ be as above. Then we have a vertex algebra isomorphism
$$\lim_{\kappa \ra \infty} \cC \cong \tilde{\cA}^G.$$
\end{theorem}

\end{document}